\documentclass[11pt]{article}

\usepackage[margin=1in]{geometry}
\usepackage{amsmath,amssymb,amsfonts,amsthm,mathtools}
\usepackage{bm}
\usepackage{enumitem}
\usepackage{booktabs}
\usepackage{microtype}
\usepackage[backend=biber,style=numeric-comp,sorting=nyt]{biblatex}
\usepackage[colorlinks=true,linkcolor=blue,citecolor=blue,urlcolor=blue]{hyperref}
\hypersetup{
  pdftitle={Large Deviations for Controlled Branching Processes: Random Slopes, Harmonic Moments, and Transfer Principles},
  pdfauthor={Miguel Gonz\'alez; Carmen Minuesa; In\'es del Puerto; Anand N. Vidyashankar}
}

\allowdisplaybreaks
\numberwithin{equation}{section}

\newtheorem{theorem}{Theorem}[section]
\newtheorem{proposition}[theorem]{Proposition}
\newtheorem{lemma}[theorem]{Lemma}
\newtheorem{corollary}[theorem]{Corollary}
\newtheorem{assumption}[theorem]{Assumption}

\theoremstyle{remark}
\newtheorem{remark}[theorem]{Remark}

\newcommand{\N}{\mathbb N}
\newcommand{\Nzero}{\mathbb N_0}
\newcommand{\R}{\mathbb R}
\newcommand{\Pbf}{\mathbf P}
\newcommand{\Ebf}{\mathbf E}

\newcommand{\one}{\mathbf 1}
\newcommand{\e}{\mathrm e}
\newcommand{\dd}{\,\mathrm d}
\newcommand{\cE}{\mathcal E}
\newcommand{\cF}{\mathcal F}
\newcommand{\cD}{\mathcal D}
\newcommand{\cG}{\mathcal G}

\newcommand{\Law}{\mathcal L}
\newcommand{\eps}{\varepsilon}
\DeclareMathOperator{\Var}{Var}

\makeatletter
\newcommand{\blfootnote}[1]{%
  \begingroup
  \renewcommand{\thefootnote}{}%
  \footnote{#1}%
  \addtocounter{footnote}{-1}%
  \endgroup
}
\makeatother

\title{Large Deviations for Controlled Branching Processes:\\
Random Slopes, Harmonic Moments, and Transfer Principles}
\author{%
\begin{tabular}{c}
Miguel Gonz\'alez\textsuperscript{1}\quad
Carmen Minuesa\textsuperscript{1}\quad
In\'es del Puerto\textsuperscript{1}\quad
Anand N. Vidyashankar\textsuperscript{2}
\\[0.65em]
{\small
\textsuperscript{1}Departamento de Matemáticas, Facultad de Ciencias,}
\\[-0.1em]{\small  e Instituto de Computación Científica Avanzada
(ICCAEx), Universidad de Extremadura}
\\[-0.1em]
{\small
Avenida de Elvas s/n, 06006 Badajoz, Spain}
\\[-0.1em]
{\small
\textsuperscript{2}Department of Statistics, George Mason University}
\\[-0.1em]
{\small
4400 University Drive, MS 4A7, Fairfax, VA 22030, USA}
\end{tabular}}
\date{}

\begin{document}
\maketitle
\blfootnote{\textit{E-mail addresses:}
\href{mailto:mvelasco@unex.es}{mvelasco@unex.es},
\href{mailto:cminuesaa@unex.es}{cminuesaa@unex.es},
\href{mailto:idelpuerto@unex.es}{idelpuerto@unex.es}, and
\href{mailto:avidyash@gmu.edu}{avidyash@gmu.edu}.}

\begin{abstract}
Let \(\{X_n:n\geq0\}\) be a controlled branching process with random
control functions \(\{\phi_n:n\geq0\}\), and let
\(N_n:=\phi_n(X_n)\) denote the number of progenitors selected in generation
\(n\).  We investigate large deviations for the one-step ratios
\(X_{n+1}/N_n\), \(N_n/X_n\), and \(X_{n+1}/X_n\), on their natural
positive-denominator events, in supercritical and critical controlled
branching processes.  We consider two control regimes: one with a generation-wide
random slope and one with an asymptotically deterministic first-order slope.

For asymptotically affine controls with a generation-wide random slope, we
use an affine recursion and a random environmental product to prove a sharp
Bahadur--Rao--Petrov upper-tail asymptotic in both supercritical and critical
environments.  A positive Cram\'er tilt identifies its prefactor through a
nonnegative, nontrivial normalized limit.  Conditional on the corresponding
upper-deviation event, all three ratios concentrate simultaneously around
their environmental centers.

When the random control has an individual-sum structure, the controlled
process admits an exact representation as a branching process in a random
environment with immigration.  Motivated by this connection, we establish
sharp upper large-deviation asymptotics for BPREI generation sizes and use
their positive-transform and harmonic-moment profiles in the controlled
process.  In the nondecreasing supercritical case, the harmonic moments
exhibit an exact environmental--boundary--persistence trichotomy.  At
criticality, a finite-horizon-minimum decomposition, integrated
bridge-occupation estimates, and an endpoint-minimum entrance law yield
exact \(n^{-1/2}\)-scale limits for all positive harmonic orders and for
lower-tail kernels bounded by a negative power of the population size, in a
centered finite-variance class.  In particular, this gives exact normalized
limits for the three critical ratio-deviation probabilities in the
corresponding individual-sum subclass.

When the first-order control slope is deterministic, the environmental-product
mechanism disappears.  Matched lower and upper probability-generating-function
envelopes yield distinct supercritical and critical harmonic-moment
trichotomies even for nonexact controls. At criticality, the normalized
harmonic moments of the population and progenitor counts converge to explicit
positive constants for \(0<r\le\sigma\), where \(\sigma\) is the common
envelope index. For \(r>\sigma\), both harmonic moments are of order
\(n^{-\sigma}\). Thus the location of
first-order randomness determines the lower-tail mechanism, while immigration
may change a sharp upper-tail constant, a phase boundary, or the decay
exponent itself.
\end{abstract}


\section{Introduction}\label{sec:introduction}

Controlled branching processes generalize the classical Galton--Watson
process by allowing the number of progenitors in each generation to depend
on the current generation size through a random control. Let
\(\{\xi_{n,j}:n\geq0,\ j\geq1\}\) be independent copies of a nonnegative
integer-valued random variable \(\xi\) with \(m:=\Ebf\xi\in(0,\infty)\).
Let \((\phi_n)_{n\geq0}\) be i.i.d. random control functions from
\(\Nzero\) to \(\Nzero\), and let \(X_0\) be the nonnegative
integer-valued initial population. Assume that \(X_0\), the control
sequence, and the offspring array are mutually independent.
The controlled process is defined by
\begin{equation}\label{eq:intro-controlled-recursion}
        X_{n+1}=\sum_{j=1}^{\phi_n(X_n)}\xi_{n,j},
        \qquad n\geq0,
\end{equation}
with the empty sum interpreted as zero. We write \(N_n:=\phi_n(X_n)\)
for the number of progenitors selected in generation \(n\). We allow
\(\phi_n(0)>0\), so an additive component of the control may let the
process leave zero.

We study large deviations for three one-step ratios:
\[
        \frac{X_{n+1}}{N_n},\qquad
        \frac{N_n}{X_n},\qquad
        \frac{X_{n+1}}{X_n}.
\]
The first is an offspring sample mean when \(N_n>0\); the second measures
the control per current individual when \(X_n>0\). On
\(\{X_n>0,\ N_n>0\}\), the third ratio is their product. If
\(X_n>0\) and \(N_n=0\), then \(X_{n+1}=0\), so the total ratio is
zero and the offspring ratio is undefined.

For a supercritical Galton--Watson process \((Z_n)\) started from one
individual, the population grows on survival at the same exponential rate
as its mean: under the usual moment assumptions,
\(n^{-1}\log Z_n\to\log m\) almost surely on survival, where
\(m=\Ebf[Z_1]>1\). This suggests two questions for the controlled
process. First, how likely is \(X_n\) to exceed \(\e^{\rho n}\) when
\(\rho\) exceeds the mean logarithmic growth rate of the environment?
Second, what features of reproduction and control govern deviations of
the three one-step ratios from their centers, both on this upper-deviation
event and when no such restriction is imposed on \(X_n\)?

In the Galton--Watson process, \(Z_{n+1}/Z_n\) is an average of \(Z_n\)
offspring counts on \(\{Z_n>0\}\). As the work of Athreya and
Vidyashankar, Athreya, and Ney and Vidyashankar
\cite{AthreyaVidyashankar1993,Athreya1994,AthreyaVidyashankar1995,
AthreyaVidyashankar1997,NeyVidyashankar2003,NeyVidyashankar2004} shows,
its deviations reflect a competition: a large generation makes a
substantial fluctuation of the average unlikely, while a small positive
generation permits such a fluctuation but has its own probability cost.
In the controlled process, the current population \(X_n\) and the number
of progenitors \(N_n=\phi_n(X_n)\) over which offspring are averaged can
differ. The control determines the number of progenitors; when its slope
is random and shared by the whole generation, it also shifts the centers
of the control and total ratios, even in a large population.

Controlled branching processes were introduced by Sevast'yanov and Zubkov
\cite{SevastyanovZubkov1974}, and random control functions were studied by
Yanev \cite{Yanev1976}. Gonz\'alez, Molina, del Puerto and their
collaborators developed much of the subsequent theory, including
extinction, geometric growth, \(L^2\) convergence, critical behavior, and
stochastic recursions; see
\cite{GonzalezMolinaDelPuerto2002,GonzalezMolinaDelPuerto2003,
GonzalezMolinaDelPuerto2005L2,GonzalezMolinaDelPuerto2005Critical,
GonzalezMolinaDelPuerto2006Growth,GonzalezMolinaDelPuerto2006Limits}
and the monograph \cite{GonzalezDelPuertoYanev2018}. A recurring quantity
in this theory is the asymptotic mean control slope. If
\(k^{-1}\Ebf[\phi_0(k)]\to\tau\), then the mean size of the next
generation per current individual tends to \(m\tau\): from state \(k\)
the expected number of offspring is \(m\Ebf[\phi_0(k)]\).
This mean alone need not describe growth along a population path.

In the random-slope model the control takes the form
\begin{equation}\label{eq:intro-random-slope-control}
        \phi_n(k)=A_n C_n(k)+B_n,\qquad M_n:=m\lambda A_n,
\end{equation}
where \(\lambda>0\) is the reference slope of the random map \(C_n\),
\(A_n\) is a positive integer multiplier shared by the whole generation,
and \(B_n\) is nonnegative. We may take the triple
\(\cE_n:=(A_n,B_n,C_n)\) as the environmental mark for generation \(n\).
The marks are i.i.d. and independent of \((X_0,(\xi_{i,j})_{i,j})\);
the components of a single mark may be dependent.
Under the ratio moment conditions of
Section~\ref{subsec:main-random-slope},
\(\phi_n(k)/k-\lambda A_n\to0\) in probability as \(k\to\infty\).
Thus the centers of the three ratios are \(m\), \(\lambda A_n\), and
\(M_n\), respectively: averaging over a large generation does not remove
its common multiplier. The product \(\Delta_n:=\prod_{j=0}^{n-1}M_j\)
provides the normalization for the sharp upper-tail result.

The additive control \(B_n\) supplies progenitors whose offspring enter
the next generation as immigration. For a general map \(C_n\), the
contributions of two subpopulations need not split into independent
families. In the individual-sum subclass,
\(C_n(k)=\sum_{i=1}^k L_{n,i}\), where the \(L_{n,i}\) are i.i.d. with
mean \(\lambda\) and are independent of the pairs \((A_n,B_n)\) and the
offspring array. Let \(\ell\) and \(f\) be the p.g.f.s of \(L\) and
\(\xi\), respectively. Conditional on \((A_n,B_n)\), each current
individual has induced offspring p.g.f. \(\ell(f(s)^{A_n})\), and the
immigration count has p.g.f. \(f(s)^{B_n}\). This gives an exact
branching process in a random environment with immigration (BPREI),
with quenched reproduction mean \(M_n\). The offspring and immigration
counts are conditionally independent, but their laws can depend on the
same mark.

Branching processes in random environments (BPRE) were introduced by
Smith and Wilkinson \cite{SmithWilkinson1969}. Athreya and Karlin
established fundamental extinction and limit results
\cite{AthreyaKarlin1971I,AthreyaKarlin1971II}; see Tanny \cite{Tanny1988}
and Kersting and Vatutin \cite{KerstingVatutin2017} for normalized growth
theory. For BPRE without immigration, large deviations of \(\log Z_n\)
were studied by Bansaye and Berestycki
\cite{BansayeBerestycki2009}, B\"oinghoff and Kersting
\cite{BoeinghoffKersting2010}, and Bansaye and B\"oinghoff
\cite{BansayeBoinghoff2011,BansayeBoinghoff2013,BansayeBoinghoff2014}.
Kozlov \cite{Kozlov2006} studied sharp upper-tail asymptotics for
geometric offspring, and Shklyaev \cite{Shklyaev2012} studied their
dependence on the initial population in the linear-fractional case.
Shklyaev studied sharp large deviation asymptotics for random affine recursions
\cite{Shklyaev2016,Shklyaev2021I} and applied this framework to branching
processes in random environments with and without immigration
\cite{Shklyaev2021II}.
Huang and Liu \cite{HuangLiu2012} developed moment and deviation results
for supercritical BPRE. Our starting point for the upper tail is the
precise Bahadur--Rao--Petrov estimate of Buraczewski and Dyszewski
\cite{BuraczewskiDyszewski2022} for BPRE without immigration.
For geometric conditional offspring, Dmitrushchenkov and Shklyaev
\cite{DmitrushchenkovShklyaev2017} studied the sharp upper tail with
immigration in critical and supercritical environments. Wang and Liu
\cite{WangLiu2017} established limit theorems for supercritical BPREI.

For a BPREI \((Z_n)\), the environmental mark in generation \(n\) is
the pair consisting of the offspring law and the immigration law.
Write \(M_n\) for the mean of the offspring law and put
\(\Pi_n:=\prod_{j=0}^{n-1}M_j\) and
\(\Lambda(\vartheta):=\log\Ebf[M^\vartheta]\). Fix \(\alpha>0\) and write
\(\rho=\Lambda'(\alpha)\), \(I(\rho)=\alpha\rho-\Lambda(\alpha)\) and
\(v_\alpha^2=\Lambda''(\alpha)\). Under the hypotheses of
Theorem~\ref{thm:bprei-sharp}, for every fixed \(y\in\R\),
\begin{equation}\label{eq:intro-bprei-sharp}
        \lim_{n\to\infty}
        \sqrt n\,\e^{nI(\rho)}
        \Pbf\bigl(Z_n>\e^{\rho n+y}\bigr)
        =\frac{c_Z(\alpha)\e^{-\alpha y}}
        {\alpha v_\alpha\sqrt{2\pi}}.
\end{equation}
Here \(Z_n/\Pi_n\to V_\infty\) almost surely and in \(L^\alpha\) under
the exponentially tilted environmental law \(\Pbf_\alpha\) (defined in
Section~\ref{subsec:bprei-model}), and
\(c_Z(\alpha)=\Ebf_\alpha[V_\infty^\alpha]\in(0,\infty)\).
In both critical and supercritical environments, immigration can affect
this constant while the exponential rate is determined by the
reproduction means. The key population hypothesis is a bound on the
one-generation innovation \(D(k)\) from a population of size \(k\):
\(\Ebf|D(k)|^\alpha\le C(1+k^\beta)\) for some \(0\le\beta<\alpha\).
Proposition~\ref{prop:bprei-primitive} gives sufficient conditions on
offspring and immigration. For \(\alpha>1\) these use a centered
offspring moment of order \(\alpha\); the comparison with the
no-immigration hypotheses is given in Remark~\ref{rem:bprei-BD}.
To show that the tilted limit is nontrivial, we adapt the large-state
argument of Collamore and Vidyashankar \cite{CollamoreVidyashankar2013}.
The proof of the tail estimate separates the early population
\(Z_{j_n}\), at a generation \(j_n\) of order \(\log n\), from the
subsequent environmental product. The early population contributes a
random logarithmic shift; averaging the shifted local-limit estimate
under the \(\alpha\)-size-biased law of \(Z_{j_n}\) gives the prefactor
in \eqref{eq:intro-bprei-sharp}.

For unconditional ratio deviations, the lower tail of the generation
size matters. In a nondecreasing supercritical BPREI started from
\(i\geq1\), two quantities compete: \(c_r:=\Ebf[M^{-r}]\) and the
one-step persistence probability \(\gamma_i:=\Pbf_i(Z_1=i)
=\Ebf[h_0p_1^i]\), where \(p_1\) is the quenched probability of one child
and \(h_0\) that of zero immigration. Under the hypotheses of
Theorem~\ref{thm:bprei-supercritical-harmonic}, the harmonic moment
\(\Ebf_i[Z_n^{-r}]\) has a positive finite limit after division by
\(c_r^n\), \(n\gamma_i^n\), or \(\gamma_i^n\), according as \(c_r\) is
greater than, equal to, or less than \(\gamma_i\). Huang, Wang and Wang
\cite{HuangWangWang2022} also studied harmonic moments for supercritical
BPREI. Our proof uses a first-exit identity under a negative tilt. For the
no-immigration case see Grama, Liu and Miqueu \cite{GramaLiuMiqueu2017};
for harmonic moments and ratio deviations with immigration in a fixed
environment see Sun and Zhang \cite{SunZhang2017}.

Turning to the critical case, we find a different scale for the harmonic
moments, determined by the oscillating environmental walk
\(S_n:=\sum_{j=0}^{n-1}\log M_j\), with \(S_0=0\). Let \(\tau_n\) be the
first time the walk attains its minimum over \(0,\ldots,n\), and put
\(d_n^-:=\Pbf(\tau_n=n)\). For a fixed initial population \(Z_0=i\), we
decompose \(\Ebf_i[Z_n^{-r};Z_n>0]\) according to the value \(k\) of
\(\tau_n\). The population \(Z_k\) at the minimum supplies the initial
population for the remaining \(n-k\) generations, during which the walk
must stay at or above \(S_k\).

The Spitzer--Doney condition (SD) gives regular variation of \(d_k^-\),
while the entrance law assumption (EL) describes the limiting
distribution of \(Z_k\) conditional on \(\tau_k=k\) (see
Section~\ref{subsec:bprei-critical-harmonic}). After the minimum, the
uniform inverse-moment bound (UIM) is summable in the number of
remaining generations. Together with the localization condition (LOC),
it ensures that long intervals after the minimum contribute negligibly
on the \(d_n^-\) scale. Consequently,
\begin{equation}\label{eq:intro-critical-harmonic}
        \lim_{n\to\infty}
        \frac{\Ebf_i[Z_n^{-r};Z_n>0]}{d_n^-}
        =C_{i,r}\in(0,\infty),
        \qquad r>0.
\end{equation}
Thus the probability that the walk first attains its minimum at the
terminal generation gives the same scale for every positive harmonic
order.

We verify these hypotheses for a centered, nondegenerate finite-variance
environment with bounded normalized quenched offspring variance,
immigration in every generation with bounded conditional mean, and
\(\Ebf[(\log M)_+^q]<\infty\) for some \(q>12\).
Here \(d_n^-\sim\kappa_0 n^{-1/2}\), and the limit of
\(\sqrt n\,\Ebf_i[Z_n^{-r};Z_n>0]\) is independent of the fixed initial
state. Two features of the proof are useful: immigrant families founded
at future minima give the uniform inverse-moment bound, and reversal of
the walk identifies a common population law at an endpoint minimum.
Appendix~\ref{sec:critical-verification} gives the verification using
fluctuation estimates related to \cite{Doney1995,BertoinDoney1994,Feller1971}.
The same decomposition gives exact critical limits for the ratio kernels
dominated by a negative power of the population, although their limiting
constants can vanish.

Returning to the controlled process, the exact recursion
\(X_{n+1}=M_nX_n+D_{n+1}\) holds for general random maps. Under the
upper-deviation hypotheses of Section~\ref{subsec:main-random-slope},
the innovation bound controls the normalized population, while
accessibility gives a strictly positive prefactor. The sharp upper tail
therefore holds without an individual-sum representation. Coupling the
processes from different initial states through the same maps gives a
state-indexed stochastic fixed-point equation; the \(\alpha\)th moments
of its solution form an eigenfunction of the annealed transition
operator (Proposition~\ref{prop:main-random-SFPE}). This identifies the
prefactor as a function of the initial state. Under the additional
one-step ratio moment bounds, all three ratios concentrate around their
environmental centers on the upper-deviation event. For unconditional
deviations, the individual-sum representation supplies the population
estimates needed for the transfer.

Finally, a deterministic first-order slope need not make the control
deterministic or restore the branching property. In this class we bound
the control p.g.f. \(h_k(s):=\Ebf[s^{\phi(k)}]\) by two factorized
envelopes:
\begin{equation}\label{eq:intro-ad-envelope}
        d_L(s)l_L(s)^k
        \leq h_k(s)
        \leq d_U(s)l_U(s)^k,
        \qquad 0\leq s\leq1,
\end{equation}
where \(l_L'(1)=l_U'(1)=\tau\). Evaluating the envelopes at the
offspring p.g.f. gives two auxiliary branching-with-immigration
recursions. Iterating these bounds lets us estimate the population and
progenitor transforms separately, as required for their respective
ratio deviations.

Under the matched-envelope hypotheses, consider the harmonic moment
\(H_n^Y(r):=\Ebf[Y_n^{-r};Y_n>0]\) for \(Y=X\) or \(Y=N\).
When \(m\tau>1\), its order is \(\theta^n\), \(n\theta^n\), or
\((m\tau)^{-rn}\), according as \(\theta(m\tau)^r\) is greater than,
equal to, or less than one; here \(\theta\) is the common small-value
rate of the two envelopes. When \(m\tau=1\), let
\(\sigma=R'(1)/(u''(1)/2)\) denote their common critical index, with
\(u\) and \(R\) the reproduction and immigration p.g.f.s of either auxiliary
recursion. Theorem~\ref{thm:main-ad-critical} gives explicit positive
limits for \(n^r H_n^Y(r)\) when \(0<r<\sigma\) and for
\(n^\sigma H_n^Y(\sigma)/\log n\) at the boundary. For \(r>\sigma\)
it gives \(H_n^Y(r)\asymp n^{-\sigma}\). In contrast to
\eqref{eq:intro-critical-harmonic}, the order here depends on the harmonic
power.

A critical environment obtained by a negative tilt presents a separate
question. If \(\Lambda'(-r_c)=0\), the environmental walk is centered
under \(\Pbf_{-r_c}\), but the tilted harmonic-moment identity still
contains the weight \(\e^{r_c S_n}Z_n^{-r_c}\) on \(\{Z_n>0\}\).
The unweighted critical limit \eqref{eq:intro-critical-harmonic} does not
determine this expectation; see Section~\ref{subsec:bprei-negative-tilt}.

We write \(\cF_n\) for the chronological history through generation \(n\),
before sampling its mark and offspring, and
\(\cG_n:=\sigma(\cE_0,\ldots,\cE_{n-1})\) for the environmental
filtration. Section~\ref{sec:main-results} states the controlled-process
results. Section~\ref{sec:bprei} develops the BPREI upper-deviation and
harmonic-moment theory, and Section~\ref{sec:controlled-process} proves
the random-slope controlled-process results. Section~\ref{sec:ad-controls}
treats asymptotically deterministic controls and Section~\ref{sec:comparisons-examples}
gives examples and compares the two regimes.
Appendix~\ref{sec:critical-verification} proves the finite-variance
critical estimates; Appendix~\ref{app:ad-standard-facts} collects the
fixed-argument estimates for the auxiliary recursions.

\section{Main results for controlled branching processes}
\label{sec:main-results}

In this section we present the main results concerning large deviations for
ratios of generation sizes and progenitor counts.  We first study ratio
concentration jointly with a sharp upper deviation of the generation size
under random asymptotically affine control.  We then remove that
upper-deviation restriction and identify the harmonic-moment and
generating-function profiles governing unconditional ratio deviations.
Finally, we turn to asymptotically deterministic controls, where lower and
upper p.g.f. envelopes replace the environmental product.
Throughout, \(N_n:=\phi_n(X_n)\) denotes the number of progenitors selected in
generation \(n\).  Unless stated otherwise, \(O(\cdot)\) and \(o(\cdot)\)
terms involving the generation index refer to \(n\to\infty\); other limiting
regimes are specified locally.  Sections~\ref{sec:controlled-process} and
\ref{sec:ad-controls} provide the primitive sufficient conditions and proofs,
while Section~\ref{sec:bprei} develops the BPREI inputs used below.

\subsection{Random asymptotically affine controls}
\label{subsec:main-random-slope}

Let \(\{\cE_n:n\geq0\}\) be i.i.d. environmental marks, independent of the
offspring array and of \(X_0\).  Associated with \(\cE_n\) are
\(A_n\in\N\), \(B_n\in\Nzero\), and a random map
\(C_n:\Nzero\to\Nzero\), with \(C_n(0)=0\); these objects may be dependent
through the common mark. We assume throughout that \(\Ebf A_n<\infty\). We consider
\begin{equation}\label{eq:main-random-slope-control}
        \phi_n(k)=A_nC_n(k)+B_n,
        \qquad k\in\Nzero,
\end{equation}
and fix \(\lambda\in(0,\infty)\).  Let \((A,B,C)\) have the law of a
generic generation mark, and let \((\xi_j)_{j\geq1}\) be an independent copy
of any row \((\xi_{n,j})_{j\geq1}\) of the offspring array.  For a
deterministic current state \(k\in\Nzero\), set
\(N(k):=AC(k)+B\),
\(\Psi(k):=\sum_{j=1}^{N(k)}\xi_j\), and
\(M:=m\lambda A\), and define the generic one-generation innovation by
\(D(k):=\Psi(k)-Mk\).  Equivalently,
\(D(k)=\sum_{j=1}^{N(k)}(\xi_j-m)
+mA\{C(k)-\lambda k\}+mB\).

For the realized transition from generation \(n\) to generation \(n+1\), put
\(M_n:=m\lambda A_n\) and
\(D_{n+1}:=X_{n+1}-M_nX_n\), so that
\(X_{n+1}=M_nX_n+D_{n+1}\).  Conditional on \(X_n=k\), the pair
\((M_n,D_{n+1})\) has the same law as \((M,D(k))\).  Thus \(D(k)\) is the
generic-state innovation, whereas \(D_{n+1}\) is the innovation realized in
the transition into generation \(n+1\).  The term ``innovation'' denotes the
remainder after removing the leading environmental linear term \(M_nX_n\); it
need not be centered because it contains the centered offspring fluctuation,
the control remainder, and the additive contribution.

The corresponding environmental normalization is
\(\Delta_n:=\prod_{j=0}^{n-1}M_j\).  Finally, set
\(\Lambda(\vartheta):=\log\Ebf[M^\vartheta]\) and
\(\cD_\Lambda:=\{\vartheta\in\R:\Lambda(\vartheta)<\infty\}\), and call
\(\cD_\Lambda\) the effective domain of \(\Lambda\).

\paragraph{Upper-deviation hypotheses.}
Fix \(\alpha>0\) in \(\operatorname{int}\cD_\Lambda\). Assume that
\(\log M\) is nonlattice, that is,
\(\Pbf(\log M\in b+h\mathbb Z)<1\) for every \(b\in\R\) and \(h>0\).
Assume also that
\begin{equation}\label{eq:main-random-cramer}
        \Ebf|\log M|<\infty,
        \qquad \mu:=\Ebf\log M\geq0,
        \qquad \rho:=\Lambda'(\alpha)>\mu,
        \qquad v_\alpha^2:=\Lambda''(\alpha)\in(0,\infty).
\end{equation}
The terminology in this subsection is environmental: the cases
\(\mu>0\) and \(\mu=0\) are called supercritical and critical,
respectively.  This differs from the classical absorbing-CBP classification
based on an asymptotic annealed one-step mean; see
\cite{GonzalezDelPuertoYanev2018}.  Here \(\phi_n(0)=B_n\), so zero need not
be absorbing, and the first-order coefficient \(M_n=m\lambda A_n\) is random
and may be less than one in individual generations.  The environmental
product \(\Delta_n\), rather than a deterministic geometric sequence, is
therefore the relevant first-order normalization.

We will also need the following moment conditions on the initial
generation size and the one-generation innovations: for some
\(0\leq\beta<\alpha\),
\begin{subequations}\label{eq:main-random-sublinear}
\begin{align}
        \Ebf X_0^\alpha &<\infty,\label{eq:main-random-initial-moment}\\
        \Ebf|D(k)|^\alpha &\leq C(1+k^\beta),
        \qquad k\in\Nzero.\label{eq:main-random-statewise-innovation}
\end{align}
\end{subequations}
\paragraph{Finite-level accessibility.}
For \(K\geq1\), let \(\tau_K:=\inf\{n\geq0:X_n\geq K\}\) and assume
\begin{equation}\label{eq:main-random-accessibility}
        \Pbf(\tau_K<\infty)>0,
        \qquad K\geq1.
\end{equation}
Accessibility is used to transfer positivity of the normalized limit
from large deterministic states to the prescribed initial law. In the
individual-sum subclass, the environmental assumptions and
\(\Pbf(X_0>0)+\Pbf(B>0)>0\) imply accessibility by
Lemma~\ref{lem:bprei-accessibility}. For general random maps we assume
\eqref{eq:main-random-accessibility} explicitly.

Below, the \emph{upper-deviation hypotheses} comprise the environmental
assumptions above (including nonlattice support), the initial and
innovation moment bounds in \eqref{eq:main-random-sublinear}, and
accessibility in \eqref{eq:main-random-accessibility} for the specified
initial law.

For the environmental change of measure, define \(\Pbf_\alpha\) on the
first \(n\) environmental marks by
\begin{equation}\label{eq:main-random-tilt}
        \frac{\dd\Pbf_\alpha}{\dd\Pbf}
        =\Delta_n^\alpha\e^{-n\Lambda(\alpha)}.
\end{equation}
Conditional on the tilted marks, the offspring law remains unchanged and
\(\Ebf_\alpha\log M=\Lambda'(\alpha)=\rho>0\).  The affine recursion
\(X_{n+1}=M_nX_n+D_{n+1}\) and the sublinear innovation estimate identify
\(\Delta_n\) as the natural normalization.

\begin{proposition}[Normalized limit and sharp population estimate]
\label{prop:main-random-normalized-limit}
Under the upper-deviation hypotheses, there is a nonnegative, nontrivial random
variable \(V_\infty\), with \(\Pbf_\alpha(V_\infty>0)>0\), such that
\begin{equation}\label{eq:main-random-V-limit}
        \frac{X_n}{\Delta_n}\longrightarrow V_\infty
        \qquad
        \Pbf_\alpha\text{-almost surely and in }L^\alpha(\Pbf_\alpha).
\end{equation}
Moreover, with
\(c_X(\alpha):=\Ebf_\alpha[V_\infty^\alpha]\in(0,\infty)\) and
\(I(\rho):=\alpha\rho-\Lambda(\alpha)\),
\begin{equation}\label{eq:main-random-moment-constant}
\begin{aligned}
\lim_{n\to\infty}
\e^{-n\Lambda(\alpha)}\Ebf[X_n^\alpha]
&=c_X(\alpha),\\
\lim_{n\to\infty}
\sqrt n\,\e^{nI(\rho)}
\Pbf(X_n>\e^{\rho n+y})
&=
\frac{c_X(\alpha)\e^{-\alpha y}}
     {\alpha v_\alpha\sqrt{2\pi}},
\qquad y\in\R.
\end{aligned}
\end{equation}
\end{proposition}

For each \(x\in\Nzero\), let \(X^{(x)}=(X_n^{(x)})_{n\ge0}\) denote
the population process started from \(X_0^{(x)}=x\). Construct these
processes using the same sequence of random maps \((\Psi_n)_{n\ge0}\):
set \(\Psi_n(k)=\sum_{j=1}^{\phi_n(k)}\xi_{n,j}\) and
\(X_{n+1}^{(x)}=\Psi_n(X_n^{(x)})\) for every \(x\). Thus all starting
states use the same environmental mark and offspring row at generation \(n\).
Write \(\mathcal V_\alpha(x):=\lim_n X_n^{(x)}/\Delta_n\).
Under \eqref{eq:main-random-cramer} and
\eqref{eq:main-random-statewise-innovation}, these limits exist simultaneously
on one \(\Pbf_\alpha\)-full event, without an accessibility assumption at
each \(x\). For \(g:\Nzero\to[0,\infty)\), let
\((Pg)(x):=\Ebf[g(\Psi(x))]\) denote the original annealed transition
operator.

\begin{proposition}[State-indexed SFPE and moment eigenfunction]
\label{prop:main-random-SFPE}
Fix \(\alpha>0\) in \(\operatorname{int}\cD_\Lambda\) and assume that
\eqref{eq:main-random-cramer} and \eqref{eq:main-random-statewise-innovation}
hold. Let \((M,\Psi)\) have its first-generation law under
\(\Pbf_\alpha\), and let \((\mathcal V_\alpha'(x):x\in\Nzero)\)
be an independent copy of the limiting random function, independent of
\((M,\Psi)\). Then
\begin{equation}\label{eq:main-random-joint-SFPE}
        (\mathcal V_\alpha(x):x\in\Nzero)
        \stackrel d=
        \bigl(M^{-1}\mathcal V_\alpha'(\Psi(x)):x\in\Nzero\bigr).
\end{equation}
If \(c_\alpha(x):=\Ebf_\alpha[\mathcal V_\alpha(x)^\alpha]\), then
\begin{equation}\label{eq:main-random-eigenfunction}
        Pc_\alpha=\e^{\Lambda(\alpha)}c_\alpha.
\end{equation}
The function \(c_\alpha\) is nonnegative.  If there exist
\(n\geq0\) and \(y\in\Nzero\) such that \(\Pbf_x(X_n=y)>0\) and
\(c_\alpha(y)>0\), then \(c_\alpha(x)>0\).  If
\(X_0\sim\nu\) is independent of the future marks and offspring array and
\(\sum_x\nu(x)x^\alpha<\infty\), then the corresponding moment constant is
\(c_\nu(\alpha)=\sum_x\nu(x)c_\alpha(x)\).
\end{proposition}

We next impose common polynomial bounds on the three one-generation ratio
kernels.  On \(\{X_n>\e^{\rho n+y}\}\), a kernel of order \(k^{-q}\)
produces an exponentially small conditional failure probability.

Fix \(p\geq2\) and put \(q=p/2\). We say that
\((\mathrm{RM}_p)\) holds if
\begin{equation}\label{eq:main-random-ratio-moments}
\begin{aligned}
  \Ebf|\xi-m|^p &<\infty,\\
  \Ebf[N(k)^q]+\Ebf|N(k)-\lambda Ak|^p
      &\le C(1+k^q),\qquad k\ge1.
\end{aligned}
\end{equation}
The bound on small progenitor counts follows from
\eqref{eq:main-random-ratio-moments}. Indeed, set
\(E_k=\{N(k)<\lambda k/2\}\). Since \(A\ge1\), on \(E_k\) we have
\(|N(k)-\lambda Ak|>\lambda k/2\). Markov's inequality gives
\[
  \Pbf(E_k)\le
  \left(\frac{2}{\lambda k}\right)^p
  \Ebf|N(k)-\lambda Ak|^p\le C k^{-q},
\]
where \(k\ge1\) and \(p=2q\). On \(E_k\cap\{N(k)>0\}\),
\(N(k)^{-q}\le1\) because \(N(k)\) is integer-valued;
on \(E_k^c\), \(N(k)^{-q}\le(2/(\lambda k))^q\). Hence
\begin{equation}\label{eq:main-random-inverse-count}
  \Pbf(N(k)=0)+\Ebf[N(k)^{-q};N(k)>0]
  \le \Pbf(E_k)+\left(\frac{2}{\lambda k}\right)^q
  \le C k^{-q}.
\end{equation}
Section~\ref{sec:controlled-process} uses this estimate to bound the
one-step offspring-ratio probability.

For every \(\eps>0\), let \(\mathcal B_{n,1}(\eps)\) be the event
\(\{X_n>0,\ |N_n/X_n-\lambda A_n|>\eps\}\), let
\(\mathcal B_{n,2}(\eps)\) be
\(\{X_n>0,\ |X_{n+1}/X_n-m\lambda A_n|>\eps\}\), and let
\(\mathcal B_{n,3}(\eps)\) be the event that \(X_n>0\) and either
\(N_n=0\), or \(N_n>0\) and
\(|X_{n+1}/N_n-m|>\eps\).  Write
\(\mathcal B_n(\eps):=\bigcup_{j=1}^3\mathcal B_{n,j}(\eps)\).
All three events require \(X_n>0\). If \(N_n=0\),
\(X_{n+1}/N_n\) is undefined; the event \(\mathcal B_{n,3}(\eps)\)
includes this case explicitly. The bound on these events conditional
on \(X_n>\e^{\rho n+y}\) follows from \((\mathrm{RM}_p)\).
The upper-deviation hypotheses are used for the sharp asymptotic of
the population tail, with \(\rho=\Lambda'(\alpha)\).

\begin{theorem}[Sharp upper deviations with environmentally centered ratios]
\label{thm:main-random-sharp-ratio}
Assume the upper-deviation hypotheses and \((\mathrm{RM}_p)\).  For every
fixed \(y\in\R\) and every sequence \(0<\eps_n\leq1\),
\begin{equation}\label{eq:main-random-sharp-ratio-error}
\Pbf\bigl(
X_n>\e^{\rho n+y},\mathcal B_n(\eps_n)
\bigr)
\leq C_y\eps_n^{-p}\e^{-q\rho n}
\Pbf\bigl(X_n>\e^{\rho n+y}\bigr).
\end{equation}
Equivalently, whenever the conditioning event has positive probability,
\[
\Pbf\!\left(
\mathcal B_n(\eps_n)\,\middle|\,X_n>\e^{\rho n+y}
\right)
\leq C_y\eps_n^{-p}\e^{-q\rho n}.
\]
If
\begin{equation}\label{eq:main-random-sharp-separation}
        \eps_n^{-p}\e^{-q\rho n}\longrightarrow0,
\end{equation}
then
\begin{equation}\label{eq:main-random-sharp-joint}
        \lim_{n\to\infty}
        \sqrt n\,\e^{nI(\rho)}
        \Pbf\bigl(
        X_n>\e^{\rho n+y},
        \mathcal B_n(\eps_n)^c
        \bigr)
        =
        \frac{c_X(\alpha)\e^{-\alpha y}}
             {\alpha v_\alpha\sqrt{2\pi}}.
\end{equation}
The theorem holds when \(\mu>0\) or \(\mu=0\). Since \(q=p/2\),
\eqref{eq:main-random-sharp-separation} is equivalent to
\(\eps_n\e^{\rho n/2}\to\infty\). The sharp joint limit also
holds for every fixed \(\eps>0\): if \(\eps>1\), use
\(\mathcal B_n(\eps)\subseteq\mathcal B_n(1)\) to bound the
omitted probability.
\end{theorem}

\subsection{Unconditional ratio deviations in the random-slope class}
\label{subsec:main-random-transfer}

We retain the three events of Theorem~\ref{thm:main-random-sharp-ratio}
but no longer intersect them with \(\{X_n>\e^{\rho n+y}\}\). Each
requires \(X_n>0\). The third event includes \(N_n=0\); when \(N_n>0\)
it tests the deviation of \(X_{n+1}/N_n\) from \(m\).
Small positive populations can now contribute to the deviation
probabilities. Averaging polynomial one-step bounds gives upper bounds
in terms of harmonic moments; averaging geometric bounds gives upper
bounds in terms of generating functions restricted to positive counts.

For \(s\in(0,1)\), write
\(F_n^+(s):=\Ebf[s^{X_n};X_n>0]\) and
\(H_n^+(s):=\Ebf[s^{N_n};N_n>0]\).  For fixed \(\eps>0\), set
\(\mathcal R_{n,j}(\eps):=\Pbf(\mathcal B_{n,j}(\eps))\),
\(1\leq j\leq3\), and
\(\mathcal R_n(\eps):=\sum_{j=1}^3\mathcal R_{n,j}(\eps)\).
Thus \(\mathcal R_n(\eps)\) is the sum of the unconditional ratio-failure
probabilities; unlike Theorem~\ref{thm:main-random-sharp-ratio}, no
upper-deviation event for \(X_n\) is imposed.

\begin{proposition}[One-step transfer through random denominators]
\label{prop:main-random-one-step-transfer}
Under \((\mathrm{RM}_p)\), for every fixed \(\eps>0\),
\begin{equation}\label{eq:main-random-polynomial-transfer}
        \mathcal R_n(\eps)
        \leq C_\eps\Ebf[X_n^{-q};X_n>0].
\end{equation}
More generally, for \(0<\eps_n\leq1\) and \(K_n\geq1\),
\begin{equation}\label{eq:main-random-cutoff-transfer}
        \mathcal R_n(\eps_n)
        \leq C\Pbf(0<X_n\leq K_n)
             +C\eps_n^{-p}K_n^{-q}.
\end{equation}
For the geometric alternative, fix \(\eps>0\).  Suppose there are
\(C_\eps<\infty\) and four numbers in \((0,1)\), denoted
\(s_{\phi,\eps}\), \(s_{D,\eps}\), \(s_{\xi,\eps}\), and
\(s_{0,\eps}\), such that, for every \(k\geq1\),
\[
\begin{aligned}
\Pbf\left(\left|\frac{N(k)}{k}-\lambda A\right|>\eps\right)
&\leq C_\eps s_{\phi,\eps}^k,
&\Pbf\bigl(|D(k)|>\eps k\bigr)
&\leq C_\eps s_{D,\eps}^k,\\
\Pbf\left(\left|\frac1k\sum_{j=1}^k\xi_j-m\right|>\eps\right)
&\leq C_\eps s_{\xi,\eps}^k,
&\Pbf\{N(k)=0\}
&\leq C_\eps s_{0,\eps}^k.
\end{aligned}
\]
Then
\begin{equation}\label{eq:main-random-exponential-transfer}
\begin{aligned}
\Pbf(\mathcal B_{n,1}(\eps))
&\leq C_\eps F_n^+(s_{\phi,\eps}),\\
\Pbf(\mathcal B_{n,2}(\eps))
&\leq C_\eps F_n^+(s_{D,\eps}),\\
\Pbf(\mathcal B_{n,3}(\eps))
&\leq C_\eps F_n^+(s_{0,\eps})
     +C_\eps H_n^+(s_{\xi,\eps}).
\end{aligned}
\end{equation}
\end{proposition}

In the third inequality of \eqref{eq:main-random-exponential-transfer},
the first term bounds the contribution from \(\{X_n>0,N_n=0\}\).
On \(\{X_n>0,N_n>0\}\), averaging the geometric offspring bound uses
the transform
\[
   H^+_{n,>0}(s):=\Ebf[s^{N_n};X_n>0,N_n>0].
\]
Since \(H^+_{n,>0}(s)\le H_n^+(s)\), the displayed upper bound follows.
A geometric lower bound on the offspring kernel transfers through
\(H^+_{n,>0}\) itself; \(H_n^+\) also counts restarts from zero. See
Proposition~\ref{prop:controlled-exponential-transfer} and the discussion
following it.

The preceding proposition applies to the random-map class, but it does
not by itself identify the lower-population profile.  The individual-sum
structure now becomes essential.  Suppose
\begin{equation}\label{eq:main-individual-sum-control}
        C_n(k)=\sum_{j=1}^{k}L_{n,j},
        \qquad \Ebf L=\lambda,
\end{equation}
where the variables \(\{L_{n,j}:j\geq1\}\) are i.i.d. and independent of
\((A_n,B_n)\) and of the offspring array.  The controlled process is then an
exact BPREI with quenched reproduction mean \(M=m\lambda A\).  Section~\ref{sec:bprei}
provides its harmonic-moment and generating-function profiles, while
Section~\ref{sec:controlled-process} verifies their primitive hypotheses for
the induced laws. The exact BPREI representation now allows us to identify the
lower-population profile appearing in
Proposition~\ref{prop:main-random-one-step-transfer}.

We begin with the nondecreasing supercritical regime.  In this case the
harmonic-moment scale is determined by the competition between an
unfavorable sequence of environmental means and repeated persistence
of the generation size at its initial value. Assume the induced BPREI satisfies \((\mathrm H)\) of
Assumption~\ref{ass:bprei-harmonic-supercritical}, and set
\(a_1:=\Pbf(L=1)\Pbf(\xi=1)\). In the individual-sum class under
\((\mathrm{RM}_p)\), conditions (H3)--(H5) follow from the moment
bounds. For (H1)--(H2) it suffices that \(\xi,L\geq1\) almost surely,
\(\Pbf(A=1,B=0)>0\), and \(0<a_1<1\). These conditions imply \(m\lambda>1\). For \(i\geq1\),
\begin{equation}\label{eq:main-induced-cr-gamma}
        c_q:=\Ebf[(m\lambda A)^{-q}],
        \qquad
        \gamma_i:=\Pbf(A=1,B=0)a_1^i,
\end{equation}
where \(c_q\) is the environmental harmonic cost and
\(\gamma_i=\Pbf_i(X_1=i)\) is the probability of remaining at \(i\)
for one generation. When \(i=0\), the persistence probability is
instead \(\gamma_0=\Pbf(B=0)\).

At criticality, the supercritical exponential trichotomy is replaced by a
fluctuation-theoretic normalization.  The abstract critical hypotheses are
\((\mathrm C):=(\mathrm{SD})+(\mathrm{EL})+(\mathrm{UIM})+(\mathrm{LOC})\):
regular variation of the descending-ladder probability, convergence of the
endpoint-minimum entrance laws, a summable uniform post-minimum inverse-moment
bound, and convolution-tail localization.  Under these inputs,
Theorem~\ref{thm:bprei-critical-harmonic} gives
\(\Ebf_i[Z_n^{-q};Z_n>0]/d_n^-\to C_{i,q}\in(0,\infty)\), where
\(d_n^-:=\Pbf(\max_{1\leq j\leq n}S_j<0)\) and
\(S_j:=\sum_{r=0}^{j-1}\log M_r\).  Theorem
\ref{thm:bprei-critical-fv-class} verifies all four inputs for the centered
finite-variance class described there; outside that class they remain
model-specific assumptions.  In the individual-sum controlled-process
verification of Section~\ref{sec:controlled-process}, the transparent
primitive persistent-immigration subclass uses
\(\Pbf(\xi=0)=0\) and \(1\leq B\leq\overline b\) almost surely for some
\(\overline b<\infty\).  This restriction concerns that primitive
verification, not the abstract BPREI theorem.  This is the critical
population profile transferred to the ratio probabilities below.

\begin{theorem}[Supercritical and critical ratio profiles in the individual-sum class]
\label{thm:main-random-harmonic-transfer}
Fix \(\eps>0\) and assume that \eqref{eq:main-individual-sum-control}
and \((\mathrm{RM}_p)\) hold, with \(q=p/2\).

If the induced BPREI is nondecreasing and satisfies \((\mathrm H)\), let
\(X_0=i\ge1\) and put
\(g_\eps(k):=\sum_{j=1}^3\Pbf_k(\mathcal B_{0,j}(\eps))\) for \(k\ge1\).
Then, for some \(C_{i,\eps}<\infty\),
\begin{equation}\label{eq:main-induced-super-ratio-profile}
\begin{cases}
 \mathcal R_n(\eps)=o(c_q^n), & c_q>\gamma_i,\\[1mm]
 \mathcal R_n(\eps)=o(n\gamma_i^n), & c_q=\gamma_i,\\[1mm]
 g_\eps(i)\gamma_i^n\le\mathcal R_n(\eps)\le C_{i,\eps}\gamma_i^n,
     & c_q<\gamma_i.
\end{cases}
\end{equation}
In the last case, \(\mathcal R_n(\eps)\asymp\gamma_i^n\) whenever
\(g_\eps(i)>0\).

If instead the induced BPREI is critical, fix an initial state
\(X_0=i\) for which \((\mathrm C)\) holds. For each \(j=1,2,3\)
there is a finite nonnegative constant \(K_{i,j}(\eps)\) such that
\begin{equation}\label{eq:main-induced-critical-ratio-profile}
        \frac{\mathcal R_{n,j}(\eps)}{d_n^-}
        \longrightarrow K_{i,j}(\eps).
\end{equation}
\end{theorem}

\begin{remark}[Critical ratio constants and further transfers]
\label{rem:main-random-transfer-refinements}
Each \(K_{i,j}(\eps)\) is positive precisely when its one-step kernel
is positive at a state of positive post-minimum probability; see
\eqref{eq:controlled-critical-transfer-limit}. The limit for
\(\mathcal R_n(\eps)/d_n^-\) is \(\sum_{j=1}^3K_{i,j}(\eps)\).
Under centered finite variance,
\(\sqrt n\,d_n^-\to\kappa_0\), and hence
\(\sqrt n\,\mathcal R_{n,j}(\eps)\to
\kappa_0K_{i,j}(\eps)\) in the verified class.  If the post-minimum profile is
nonsummable, a different normalization is available only under the
convolution-replacement hypothesis of
Proposition~\ref{prop:bprei-critical-divergent-profile}.
Under geometric one-step bounds, the corresponding generating-function
profiles give the analogous exponential transfer.
\end{remark}

The two clauses are disjoint at the primitive level.  The nondecreasing case
requires \(\xi,L\geq1\) almost surely.  In the persistent-immigration
critical subclass verified in Section~\ref{sec:controlled-process}, one uses
\(\Pbf(\xi=0)=0\) and \(1\leq B\leq\overline b\) almost surely, so positive probability
of zero induced offspring must come from \(\Pbf(L=0)>0\).  More general
critical BPREI covered by the abstract theorem may instead allow offspring
zeros directly.

The preceding harmonic-moment rates are derived under the original
environmental law.  Lower population deviations introduce a second critical
boundary through a negative environmental tilt.  The next proposition records
the exact change-of-measure identity.

\begin{proposition}[Negative-tilt identity]
\label{prop:main-negative-tilt-boundary}
Write \(Z_n\) for the induced BPREI population, let
\(\Pi_n:=\prod_{j=0}^{n-1}M_j\), and set \(V_n:=Z_n/\Pi_n\).  For every
\(r>0\) such that \(-r\in\operatorname{int}\cD_\Lambda\),
\begin{equation}\label{eq:main-negative-tilt-identity}
        \Ebf[Z_n^{-r};Z_n>0]
        =\e^{n\Lambda(-r)}
         \Ebf_{-r}[V_n^{-r};Z_n>0].
\end{equation}
\end{proposition}

Suppose \(r_c>0\) satisfies \(-r_c\in\operatorname{int}\cD_\Lambda\)
and \(\Lambda'(-r_c)=0\). Then the environmental walk \(S_n=\log\Pi_n\)
is centered under \(\Pbf_{-r_c}\). On \(\{Z_n>0\}\), the integrand
in \eqref{eq:main-negative-tilt-identity} is
\(V_n^{-r_c}=\e^{r_cS_n}Z_n^{-r_c}\).
The unweighted critical theorem does not determine this expectation.
Its sharp asymptotics are outside the scope of this paper; see
Section~\ref{subsec:bprei-negative-tilt}.

\subsection{Asymptotically deterministic controls}
\label{subsec:main-ad}

We finally turn to controls with deterministic first-order slope.  The
environmental product no longer carries a nondegenerate fluctuation, and the
original controlled process need not satisfy the branching property.  A
normalized lower and upper p.g.f. sandwich supplies two exact auxiliary
recursions.  Their generating-function and harmonic-moment profiles are then
transferred through the same random-denominator principle used above.

\paragraph{Normalized envelope condition \textup{(AD0)}.}
Let \(h_k(s):=\Ebf[s^{\phi(k)}]\). Assume there are p.g.f.s
\(l_L,l_U,d_L,d_U\) such that, for every \(k\in\Nzero\),
\begin{equation}\label{eq:main-ad-envelope}
        d_L(s)l_L(s)^k
        \leq h_k(s)
        \leq d_U(s)l_U(s)^k,
        \qquad 0\leq s\leq1,
\end{equation}
with \(d_e(1)=l_e(1)=1\) for \(e\in\{L,U\}\) and
\(l_L'(1)=l_U'(1)=:\tau\in(0,\infty)\).  Whenever harmonic moments of
\(N_n\) are asserted, also assume \(l_L''(1)+l_U''(1)<\infty\).  Set
\(u_e:=l_e\circ f\), \(R_e:=d_e\circ f\), and \(a:=m\tau\).  With
\(F_0^e=F_0\), define the auxiliary transforms by
\begin{equation}\label{eq:main-ad-exact-recursions}
        F_{n+1}^e(s)=R_e(s)F_n^e(u_e(s)),
        \qquad
        H_n^e(z)=d_e(z)F_n^e(l_e(z)),
        \qquad e\in\{L,U\}.
\end{equation}
Section~\ref{sec:ad-controls} proves that these transforms provide the needed
lower and upper bounds for the population and progenitor transforms.

We call the envelopes \emph{coincident} when \(d_L=d_U\) and \(l_L=l_U\);
then the original process itself satisfies one exact transformed recursion.
When the envelopes are distinct, \textup{(AD-S)} matches their
supercritical small-value rates, while \textup{(AD-C)} matches their
critical indices. Neither requires the two p.g.f. envelopes to coincide.

For fixed \(\eps>0\), let \(J_{n,1}(\eps)\) be the probability of
\(\{N_n>0,\ |X_{n+1}/N_n-m|>\eps\}\), and let \(J_{n,2}(\eps)\) and
\(J_{n,3}(\eps)\) be the probabilities of
\(\{X_n>0,\ |N_n/X_n-\tau|>\eps\}\) and
\(\{X_n>0,\ |X_{n+1}/X_n-a|>\eps\}\), respectively.
The offspring, control, and total ratios correspond, respectively, to
indices \(3,1,2\) in the random-slope notation \(\mathcal B_{n,j}\).
Here the offspring ratio is tested whenever \(N_n>0\), including
\(X_n=0<N_n\). Put
\(\mathcal R_n^{\mathrm{ad}}(\eps)=\sum_{j=1}^3J_{n,j}(\eps)\).
For polynomial transfer, assume that, for some \(p\geq2\) and \(q:=p/2\),
\begin{equation}\label{eq:main-ad-poly-conditions}
        \Ebf|\xi-m|^p<\infty,
        \qquad
        \sup_{k\geq1}k^{-p/2}
        \Ebf|\phi(k)-\tau k|^p<\infty.
\end{equation}

\paragraph{Exponential one-step condition \textup{(AD-E)}.}
Assume that the offspring law has a moment generating function in a
neighborhood of zero.  For the control, assume that the scaled log-m.g.f.
\[
        \Lambda_\phi(\eta)
        :=\lim_{k\to\infty}\frac1k
          \log\Ebf[\e^{\eta\phi(k)}]
\]
exists and is finite on an open interval containing zero, and is
differentiable at zero with \(\Lambda_\phi'(0)=\tau\).
The polynomial condition transfers harmonic moments, whereas
\textup{(AD-E)} transfers generating-function profiles.

\paragraph{Matched supercritical envelope condition \textup{(AD-S)}.}
Assume \(a>1\).  For each \(e\in\{L,U\}\), let \(q_e\in[0,1)\) be the
attracting fixed point of \(u_e\), with \(0<u_e'(q_e)<1\), and let
\(\eta_e\) have p.g.f. \(R_e\).  Define the auxiliary small-value rate
\[
\theta_e:=
\begin{cases}
R_e(0)u_e'(0), & q_e=0,\\
R_e(q_e), & q_e>0\text{ and }R_e(q_e)<1,\\
u_e'(q_e), & q_e>0\text{ and }R_e\equiv1.
\end{cases}
\]
Assume \(\theta_L=\theta_U=:\theta\in(0,1)\).  If
\(u_e(s)=\sum_{k\geq0}u_{e,k}s^k\), require
\(\sum_{k\geq1}k\log^+k\,u_{e,k}<\infty\) and
\(\Ebf\log(1+\eta_e)<\infty\) for both envelopes.  Impose the corresponding
initial-support condition for each branch: when \(q_e=0\),
\(F_0(s)=\mu_1s+O(s^2)\) as \(s\downarrow0\) for some \(\mu_1>0\); when
\(q_e>0\) and \(R_e\equiv1\), require \(F_0'(q_e)>0\).  If one envelope is
in the positive-fixed-point no-immigration branch, assume both are in that
branch and \(q_L=q_U\).

With the common auxiliary rate \(\theta\), define
\[
        b_n^{\mathrm{sc}}(r):=
        \begin{cases}
        \theta^n, & \theta a^r>1,\\[1mm]
        n\theta^n, & \theta a^r=1,\\[1mm]
        a^{-rn}, & \theta a^r<1.
        \end{cases}
\]

\begin{theorem}[Supercritical AD harmonic-moment and ratio profiles]
\label{thm:main-ad-supercritical}
Under conditions \textup{(AD0)} and \textup{(AD-S)}, for
\(Y_n=X_n\) and for \(Y_n=N_n\), and every \(r>0\),
\begin{equation}\label{eq:main-ad-super-harmonic}
        0<\liminf_{n\to\infty}
        \frac{\Ebf[Y_n^{-r};Y_n>0]}{b_n^{\mathrm{sc}}(r)}
        \leq
        \limsup_{n\to\infty}
        \frac{\Ebf[Y_n^{-r};Y_n>0]}{b_n^{\mathrm{sc}}(r)}<\infty.
\end{equation}
If \eqref{eq:main-ad-poly-conditions} holds, then, for \(j=1,2,3\),
\begin{equation}\label{eq:main-ad-super-ratio-profile}
        \limsup_{n\to\infty}
        \frac{J_{n,j}(\eps)}{b_n^{\mathrm{sc}}(q)}<\infty.
\end{equation}
Under \textup{(AD-E)},
\(\limsup_{n\to\infty}\theta^{-n}J_{n,j}(\eps)<\infty\) for
\(j=1,2,3\).
\end{theorem}

Corollary~\ref{cor:main-ad-ratio-lower} gives sufficient conditions on
the one-step kernels for matching lower bounds on \(J_{n,j}(\eps)\) in
the supercritical and critical regimes.

\paragraph{Matched critical envelope condition \textup{(AD-C)}.}
Assume \(a=1\) and, for each \(e\in\{L,U\}\), that there are
\(\gamma_e,\beta_e\in(0,\infty)\) and \(\eta_e\in(0,1]\) such that, as
\(x\downarrow0\),
\[
        u_e(1-x)=1-x+\gamma_ex^2+O(x^{2+\eta_e}),
        \qquad
        \log R_e(1-x)=-\beta_ex+O(x^2),
\]
with \(u_e(1)=u_e'(1)=R_e(1)=1\). Assume
\(\beta_L/\gamma_L=\beta_U/\gamma_U=:\sigma>0\).
By \textup{(AD0)} at \(k=0\) and by taking \(k\)th roots and passing
to the limit for \(0<s<1\), we have \(d_L\le d_U\) and \(l_L\le l_U\).
The expansions then give \(\beta_L\ge\beta_U\) and
\(\gamma_L\le\gamma_U\). Hence \(\beta_L=\beta_U=:\beta\) and
\(\gamma_L=\gamma_U=:\gamma\).
The two bounds on \(h_k\) now have the same derivative at one, so
\(\Ebf\phi(k)=\tau k+\beta/m\) for every \(k\ge0\). Define
\[
        b_n^{\mathrm{cr}}(r):=
        \begin{cases}
        n^{-\sigma}, & r>\sigma,\\[1mm]
        n^{-\sigma}\log n, & r=\sigma,\\[1mm]
        n^{-r}, & 0<r<\sigma.
        \end{cases}
\]

\begin{theorem}[Critical AD harmonic-moment and ratio profiles]
\label{thm:main-ad-critical}
Under conditions \textup{(AD0)} and \textup{(AD-C)}, let
\(Y=X\) or \(N\) and write \(H_n^Y(r):=\Ebf[Y_n^{-r};Y_n>0]\).
Set \(\omega_X:=\gamma\) and \(\omega_N:=\gamma\tau\),
where \(\gamma=\gamma_L=\gamma_U\). For every \(r>0\),
\begin{equation}\label{eq:main-ad-critical-harmonic}
  0<\liminf_{n\to\infty}\frac{H_n^Y(r)}{b_n^{\mathrm{cr}}(r)}
  \le \limsup_{n\to\infty}\frac{H_n^Y(r)}{b_n^{\mathrm{cr}}(r)}
  <\infty.
\end{equation}
When \(0<r\le\sigma\) the normalized moments in
\eqref{eq:main-ad-critical-harmonic} converge. For \(0<r<\sigma\) their
limit is \(\omega_Y^{-r}\Gamma(\sigma-r)/\Gamma(\sigma)\); at \(r=\sigma\)
it is \(\omega_Y^{-\sigma}/\Gamma(\sigma)\).
If \eqref{eq:main-ad-poly-conditions} holds, then, for \(j=1,2,3\),
\begin{equation}\label{eq:main-ad-critical-ratio-profile}
        \limsup_{n\to\infty}
        \frac{J_{n,j}(\eps)}{b_n^{\mathrm{cr}}(q)}<\infty.
\end{equation}
Under \textup{(AD-E)},
\(\limsup_{n\to\infty}n^\sigma J_{n,j}(\eps)<\infty\) for
\(j=1,2,3\).
\end{theorem}

For \(Y_n=X_n\) or \(Y_n=N_n\), the upper envelope gives
\(\Pbf(Y_n=0)=O(n^{-\sigma})\). Hence \(\Pbf(Y_n>0)\to1\), and
\(\Ebf[Y_n^{-r}\mid Y_n>0]\sim\Ebf[Y_n^{-r};Y_n>0]\) for each
\(r>0\).

\begin{corollary}[Lower bounds for the ratio-deviation probabilities]
\label{cor:main-ad-ratio-lower}
Under the hypotheses of Theorem~\ref{thm:main-ad-supercritical}
or Theorem~\ref{thm:main-ad-critical}, fix \(\eps>0\) and \(j\in\{1,2,3\}\). Let
\(g_j(k)\) denote the corresponding one-step kernel
\(\psi_\xi(k;\eps),\psi_{\mathrm{ad}}(k;\eps)\), or
\(\psi_{\mathrm{comb}}(k;\eps)\), as defined in
Section~\ref{subsec:ad-ratio-transfer}. Its argument is the progenitor
count for \(j=1\) and the population size for \(j=2,3\).
If \(g_j(k)\ge c k^{-q}\) for every \(k\ge1\), with \(c>0\), then
\(J_{n,j}(\eps)/b_n(q)\) has positive liminf, where
\(b_n=b_n^{\mathrm{sc}}\) or \(b_n^{\mathrm{cr}}\) in the respective
regime. If instead \(g_j(k)\ge c s^k\) for every \(k\ge1\), for some
\(c>0\) and \(0<s<1\), the same assertion holds with \(b_n(q)\)
replaced by \(\theta^n\) or \(n^{-\sigma}\), respectively.
\end{corollary}

Section~\ref{sec:ad-controls} uses the auxiliary recursions in
\eqref{eq:main-ad-exact-recursions} to prove
Theorems~\ref{thm:main-ad-supercritical} and~\ref{thm:main-ad-critical}.  Section~\ref{sec:comparisons-examples} compares
the resulting rates and identifies the deterministic-multiplier boundary.

\section{Upper deviations and harmonic moments for BPREI}
\label{sec:bprei}

In this section we derive precise upper large-deviation estimates for the
generation sizes of a BPREI, together with the Laplace-transform and
harmonic-moment estimates that govern the ratio deviations described in
Sections~\ref{sec:introduction} and~\ref{sec:main-results}.  Sharp upper
deviations concern unusually large generation sizes.  For unconditional ratio
deviations, two mechanisms interact: an atypical one-generation reproduction
fluctuation and a positive denominator that is too small for averaging to
suppress that fluctuation.  Conditional one-step probabilities decay with the
denominator, so after averaging the small-denominator profile determines the
leading rate.  Polynomial one-step estimates are therefore averaged through
harmonic moments, whereas exponential one-step estimates are averaged through
positive Laplace transforms; see
\cite{AthreyaVidyashankar1993,Athreya1994,
AthreyaVidyashankar1995,AthreyaVidyashankar1997,
NeyVidyashankar2003,NeyVidyashankar2004}.

We state the results for a branching process in an i.i.d.\ random environment
with immigration.  Reproduction and immigration may depend through the same
generation-level environmental mark, and hence need not be independent under
the annealed law.
Unless otherwise stated, every \(O(\cdot)\) or \(o(\cdot)\) term involving the generation index \(n\) is understood as \(n\to\infty\); when another index or a real parameter tends to a limit, we state that regime locally.

\subsection{The BPREI model and common analytic tools}
\label{subsec:bprei-model}

Let \(\cE_n=(Q_n,G_n)\), \(n\geq0\), be i.i.d.\ environmental marks taking values in a standard Borel space.  Here
\(Q_n\) is the reproduction law in generation \(n\), and \(G_n\) is the law
of the immigration entering at time \(n+1\).  Conditional on
\(\cE=(\cE_n)_{n\geq0}\),
\begin{equation}\label{eq:bprei-recursion}
        Z_{n+1}=\sum_{i=1}^{Z_n}\xi_{n,i}+\eta_{n+1},
        \qquad n\geq0.
\end{equation}
Conditional on the environment \(\cE\), the variables \(\xi_{n,j}\) (\(n\ge0,\ j\ge1\)) and \(\eta_{n+1}\) (\(n\ge0\)) are mutually independent, with respective laws \(Q_n\) and \(G_n\). The initial population \(Z_0\) is independent of the joint collection of environmental marks, offspring variables, and immigration variables.

Write \(f_n(s):=\Ebf[s^{\xi_{n,1}}\mid\cE_n]\),
\(g_n(s):=\Ebf[s^{\eta_{n+1}}\mid\cE_n]\), and
\(M_n:=f_n'(1)\in(0,\infty)\) almost surely.  For \(0\leq j<n\), put
\begin{equation}\label{eq:bprei-products}
        \Pi_{j,n}:=\prod_{r=j}^{n-1}M_r,
        \qquad \Pi_n:=\Pi_{0,n},
        \qquad S_n:=\log\Pi_n,
\end{equation}
with \(\Pi_{n,n}=1\).  If
\(f_{j,n}:=f_j\circ\cdots\circ f_{n-1}\), with \(f_{n,n}(s)=s\), then the
quenched p.g.f.\ from \(Z_0=i\) is
\begin{equation}\label{eq:bprei-quenched-product}
        F^{\cE}_{i,n}(s)
        =f_{0,n}(s)^i
          \prod_{j=0}^{n-1}g_j\bigl(f_{j+1,n}(s)\bigr).
\end{equation}
The immigration p.g.f.\ therefore enters every iteration and may alter both
the Laplace-transform rate and its limiting profile.

Following the affine decomposition introduced in
Subsection~\ref{subsec:main-random-slope}, put
\(D_{n+1}:=Z_{n+1}-M_nZ_n\) and \(V_n:=Z_n/\Pi_n\).  We call
\(\{D_{n+1}:n\geq0\}\) the innovation sequence.  It need not be centered,
since it consists of the centered offspring fluctuation together with the
immigration entering generation \(n+1\).  The exact linear recursion and its
normalized form are
\begin{equation}\label{eq:bprei-linear-recursion}
        Z_{n+1}=M_nZ_n+D_{n+1},
        \qquad
        V_{n+1}-V_n=\frac{D_{n+1}}{\Pi_{n+1}}.
\end{equation}

For a generic environmental mark, write $M$ for its reproduction mean and
let $\xi_1,\xi_2,\ldots$ and $\eta$ have the associated one-generation
conditional laws.  For $k\in\Nzero$, set
\[
        D(k):=\sum_{i=1}^{k}(\xi_i-M)+\eta.
\]
Conditional on $Z_n=k$ and $\cE_n$, the variable $D_{n+1}$ has the law of
$D(k)$ under the mark $\cE_n$; hence, conditional on $Z_n=k$, it has the
corresponding annealed law.

Let \(\Lambda(\vartheta):=\log\Ebf[M^\vartheta]\) and
\(\cD_\Lambda:=\{\vartheta\in\R:\Lambda(\vartheta)<\infty\}\).  Whenever
\(\Ebf|\log M|<\infty\), we write \(\mu:=\Ebf\log M\).

Let \(\cG_n=\sigma(\cE_j:0\leq j<n)\) and let
\[
\cF_n=\sigma(Z_0)\vee\sigma\!\left(\cE_j,(\xi_{j,i})_{i\geq1},
                         \eta_{j+1}:0\leq j<n\right).
\]
Thus \(\cF_n\) contains the history through \(Z_n\), before the mark and row for generation \(n\) are sampled. For \(\vartheta\in\operatorname{int}\cD_\Lambda\), define the exponentially tilted environmental law by
\begin{equation}\label{eq:bprei-tilt}
        \left.\frac{\dd\Pbf_\vartheta}{\dd\Pbf}\right|_{\cG_n}
        =\exp\{\vartheta S_n-n\Lambda(\vartheta)\}.
\end{equation}
We keep the law of \(Z_0\) and the conditional laws of the generation rows given the environment unchanged. The density in \eqref{eq:bprei-tilt} therefore also gives the change of measure on \(\cF_n\), and \(\Ebf_\vartheta\log M=\Lambda'(\vartheta)\).  For the positive tilt used in
the upper-deviation analysis, write \(\rho:=\Lambda'(\alpha)\).  Thus
\(\rho\) is the target logarithmic growth rate, and \(\alpha>0\) is its
Cram\'er dual parameter.  The supercritical harmonic-moment analysis uses the
same family of changes of measure at the negative parameter \(-r\).

For \(t>0\), put
\(G_{i,n}(t):=\Ebf_i[\e^{-tZ_n};Z_n>0]\).  The common Laplace--Gamma
representation is
\begin{equation}\label{eq:bprei-laplace-gamma}
        \Gamma(r)\Ebf_i[Z_n^{-r};Z_n>0]
        =\int_0^\infty G_{i,n}(t)t^{r-1}\dd t,
        \qquad r>0.
\end{equation}
Section~\ref{subsec:bprei-upper-deviations} studies the normalized recursion under the positive \(\alpha\)-tilt. For harmonic moments, Section~\ref{subsec:bprei-supercritical-harmonic} conditions on the first exit from the current population size, while Section~\ref{subsec:bprei-critical-harmonic} decomposes the harmonic moment according to the first time the environmental walk attains its minimum. The Laplace--Gamma identity \eqref{eq:bprei-laplace-gamma} relates the resulting harmonic and Laplace profiles.

\subsection{Precise upper deviations}
\label{subsec:bprei-upper-deviations}

\paragraph{BPREI specialization of the upper-deviation hypotheses.}
The environmental and innovation assumptions below are the BPREI
counterparts of \eqref{eq:main-random-cramer} and
\eqref{eq:main-random-sublinear}.  We record them in the present notation
because Theorem~\ref{thm:bprei-sharp} is a standalone BPREI result.  Unlike
the general random-map result, no separate accessibility condition is
imposed: the nontriviality condition below, together with the BPREI structure,
yields finite-level accessibility in Lemma~\ref{lem:bprei-accessibility}.

\begin{assumption}[Environmental Cram\'er condition]\label{ass:bprei-cramer}
Fix \(\alpha>0\) in the interior of \(\cD_\Lambda\).  Assume that the law of
\(\log M\) is nonlattice and
\begin{equation}\label{eq:bprei-cramer-parameter}
        \Ebf|\log M|<\infty,
        \qquad \mu:=\Ebf\log M\geq0,
        \qquad \rho:=\Lambda'(\alpha)>\mu,
        \qquad v_\alpha^2:=\Lambda''(\alpha)\in(0,\infty).
\end{equation}
Put \(I(\rho):=\alpha\rho-\Lambda(\alpha)\).
\end{assumption}

In both the original supercritical case \(\mu>0\) and the original critical case \(\mu=0\), the environmental walk has drift \(\Lambda'(\alpha)=\rho>0\) under \(\Pbf_\alpha\).  Thus the BPREI is supercritical under the \(\alpha\)-tilted law.

\begin{assumption}[Sublinear one-generation innovation]\label{ass:bprei-linearization}
Assume \(\Ebf Z_0^\alpha<\infty\), and suppose that, for some
\(C<\infty\) and \(0\leq\beta<\alpha\),
\begin{equation}\label{eq:bprei-sublinear-innovation}
        \Ebf|D(k)|^\alpha\leq C(1+k^\beta),
        \qquad k\in\Nzero.
\end{equation}
\end{assumption}

\begin{assumption}[Nontriviality]\label{ass:bprei-nontrivial}
Assume \(\Pbf(Z_0>0)+\Pbf(\eta>0)>0\).
\end{assumption}

We refer to Assumptions~\ref{ass:bprei-cramer}--\ref{ass:bprei-nontrivial}
collectively as \((\mathrm U_\alpha)\).  Proposition~\ref{prop:bprei-primitive}
gives primitive branching-model conditions for the innovation estimate.

\begin{theorem}[Precise upper deviations for critical and supercritical BPREI]
\label{thm:bprei-sharp}
Under \((\mathrm U_\alpha)\), there is a nonnegative, nontrivial random
variable \(V_\infty\), with \(\Pbf_\alpha(V_\infty>0)>0\), such that
\begin{equation}\label{eq:bprei-Lalpha-limit}
        \frac{Z_n}{\Pi_n}\longrightarrow V_\infty
        \qquad
        \Pbf_\alpha\text{-almost surely and in }L^\alpha(\Pbf_\alpha).
\end{equation}
Moreover,
\begin{equation}\label{eq:bprei-moment-limit-main}
        \lim_{n\to\infty}
        \e^{-n\Lambda(\alpha)}\Ebf[Z_n^\alpha]
        =c_Z(\alpha):=\Ebf_\alpha[V_\infty^\alpha]
        \in(0,\infty),
\end{equation}
and, for every fixed \(y\in\R\),
\begin{equation}\label{eq:bprei-sharp-shifted}
        \lim_{n\to\infty}
        \sqrt n\,\e^{nI(\rho)}
        \Pbf\bigl(Z_n>\e^{\rho n+y}\bigr)
        =\frac{c_Z(\alpha)\e^{-\alpha y}}
        {\alpha v_\alpha\sqrt{2\pi}}.
\end{equation}
The conclusion holds in both the supercritical case \(\mu>0\) and the
critical case \(\mu=0\).
\end{theorem}

Proposition~\ref{prop:bprei-normalized-limit} identifies the limit of
\(\e^{-n\Lambda(\alpha)}\Ebf Z_n^\alpha\) as
\(\Ebf_\alpha[V_\infty^\alpha]\). Lemma~\ref{lem:bprei-accessibility}
and Proposition~\ref{prop:bprei-positive-limit} show that this constant
is strictly positive. The sharp upper-tail estimate then follows from the
logarithmic-block comparison in Proposition~\ref{prop:bprei-block}.

\begin{remark}[Relation with Buraczewski--Dyszewski]
\label{rem:bprei-BD}
Setting \(\eta\equiv0\) recovers the critical and supercritical
upper-deviation problem of \cite{BuraczewskiDyszewski2022}.  Theorem
\ref{thm:bprei-sharp} additionally allows immigration, including dependence
of reproduction and immigration through the same environmental mark.  For \(\alpha>1\),
the first-generation moment required in
\cite[Theorem~2.1\textup{(H3)}, p.~1672]{BuraczewskiDyszewski2022} is
\(\Ebf Z_1^{\alpha+\delta}<\infty\) for some \(\delta>0\).
Here Proposition~\ref{prop:bprei-primitive}\textup{(ii)} verifies
\eqref{eq:bprei-sublinear-innovation} using the \(\alpha\)-moments of
\(\xi_1-M\) and \(\eta\). At \(\alpha=1\), the moment above one in
Proposition~\ref{prop:bprei-primitive}\textup{(i)} is used to obtain
\(\beta<1\).

The environmental moment neighborhood in Assumption~\ref{ass:bprei-cramer}
is retained. To obtain the Bahadur--Rao--Petrov prefactor
\cite{BahadurRao1960,Petrov1965}, we apply Stone's theorem
\cite{Stone1965} to the tilted environmental sum and average the
resulting estimate over the \(\alpha\)-size-biased early generation.
The required small-shift condition follows from
\eqref{eq:bprei-Lalpha-limit}.
\end{remark}

\begin{remark}[The critical environmental regime]
\label{rem:bprei-Afanasyev}
Afanasyev \cite{Afanasyev2021} derives a functional limit theorem for a
critical BPREI under a random environmental normalization.  The critical part
of Theorem~\ref{thm:bprei-sharp} concerns a different event: the original walk
has zero drift, but the positive \(\alpha\)-tilt changes its drift to
\(\Lambda'(\alpha)=\rho>0\).  We analyze the resulting positive-drift
recursion and then return to the original law.
\end{remark}

The theorem is stated at the innovation level; the following proposition
records standard branching-model conditions that verify it.

\begin{proposition}[Primitive sufficient conditions]\label{prop:bprei-primitive}
The following conditions imply \eqref{eq:bprei-sublinear-innovation}.
\begin{enumerate}[label=\textup{(\roman*)}]
\item If $0<\alpha\leq1$, suppose there exists $r\in(1,2]$ such that
\begin{equation}\label{eq:bprei-primitive-small-alpha}
        \Ebf\left[
        \left\{\Ebf\bigl(|\xi_1-M|^r\mid\cE_0\bigr)\right\}^{\alpha/r}
        \right]<\infty,
        \qquad
        \Ebf[\eta^\alpha]<\infty.
\end{equation}
Then \eqref{eq:bprei-sublinear-innovation} holds with
$\beta=\alpha/r$.

\item If $\alpha>1$, suppose
\begin{equation}\label{eq:bprei-primitive-large-alpha}
        \Ebf|\xi_1-M|^\alpha<\infty,
        \qquad
        \Ebf[\eta^\alpha]<\infty.
\end{equation}
Then \eqref{eq:bprei-sublinear-innovation} holds with
\begin{equation}\label{eq:bprei-beta-large-alpha}
        \beta=1\vee\frac{\alpha}{2}<\alpha.
\end{equation}
\end{enumerate}
\end{proposition}

\begin{proof}
Let \(T_k:=\sum_{i=1}^{k}(\xi_i-M)\).  If \(0<\alpha\leq1\), then
subadditivity gives
\(\Ebf(|T_k+\eta|^\alpha\mid\cE_0)
 \leq\Ebf(|T_k|^\alpha\mid\cE_0)+\Ebf(\eta^\alpha\mid\cE_0)\).
Conditional Jensen and the Marcinkiewicz--Zygmund inequality at exponent
\(r\) yield
\[
\Ebf(|T_k|^\alpha\mid\cE_0)
\leq\{\Ebf(|T_k|^r\mid\cE_0)\}^{\alpha/r}
\leq C_rk^{\alpha/r}
\{\Ebf(|\xi_1-M|^r\mid\cE_0)\}^{\alpha/r}.
\]
Annealing proves \eqref{eq:bprei-sublinear-innovation} with
\(\beta=\alpha/r\).

If \(1<\alpha\leq2\), conditional von Bahr--Esseen gives
\(\Ebf(|T_k|^\alpha\mid\cE_0)
 \leq C_\alpha k\Ebf(|\xi_1-M|^\alpha\mid\cE_0)\).
If \(\alpha>2\), conditional Rosenthal and conditional Lyapunov give
\[
\begin{aligned}
\Ebf(|T_k|^\alpha\mid\cE_0)
&\leq C_\alpha\left[k\Ebf(|\xi_1-M|^\alpha\mid\cE_0)
+k^{\alpha/2}\{\Ebf(|\xi_1-M|^2\mid\cE_0)\}^{\alpha/2}\right]\\
&\leq C_\alpha(k+k^{\alpha/2})
\Ebf(|\xi_1-M|^\alpha\mid\cE_0).
\end{aligned}
\]
Annealing and
\(|x+y|^\alpha\leq2^{\alpha-1}(|x|^\alpha+|y|^\alpha)\) prove the claim
with \(\beta=1\) for \(1<\alpha\leq2\) and
\(\beta=\alpha/2\) for \(\alpha>2\).
\end{proof}

\subsubsection{Tilted normalization and the innovation estimate}
\label{subsec:bprei-normalization}

\paragraph{Random-map formulation.}
Let \((M_n,\Psi_n)_{n\ge0}\) be i.i.d. pairs, independent of the initial population \(Z_0\in\Nzero\), where \(M_n>0\) and \(\Psi_n:\Nzero\to\Nzero\) is a random map. Consider the recursion
\[
Z_{n+1}=\Psi_n(Z_n),\qquad n\ge0,
\]
and write \(D(k):=\Psi_0(k)-M_0k\) for the innovation at state \(k\) and \(D_{n+1}:=\Psi_n(Z_n)-M_nZ_n\) for the innovation at generation \(n+1\). Set \(\Pi_0=1\), \(\Pi_n=\prod_{\ell=0}^{n-1}M_\ell\), and \(\tau_K=\inf\{n\ge0:Z_n\ge K\}\).

Under \(\Pbf_\alpha\), the pairs \((M_n,\Psi_n)_{n\ge0}\) are i.i.d. with common \(\alpha\)-tilted distribution
\[
\Pbf_\alpha(M_0\in dm,\,\Psi_0\in d\psi)
=m^\alpha \e^{-\Lambda(\alpha)}\,\Pbf(M_0\in dm,\,\Psi_0\in d\psi).
\]
The initial population \(Z_0\) retains its original law and is independent of the pairs under \(\Pbf_\alpha\).
For this recursion \(V_n=Z_n/\Pi_n\), and \(\cF_n\) denotes the history
through \(Z_n\), before the generation-\(n\) pair is sampled.

We first convert the sublinear one-generation estimate into control of the
normalized recursion.  The following identities convert the relevant
original-law moments into normalized moments under the \(\alpha\)-tilted law.

\begin{lemma}[Tilt cancellation]\label{lem:bprei-tilt-cancellation}
For every $n\geq0$,
\begin{equation}\label{eq:bprei-moment-tilt-identity}
        \Ebf_\alpha[V_n^\alpha]
        =\e^{-n\Lambda(\alpha)}\Ebf[Z_n^\alpha],
\end{equation}
and
\begin{equation}\label{eq:bprei-increment-tilt-identity}
        \Ebf_\alpha|V_{n+1}-V_n|^\alpha
        =\e^{-(n+1)\Lambda(\alpha)}\Ebf|D_{n+1}|^\alpha.
\end{equation}
More generally, for every $k\geq1$,
\begin{equation}\label{eq:bprei-relative-tilt-identity}
        \Ebf_\alpha\left[M^{-\alpha}|D(k)|^\alpha\right]
        =\e^{-\Lambda(\alpha)}\Ebf|D(k)|^\alpha.
\end{equation}
\end{lemma}

\begin{proof}
Since \(V_n=Z_n/\Pi_n\),
\[
\Ebf_\alpha[V_n^\alpha]
=\e^{-n\Lambda(\alpha)}
\Ebf\!\left[\Pi_n^\alpha\frac{Z_n^\alpha}{\Pi_n^\alpha}\right]
=\e^{-n\Lambda(\alpha)}\Ebf[Z_n^\alpha],
\]
which proves \eqref{eq:bprei-moment-tilt-identity}.  Since
\(V_{n+1}-V_n=D_{n+1}/\Pi_{n+1}\), the same calculation at time \(n+1\)
proves \eqref{eq:bprei-increment-tilt-identity}.  Under the one-generation Radon--Nikodym derivative
\(M^\alpha\e^{-\Lambda(\alpha)}\),
\(\Ebf_\alpha[M^{-\alpha}|D(k)|^\alpha]
 =\e^{-\Lambda(\alpha)}\Ebf|D(k)|^\alpha\), which proves
\eqref{eq:bprei-relative-tilt-identity}.
\end{proof}

We use the following elementary recursion bound.

\begin{lemma}\label{lem:bprei-recursion-bound}
Let \(q\in(0,1)\), \(r\in[0,1)\), and \(C>0\). If a nonnegative
sequence satisfies
\begin{equation}\label{eq:bprei-deterministic-recursion}
 a_{n+1}\le a_n+Cq^n(1+a_n^r),\qquad n\ge0,
\end{equation}
then \(\sup_{n\ge0}a_n\le C_*(1\vee a_0)\), where \(C_*\) depends
only on \(C\) and \(q\). For \(r=0\), the term \(a_n^r\) is read as one.
\end{lemma}
\begin{proof}
For \(b_n=1\vee a_n\), we have \(1+a_n^r\le2b_n\) and hence
\(b_{n+1}\le b_n(1+2Cq^n)\). Iteration gives the assertion with
\(C_*=\exp\{2C/(1-q)\}\).
\end{proof}

\begin{proposition}[Normalized limit and innovation decay]
\label{prop:bprei-normalized-limit}
Fix \(\alpha>0\) in \(\operatorname{int}\cD_\Lambda\) and suppose that
\eqref{eq:bprei-cramer-parameter} and Assumption~\ref{ass:bprei-linearization}
hold. Then \(V_n=Z_n/\Pi_n\) converges \(\Pbf_\alpha\)-almost surely and in
\(L^\alpha(\Pbf_\alpha)\) to a finite nonnegative random variable \(V_\infty\),
and
\[
 \e^{-n\Lambda(\alpha)}\Ebf Z_n^\alpha\longrightarrow
 \Ebf_\alpha V_\infty^\alpha.
\]
There is a constant \(C_\alpha<\infty\), independent of the deterministic
starting state \(x\), such that
\[
 \sup_{n\ge0}\Ebf_{\alpha,x}V_n^\alpha\le C_\alpha(1+x^\alpha),
 \qquad x\in\Nzero.
\]
Moreover, for some \(C_0<\infty\),
\begin{equation}\label{eq:bprei-innovation-decay}
 \Ebf|D_n|^\alpha\le C_0\chi_\alpha^n,\qquad n\ge1,
\end{equation}
where
\begin{equation}\label{eq:bprei-chi-range}
 1\le\chi_\alpha<\e^{\Lambda(\alpha)}.
\end{equation}
\end{proposition}
\begin{proof}
Write \(L=\Lambda(\alpha)\) and \(r=\beta/\alpha<1\).
Since \(\log M\) is nonconstant, strict Jensen gives \(L>\alpha\mu\ge0\).
For \(0<\beta<\alpha\), conditioning on \(Z_n\), applying Lyapunov's
inequality, and using \eqref{eq:bprei-moment-tilt-identity} give
\begin{equation}\label{eq:bprei-D-bootstrap}
 \Ebf|D_{n+1}|^\alpha
 \le C\{1+\Ebf Z_n^\beta\}
 \le C\{1+\e^{nrL}(\Ebf_\alpha V_n^\alpha)^r\}.
\end{equation}
For \(\beta=0\), the uniform innovation bound in
\eqref{eq:bprei-sublinear-innovation} gives \(\Ebf|D_{n+1}|^\alpha\le C\)
directly, and the same estimates below hold with \(r=0\).

If \(0<\alpha\le1\), put \(a_n=\Ebf_\alpha V_n^\alpha\) and apply
subadditivity to \(V_{n+1}=V_n+(V_{n+1}-V_n)\). If \(\alpha>1\), put
\(a_n=\|V_n\|_{L^\alpha(\Pbf_\alpha)}\) and apply Minkowski's inequality
to the same identity. In both cases, \eqref{eq:bprei-increment-tilt-identity}
and \eqref{eq:bprei-D-bootstrap} give
\[
 a_{n+1}\le a_n+Cq^n(1+a_n^r),\qquad
 q=\exp\{-L(1-r)/(1\vee\alpha)\}<1.
\]
The constants \(C,q\) do not depend on the law of \(Z_0\).
Lemma~\ref{lem:bprei-recursion-bound} therefore gives
\(\sup_n\Ebf_\alpha V_n^\alpha\le C_\alpha(1+\Ebf Z_0^\alpha)\).
Taking \(Z_0=x\) gives the claimed bound for every starting state with
the same \(C_\alpha\).

Substitution into \eqref{eq:bprei-D-bootstrap} proves
\eqref{eq:bprei-innovation-decay} with \(\chi_\alpha=\e^{rL}\).
By \eqref{eq:bprei-increment-tilt-identity},
\(\Ebf_\alpha|V_{n+1}-V_n|^\alpha\le C_1\e^{-n(1-r)L}\).
For \(0<\alpha\le1\), the summability of these expectations implies
\(\sum_n|V_{n+1}-V_n|^\alpha<\infty\) almost surely and hence absolute
convergence of the increments; subadditivity also gives convergence in
\(L^\alpha\). For \(\alpha>1\), the \(L^\alpha\) norms of the increments
are summable, so the series converges absolutely almost surely and in
\(L^\alpha\). The moment limit now follows from
\eqref{eq:bprei-moment-tilt-identity}.
\end{proof}

\subsubsection{Positivity of the tilted limit}
\label{subsec:bprei-positivity}

We use the large-state approach of Collamore and Vidyashankar
\cite{CollamoreVidyashankar2013}. For the recursion considered here,
we estimate \(V_\infty-x\) under \(\Pbf_{\alpha,x}\) with a constant
that is independent of the starting population \(x\). Accessibility then
transfers positivity from large states to the prescribed initial law.

\begin{lemma}[Finite-level accessibility]\label{lem:bprei-accessibility}
For the BPREI in \eqref{eq:bprei-recursion}, under
Assumptions~\ref{ass:bprei-cramer} and~\ref{ass:bprei-nontrivial},
\begin{equation}\label{eq:bprei-accessibility}
 \Pbf_\alpha(\tau_K<\infty)>0,\qquad K\ge1 .
\end{equation}
\end{lemma}
\begin{proof}
Since \(\Ebf_\alpha\log M=\rho>0\), we have \(\Pbf_\alpha(M>1)>0\).
On \(\{M>1\}\) an individual has a positive conditional probability of
at least two offspring. Thus for some \(\delta>0\) the event
\(Q_0(\{2,3,\ldots\})\ge\delta\) has tilted probability \(\gamma>0\).
Choose an integer \(s\ge1\) with \(2^s\ge K\). Starting from at least one individual,
require that at each generation \(\ell=0,\ldots,s-1\) this mark event
occurs and the first \(2^\ell\) individuals each have at least two
offspring. The probability is at least
\(\gamma^s\delta^{\,2^s-1}>0\); on this event \(Z_s\ge2^s\).
If \(\Pbf(Z_0>0)>0\), this proves the assertion. Otherwise
Assumption~\ref{ass:bprei-nontrivial} gives \(\Pbf(\eta_1>0)>0\),
and the strictly positive one-generation tilt density gives
\(\Pbf_\alpha(\eta_1>0)>0\). On that event \(Z_1\ge1\), and the same
construction uses only the independent rows from generation \(1\) onward.
\end{proof}

\begin{proposition}[Nondegeneracy of the normalized limit]
\label{prop:bprei-positive-limit}
Under the hypotheses of Proposition~\ref{prop:bprei-normalized-limit},
suppose \(\Pbf_\alpha(\tau_K<\infty)>0\) for every \(K\ge1\). Then
\begin{equation}\label{eq:bprei-positive-limit}
 \Pbf_\alpha(V_\infty>0)>0,
\end{equation}
and consequently
\begin{equation}\label{eq:bprei-positive-constant}
 0<\Ebf_\alpha[V_\infty^\alpha]<\infty.
\end{equation}
\end{proposition}
\begin{proof}
First start from \(Z_0=x\ge1\). Write \(r=\beta/\alpha\) and
\(L=\Lambda(\alpha)>0\). Proposition~\ref{prop:bprei-normalized-limit}
gives \(\sup_n\Ebf_{\alpha,x}V_n^\alpha\le C_\alpha(1+x^\alpha)\),
with the same constant for all \(x\). Substituting this bound into
\eqref{eq:bprei-D-bootstrap} and using \eqref{eq:bprei-increment-tilt-identity}
gives
\[
 \Ebf_{\alpha,x}|V_{n+1}-V_n|^\alpha
 \le C(1+x^\beta)\e^{-n(1-r)L},\qquad x\ge1,\ n\ge0.
\]
For \(\beta=0\) the same inequality follows directly from the constant
innovation bound. The constant \(C\) is independent of both \(n\) and \(x\).

Since \(V_0=x\), sum the increment moments when \(0<\alpha\le1\).
When \(\alpha>1\), sum their \(L^\alpha\) norms by Minkowski's inequality
and raise to the power \(\alpha\). Passing to the limit gives, in both cases,
\[
 \Ebf_{\alpha,x}|V_\infty-x|^\alpha\le C'(1+x^\beta).
\]
Consequently
\[
 \Pbf_{\alpha,x}(V_\infty\le x/2)
 \le 2^\alpha C'(1+x^\beta)x^{-\alpha}.
\]
As \(\beta<\alpha\), there exists \(K\) for which
\(\inf_{x\ge K}\Pbf_{\alpha,x}(V_\infty>0)\ge1/2\).

Write \(p_\alpha(x)=\Pbf_{\alpha,x}(V_\infty>0)\). On \(\{\tau_K<\infty\}\), the strong Markov property of the marked recursion under \(\Pbf_\alpha\) gives, conditionally on \(\cF_{\tau_K}\), the law of the recursion started from \(Z_{\tau_K}\) and driven by fresh generation pairs. Its normalized limit is
\[
\lim_{m\to\infty}\frac{Z_{\tau_K+m}}{\Pi_{\tau_K,\tau_K+m}}=\Pi_{\tau_K}V_\infty.
\]
Since \(\Pi_{\tau_K}>0\) and \(Z_{\tau_K}\ge K\) on this event,
\[
\Pbf_\alpha(V_\infty>0,\,\tau_K<\infty)
=\Ebf_\alpha[\one_{\{\tau_K<\infty\}}\,p_\alpha(Z_{\tau_K})]
\ge\tfrac12\Pbf_\alpha(\tau_K<\infty)>0.
\]
\end{proof}

\subsubsection{A size-biased random-prefactor estimate}
\label{subsec:bprei-prefactor}

With positivity established, we turn to the product with the random shift
created by the early generation size.  Under $\Pbf_\alpha$, let
\begin{equation}\label{eq:bprei-centered-random-walk}
        T_N:=\sum_{j=1}^{N}(\log M_j-\rho).
\end{equation}
For $N\geq1$ and $y\in\R$, define
\begin{equation}\label{eq:bprei-Petrov-kernel}
        K_N(y):=
        \sqrt N\,
        \Ebf_\alpha\left[
        \e^{-\alpha(T_N-y)};
        T_N>y
        \right].
\end{equation}

\begin{lemma}[The integrated Stone kernel]\label{lem:bprei-Stone-kernel}
Under Assumption~\ref{ass:bprei-cramer},
\begin{equation}\label{eq:bprei-Stone-uniform-bound}
 \sup_{N\ge1}\sup_{y\in\R}K_N(y)<\infty.
\end{equation}
For every sequence \(b_N\ge0\) with \(b_N=o(\sqrt N)\),
\begin{equation}\label{eq:bprei-Stone-central-limit}
 \sup_{|y|\le b_N}
 \left|K_N(y)-\frac{1}{\alpha v_\alpha\sqrt{2\pi}}\right|
 \longrightarrow0.
\end{equation}
\end{lemma}

\begin{proof}
Under \(\Pbf_\alpha\), the increments of \(T_N\) are i.i.d., centered,
and have variance \(v_\alpha^2\). Their law is nonlattice, since the
tilting density is strictly positive and centering only translates the
support. Stone's interval local limit theorem \cite{Stone1965} gives
\[
 C:=\sup_{N\ge1}\sup_{x\in\R}\sqrt N\,
       \Pbf_\alpha(x<T_N\le x+1)<\infty.
\]
For \(R\ge0\), partition \((R,\infty)\) into unit intervals open on
the left. Then, uniformly in \(N\) and \(y\),
\[
 \sqrt N\,\Ebf_\alpha\!\left[
   \e^{-\alpha(T_N-y)};T_N-y>R\right]
 \le C\sum_{j=0}^{\infty}\e^{-\alpha(R+j)}
 =\frac{C\e^{-\alpha R}}{1-\e^{-\alpha}} .
\]
Taking \(R=0\) proves \eqref{eq:bprei-Stone-uniform-bound}.

For any fixed \(0\le a<b\le R<\infty\), the same interval theorem gives
\[
 \sup_{|y|\le b_N}
 \left|\sqrt N\,\Pbf_\alpha(y+a<T_N\le y+b)
       -\frac{b-a}{v_\alpha\sqrt{2\pi}}\right|\longrightarrow0,
\]
since \((y+a)/\sqrt N\to0\) uniformly on this range.
The half-open convention follows also by applying the left-closed
interval form to \(-T_N\). Upper and lower Riemann sums for
\(\e^{-\alpha u}\) on \((0,R]\) therefore give, uniformly for
\(|y|\le b_N\),
\[
 \sqrt N\,\Ebf_\alpha\!\left[
 \e^{-\alpha(T_N-y)};0<T_N-y\le R\right]
 \longrightarrow
 \frac{1}{v_\alpha\sqrt{2\pi}}\int_0^R\e^{-\alpha u}\dd u .
\]
Letting \(R\to\infty\) and using the uniform tail bound proves
\eqref{eq:bprei-Stone-central-limit}.
\end{proof}

\begin{lemma}[Random prefactor under an \(\alpha\)-size-biased shift]
\label{lem:bprei-random-prefactor}
Suppose Assumption~\ref{ass:bprei-cramer} holds. Let \(j_n\) be integers
with \(j_n=o(n)\) and \(0\le j_n<n\) for all sufficiently large \(n\). Put \(\mathfrak n_n=n-j_n\).
Let \(Y_n\ge0\) be independent of \(\Pi_{j_n,n}\) under \(\Pbf\). Suppose
\begin{equation}\label{eq:bprei-prefactor-moment}
 c_n:=\e^{-j_n\Lambda(\alpha)}\Ebf[Y_n^\alpha]\longrightarrow c\in(0,\infty).
\end{equation}
For all sufficiently large \(n\), \(0<c_n<\infty\). Define
\begin{equation}\label{eq:bprei-size-biased-law}
 \widehat\Pbf_n(A)=
 \frac{\e^{-j_n\Lambda(\alpha)}\Ebf[Y_n^\alpha\one_A]}{c_n}.
\end{equation}
This law assigns probability one to \(\{Y_n>0\}\). If
\begin{equation}\label{eq:bprei-size-biased-small-shift}
 (\rho j_n-\log Y_n)/\sqrt{\mathfrak n_n}\longrightarrow0
 \qquad\text{in }\widehat\Pbf_n\text{-probability},
\end{equation}
then, for every fixed \(y\in\R\),
\begin{equation}\label{eq:bprei-prefactor-asymptotic}
 \sqrt n\,\e^{nI(\rho)}
 \Pbf(Y_n\Pi_{j_n,n}>\e^{\rho n+y})
 \longrightarrow\frac{c\,\e^{-\alpha y}}{\alpha v_\alpha\sqrt{2\pi}}.
\end{equation}
\end{lemma}

\begin{proof}
Conditioning on \(Y_n\) and tilting the independent environmental product
of length \(\mathfrak n_n\) gives
\begin{equation}\label{eq:bprei-prefactor-exact-identity}
 \Pbf(Y_n\Pi_{j_n,n}>\e^{\rho n+y})
 =\frac{c_n\e^{-nI(\rho)-\alpha y}}{\sqrt{\mathfrak n_n}}\,
   \widehat\Ebf_n K_{\mathfrak n_n}(\rho j_n+y-\log Y_n).
\end{equation}
Write \(U_n=\rho j_n+y-\log Y_n\).
By \eqref{eq:bprei-size-biased-small-shift}, there are deterministic
\(b_n\ge0\), \(b_n=o(\sqrt{\mathfrak n_n})\), such that
\(\widehat\Pbf_n(|U_n|>b_n)\to0\).
Lemma~\ref{lem:bprei-Stone-kernel} gives uniform convergence of
\(K_{\mathfrak n_n}\) to \(1/(\alpha v_\alpha\sqrt{2\pi})\) on
\([-b_n,b_n]\), and a uniform bound outside it. Its expectation in
\eqref{eq:bprei-prefactor-exact-identity} therefore converges to this
constant. Now use \(\mathfrak n_n/n\to1\) and \(c_n\to c\).
\end{proof}

\begin{corollary}[The early generation size as a prefactor]
\label{cor:bprei-early-prefactor}
Let \(j_n=\lceil K\log n\rceil\) for a fixed \(K>0\) and all sufficiently large \(n\). Under
Assumptions~\ref{ass:bprei-cramer} and~\ref{ass:bprei-linearization},
if \(c_Z(\alpha)=\Ebf_\alpha[V_\infty^\alpha]>0\), then for every fixed \(y\in\R\),
\begin{equation}\label{eq:bprei-early-prefactor-asymptotic}
 \sqrt n\,\e^{nI(\rho)}\Pbf(Z_{j_n}\Pi_{j_n,n}>\e^{\rho n+y})
 \longrightarrow\frac{c_Z(\alpha)\e^{-\alpha y}}{\alpha v_\alpha\sqrt{2\pi}}.
\end{equation}
\end{corollary}

\begin{proof}
The variable \(Z_{j_n}\) is \(\cF_{j_n}\)-measurable and the product
\(\Pi_{j_n,n}\) is independent of \(\cF_{j_n}\).
Take \(Y_n=Z_{j_n}\) in \eqref{eq:bprei-size-biased-law}. The tilt density
on \(\cF_{j_n}\) gives, for every bounded \(\cF_{j_n}\)-measurable random
variable \(F\),
\begin{equation}\label{eq:bprei-size-biased-V-representation}
 \widehat\Ebf_n F
 =\frac{\Ebf_\alpha[V_{j_n}^\alpha F]}{c_n},
 \qquad c_n=\Ebf_\alpha V_{j_n}^\alpha\longrightarrow c_Z(\alpha)>0.
\end{equation}
Indeed, Proposition~\ref{prop:bprei-normalized-limit} gives
\(V_n\to V_\infty\) in \(L^\alpha(\Pbf_\alpha)\), so
\(V_n^\alpha\to V_\infty^\alpha\) in \(L^1(\Pbf_\alpha)\).
In particular, the weights \(V_{j_n}^\alpha\) are uniformly integrable.
For \(\eps>0\), Chebyshev's inequality gives
\[
 \Pbf_\alpha(|S_{j_n}-\rho j_n|>\eps\sqrt n)
 \le\frac{v_\alpha^2j_n}{\eps^2 n}\longrightarrow0.
\]
By uniform integrability and \eqref{eq:bprei-size-biased-V-representation},
this event also has \(\widehat\Pbf_n\)-probability tending to zero. Moreover,
\[
\begin{aligned}
 \widehat\Pbf_n(V_{j_n}>\e^{\eps\sqrt n})
 &=c_n^{-1}\Ebf_\alpha[V_{j_n}^\alpha;V_{j_n}>\e^{\eps\sqrt n}]
   \longrightarrow0,\\
 \widehat\Pbf_n(0<V_{j_n}<\e^{-\eps\sqrt n})
 &\le c_n^{-1}\e^{-\alpha\eps\sqrt n}\longrightarrow0.
\end{aligned}
\]
The first limit uses uniform integrability; the second uses the weight
itself on the indicated event. Since \(V_{j_n}>0\) \(\widehat\Pbf_n\)-almost
surely, \(\log Z_{j_n}=S_{j_n}+\log V_{j_n}\) under this law.
Thus \eqref{eq:bprei-size-biased-small-shift} holds, since
\(\mathfrak n_n/n\to1\). Apply Lemma~\ref{lem:bprei-random-prefactor}.
\end{proof}

\subsubsection{Logarithmic-block approximation}
\label{subsec:bprei-block}

It remains to show that the innovations after the logarithmic split are
negligible on the sharp scale.

\begin{proposition}[Sharp upper deviations for a random-map recursion]
\label{prop:bprei-block}
Suppose the hypotheses of Proposition~\ref{prop:bprei-normalized-limit}
hold and \(\Pbf_\alpha(\tau_K<\infty)>0\) for every \(K\ge1\).
Then \(c_Z(\alpha):=\Ebf_\alpha[V_\infty^\alpha]>0\).
If the law of \(\log M_0\) is nonlattice, the sharp upper-tail limit
\eqref{eq:bprei-sharp-shifted} holds for every fixed \(y\in\R\).
\end{proposition}

\begin{proof}
Proposition~\ref{prop:bprei-positive-limit} gives \(c_Z(\alpha)>0\).
Assume now that the law of \(\log M_0\) is nonlattice.

To prove the tail limit, write \(j=j_n=\lceil K\log n\rceil\), where \(K\)
will be chosen below. The identity
\begin{equation}\label{eq:bprei-exact-telescope}
 Z_n-Z_j\Pi_{j,n}=\sum_{k=j+1}^n D_k\Pi_{k,n}
\end{equation}
pairs \(D_k=Z_k-M_{k-1}Z_{k-1}\), which is \(\cF_k\)-measurable,
with \(\Pi_{k,n}=\prod_{\ell=k}^{n-1}M_\ell\), which is independent of
\(\cF_k\) under \(\Pbf\). Choose \(a>0\) so that \(a\sum_{\ell\ge1}\ell^{-2}=1\).
If the absolute value of the sum exceeds \(t\), at least one summand has
absolute value greater than \(at(k-j)^{-2}\). By Markov's inequality and
\eqref{eq:bprei-innovation-decay}, with
\(h=\chi_\alpha\e^{-\Lambda(\alpha)}<1\),
\[
 \Pbf(|Z_n-Z_j\Pi_{j,n}|>t)
 \le C t^{-\alpha}\e^{n\Lambda(\alpha)}
       \sum_{k=j+1}^n(k-j)^{2\alpha}h^k
 \le C' t^{-\alpha}\e^{n\Lambda(\alpha)}h^j .
\]
Take \(K>1/(2|\log h|)\). Since \(I(\rho)=\alpha\rho-\Lambda(\alpha)\),
for each \(\varepsilon>0\) this gives
\begin{equation}\label{eq:bprei-block-approximation}
 \Pbf\!\left(|Z_n-Z_{j_n}\Pi_{j_n,n}|>\varepsilon\e^{\rho n}\right)
       =o\!\left(n^{-1/2}\e^{-nI(\rho)}\right).
\end{equation}
Fix \(y\in\R\) and \(0<\delta<1\). Apply
Corollary~\ref{cor:bprei-early-prefactor} at
\(y+\log(1+\delta)\) and \(y+\log(1-\delta)\). By
\eqref{eq:bprei-block-approximation} with \(\varepsilon=\delta\e^y\),
the liminf and limsup of
\(\sqrt n\,\e^{nI(\rho)}\Pbf(Z_n>\e^{\rho n+y})\) lie between
\((1+\delta)^{-\alpha}C_y\) and \((1-\delta)^{-\alpha}C_y\), where
\[
 C_y=\frac{c_Z(\alpha)\e^{-\alpha y}}{\alpha v_\alpha\sqrt{2\pi}} .
\]
Letting \(\delta\downarrow0\) proves \eqref{eq:bprei-sharp-shifted}.
\end{proof}

\subsubsection{Proof of Theorem~\ref{thm:bprei-sharp}}
\label{subsec:bprei-main-proof}

\begin{proof}[Proof of Theorem~\ref{thm:bprei-sharp}]
For the BPREI, take \(\Psi_n(k)=\sum_{j=1}^k\xi_{n,j}+\eta_{n+1}\) and
\(M_n=f_n'(1)\). By the construction in Section~\ref{subsec:bprei-model},
the pairs \((M_n,\Psi_n)\) are i.i.d.\ and independent of \(Z_0\).
Assumptions~\ref{ass:bprei-cramer} and~\ref{ass:bprei-linearization}
give the environmental and innovation conditions. Under
Assumption~\ref{ass:bprei-nontrivial}, Lemma~\ref{lem:bprei-accessibility}
gives \(\Pbf_\alpha(\tau_K<\infty)>0\) for every \(K\ge1\).
The conclusions follow from Propositions~\ref{prop:bprei-normalized-limit}
and~\ref{prop:bprei-block}.
\end{proof}

\subsection{Supercritical harmonic moments in the nondecreasing case}
\label{subsec:bprei-supercritical-harmonic}

We now study the small-generation-size profile that enters the unconditional
ratio bounds.  Every individual has at least one child, so the BPREI is
nondecreasing and the first exit from the current generation size gives the
natural renewal variable.

\paragraph{Standing supercritical hypotheses.}
For a generic environmental mark, write \(p_j:=Q(\{j\})\) and
\(h_j:=G(\{j\})\).  Thus \(p_j\) is the quenched probability that one
individual has \(j\) offspring, while \(h_j\) is the quenched probability of
\(j\) immigrants; both are random through the environmental mark.  Put
\(M:=\sum_{j\geq0}jp_j\), and write
\(\mathsf p_{ij}:=\Pbf_i(Z_1=j)\) for the annealed one-step transition
probability of the full BPREI from state \(i\) to state \(j\).

\begin{assumption}[Nondecreasing supercritical BPREI]
\label{ass:bprei-harmonic-supercritical}
Assume the following.
\begin{enumerate}[label=\textup{(H\arabic*)}]
\item
      \(p_0=0\) almost surely and \(\mu:=\Ebf\log M\in(0,\infty)\).
\item\label{cond:bprei-harmonic-schroder}
      \(\Pbf(h_0>0,\ 0<p_1<1)>0\).
\item\label{cond:bprei-harmonic-ks}
      \(\Ebf[M^{-1}\sum_{j\geq1}j\log^+j\,p_j]<\infty\).
\item\label{cond:bprei-harmonic-moment}
      \(\Ebf M^\delta<\infty\) for some \(\delta>0\).
\item\label{cond:bprei-harmonic-immigration}
      \(\Ebf\eta<\infty\).
\end{enumerate}
Reproduction and immigration may depend through the same environmental mark.
\end{assumption}

We refer to this assumption as \((\mathrm H)\).  Condition \textup{(H3)} is
the annealed Kesten--Stigum condition for the conditionally size-biased
offspring count: if
\(\Pbf_{\cE}(\xi^\star=j)=jp_j/M\), then it says
\(\Ebf\log^+\xi^\star<\infty\).  Condition \textup{(H5)} makes the normalized immigrant-clan
series summable.  Since \(M\geq1\) under \textup{(H1)}, a first immigration
moment is enough for this last step.

For each \(i\geq1\), put \(\gamma_i:=\mathsf p_{ii}=\Ebf[h_0p_1^i]\).  This is the one-step
probability that the generation size remains equal to \(i\); accordingly,
\(\gamma_i^\ell\) is the cost of remaining at \(i\) for \(\ell\) consecutive
generations.  Condition \textup{(H2)} gives
\(0<\gamma_{i+1}<\gamma_i<1\), since
\(\gamma_i-\gamma_{i+1}=\Ebf[h_0p_1^i(1-p_1)]>0\).  For \(r>0\), put
\[
        c_r:=\Ebf[M^{-r}].
\]
Since \(M\geq1\) almost surely under \textup{(H1)},
\(c_r\in(0,1]\).  Moreover, \(\mu>0\) implies
\(\Pbf(M>1)>0\), and hence \(c_r<1\) for every \(r>0\).  The function
\(r\mapsto c_r=\exp\{\Lambda(-r)\}\) is continuous and strictly decreasing
from \(1\) to \(\Pbf(M=1)\).  Finally,
let \(r_i:=\sup\{r>0:c_r>\gamma_i\}\in(0,\infty]\), and write
\(H_{i,n}(r):=\Ebf_iZ_n^{-r}\) and
\(\Phi_{i,n}(t):=\Ebf_i\e^{-tZ_n}\).  Since \(Z_n\geq i\), these quantities
are finite.

\begin{theorem}[Exact supercritical harmonic moments for BPREI]
\label{thm:bprei-supercritical-harmonic}
Under \((\mathrm H)\), let \(i\geq1\).  The following limits hold in the
indicated ranges of \(r\):
\begin{equation}\label{eq:bprei-supercritical-harmonic-main}
\begin{aligned}
\lim_{n\to\infty}c_r^{-n}H_{i,n}(r)
    &=C^{\mathrm{env}}_{i,r}, &&0<r<r_i,\\
\lim_{n\to\infty}\frac{H_{i,n}(r)}{n\gamma_i^n}
    &=C^{\mathrm{bd}}_i, &&r=r_i<\infty,\\
\lim_{n\to\infty}\gamma_i^{-n}H_{i,n}(r)
    &=C^{\mathrm{sm}}_{i,r}, &&r>r_i.
\end{aligned}
\end{equation}
Every constant on the right is finite and strictly positive.
\end{theorem}

If \(r_i=\infty\), the boundary and small-state regimes are absent, so only
the environmental limit in \eqref{eq:bprei-supercritical-harmonic-main}
occurs.  Proposition~\ref{prop:bprei-harmonic-gamma-constants} below gives
representations of the constants.

Conditioning on the first exit from \(i\) gives a renewal equation with
persistence factor \(\gamma_i\). Under \(\Pbf_{-r}\), the inverse moments
of the normalized population converge for sufficiently large initial states.
The first-exit identity and the renewal equation then give the limits for
smaller initial states. At \(c_r=\gamma_i\), summing over the possible
exit times gives the factor \(n\).

\begin{remark}[Related results]
\label{rem:bprei-harmonic-improvement}
For harmonic moments of supercritical branching processes in random
environments, see \cite{HuangWangWang2022} for the case with immigration
and \cite{GramaLiuMiqueu2017} for the case without immigration.
\end{remark}

\begin{remark}[Why there are three regimes here]
\label{rem:bprei-harmonic-no-critical-tilt}
Under \textup{(H1)}, \(M\geq1\) almost surely and
\(\Pbf(M>1)>0\), since \(\mu>0\). Thus, for every \(r>0\),
\[
        \Ebf_{-r}\log M=\frac{\Ebf[M^{-r}\log M]}{c_r}>0.
\]
The negatively tilted process therefore remains supercritical.  A critical
negative tilt can occur only when environments with \(M<1\) are allowed.  The
unweighted critical BPREI theory is developed in
Section~\ref{subsec:bprei-critical-harmonic}, while the additional weighted
problem arising at a critical negative tilt is discussed in
Section~\ref{subsec:bprei-negative-tilt}.
\end{remark}

\subsubsection{Normalized generation sizes under the negative tilt}

In Theorem~\ref{thm:bprei-sharp}, we analyzed the normalized population
\[
        V_n=\frac{Z_n}{\Pi_n}
\]
under the positive Cram\'er tilt \(\Pbf_\alpha\), and used tilt cancellation
to translate its \(\alpha\)-moment behavior back to the original law.  The
proof of Theorem~\ref{thm:bprei-supercritical-harmonic} requires the
corresponding analysis under the negative tilt \(\Pbf_{-r}\), since
\[
        H_{i,n}(r)
        =c_r^n\Ebf_{-r,i}\!\left[(V_n^{(i)})^{-r}\right].
\]
We therefore begin by establishing the convergence and strict positivity of
\(V_n^{(i)}\) under \(\Pbf_{-r}\).

The proof uses the clan decomposition rather than the innovation recursion of
Subsection~\ref{subsec:bprei-upper-deviations}.
Condition \textup{(H1)} ensures that the environment remains supercritical under
every negative tilt and that ancestral clans cannot become extinct.  Together
with \textup{(H1)},
Condition~\ref{cond:bprei-harmonic-ks} supplies the Kesten--Stigum
integrability needed for the normalized ancestral-clan limits, while
Condition~\ref{cond:bprei-harmonic-immigration} makes the normalized
immigrant-clan series summable. For the BPRE without immigration, Grama, Liu and Miqueu
use the same negative tilt and invoke an inverse-moment criterion
\cite[Lemma~3.1 and (3.8)--(3.11), pp.~12--14]{GramaLiuMiqueu2017}.
The criterion for a prescribed order in
\cite[Theorem~2.1(b), p.~1937]{GramaLiuMiqueu2023} assumes the uniform
bounds (2.11) of that paper. We prove the large-state estimate in
Lemma~\ref{lem:bprei-harmonic-large-state} directly under \((\mathrm H)\).

\begin{lemma}[Normalized BPREI under the negative tilt]
\label{lem:bprei-harmonic-normalized-limit}
Under the standing hypothesis \((\mathrm H)\), fix $r>0$ and
$i\geq1$.  Under $\Pbf_{-r,i}$,
\begin{equation}\label{eq:bprei-harmonic-normalized-convergence}
        V_n^{(i)}=\frac{Z_n}{\Pi_n}
        \longrightarrow V_\infty^{(i)}
        \quad\text{in }L^1\text{ and almost surely},
\end{equation}
and
\begin{equation}\label{eq:bprei-harmonic-normalized-positive}
        0<V_\infty^{(i)}<\infty
        \qquad\Pbf_{-r,i}\text{-almost surely}.
\end{equation}
\end{lemma}

\begin{proof}
Under \(\Pbf_{-r}\), the environmental marks remain i.i.d. and
\(\Ebf_{-r}\log M=c_r^{-1}\Ebf[M^{-r}\log M]>0\).  Moreover,
\[
\Ebf_{-r}\left[M^{-1}\sum_{j\geq1}j\log^+j\,p_j\right]
=\frac1{c_r}\Ebf\left[M^{-r-1}\sum_{j\geq1}j\log^+j\,p_j\right]
\leq\frac1{c_r}\Ebf\left[M^{-1}\sum_{j\geq1}j\log^+j\,p_j\right]<\infty.
\]
Thus the annealed Kesten--Stigum condition holds under the negative tilt.
For \(s=0\), let \(Z_{0,n}^{(a)}\) denote the number of descendants at time
\(n\) of the \(a\)th initial particle.  For \(1\leq s\leq n\), let
\(Z_{s,n}^{(a)}\) denote the number of descendants at time \(n\) of the
\(a\)th immigrant entering at time \(s\).  In each case, only descendants
produced through reproduction are included; immigrants entering at later
times initiate separate clans.  Put
\[
        W_{s,n}^{(a)}:=\frac{Z_{s,n}^{(a)}}{\Pi_{s,n}},
        \qquad
        \Pi_{s,n}:=\prod_{k=s}^{n-1}M_k,
\]
with \(\Pi_{n,n}=1\).  Conditional on the environment, each clan is a
no-immigration BPRE started from one particle at time \(s\).  Hence, by the
preceding Kesten--Stigum verification,
\[
        W_{s,n}^{(a)}\longrightarrow W_{s,\infty}^{(a)}
\]
almost surely and in \(L^1(\Pbf_{-r})\), with
\[
        \Ebf_{-r}W_{s,\infty}^{(a)}=1,
        \qquad W_{s,\infty}^{(a)}>0
        \quad\text{almost surely}.
\]

The population at time \(n\) is the sum of the initial and immigrant clans:
\[
        Z_n=\sum_{a=1}^{i}Z_{0,n}^{(a)}
        +\sum_{s=1}^{n}\sum_{a=1}^{\eta_s}Z_{s,n}^{(a)}.
\]
Since \(\Pi_n=\Pi_s\Pi_{s,n}\), division by \(\Pi_n\) gives the clan
decomposition
\begin{equation}\label{eq:bprei-harmonic-clan-decomposition}
V_n^{(i)}=\sum_{a=1}^{i}W_{0,n}^{(a)}
+\sum_{s=1}^{n}\Pi_s^{-1}\sum_{a=1}^{\eta_s}W_{s,n}^{(a)}.
\end{equation}
For \(s\geq1\), the pair \((M_{s-1},\eta_s)\) is independent of the
first \(s-1\) generations. Using their joint law for the last factor gives
\begin{equation}\label{eq:bprei-harmonic-immigration-sum}
 \Ebf_{-r}[\Pi_s^{-1}\eta_s]
 =\frac{\Ebf[M^{-r-1}\eta]}{c_r}
   \left(\frac{c_{r+1}}{c_r}\right)^{s-1}.
\end{equation}
Condition~\ref{cond:bprei-harmonic-immigration} gives
\(\Ebf[M^{-r-1}\eta]<\infty\), while \(c_{r+1}<c_r\); hence
\(\sum_{s\geq1}\Ebf_{-r}[\Pi_s^{-1}\eta_s]<\infty\).
Define
\begin{equation}\label{eq:bprei-harmonic-limit-series}
V_\infty^{(i)}:=\sum_{a=1}^{i}W_{0,\infty}^{(a)}
+\sum_{s=1}^{\infty}\Pi_s^{-1}\sum_{a=1}^{\eta_s}W_{s,\infty}^{(a)}.
\end{equation}
For \(n\geq N\), let \(V_{n,N}^{(i)}\) and \(V_{\infty,N}^{(i)}\) denote
the sums restricted to the initial clans and the immigrants entering by
time \(N\). Put \(d_m:=\Ebf_{-r}|W_{0,m}^{(1)}-W_{0,\infty}^{(1)}|\),
so \(d_m\to0\). Conditional on \(\cF_s\), each clan founded at time
\(s\) has the law of a fresh ancestral clan, while \(\eta_s\) and
\(\Pi_s\) are known. Thus
\[
\Ebf_{-r}|V_{n,N}^{(i)}-V_{\infty,N}^{(i)}|
\leq i d_n+\sum_{s=1}^N\Ebf_{-r}[\Pi_s^{-1}\eta_s]d_{n-s}
\longrightarrow0
\]
for fixed \(N\). Each clan has normalized mean one, both at finite times
and in the limit, so the two omitted tails have total expectation at most
\(2\sum_{s>N}\Ebf_{-r}[\Pi_s^{-1}\eta_s]\). This tends to zero by
\eqref{eq:bprei-harmonic-immigration-sum}, proving convergence in \(L^1\).  For almost-sure convergence, use the chronological
filtration \((\cF_n)\).  Under the negative tilt,
\[
\Ebf_{-r}[V_{n+1}^{(i)}\mid\cF_n]
=V_n^{(i)}+\Pi_n^{-1}\Ebf_{-r}\!\left[\frac{\eta}{M}\right]
\geq V_n^{(i)}.
\]
Thus \((V_n^{(i)})\) is a nonnegative submartingale, and
\eqref{eq:bprei-harmonic-immigration-sum} gives
\[
\sup_n\Ebf_{-r}V_n^{(i)}
\leq i+\sum_{s\geq1}\Ebf_{-r}[\Pi_s^{-1}\eta_s]<\infty.
\]
The submartingale convergence theorem yields a finite almost-sure limit,
which the established \(L^1\) convergence identifies with
\eqref{eq:bprei-harmonic-limit-series}.  The first sum in
\eqref{eq:bprei-harmonic-limit-series} is strictly positive almost surely,
so \(0<V_\infty^{(i)}<\infty\).  See \cite{WangLiu2017} for related
normalized convergence results under weaker immigration assumptions.
\end{proof}

The positive normalized limit in
Lemma~\ref{lem:bprei-harmonic-normalized-limit} does not by itself imply
convergence of inverse moments.  Recall that
\[
        H_{j,n}(r):=\Ebf_j[Z_n^{-r}].
\]
The negative-tilt identity gives
\[
        c_r^{-n}H_{j,n}(r)
        =\Ebf_{-r,j}\!\left[(V_n^{(j)})^{-r}\right].
\]
The next lemma chooses \(K=K(r)\) as a sufficiently large number of initial
ancestors for a comparison BPRE without immigration.  This same \(K\) serves
as a large-state threshold for the BPREI: coupling with the first \(K\)
ancestral families yields uniform integrability of the inverse moments for
every initial state \(j\geq K\).  The lemma therefore identifies the limiting
contribution on the environmental scale \(c_r^n\) once the BPREI has entered
the region \([K,\infty)\).  Since the BPREI is nondecreasing, its first
entrance into \([K,\infty)\) is also its last exit from
\(\{1,\ldots,K-1\}\).  For \(j\geq K\), the quantity
\[
        U_j(r):=\Ebf_{-r,j}\!\left[(V_\infty^{(j)})^{-r}\right]
\]
is the corresponding limiting inverse-moment constant after that entrance.

\begin{lemma}[Large-state inverse moments under the negative tilt]
\label{lem:bprei-harmonic-large-state}
Fix $r>0$.  There is an integer $K=K(r)$ such that, for every $j\geq K$,
\begin{equation}\label{eq:bprei-harmonic-large-state-limit}
        c_r^{-n}H_{j,n}(r)
        \longrightarrow
        U_j(r):=\Ebf_{-r,j}[(V_\infty^{(j)})^{-r}]
        \in(0,\infty),
\end{equation}
and
\begin{equation}\label{eq:bprei-harmonic-large-state-uniform}
        \sup_{j\geq K}\sup_{n\geq0}
        c_r^{-n}H_{j,n}(r)<\infty.
\end{equation}
\end{lemma}

\begin{proof}
Since \(0\leq p_1\leq1\),
\(\Ebf[p_1^K]\to\Pbf(p_1=1)=\Pbf(M=1)\), whereas
\(c_r=\Pbf(M=1)+\Ebf[M^{-r};M>1]>\Pbf(M=1)\).
Choose \(K\) so that
\begin{equation}\label{eq:bprei-harmonic-K-choice}
        \Ebf[p_1^K]<c_r.
\end{equation}

For a BPRE without immigration, let \(T_k=\xi_1+\cdots+\xi_k\) be the
number of offspring in one generation from \(k\) individuals. Under the original
law write \(U_j=\xi_j/M\), \(U=U_1\), and \(\bar U_k=T_k/(kM)\). By (H3),
\(\Ebf[U\log^+U]<\infty\), since \(M\ge1\). For \(x\ge1\) set
\[
 e(x)=C_r\bigl(x^{-1/4}+\Ebf[U;U>\sqrt x]\bigr).
\]
We claim that \(\Ebf[(T_k/k)^{-r}]\le c_r+e(k)\). Since \(T_k\ge k\),
\[
 (T_k/k)^{-r}-M^{-r}\le C_r(1-\bar U_k)_+ .
\]
Indeed, for \(\bar U_k<1/2\) the left side is at most one; for
\(\bar U_k\ge1/2\) use that \(u^{-r}\) is Lipschitz on \([1/2,\infty)\).
Conditionally on the mark, the \(U_j\) are independent with mean one.
Truncating each at \(\sqrt k\) costs at most \(\Ebf[U;U>\sqrt k]\) in
the annealed mean. The conditional variance of each truncated summand
is at most \(\sqrt k\), so the expected absolute deviation of its
sample mean from its conditional mean is at most \(k^{-1/4}\).
This proves the claim.

Choose \(1<a<b\) with \(\Pbf(a<M<b)>0\). The conditional law of large
numbers gives
\(\Ebf[(T_k/k)^{-r};T_k\ge ak]\ge p>0\) for all sufficiently large
\(k\), by restriction to \(a<T_k/k<b\).
Define \(L(x)=1+(2/p)\sum_{j\ge0} e(a^jx)\).
The series converges and \(L(x)\downarrow1\): each \(U\) contributes to at
most \(1+2\log^+U/\log a\) of the tail terms, and
\(\Ebf[U(1+\log^+U)]<\infty\). Increase \(K\) so that \(L(K)\le2\)
and the preceding lower bound holds for all \(k\ge K\).
Since \(L(k)-L(ak)=2e(k)/p\) and \(T_k\ge k\),
\[
 \Ebf[(T_k/k)^{-r}L(T_k)]
 \le (c_r+e(k))L(k)-p\{L(k)-L(ak)\}
 \le c_rL(k).
\]
Iterating this one-generation estimate for \(k^{-r}L(k)\) gives
\(\Ebf_K[\overline Z_n^{-r}]\le c_r^nK^{-r}L(K)\), where
\(\overline Z_n\) denotes the population of the BPRE without immigration.
Write \(\overline V_n^{(K)}=\overline Z_n/\Pi_n\).
The change of measure gives
\(\Ebf_{-r}[(\overline V_n^{(K)})^{-r}]=c_r^{-n}\Ebf_K[\overline Z_n^{-r}]\).
The almost sure convergence of the no-immigration martingale under
\(\Pbf_{-r}\) and Fatou's lemma now give
\begin{equation}\label{eq:bprei-harmonic-noimm-inverse}
 \Ebf_{-r}[(\overline V_\infty^{(K)})^{-r}]<\infty .
\end{equation}

The no-immigration martingale is uniformly integrable, so
\(\overline V_n^{(K)}
=\Ebf_{-r}[\overline V_\infty^{(K)}\mid\cF_n]\).  Convexity of
\(x\mapsto x^{-r}\) therefore gives
\begin{equation}\label{eq:bprei-harmonic-jensen-large-state}
        (\overline V_n^{(K)})^{-r}
        \leq
        \Ebf_{-r}[(\overline V_\infty^{(K)})^{-r}\mid\cF_n].
\end{equation}
For \(j\geq K\), couple the BPREI started from \(j\) individuals with
the no-immigration BPRE started from \(K\) individuals by retaining the first
\(K\) ancestral families and discarding immigration.  Then
\[
        V_n^{(j)}\geq\overline V_n^{(K)},
\]
and hence
\[
        (V_n^{(j)})^{-r}\leq(\overline V_n^{(K)})^{-r}.
\]
It follows from \eqref{eq:bprei-harmonic-jensen-large-state} and
\eqref{eq:bprei-harmonic-noimm-inverse} that
\[
\sup_{j\geq K}\sup_{n\geq0}
\Ebf_{-r,j}\!\left[(V_n^{(j)})^{-r}\right]<\infty,
\]
and that, for every fixed \(j\geq K\), the family
\(\{(V_n^{(j)})^{-r}:n\geq0\}\) is uniformly integrable.
Lemma~\ref{lem:bprei-harmonic-normalized-limit} gives
\[
        V_n^{(j)}\longrightarrow V_\infty^{(j)}
        \qquad \Pbf_{-r,j}\text{-almost surely}.
\]
Therefore, Vitali's theorem yields
\[
\Ebf_{-r,j}\!\left[(V_n^{(j)})^{-r}\right]
\longrightarrow
\Ebf_{-r,j}\!\left[(V_\infty^{(j)})^{-r}\right]
=U_j(r).
\]

It remains to translate these inverse-moment statements back to the original
law.  Under \textup{(H1)}, a process started from \(j\geq1\) remains
positive, and the definition of the negative tilt gives
\begin{equation}\label{eq:bprei-harmonic-change-measure}
\begin{aligned}
\Ebf_{-r,j}\!\left[(V_n^{(j)})^{-r}\right]
&=c_r^{-n}\Ebf_j\!\left[
  \Pi_n^{-r}\left(\frac{Z_n}{\Pi_n}\right)^{-r}\right]\\
&=c_r^{-n}\Ebf_j[Z_n^{-r}]
=c_r^{-n}H_{j,n}(r).
\end{aligned}
\end{equation}
Consequently, the inverse-moment convergence gives
\[
        c_r^{-n}H_{j,n}(r)\longrightarrow U_j(r),
\]
which is \eqref{eq:bprei-harmonic-large-state-limit}, while the uniform
inverse-moment bound gives
\eqref{eq:bprei-harmonic-large-state-uniform}.  The finiteness of \(U_j(r)\)
follows from the preceding uniform-integrability argument, and its strict
positivity follows from
\[
        0<V_\infty^{(j)}<\infty
        \qquad\Pbf_{-r,j}\text{-almost surely}
\]
in \eqref{eq:bprei-harmonic-normalized-positive}.
\end{proof}

For every \(i\geq1\), set
\(U_i(r):=\Ebf_{-r,i}[(V_\infty^{(i)})^{-r}]\in(0,\infty]\).
We now condition on the first exit from \(i\).  The resulting identity couples
the negative tilt with the renewal decomposition and avoids assuming inverse
uniform integrability before it has been proved.

\begin{lemma}[First exit under the negative tilt]
\label{lem:bprei-harmonic-tilted-first-exit}
Under the standing hypothesis \((\mathrm H)\), for every $i\geq1$
and $r>0$,
\begin{equation}\label{eq:bprei-harmonic-tilted-first-exit}
        U_i(r)
        =\sum_{\ell=0}^{\infty}
          \frac{\gamma_i^\ell}{c_r^{\ell+1}}
          \sum_{j>i}\mathsf p_{ij}U_j(r),
\end{equation}
where equality is understood in $[0,\infty]$.  Consequently, if
$c_r>\gamma_i$ and $\sum_{j>i}\mathsf p_{ij}U_j(r)<\infty$, then
\begin{equation}\label{eq:bprei-harmonic-Ui-recursion}
        U_i(r)
        =\frac{\sum_{j>i}\mathsf p_{ij}U_j(r)}{c_r-\gamma_i}
        <\infty.
\end{equation}
\end{lemma}

\begin{proof}
Let \(\tau_i:=\inf\{n\geq1:Z_n>i\}\) when \(Z_0=i\).  Under
\(\Pbf_{-r}\), the probability of remaining at \(i\) for one more generation
is
\[
        \widetilde\gamma_i(r)
        :=\frac{\Ebf[M^{-r}h_0p_1^i]}{c_r}<1,
\]
where strict inequality follows from
Condition~\ref{cond:bprei-harmonic-schroder}.  Hence
\(\Pbf_{-r,i}(\tau_i<\infty)=1\).

Now, on \(\{\tau_i=\ell+1,Z_{\ell+1}=j\}\), write
\(V_\infty^{(i)}
=\Pi_{\ell+1}^{-1}V_{\infty,\ell+1}^{(j)}\), where the second factor is
constructed from the shifted future and has the law of
\(V_\infty^{(j)}\).  The negative-tilt Radon--Nikodym derivative over the first
\(\ell+1\) generations cancels the factor \(\Pi_{\ell+1}^{r}\) arising
from \((V_\infty^{(i)})^{-r}\).  Independence of the shifted future then
gives
\[
\Ebf_{-r,i}\left[
(V_\infty^{(i)})^{-r};
\tau_i=\ell+1,\ Z_{\ell+1}=j
\right]
=
\frac{\gamma_i^\ell\mathsf p_{ij}}
     {c_r^{\ell+1}}U_j(r).
\]
Summing over \(\ell\geq0\) and \(j>i\) by Tonelli's theorem proves
\eqref{eq:bprei-harmonic-tilted-first-exit}.  Put
\[
        B_i(r):=\sum_{j>i}\mathsf p_{ij}U_j(r).
\]
This quantity is the post-exit contribution: it averages the future
inverse-moment constants \(U_j(r)\) over the possible states \(j>i\) entered
when the process first leaves \(i\).  In renewal-theory terminology,
\(B_i(r)/c_r\) is the forcing term in
\eqref{eq:bprei-harmonic-tilted-first-exit}, since that identity can be
written as
\[
        U_i(r)=\sum_{\ell=0}^{\infty}
        \left(\frac{\gamma_i}{c_r}\right)^\ell\frac{B_i(r)}{c_r}.
\]
Consequently, if \(c_r>\gamma_i\) and \(B_i(r)<\infty\), then
\[
\begin{aligned}
U_i(r)
&=\frac{B_i(r)}{c_r}
  \sum_{\ell=0}^{\infty}\left(\frac{\gamma_i}{c_r}\right)^\ell\\
&=\frac{B_i(r)}{c_r-\gamma_i},
\end{aligned}
\]
which is \eqref{eq:bprei-harmonic-Ui-recursion}.
\end{proof}

\subsubsection{First-exit renewal identities}

We now return to the Laplace transforms and isolate the exit-time regions that
produce the three scales.

\begin{lemma}[First-exit renewal identities]
\label{lem:bprei-harmonic-first-exit}
For every \(j\geq1\), \(n\geq0\), and \(t>0\),
\begin{equation}\label{eq:bprei-harmonic-first-exit-laplace}
\Phi_{j,n}(t)=\gamma_j^n\e^{-jt}
+\sum_{\ell=0}^{n-1}\gamma_j^\ell
\sum_{k>j}\mathsf p_{jk}\Phi_{k,n-\ell-1}(t).
\end{equation}
Let \(B_{j,n}(r):=\sum_{k>j}\mathsf p_{jk}H_{k,n}(r)\).  Then
\(H_{j,n+1}(r)=\gamma_jH_{j,n}(r)+B_{j,n}(r)\), and
\begin{equation}\label{eq:bprei-harmonic-renewal-iterated}
H_{j,n}(r)=j^{-r}\gamma_j^n
+\sum_{m=0}^{n-1}\gamma_j^{n-m-1}B_{j,m}(r).
\end{equation}
Moreover, \(H_{k,n}(r)\leq H_{j,n}(r)\) whenever \(k\geq j\).
\end{lemma}

\begin{proof}
Let \(\tau_j:=\inf\{m\geq1:Z_m>j\}\) under \(\Pbf_j\).  Since the process is
nondecreasing, on \(\{\tau_j>n\}\) one has
\(Z_0=\cdots=Z_n=j\), and therefore
\(\Ebf_j[\e^{-tZ_n};\tau_j>n]=\gamma_j^n\e^{-jt}\).  For
\(0\leq\ell\leq n-1\) and \(k>j\), the Markov property gives
\[
\Ebf_j[\e^{-tZ_n};\tau_j=\ell+1,Z_{\ell+1}=k]
=\gamma_j^\ell\mathsf p_{jk}\Phi_{k,n-\ell-1}(t).
\]
Summing over \(\ell\) and \(k\) proves
\eqref{eq:bprei-harmonic-first-exit-laplace}; the Laplace--Gamma identity
and iteration give the two harmonic relations.

For monotonicity, construct the processes from \(j\leq k\) with the same
environment, immigration variables, and offspring array.  If
\(Z_n^{(j)}\leq Z_n^{(k)}\), then
\(Z_{n+1}^{(k)}-Z_{n+1}^{(j)}
 =\sum_{a=Z_n^{(j)}+1}^{Z_n^{(k)}}\xi_{n,a}\geq0\).
Thus \(Z_n^{(j)}\leq Z_n^{(k)}\) for every \(n\), and negative powers give
\(H_{k,n}(r)\leq H_{j,n}(r)\).
\end{proof}

In \eqref{eq:bprei-harmonic-renewal-iterated}, the index \(m\) is the number
of generations available after the first exit.  Terms with bounded \(m\)
form the late-exit region, terms with bounded \(n-1-m\) form the early-exit
region, and the middle range contributes at leading order only when
\(c_r=\gamma_i\).  This geometric convolution yields exact constants rather
than only the three orders.

\subsubsection{Proof of the harmonic-moment theorem}

\begin{proof}[Proof of Theorem~\ref{thm:bprei-supercritical-harmonic}]
Fix \(r>0\). Choose \(K\) as in
Lemma~\ref{lem:bprei-harmonic-large-state}. For \(j\ge K\),
\(\gamma_j=\Ebf[h_0p_1^j]\le\Ebf[p_1^K]<c_r\), and that lemma gives
\(c_r^{-n}H_{j,n}(r)\to U_j(r)\in(0,\infty)\).
We proceed by downward induction on \(i\), proving the three limits in
\eqref{eq:bprei-supercritical-harmonic-main} and, whenever \(c_r>\gamma_i\),
identifying the limit as \(U_i(r)\).

Suppose the assertion holds for all \(j>i\). If \(c_r\ge\gamma_i\),
then \(\gamma_j\le\gamma_{i+1}<\gamma_i\le c_r\) for every \(j>i\).
Thus \(c_r^{-n}H_{j,n}(r)\to U_j(r)\) for each such \(j\).
Moreover, by Lemma~\ref{lem:bprei-harmonic-first-exit},
\[
 0\le c_r^{-n}H_{j,n}(r)\le c_r^{-n}H_{i+1,n}(r)
 \le \sup_{m\ge0} c_r^{-m}H_{i+1,m}(r)<\infty .
\]
Dominated convergence against the transition probabilities
\(\mathsf p_{ij}\) gives
\[
 c_r^{-n}B_{i,n}(r)\longrightarrow
 B_i(r):=\sum_{j>i}\mathsf p_{ij}U_j(r)\in(0,\infty),
\]
where strict positivity follows from
\(\sum_{j>i}\mathsf p_{ij}=1-\gamma_i>0\).

If \(c_r>\gamma_i\), divide \eqref{eq:bprei-harmonic-renewal-iterated}
by \(c_r^n\). The initial term tends to zero, and the convolution of
\(c_r^{-m}B_{i,m}(r)\) with the summable geometric weights gives
\[
 \lim_{n\to\infty}c_r^{-n}H_{i,n}(r)
 =\frac{B_i(r)}{c_r}\sum_{\ell=0}^\infty
       \left(\frac{\gamma_i}{c_r}\right)^\ell
 =\frac{B_i(r)}{c_r-\gamma_i}
 =U_i(r),
\]
the last identity following from
Lemma~\ref{lem:bprei-harmonic-tilted-first-exit}.
In this regime, \eqref{eq:bprei-harmonic-change-measure} and
Lemma~\ref{lem:bprei-harmonic-normalized-limit} also give
\((V_n^{(i)})^{-r}\to(V_\infty^{(i)})^{-r}\) in
\(L^1(\Pbf_{-r,i})\) by Scheff\'e's lemma.
If \(c_r=\gamma_i\), the same renewal identity gives the Ces\`aro limit
\[
 \frac{H_{i,n}(r)}{n c_r^n}
 =\frac{i^{-r}}n+\frac1{n c_r}\sum_{m=0}^{n-1}c_r^{-m}B_{i,m}(r)
 \longrightarrow \frac{B_i(r)}{c_r}.
\]

Finally, if \(c_r<\gamma_i\), put \(d_i=c_r\vee\gamma_{i+1}<\gamma_i\).
The induction hypothesis at \(i+1\), including its boundary case, gives
\(H_{i+1,m}(r)\le C(m+1)d_i^m\).
By monotonicity, \(B_{i,m}(r)\le H_{i+1,m}(r)\), so
\(\sum_{m\ge0}\gamma_i^{-m}B_{i,m}(r)<\infty\).
Dividing \eqref{eq:bprei-harmonic-renewal-iterated} by \(\gamma_i^n\)
therefore gives
\[
 \gamma_i^{-n}H_{i,n}(r)\longrightarrow
 i^{-r}+\frac1{\gamma_i}\sum_{m=0}^\infty\gamma_i^{-m}B_{i,m}(r)
 \in(0,\infty).
\]
This completes the induction. The strict decrease and continuity of
\(r\mapsto c_r\) identify the three cases with the ranges of \(r\) in
\eqref{eq:bprei-supercritical-harmonic-main}.
\end{proof}

\subsubsection{Identification of the constants}

\begin{proposition}[Representations of the constants]
\label{prop:bprei-harmonic-gamma-constants}
Under the assumptions of Theorem~\ref{thm:bprei-supercritical-harmonic}, let
\(V_\infty^{(j)}\) be the tilted limit in
Lemma~\ref{lem:bprei-harmonic-normalized-limit} and put
\(U_j(r):=\Ebf_{-r,j}[(V_\infty^{(j)})^{-r}]\).  In the environmental regime,
\begin{equation}\label{eq:bprei-harmonic-env-constant}
        C^{\mathrm{env}}_{i,r}=U_i(r)\in(0,\infty).
\end{equation}
At the boundary \(c_r=\gamma_i\),
\begin{equation}\label{eq:bprei-harmonic-boundary-constant}
        C^{\mathrm{bd}}_i
        =\frac1{c_r}\sum_{j>i}\mathsf p_{ij}U_j(r).
\end{equation}
In the small-state regime,
\begin{equation}\label{eq:bprei-harmonic-small-constant}
        C^{\mathrm{sm}}_{i,r}
        =i^{-r}
        +\frac1{\gamma_i}\sum_{m=0}^{\infty}\gamma_i^{-m}
          \sum_{j>i}\mathsf p_{ij}H_{j,m}(r).
\end{equation}
These constants have the following Laplace--Gamma forms.  If
\(c_r>\gamma_i\), let
\[
        \phi_{i,r}(t)
        :=\Ebf_{-r,i}[\e^{-tV_\infty^{(i)}}].
\]
Then
\begin{equation}\label{eq:bprei-harmonic-env-gamma-constant}
        C^{\mathrm{env}}_{i,r}
        =\frac1{\Gamma(r)}
          \int_0^\infty\phi_{i,r}(t)t^{r-1}\dd t.
\end{equation}

If $c_r=\gamma_i$, then
\begin{equation}\label{eq:bprei-harmonic-boundary-gamma-constant}
        C^{\mathrm{bd}}_i
        =\frac1{c_r\Gamma(r)}
          \int_0^\infty
          \left\{\sum_{j>i}\mathsf p_{ij}\phi_{j,r}(t)\right\}
          t^{r-1}\dd t.
\end{equation}

If $c_r<\gamma_i$, define
\begin{equation}\label{eq:bprei-harmonic-Qi}
        \mathcal Q_i(t)
        :=\e^{-it}
        +\frac1{\gamma_i}
          \sum_{m=0}^{\infty}\gamma_i^{-m}
          \sum_{j>i}\mathsf p_{ij}\Phi_{j,m}(t),
        \qquad t>0.
\end{equation}
Then
\begin{equation}\label{eq:bprei-harmonic-small-gamma-constant}
        C^{\mathrm{sm}}_{i,r}
        =\frac1{\Gamma(r)}
          \int_0^\infty\mathcal Q_i(t)t^{r-1}\dd t,
\end{equation}
and
\begin{equation}\label{eq:bprei-harmonic-fixed-laplace-limit}
        \gamma_i^{-n}\Phi_{i,n}(t)
        \uparrow\mathcal Q_i(t),
        \qquad t>0.
\end{equation}
\end{proposition}

\begin{proof}
If \(c_r>\gamma_i\), the Gamma identity gives
\[
C^{\mathrm{env}}_{i,r}=U_i(r)
=\frac1{\Gamma(r)}\int_0^\infty
\Ebf_{-r,i}[\e^{-tV_\infty^{(i)}}]t^{r-1}\dd t.
\]
If \(c_r=\gamma_i\), Tonelli and the preceding identity for \(j>i\) give
\[
C_i^{\mathrm{bd}}
=\frac1{c_r}\sum_{j>i}\mathsf p_{ij}U_j(r)
=\frac1{c_r\Gamma(r)}\int_0^\infty
\left\{\sum_{j>i}\mathsf p_{ij}\phi_{j,r}(t)\right\}t^{r-1}\dd t.
\]
If \(c_r<\gamma_i\), reindex
\eqref{eq:bprei-harmonic-first-exit-laplace} to obtain
\[
\gamma_i^{-n}\Phi_{i,n}(t)
=\e^{-it}+\frac1{\gamma_i}\sum_{m=0}^{n-1}\gamma_i^{-m}
\sum_{j>i}\mathsf p_{ij}\Phi_{j,m}(t)\uparrow\mathcal Q_i(t).
\]
For \(x\geq1\), \(\e^{-tx}\leq C_{r,t}x^{-r}\); hence
\[
\sum_{m\geq0}\gamma_i^{-m}
\sum_{j>i}\mathsf p_{ij}\Phi_{j,m}(t)
\leq C_{r,t}\sum_{m\geq0}\gamma_i^{-m}B_{i,m}(r)<\infty.
\]
Tonelli and \eqref{eq:bprei-laplace-gamma} then give
\[
\frac1{\Gamma(r)}\int_0^\infty\mathcal Q_i(t)t^{r-1}\dd t
=i^{-r}+\frac1{\gamma_i}\sum_{m\geq0}\gamma_i^{-m}
\sum_{j>i}\mathsf p_{ij}H_{j,m}(r),
\]
which is \eqref{eq:bprei-harmonic-small-constant}.
\end{proof}

\begin{remark}[Equality of persistence probabilities]
\label{rem:bprei-harmonic-repeated-persistence}
Under \textup{(H1)}, if two distinct positive states have the same
persistence probability, then all positive states do, and the common
value is strictly less than \(c_r\) for every \(r>0\). Indeed, for \(i\ge1\),
\[
 \gamma_i-\gamma_{i+1}=\Ebf[h_0p_1^i(1-p_1)]\ge0.
\]
Equality at two distinct positive states forces this difference to vanish
for some \(i\ge1\), so \(h_0=0\) almost surely on \(\{0<p_1<1\}\).
Since \(p_0=0\) almost surely, \(\{p_1=1\}=\{M=1\}\)
up to null sets. Hence, for every \(j\ge1\),
\[
 \gamma_j=\Ebf[h_0;M=1]\le\Pbf(M=1)<c_r,
\]
where the strict inequality follows from \(\Pbf(M>1)>0\) and
\(c_r=\Pbf(M=1)+\Ebf[M^{-r};M>1]\).
\end{remark}

\begin{remark}[Downward transitions]
\label{rem:bprei-harmonic-p0-positive}
If downward transitions are possible, a first exit from \(i\) may lead to
a smaller state, and the process may subsequently return to \(i\).
At state \(i\), the downward induction in
Theorem~\ref{thm:bprei-supercritical-harmonic} provides estimates only for
\(j>i\); these no longer cover every possible exit state. Knowing the exponential
rate of \(G_{i,n}(t)=\Ebf_i[\e^{-tZ_n};Z_n>0]\) for each fixed \(t>0\)
does not by itself determine the asymptotics of its integral against
\(t^{r-1}\dd t\) in \eqref{eq:bprei-laplace-gamma}.
\end{remark}

This geometric renewal argument relies on the nondecreasing population sizes
and on the fact that every negative \(r\)-tilt remains supercritical.  In a
critical environment the environmental walk oscillates, and a small
terminal population is organized by the finite-horizon minimum of that walk.
The next subsection replaces the geometric renewal kernel by a regularly
varying ladder convolution, bringing Spitzer--Doney and local fluctuation
estimates into the analysis.

\subsection{Critical BPREI and the finite-horizon minimum decomposition}
\label{subsec:bprei-critical-harmonic}

We now consider a critical BPREI.  Recall that
\(S_0=0\) and
\(S_n:=\log\Pi_n=\sum_{j=0}^{n-1}\log M_j\) is the environmental random
walk.  For a fixed terminal generation \(n\), we decompose according to the
first time at which \(S\) attains its minimum over
\(\{0,\ldots,n\}\).  The resulting ladder convolution separates the
fluctuation theory of the environment from the lower tail of the
post-minimum population.

For \(n\geq0\), put
\begin{equation}\label{eq:critical-minimum-time}
        L_n:=\min_{0\leq j\leq n}S_j,
        \qquad
        \tau_n:=\min\{0\leq j\leq n:S_j=L_n\}.
\end{equation}
Set \(d_0^-=1\). For \(n\ge1\), time reversal gives
\begin{equation}\label{eq:critical-dminus}
 d_n^-:=\Pbf(\tau_n=n)=\Pbf\!\left(\max_{1\le j\le n}S_j<0\right).
\end{equation}
Indeed, \(\tau_n=n\) precisely when \(S_n<S_j\) for every \(j<n\),
so all the reversed partial sums \(S_n-S_{n-j}\), \(1\le j\le n\),
are strictly negative.

\begin{assumption}[SD: Spitzer--Doney condition]
\label{ass:bprei-critical-SD}
The walk \(S\) is nondegenerate and oscillating, and for some
\(\varrho\in(0,1)\),
\begin{equation}\label{eq:critical-SD}
 \Pbf(S_n>0)\longrightarrow\varrho .
\end{equation}
Equivalently, one may impose Spitzer's averaged condition
\cite{Doney1995}.
\end{assumption}

Under \(\mathrm{(SD)}\), there are positive slowly varying functions
\(\ell_-\) and \(\ell_+\) such that, for \(n\ge1\),
\begin{equation}\label{eq:critical-ladder-rv}
 d_n^-=n^{-\varrho}\ell_-(n),\qquad
 \Pbf(L_n\ge0)=n^{-(1-\varrho)}\ell_+(n);
\end{equation}
see \cite{Doney1995,Feller1971,AfanasyevGeigerKerstingVatutin2005}.

\begin{remark}[Centered finite variance]
\label{rem:bprei-critical-finite-variance}
If \(\Ebf\log M=0\) and \(0<\Var(\log M)<\infty\), then
\(\varrho=1/2\) and
\(d_n^-\sim\kappa_0n^{-1/2}\) for some \(\kappa_0\in(0,\infty)\); see
\cite[Chapter~XII]{Feller1971}.  No arithmetic restriction is needed for
this persistence asymptotic.
\end{remark}

If \(p_0=Q(\{0\})=0\) almost surely, then \(M\geq1\) almost surely.  The
critical condition \(\Ebf\log M=0\) would then force \(M=1\) almost surely
and the integer-valued reproduction law to be concentrated at one.  Thus a
nondegenerate critical environment necessarily satisfies
\begin{equation}\label{eq:critical-p0-positive}
        \Pbf(p_0>0)>0.
\end{equation}

For \(z\in\Nzero\), \(t>0\), and \(m\geq0\), define
\begin{equation}\label{eq:critical-Phi}
        \Phi^{\mathrm{min}}_{m,t}(z)
        :=\Ebf_z[\e^{-tZ_m};Z_m>0,L_m\geq0],
\end{equation}
and, for \(r>0\),
\begin{equation}\label{eq:critical-Psi}
        \Psi_{m,r}(z)
        :=\Ebf_z[Z_m^{-r};Z_m>0,L_m\geq0].
\end{equation}
Since \(S\) oscillates, \(\Pbf(S_1<0)>0\). Taking \(k\) strictly
negative increments gives \(d_k^-\ge\Pbf(S_1<0)^k>0\) for \(k\ge1\).
Thus, for every \(k\ge0\), the conditional law
\begin{equation}\label{eq:critical-entrance-law-finite}
 \nu_{i,k}^-(A):=\Pbf_i(Z_k\in A\mid\tau_k=k),
 \qquad A\subseteq\Nzero,
\end{equation}
is well defined, with \(\nu_{i,0}^-=\delta_i\).

\begin{assumption}[EL: Entrance law at an endpoint minimum]
\label{ass:bprei-critical-entrance}
For the fixed initial state \(i\), there is a probability law
\(\nu_i^-\) on \(\Nzero\) such that
\begin{equation}\label{eq:critical-entrance-TV}
 \|\nu_{i,k}^--\nu_i^-\|_{\mathrm{TV}}\longrightarrow0
 \qquad (k\to\infty).
\end{equation}
For some \(m_0\ge0\),
\(\Pbf_{\nu_i^-}(Z_{m_0}>0,L_{m_0}\ge0)>0\).
\end{assumption}

Write
\(\mathfrak N_i^{\mathrm{fin}}:=\{\nu_{i,k}^-:k\geq0\}\), and, when the
limiting law is available, put
\(\mathfrak N_i:=\mathfrak N_i^{\mathrm{fin}}\cup\{\nu_i^-\}\).

\begin{assumption}[UIM: Uniform post-minimum inverse moments]
\label{ass:bprei-critical-UIM}
There is \(s_0>0\) such that, for every \(s\in(0,s_0]\), there is a
nonnegative summable sequence \((\overline a_m(s))_{m\ge0}\) satisfying
\begin{equation}\label{eq:critical-UIM}
 \sup_{\nu\in\mathfrak N_i}
 \Ebf_\nu[Z_m^{-s};Z_m>0,L_m\ge0]\le\overline a_m(s),
 \qquad m\ge0 .
\end{equation}
When \(\mathrm{(EL)}\) is not assumed, take the supremum over
\(\mathfrak N_i^{\mathrm{fin}}\).
\end{assumption}

The bound in \eqref{eq:critical-UIM} controls the lower tail of the
post-minimum population. Indeed, Tonelli's theorem gives
\[
 \Ebf_\nu[Z_m^{-s};Z_m>0,L_m\ge0]
 =s\int_1^\infty x^{-s-1}\Pbf_\nu(0<Z_m\le x,L_m\ge0)\dd x .
\]

\begin{assumption}[LOC: Convolution-tail localization]
\label{ass:bprei-critical-terminal-localization}
For every \(s\in(0,s_0]\), the majorant in \eqref{eq:critical-UIM}
satisfies
\begin{equation}\label{eq:critical-terminal-localization}
 \frac1{d_n^-}\sum_{k=0}^{\lfloor n/2\rfloor}
       d_k^-\overline a_{n-k}(s)\longrightarrow0 .
\end{equation}
\end{assumption}

Write \((\mathrm C)\) for \(\mathrm{(SD)}\), \(\mathrm{(EL)}\),
\(\mathrm{(UIM)}\), and \(\mathrm{(LOC)}\) together. To see the
distinct roles of the last two conditions, write \(m=n-k\) for the
number of generations after the minimum. On \(m\le n/2\),
\(d_{n-m}^-/d_n^-\) is uniformly bounded by regular variation, so
summability controls the terms with \(N<m\le n/2\) as \(N\to\infty\).
Condition \(\mathrm{(LOC)}\) controls the remaining terms \(m>n/2\),
where \(n-m\) can be small.

\begin{remark}[A verified critical class]
Section~\ref{subsubsec:bprei-critical-fv-class} introduces a centered
finite-variance class \(\mathcal C_{\mathrm{fv}}\) of BPREI.
Theorem~\ref{thm:bprei-critical-fv-class} below, using the results proved in
Appendix~\ref{sec:critical-verification}, shows that every BPREI in this class
satisfies the assumptions \((\mathrm C)\), namely \((\mathrm{SD})\),
\((\mathrm{EL})\), \((\mathrm{UIM})\), and \((\mathrm{LOC})\).
\end{remark}

Define
\begin{equation}\label{eq:critical-qim}
\begin{aligned}
        q_{i,m}(t)
        &:=\int_{\Nzero}\Phi^{\mathrm{min}}_{m,t}(z)\,\nu_i^-(\dd z),
        &Q_i(t)&:=\sum_{m=0}^{\infty}q_{i,m}(t),\\
        h_{i,m}(r)
        &:=\int_{\Nzero}\Psi_{m,r}(z)\,\nu_i^-(\dd z).
\end{aligned}
\end{equation}

\begin{theorem}[Critical Laplace profiles and harmonic moments]
\label{thm:bprei-critical-harmonic}
Assume \((\mathrm C)\).  Then, locally uniformly for \(t\in(0,\infty)\),
\begin{equation}\label{eq:critical-Laplace-limit}
        \frac{G_{i,n}(t)}{d_n^-}\longrightarrow Q_i(t)\in(0,\infty).
\end{equation}
Moreover, for every \(r>0\),
\begin{equation}\label{eq:critical-harmonic-main}
        \frac{\Ebf_i[Z_n^{-r};Z_n>0]}{d_n^-}
        \longrightarrow C_{i,r}\in(0,\infty),
\end{equation}
where
\begin{equation}\label{eq:critical-harmonic-constant}
        C_{i,r}
        =\sum_{m=0}^{\infty}h_{i,m}(r)
        =\frac1{\Gamma(r)}\int_0^\infty Q_i(t)t^{r-1}\dd t.
\end{equation}
\end{theorem}

\begin{corollary}[Critical transfer for bounded lower-tail kernels]
\label{cor:bprei-critical-bounded-kernel}
Assume \((\mathrm C)\). Let \(g:\Nzero\to[0,\infty)\) satisfy \(g(0)=0\)
and \(g(k)\le C_g k^{-q}\) for \(k\ge1\), for some \(q>0\) and \(C_g<\infty\).
Then
\begin{equation}\label{eq:critical-bounded-kernel-limit}
 \frac{\Ebf_i[g(Z_n)]}{d_n^-}\longrightarrow
 \sum_{m=0}^{\infty}\Ebf_{\nu_i^-}[g(Z_m);L_m\ge0]<\infty .
\end{equation}
The limit is strictly positive if and only if there exist integers \(m\ge0\)
and \(k\ge1\) such that \(g(k)>0\) and
\(\Pbf_{\nu_i^-}(Z_m=k,L_m\ge0)>0\).
\end{corollary}

\begin{theorem}[Critical order bounds without an entrance law]
\label{thm:bprei-critical-two-sided}
Fix \(i\in\Nzero\). Assume \textup{(SD)}, \textup{(UIM)}, and
\textup{(LOC)}, with \textup{(UIM)} taken over
\(\mathfrak N_i^{\mathrm{fin}}\). Then, for every \(r>0\),
\begin{equation}\label{eq:critical-harmonic-upper}
 \sup_{n\ge1}\frac{\Ebf_i[Z_n^{-r};Z_n>0]}{d_n^-}<\infty .
\end{equation}
If, in addition, \(\Pbf_i(Z_n\ge1)=1\) for every \(n\ge1\) and
\(\kappa_i:=\sup_{n\ge1}\Ebf_i[Z_n\mid\tau_n=n]<\infty\), then
\begin{equation}\label{eq:critical-harmonic-two-sided}
 0<\inf_{n\ge1}\frac{\Ebf_i[Z_n^{-r}]}{d_n^-}
 \le\sup_{n\ge1}\frac{\Ebf_i[Z_n^{-r}]}{d_n^-}<\infty .
\end{equation}
\end{theorem}

\subsubsection{The exact minimum decomposition}

\begin{lemma}[Exact finite-horizon minimum decomposition]
\label{lem:bprei-critical-minimum-decomposition}
For every bounded \(f:\Nzero\to\R\) and every \(n\geq0\),
\begin{equation}\label{eq:critical-minimum-decomposition}
\Ebf_i[f(Z_n);Z_n>0]
=\sum_{m=0}^{n}d_{n-m}^-
  \int_{\Nzero}
  \Ebf_z[f(Z_m);Z_m>0,L_m\geq0]\,
  \nu_{i,n-m}^-(\dd z).
\end{equation}
\end{lemma}

\begin{proof}
Set \(k=n-m\) and write
\(\Phi_{m,f}(z)=\Ebf_z[f(Z_m);Z_m>0,L_m\ge0]\).
By first attainment,
\[
 \{\tau_n=k\}
 =\{S_j>S_k\text{ for all }0\le j<k\}
   \cap\{S_{k+j}-S_k\ge0\text{ for all }1\le j\le m\}.
\]
The first event is \(\{\tau_k=k\}\); ties with \(S_k\) are allowed
after time \(k\). An empty set of inequalities imposes no restriction.
Conditional on the history through time \(k\) and on \(Z_k=z\), the shifted
environmental marks and reproduction--immigration variables are independent
of the past and have the original \(m\)-generation BPREI law.  Therefore
\[
\Ebf_i[f(Z_n);Z_n>0,\tau_n=k]
=\Ebf_i[\one_{\{\tau_k=k\}}\Phi_{m,f}(Z_k)]
=d_k^-\int_{\Nzero}\Phi_{m,f}(z)\,\nu_{i,k}^-(\dd z).
\]
Sum over \(k=n-m\), \(0\leq m\leq n\).
\end{proof}

\begin{lemma}[Localized convolution with a regularly varying sequence]
\label{lem:bprei-critical-summable-convolution}
Let \(d_0>0\) and \(d_n=n^{-\varrho}\ell(n)\) for \(n\ge1\), where
\(0<\varrho<1\) and \(\ell\) is positive and slowly varying. Let
\(a_{n,m}\ge0\) for \(0\le m\le n\), and suppose there are nonnegative
sequences \((a_m)_{m\ge0}\) and \((\overline a_m)_{m\ge0}\) such that
\begin{enumerate}[label=\textup{(\roman*)}]
\item \(a_{n,m}\to a_m\) for each fixed \(m\);
\item \(a_{n,m}\le\overline a_m\) and
      \(\sum_{m\ge0}\overline a_m<\infty\);
\item \(d_n^{-1}\sum_{k=0}^{\lfloor n/2\rfloor}
      d_k\overline a_{n-k}\to0\).
\end{enumerate}
Then
\[
 \frac1{d_n}\sum_{m=0}^nd_{n-m}a_{n,m}
 \longrightarrow\sum_{m\ge0}a_m .
\]
If the array depends on a parameter, the convergence is uniform on
each compact parameter set on which \textup{(i)} is uniform for every
fixed \(m\) and a common majorant satisfies \textup{(ii)}--\textup{(iii)}.
\end{lemma}

\begin{proof}
Let
\(T_n:=d_n^{-1}\sum_{m=0}^{n}d_{n-m}a_{n,m}\).  For fixed \(N\), split
\(T_n=T_{n,1}(N)+T_{n,2}(N)+T_{n,3}\) over
\(m\leq N\), \(N<m\leq n/2\), and \(m>n/2\).  For the first part,
\[
T_{n,1}(N)=\sum_{m=0}^{N}\frac{d_{n-m}}{d_n}a_{n,m}
\longrightarrow\sum_{m=0}^{N}a_m,
\]
because \(d_{n-m}/d_n\to1\) for fixed \(m\).  The uniform convergence
theorem for regularly varying sequences gives
\(\sup_{0\leq m\leq n/2}d_{n-m}/d_n\leq C\), so
\(T_{n,2}(N)\leq C\sum_{m>N}\overline a_m\).  Finally, with
\(k=n-m\),
\(T_{n,3}\leq d_n^{-1}\sum_{k<n/2}d_k\overline a_{n-k}\to0\) by
\textup{(iii)}.  Hence
\[
\limsup_{n\to\infty}
\left|T_n-\sum_{m\geq0}a_m\right|
\leq C\sum_{m>N}\overline a_m+\sum_{m>N}a_m.
\]
Let \(N\to\infty\).  The locally uniform statement follows from the same
three-part decomposition when the fixed-\(m\) convergence is uniform on
compact parameter sets.
\end{proof}

\begin{proof}[Proof of Theorem~\ref{thm:bprei-critical-harmonic}]
Fix a compact \(K\Subset(0,\infty)\), set \(t_K:=\min K\), choose
\(s\in(0,s_0]\), and define
\[
a_{n,m}(t):=\int_{\Nzero}\Phi^{\mathrm{min}}_{m,t}(z)
\,\nu_{i,n-m}^-(\dd z).
\]
By \eqref{eq:critical-minimum-decomposition},
\(G_{i,n}(t)=\sum_{m=0}^{n}d_{n-m}^-a_{n,m}(t)\).  For fixed \(m\),
\[
\sup_{t\in K}|a_{n,m}(t)-q_{i,m}(t)|
\leq\|\nu_{i,n-m}^--\nu_i^-\|_{\mathrm{TV}}\longrightarrow0.
\]
For \(z\geq1\) and \(t\in K\),
\(\e^{-tz}\leq(s/(\e t_K))^sz^{-s}\); hence
\(a_{n,m}(t)\leq(s/(\e t_K))^s\overline a_m(s)\).  Assumption
\ref{ass:bprei-critical-terminal-localization} supplies condition
\textup{(iii)} of Lemma~\ref{lem:bprei-critical-summable-convolution}.
Thus
\[
\sup_{t\in K}\left|\frac{G_{i,n}(t)}{d_n^-}-Q_i(t)\right|
\longrightarrow0.
\]
The series defining \(Q_i(t)\) is finite by the same majorant, and
\(q_{i,m_0}(t)>0\) by the nontriviality clause of
Assumption~\ref{ass:bprei-critical-entrance}.

For harmonic moments, fix \(r>0\), choose
\(s\in(0,\min\{r,s_0\})\), and set
\[
b_{n,m}(r):=\int_{\Nzero}\Psi_{m,r}(z)\,\nu_{i,n-m}^-(\dd z).
\]
Then
\(\Ebf_i[Z_n^{-r};Z_n>0]=\sum_{m=0}^{n}d_{n-m}^-b_{n,m}(r)\),
\(b_{n,m}(r)\to h_{i,m}(r)\) for fixed \(m\), and
\(0\leq b_{n,m}(r)\leq\overline a_m(s)\) because
\(z^{-r}\leq z^{-s}\) on \(\{z\geq1\}\).  The convolution lemma yields
\[
\frac{\Ebf_i[Z_n^{-r};Z_n>0]}{d_n^-}
\longrightarrow\sum_{m\geq0}h_{i,m}(r).
\]
Finally, Tonelli and the Gamma identity give
\[
\int_0^\infty Q_i(t)t^{r-1}\dd t
=\Gamma(r)\sum_{m\geq0}h_{i,m}(r).
\]
Summability gives finiteness, and the entrance-law nontriviality gives
positivity.
\end{proof}

\begin{proof}[Proof of Corollary~\ref{cor:bprei-critical-bounded-kernel}]
Let \(a_{n,m}:=\Ebf_{\nu_{i,n-m}^-}[g(Z_m);L_m\ge0]\).
Since \(g(0)=0\), \eqref{eq:critical-minimum-decomposition} gives
\(\Ebf_i[g(Z_n)] = \sum_{m=0}^n d_{n-m}^-a_{n,m}\).

For fixed \(m\), \textup{(EL)} and boundedness of \(g\) give
\(a_{n,m}\to\Ebf_{\nu_i^-}[g(Z_m);L_m\ge0]\). With
\(s=\min\{q,s_0\}\), \textup{(UIM)} gives
\(a_{n,m}\le C_g\overline a_m(s)\). Hence \textup{(LOC)} and
Lemma~\ref{lem:bprei-critical-summable-convolution} give
\eqref{eq:critical-bounded-kernel-limit}. The positivity criterion follows
by writing the \(m\)th summand as
\[
 \sum_{k\ge1}g(k)\Pbf_{\nu_i^-}(Z_m=k,L_m\ge0).
 \qedhere
\]
\end{proof}

\begin{proof}[Proof of Theorem~\ref{thm:bprei-critical-two-sided}]
Fix \(r>0\) and put \(s=\min\{r,s_0\}\). Since \(z^{-r}\le z^{-s}\)
for \(z\ge1\), \eqref{eq:critical-minimum-decomposition} and
\textup{(UIM)} give
\[
 \frac{\Ebf_i[Z_n^{-r};Z_n>0]}{d_n^-}
 \le\frac1{d_n^-}\sum_{m=0}^n d_{n-m}^-\overline a_m(s).
\]
By \textup{(SD)}, \textup{(LOC)}, and
Lemma~\ref{lem:bprei-critical-summable-convolution} applied to
\(a_{n,m}=\overline a_m(s)\), the right-hand side converges to
\(\sum_{m\ge0}\overline a_m(s)<\infty\). For the lower bound, conditional
Jensen's inequality gives, for every \(n\ge1\),
\[
 \frac{\Ebf_i[Z_n^{-r}]}{d_n^-}
 \ge \Ebf_i[Z_n^{-r}\mid\tau_n=n]
 \ge \bigl(\Ebf_i[Z_n\mid\tau_n=n]\bigr)^{-r}
 \ge \kappa_i^{-r}>0 .
\]
\end{proof}

\subsubsection{A verified centered finite-variance class}
\label{subsubsec:bprei-critical-fv-class}

Let \(\mathcal C_{\mathrm{fv}}\) be the class of critical BPREI satisfying:
\begin{enumerate}[label=\textup{(C\arabic*)}]
\item \(\Ebf\log M=0\) and
      \(0<\Var(\log M)<\infty\);
\item the normalized quenched offspring variance is bounded,
      \(\Var(\xi_{0,1}\mid\cE_0)/M_0^2\leq\overline C\) almost surely;
\item immigration is present in every generation and has bounded quenched
      mean: \(\eta_{n+1}\geq1\) and
      \(\Ebf[\eta_{n+1}\mid\cE_n]\leq\overline\eta\) almost surely;
\item \(\Ebf(\log M)_+^q<\infty\) for some \(q>12\).
\end{enumerate}

The stronger condition
\(\Ebf\exp\{\theta(\log M)_+\}<\infty\) for some \(\theta>0\)
implies \textup{(C4)} for every finite \(q\).  In that case
Proposition~\ref{prop:bprei-critical-fv-uim}\textup{(i)} gives the sharper
majorant \(C_sm^{-3/2}[1+\log(m+1)]^4\) for \(m\geq1\).

\begin{theorem}[Centered finite-variance critical BPREI]
\label{thm:bprei-critical-fv-class}
For a BPREI in \(\mathcal C_{\mathrm{fv}}\), the assumptions
\((\mathrm C)\) hold for every fixed \(i\in\Nzero\), with a common
endpoint-minimum entrance law \(\nu^-\) on \(\{1,2,\ldots\}\):
\(\|\nu_{i,n}^--\nu^-\|_{\mathrm{TV}}\to0\). Write
\(\kappa_0:=\lim_{n\to\infty}\sqrt n\,d_n^-\in(0,\infty)\).
Then for every \(r>0\),
\begin{equation}\label{eq:critical-harmonic-finite-variance}
  \sqrt n\,\Ebf_i[Z_n^{-r};Z_n>0]\longrightarrow\kappa_0 C_r,
  \qquad
  C_r:=\sum_{m=0}^{\infty}
        \Ebf_{\nu^-}[Z_m^{-r};Z_m>0,L_m\geq0]\in(0,\infty).
\end{equation}
Moreover, locally uniformly for \(t\in(0,\infty)\),
\[
  \sqrt n\,G_{i,n}(t)\longrightarrow\kappa_0 Q(t),
  \qquad
  Q(t):=\sum_{m=0}^{\infty}
        \Ebf_{\nu^-}[\e^{-tZ_m};Z_m>0,L_m\geq0]\in(0,\infty).
\]
Neither \(C_r\) nor \(Q\) depends on the initial state \(i\).
\end{theorem}

\begin{proof}
Under \textup{(C1)}, the central limit theorem gives \textup{(SD)} with
\(\varrho=1/2\), and the persistence estimate in
Remark~\ref{rem:bprei-critical-finite-variance} gives
\(\sqrt n\,d_n^-\to\kappa_0\). Proposition~\ref{prop:bprei-critical-fv-uim}\textup{(ii)}
gives \textup{(UIM)} and \textup{(LOC)} uniformly over the initial laws.
By Proposition~\ref{prop:bprei-critical-fv-entrance}, the conditional
laws \(\nu_{i,n}^-\) converge in total variation to the same \(\nu^-\)
for all fixed \(i\). Since \(\nu^-(\{1,2,\ldots\})=1\), the positivity
clause of \textup{(EL)} holds with \(m_0=0\).
Theorem~\ref{thm:bprei-critical-harmonic} now gives both limits after
multiplication by \(\sqrt n\,d_n^-\), with local uniformity in \(t\) for
the Laplace limit.
\end{proof}

For \(g\) as in Corollary~\ref{cor:bprei-critical-bounded-kernel}, the same
normalization gives
\[
  \sqrt n\,\Ebf_i[g(Z_n)]\longrightarrow\kappa_0 K_g,\qquad
  K_g:=\sum_{m=0}^\infty\Ebf_{\nu^-}[g(Z_m);L_m\geq0].
\]
The constant \(K_g\) is independent of \(i\); it is positive exactly when
\(\Pbf_{\nu^-}(g(Z_m)>0,L_m\geq0)>0\) for some \(m\geq0\).

\begin{remark}[Role of the positive-tail moment]
The bound \(q>12\) in \textup{(C4)} suffices for the summable estimate
in Proposition~\ref{prop:bprei-critical-fv-uim}\textup{(ii)}; the argument
does not establish its optimality. The entrance-law construction uses only
\textup{(C1)} and \textup{(C3)}. Theorem~\ref{thm:bprei-critical-fv-class}
requires no lattice restriction on \(\log M\).
\end{remark}

\subsubsection{Conditional nonsummable post-minimum profiles}

When the post-minimum harmonic profile is not summable, the duration of the
post-minimum block can affect the normalization. The following result gives
the two scales for a regularly varying profile.

\begin{proposition}[Regularly varying nonsummable post-minimum profile]
\label{prop:bprei-critical-divergent-profile}
Fix \(i\in\Nzero\) and \(r>0\), and assume \textup{(SD)} and
\textup{(EL)} for this initial state. Write \(h_m=h_{i,m}(r)\) and
\(T_n=\sum_{m=0}^n d_{n-m}^-h_m\).
The convolution replacement is the equivalence
\begin{equation}\label{eq:critical-h-convolution-approx}
 \Ebf_i[Z_n^{-r};Z_n>0]\sim T_n .
\end{equation}
Suppose that, for some \(a_{i,r}>0\), \(\delta_r>0\), and an eventually
positive slowly varying \(L_r\),
\begin{equation}\label{eq:critical-h-rv}
 h_m\sim a_{i,r}m^{-\delta_r}L_r(m).
\end{equation}
Suppose that, for some \(C<\infty\) and all sufficiently large \(m\),
\[
 \sup_{z\in\mathcal S_i}\Psi_{m,r}(z)\le C h_m,\qquad
 \mathcal S_i=\bigcup_{\nu\in\mathfrak N_i}\operatorname{supp}\nu .
\]
Then in each of the following cases \eqref{eq:critical-h-convolution-approx}
holds and gives the indicated limit.
\begin{enumerate}[label=\textup{(\roman*)}]
\item If \(\delta_r=1\) and \(A_r(n):=\sum_{m=1}^n L_r(m)/m\to\infty\),
then
\begin{equation}\label{eq:critical-log-boundary}
 \frac{\Ebf_i[Z_n^{-r};Z_n>0]}{d_n^-A_r(n)}\longrightarrow a_{i,r}.
\end{equation}
In particular, if \(L_r(n)\to\ell_r\in(0,\infty)\), then
\(A_r(n)\sim\ell_r\log n\).
\item If \(1-\varrho\le\delta_r<1\), then
\begin{equation}\label{eq:critical-beta-region}
 \frac{\Ebf_i[Z_n^{-r};Z_n>0]}
 {n^{1-\delta_r-\varrho}L_r(n)\ell_-(n)}
 \longrightarrow a_{i,r}B(1-\delta_r,1-\varrho).
\end{equation}
\end{enumerate}
Both limits also hold if \eqref{eq:critical-h-convolution-approx} is
established directly in place of the statewise bound.
\end{proposition}

\begin{proof}
Put \(d_k=d_k^-\). Since \(0\le\Psi_{m,r}(z)\le\Pbf(L_m\ge0)\),
\eqref{eq:critical-ladder-rv} and \eqref{eq:critical-h-rv} imply
\(\delta_r\ge1-\varrho\). The statewise bound holds, in particular, if
\(\liminf_m h_m/\Pbf(L_m\ge0)>0\).
It can also be checked by coupling when at least one immigrant arrives
in every generation and \(\nu_i^-(\{z_*\})>0\) for
\(z_*=\min\mathcal S_i\). Indeed, using the same environment, offspring
rows and immigrants for all initial states gives
\(Z_m^{(z)}\ge Z_m^{(z_*)}\ge1\) for \(z\in\mathcal S_i\) and \(m\ge1\).
Thus \(\sup_{z\in\mathcal S_i}\Psi_{m,r}(z)=\Psi_{m,r}(z_*)
\le h_m/\nu_i^-(\{z_*\})\).

Write \(\nu_k=\nu_{i,k}^-\), \(\nu=\nu_i^-\), and
\(b_{k,m}=\int\Psi_{m,r}\,\dd\nu_k\). By
\eqref{eq:critical-minimum-decomposition}, it suffices to show that
\[
 \frac{1}{T_n}\sum_{m=0}^{n} d_{n-m}|b_{n-m,m}-h_m|\longrightarrow0.
\]
For every fixed \(J\), regular variation of \(d_n\) and \(h_n\) gives
\[
 \liminf_{n\to\infty}\frac{T_n}{d_n}\ge\sum_{m=0}^J h_m,
 \qquad
 \liminf_{n\to\infty}\frac{T_n}{h_n}\ge\sum_{k=0}^J d_k .
\]
Both sums diverge as \(J\to\infty\): the profile \(h_m\) is nonsummable
in cases \textup{(i)} and \textup{(ii)}, and \(d_k\) is regularly varying
with index \(-\varrho>-1\). Hence \(d_n=o(T_n)\) and \(h_n=o(T_n)\).

Choose \(M\) such that \(h_m>0\) and
\(\sup_{z\in\mathcal S_i}\Psi_{m,r}(z)\le Ch_m\) for every \(m\ge M\),
with the same constant \(C\). Since \(0\le b_{k,m},h_m\le1\), the terms
with \(m<M\) have total absolute error \(O_M(d_n)=o(T_n)\).
For \(m\ge M\), applying the total-variation bound to the normalized kernel gives
\[
 \sup_{m\ge M}\left|\frac{b_{k,m}}{h_m}-1\right|
 \le 2C\|\nu_k-\nu\|_{\mathrm{TV}}.
\]
For fixed \(K\), the terms with \(k=n-m<K\) and \(m\ge M\) contribute at
most \((C+1)\sum_{k<K}d_k h_{n-k}=O_K(h_n)=o(T_n)\). On the remaining
range, \(m\ge M\) and \(k\ge K\), the total absolute error is bounded by
\(2C T_n\sup_{k\ge K}\|\nu_k-\nu\|_{\mathrm{TV}}\). Divide by \(T_n\),
let \(n\to\infty\), and then let \(K\to\infty\). Assumption \textup{(EL)}
now gives \eqref{eq:critical-h-convolution-approx}.

It remains to evaluate \(T_n\). In case \textup{(i)}, for fixed
\(0<\eps<1\), uniform slow variation gives
\[
 A_r(n)-A_r(\lfloor\eps n\rfloor)
 \sim L_r(n)\log(1/\eps).
\]
Allowing \(\eps\downarrow0\) shows \(L_r(n)=o(A_r(n))\); hence
\(A_r(\lfloor\eps n\rfloor)\sim A_r(n)\) for each fixed \(\eps\).
Thus \(\sum_{m\le\eps n}h_m\sim a_{i,r}A_r(n)\). Uniform regular variation
bounds \(d_{n-m}/d_n\) on this range between \(1-o(1)\) and
\((1-\eps)^{-\varrho}+o(1)\). On \(m>\eps n\),
\(h_m=O_\eps(L_r(n)/n)\), while
\(\sum_{k\le n}d_k=O(nd_n)\) by Karamata's theorem
\cite{BinghamGoldieTeugels1987}. This part is therefore
\(O_\eps(d_nL_r(n))=o(d_nA_r(n))\). Letting \(\eps\downarrow0\) proves
\eqref{eq:critical-log-boundary}.

In case \textup{(ii)}, normalize \(T_n\) by \(nd_nh_n\). Any fixed number
of terms at either endpoint contributes \(O(d_n+h_n)=o(nd_nh_n)\).
On \(\eps n\le m\le(1-\eps)n\), uniform regular variation gives the
Riemann integral of \(x^{-\delta_r}(1-x)^{-\varrho}\) over
\([\eps,1-\eps]\). For \(m\) and \(n-m\) sufficiently large, Potter
bounds give, with \(x=m/n\) and
\(0<\eta<\min(1-\delta_r,1-\varrho)\),
\[
 \frac{h_m d_{n-m}}{h_n d_n}
 \le C x^{-\delta_r-\eta}(1-x)^{-\varrho-\eta}.
\]
The two remaining endpoint sums, after normalization, are consequently
\(O(\eps^{1-\delta_r-\eta}+\eps^{1-\varrho-\eta})\).
Letting \(\eps\downarrow0\) gives \(T_n/(nd_nh_n)\to
B(1-\delta_r,1-\varrho)\). Now \eqref{eq:critical-h-convolution-approx}
and \eqref{eq:critical-h-rv} give \eqref{eq:critical-beta-region}.
\end{proof}

\subsection{Negative tilts and lower-deviation directions}
\label{subsec:bprei-negative-tilt}

Fix \(q>0\) with \(-q\in\operatorname{int}\cD_\Lambda\), and suppose
\(\mu:=\Ebf\log M\in(0,\infty)\). Then \(\mu=\Lambda'_-(0)\), and
\(\mu=\Lambda'(0)\) if \(0\in\operatorname{int}\cD_\Lambda\).
The change of measure gives
\begin{equation}\label{eq:bprei-harmonic-negative-tilt}
\begin{aligned}
\Ebf_{-q}[V_n^{-q};Z_n>0]
&=\e^{-n\Lambda(-q)}
\Ebf[\e^{-qS_n}V_n^{-q};Z_n>0]\\
&=\e^{-n\Lambda(-q)}\Ebf[Z_n^{-q};Z_n>0].
\end{aligned}
\end{equation}
Thus
\(\Ebf[Z_n^{-q};Z_n>0]
 =\e^{n\Lambda(-q)}\Ebf_{-q}[V_n^{-q};Z_n>0]\).

If \(q_c>0\) satisfies \(\Lambda'(-q_c)=0\), the environmental walk is
critical under \(\Pbf_{-q_c}\), but
\(V_n^{-q_c}=\e^{q_cS_n}Z_n^{-q_c}\).  Hence
Theorem~\ref{thm:bprei-critical-harmonic} does not apply directly: its
terminal functions depend only on \(Z_n\), whereas the tilted functional also
contains the terminal environmental position.  The corresponding weighted
post-minimum profile is
\[
\Psi^{\mathrm w}_{m,q_c}(z)
:=\Ebf_{-q_c,z}[\e^{q_cS_m}Z_m^{-q_c};Z_m>0,L_m\geq0].
\]
On \(\{\tau_n=k\}\), write
\(\e^{q_cS_n}=\e^{q_cS_k}\e^{q_c(S_n-S_k)}\).
Conditioning on the history through time \(k\) as in
Lemma~\ref{lem:bprei-critical-minimum-decomposition} gives, for \(Z_0=i\),
\[
 \Ebf_{-q_c,i}[V_n^{-q_c};Z_n>0]
 =\sum_{k=0}^n
   \Ebf_{-q_c,i}\!\left[
    \e^{q_cS_k}\Psi^{\mathrm w}_{n-k,q_c}(Z_k);\,\tau_k=k
   \right].
\]

For a linear-fractional BPRE without immigration, the composition formula
singles out \(q=1\).  Provided \(-1\in\operatorname{int}\cD_\Lambda\), the
signs \(\Lambda'(-1)>0\), \(=0\), and \(<0\) give the strongly,
intermediately, and weakly supercritical regimes, respectively.  For a BPREI,
the distinguished negative power depends on the immigration transform; no
universal value \(1\) is asserted.  Outside the no-immigration
linear-fractional comparison, the critical negative tilt is determined by a
solution of \(\Lambda'(-q_c)=0\), when one exists.

\section{Random-slope controlled processes: proofs and transfer}
\label{sec:controlled-process}

For the random-slope controls of Section~\ref{subsec:main-random-slope},
the population update has the random-map form of
Section~\ref{subsec:bprei-upper-deviations}. We identify its innovation
and apply Propositions~\ref{prop:bprei-normalized-limit}
and~\ref{prop:bprei-block} to obtain the sharp upper-tail estimate for
\(X_n\). On the event \(X_n>\e^{\rho n+y}\), the one-step ratio bounds
then give Theorem~\ref{thm:main-random-sharp-ratio}.

In the individual-sum model the controlled population also has an exact
BPREI representation. We use this representation to study ratio
deviations without a lower bound on \(X_n\), by averaging the one-step
kernels over the population law.

Throughout this section, all occurrences of $O(\cdot)$ and $o(\cdot)$ that
involve the generation index are understood as $n\to\infty$ unless another
limiting regime is stated locally.

\subsection{Exact affine recursion and primitive innovation estimates}
\label{subsec:controlled-linearization}

Recall that \(\lambda\in(0,\infty)\) is the deterministic reference slope of
the random maps \(C_n\), and write
\(\mathfrak r_n(k):=C_n(k)-\lambda k\).  Thus \(\mathfrak r_n(k)\)
measures the departure of the control map from its first-order linear part.
In the individual-sum subclass
\(C_n(k)=\sum_{j=1}^{k}L_{n,j}\), the natural choice is
\(\lambda=\Ebf L\).

Put \(N_n:=\phi_n(X_n)\), \(M_n:=m\lambda A_n\), and define the random
population map
\begin{equation}\label{eq:controlled-random-map}
        \Psi_n(k):=\sum_{j=1}^{A_nC_n(k)+B_n}\xi_{n,j},
        \qquad k\in\Nzero.
\end{equation}
Then \(X_{n+1}=\Psi_n(X_n)\), and processes started from different
deterministic states can be coupled through the same sequence of maps.

The innovation has three terms: the centered offspring sum, the
departure of \(C_n(X_n)\) from \(\lambda X_n\), multiplied by \(mA_n\),
and the additive contribution \(mB_n\).

For a generic one-generation mark and a deterministic state
\(k\in\Nzero\), put \(N(k):=AC(k)+B\), and define
\[
\begin{aligned}
D_{n+1}
&:=\sum_{j=1}^{N_n}(\xi_{n,j}-m)
   +mA_n\mathfrak r_n(X_n)+mB_n,\\
D(k)
&:=\sum_{j=1}^{N(k)}(\xi_j-m)
   +mA\{C(k)-\lambda k\}+mB.
\end{aligned}
\]
Conditional on \(X_n=k\), the pair \((M_n,D_{n+1})\) has the same law as
\((M,D(k))\).  We call \(\{D_{n+1}:n\geq0\}\) the innovation sequence.  It
need not be centered.  For \(0\leq j\leq n\), write
\(\Delta_{j,n}:=\prod_{r=j}^{n-1}M_r\), with
\(\Delta_n:=\Delta_{0,n}\) and \(\Delta_{n,n}:=1\).

\begin{proposition}[Exact environmental linearization]
\label{prop:controlled-linearization}
For every \(n\geq0\),
\begin{equation}\label{eq:controlled-linear-recursion}
        X_{n+1}=M_nX_n+D_{n+1}.
\end{equation}
Moreover, for \(0\leq\ell<n\),
\begin{equation}\label{eq:controlled-exact-telescope}
        X_n-X_\ell\Delta_{\ell,n}
        =\sum_{j=\ell+1}^{n}D_j\Delta_{j,n}.
\end{equation}
In particular,
\(X_n/\Delta_n=X_0+\sum_{j=1}^{n}D_j/\Delta_j\).
\end{proposition}

\begin{proof}
Since \(N_n=A_nC_n(X_n)+B_n\), centering the offspring sum at its mean gives
\[
X_{n+1}
=\sum_{j=1}^{N_n}(\xi_{n,j}-m)+mN_n
=M_nX_n+D_{n+1}.
\]
Iterating this identity from generation \(\ell\) to generation \(n\) gives
\eqref{eq:controlled-exact-telescope}.  Taking \(\ell=0\) and dividing by
\(\Delta_n\) gives the normalized series representation.
\end{proof}

Proposition~\ref{prop:main-random-normalized-limit} assumes the
statewise innovation bound \eqref{eq:main-random-statewise-innovation}.
The next lemma gives sufficient conditions for this bound at \(s=\alpha\).
Fix \(s>0\).  When \(0<s\leq1\), choose \(r_s\in(1,2]\), and define
\begin{equation}\label{eq:controlled-bs}
        \mathfrak b_s:=
        \begin{cases}
        s/r_s, & 0<s\leq1,\\[1mm]
        1, & 1<s\leq2,\\[1mm]
        s/2, & s>2.
        \end{cases}
\end{equation}
Then \(0<\mathfrak b_s<s\).

\begin{assumption}[Sum-scale control innovations]
\label{ass:controlled-sum-scale}
Assume that there is \(K_s<\infty\) such that, for every
\(k\in\Nzero\),
\begin{align}
        \Ebf[N(k)^{\mathfrak b_s}]
        &\leq K_s(1+k^{\mathfrak b_s}),
        \label{eq:controlled-N-moment}\\
        \Ebf\left|A\{C(k)-\lambda k\}+B\right|^s
        &\leq K_s(1+k^{\mathfrak b_s}).
        \label{eq:controlled-control-remainder-moment}
\end{align}
If \(0<s\leq1\), assume in addition
\(\Ebf|\xi-m|^{r_s}<\infty\); if \(s>1\), assume
\(\Ebf|\xi-m|^s<\infty\).
\end{assumption}

The first bound controls the random length of the centered offspring sum,
whereas the second controls the \(s\)th moment of the control remainder.
Since \(\mathfrak b_s<s\), neither estimate follows from the other under the
assumptions as stated.

\begin{lemma}[Sublinear innovation bound]
\label{lem:controlled-innovation-bound}
Under Assumption~\ref{ass:controlled-sum-scale},
\begin{equation}\label{eq:controlled-D-moment}
        \Ebf|D(k)|^s
        \leq C_s(1+k^{\mathfrak b_s}),
        \qquad k\in\Nzero,
\end{equation}
for a finite constant \(C_s\).  Moreover,
\begin{equation}\label{eq:controlled-slope-Ls}
        \lim_{k\to\infty}
        \Ebf\left|\frac{\phi(k)}{k}-\lambda A\right|^s=0.
\end{equation}
\end{lemma}

\begin{proof}
Put \(S_{N(k)}:=\sum_{j=1}^{N(k)}(\xi_j-m)\) and
\(R(k):=A\{C(k)-\lambda k\}+B\), so that
\(D(k)=S_{N(k)}+mR(k)\).

Suppose first that \(0<s\leq1\).  Subadditivity gives
\[
        \Ebf|D(k)|^s
        \leq \Ebf|S_{N(k)}|^s+m^s\Ebf|R(k)|^s.
\]
Conditional Marcinkiewicz--Zygmund \cite{MarcinkiewiczZygmund1937} at exponent \(r_s\), followed by
conditional Jensen, gives
\[
\begin{aligned}
\Ebf[|S_{N(k)}|^{r_s}\mid N(k)]
&\leq C_{r_s}N(k)\Ebf|\xi-m|^{r_s},\\
\Ebf[|S_{N(k)}|^s\mid N(k)]
&\leq
\{\Ebf[|S_{N(k)}|^{r_s}\mid N(k)]\}^{s/r_s}
\leq C_sN(k)^{s/r_s}.
\end{aligned}
\]
Since \(s/r_s=\mathfrak b_s\), the two moment bounds in
Assumption~\ref{ass:controlled-sum-scale} prove
\eqref{eq:controlled-D-moment}.

For \(1<s\leq2\), conditional von Bahr--Esseen \cite{vonBahrEsseen1965} gives
\[
        \Ebf[|S_{N(k)}|^s\mid N(k)]
        \leq C_sN(k)\Ebf|\xi-m|^s.
\]
For \(s>2\), conditional Rosenthal \cite{Rosenthal1970} and Lyapunov give
\[
\begin{aligned}
\Ebf[|S_{N(k)}|^s\mid N(k)]
&\leq C_s\left\{N(k)\Ebf|\xi-m|^s
+N(k)^{s/2}(\Ebf|\xi-m|^2)^{s/2}\right\}\\
&\leq C_s\{N(k)+N(k)^{s/2}\}.
\end{aligned}
\]
Since \(N(k)\leq1+N(k)^{s/2}\), the choices
\(\mathfrak b_s=1\) and \(\mathfrak b_s=s/2\) give
\eqref{eq:controlled-D-moment} in the two remaining cases.
Finally,
\[
\begin{aligned}
\Ebf\left|\frac{\phi(k)}k-\lambda A\right|^s
&=k^{-s}\Ebf|R(k)|^s\\
&\leq K_s\{k^{-s}+k^{\mathfrak b_s-s}\}\longrightarrow0,
\end{aligned}
\]
because \(\mathfrak b_s<s\).
\end{proof}

\begin{remark}[Scope of the sum-scale conditions]
\label{rem:controlled-sum-assumptions}
The exponents in Assumption~\ref{ass:controlled-sum-scale} are matched to
the classical random-sum inequalities and produce the strict gap
\(\mathfrak b_s<s\) required by the normalized recursion.  With weaker
moments one may work at a smaller exponent or use truncation, at the cost of
a weaker polynomial ratio kernel.  If no useful power moment is available,
the polynomial transfer must be replaced by a tail-specific estimate.
\end{remark}

\subsection{Normalized convergence, nondegeneracy, and the sharp population tail}
\label{subsec:controlled-normalized-sharp}

\begin{proof}[Proof of Proposition~\ref{prop:main-random-normalized-limit}]
Take \(\Psi_n(k)=\sum_{j=1}^{\phi_n(k)}\xi_{n,j}\) and \(M_n=m\lambda A_n\).
Then \((M_n,\Psi_n)_{n\ge0}\) are i.i.d.\ pairs independent of \(X_0\),
and \(\Pi_n=\Delta_n\). The environmental hypotheses of
Section~\ref{subsec:main-random-slope} and the moment bounds
\eqref{eq:main-random-initial-moment}--\eqref{eq:main-random-statewise-innovation}
give Assumptions~\ref{ass:bprei-cramer} and~\ref{ass:bprei-linearization}.
For each \(K\ge1\), \eqref{eq:main-random-accessibility} gives a finite
\(n\) such that \(\Pbf(\tau_K\le n)>0\). Since the tilt density is
strictly positive on this finite history, \(\Pbf_\alpha(\tau_K\le n)>0\)
as well. Propositions~\ref{prop:bprei-normalized-limit}
and~\ref{prop:bprei-block} now give the asserted limits with \(Z_n=X_n\).
\end{proof}
\begin{proof}[Proof of Proposition~\ref{prop:main-random-SFPE}]
For each deterministic \(z\in\Nzero\),
Proposition~\ref{prop:bprei-normalized-limit}, applied with \(Z_0=z\)
and \(\Pi_n=\Delta_n\), gives almost-sure and \(L^\alpha(\Pbf_\alpha)\)
convergence of \(X_n^{(z)}/\Delta_n\) to \(\mathcal V_\alpha(z)\).
Since \(\Nzero\) is countable, the almost-sure limits exist simultaneously
for all \(z\) on one event of probability one.

Let \(X_n^{\langle1\rangle,z}\) be the iteration started from \(z\) and
driven by \(\Psi_1,\Psi_2,\ldots\), and put
\(V_n^{\langle1\rangle}(z):=X_n^{\langle1\rangle,z}/\Delta_{1,n+1}\).
Then
\[
        \frac{X_{n+1}^{(x)}}{\Delta_{n+1}}
        =M_0^{-1}V_n^{\langle1\rangle}(\Psi_0(x)).
\]
For \(z\in\Nzero\), let
\[
        \mathcal V_\alpha^{\langle1\rangle}(z)
        :=\lim_{n\to\infty}V_n^{\langle1\rangle}(z).
\]
Since \(\Nzero\) is countable, these shifted limits exist simultaneously
for all \(z\in\Nzero\) on a \(\Pbf_\alpha\)-full event.  In particular, the
limit may be evaluated at the random state \(z=\Psi_0(x)\).  Therefore, on
the intersection of the original and shifted full-measure events, letting
\(n\to\infty\) in the preceding identity gives
\[
        \mathcal V_\alpha(x)
        =M_0^{-1}\mathcal V_\alpha^{\langle1\rangle}(\Psi_0(x)),
        \qquad x\in\Nzero.
\]
Under \(\Pbf_\alpha\), the generation maps from time \(1\) onward are
independent of \((M_0,\Psi_0)\) and have the same law as the original
sequence of generation maps.  Hence
\((\mathcal V_\alpha^{\langle1\rangle}(x):x\in\Nzero)\) is an independent
copy of \((\mathcal V_\alpha(x):x\in\Nzero)\), and the preceding pathwise
identity yields \eqref{eq:main-random-joint-SFPE}.

Taking \(\alpha\)th moments and using the one-generation tilt density gives
\[
        c_\alpha(x)
        =\Ebf_\alpha[M^{-\alpha}c_\alpha(\Psi(x))]
        =\e^{-\Lambda(\alpha)}(Pc_\alpha)(x),
\]
and hence \eqref{eq:main-random-eigenfunction}.  If
\(P^n(x,y)>0\) and \(c_\alpha(y)>0\), then
\(\e^{n\Lambda(\alpha)}c_\alpha(x)
=(P^nc_\alpha)(x)\geq P^n(x,y)c_\alpha(y)>0\).

If \(X_0\sim\nu\), then tilt cancellation and conditioning on \(X_0\) give
\[
\e^{-n\Lambda(\alpha)}\Ebf_\nu[X_n^\alpha]
=\sum_{x\in\Nzero}\nu(x)
\Ebf_{\alpha,x}\left[\left(\frac{X_n}{\Delta_n}\right)^\alpha\right].
\]
For each \(x\), the conditional moment converges to \(c_\alpha(x)\).
Proposition~\ref{prop:bprei-normalized-limit} bounds it by
\(C_\alpha(1+x^\alpha)\), with \(C_\alpha\) independent of \(n\) and \(x\).
This bound is \(\nu\)-integrable, so dominated convergence gives
\[
 c_\nu(\alpha)=\sum_{x\in\Nzero}\nu(x)c_\alpha(x).
\]
\end{proof}

\subsection{One-step kernels and the sharp ratio theorem}
\label{subsec:controlled-ratio-kernels}

We now condition on the current generation size and derive the one-step
kernels needed for the ratio part of
Theorem~\ref{thm:main-random-sharp-ratio}.  The sharp population tail has
already been proved in Proposition~\ref{prop:main-random-normalized-limit};
the estimates below show that, on that upper-deviation event, each one-step
ratio failure is negligible.  Conditional on \(X_n\), the control
fluctuation, the innovation, and the offspring empirical mean are all
one-generation random-sum events.

For $k\geq1$ and $\eps>0$, define
\begin{align}
\psi_\phi(k;\eps)
&:=\Pbf\left(\left|\frac{\phi(k)}k-\lambda A\right|>\eps\right),
\label{eq:controlled-kernel-control}\\
\psi_D(k;\eps)
&:=\Pbf(|D(k)|>\eps k),
\label{eq:controlled-kernel-D}\\
\psi_\xi(k;\eps)
&:=\Pbf\left(
\left|\frac1k\sum_{j=1}^{k}\xi_j-m\right|>\eps\right).
\label{eq:controlled-kernel-offspring}
\end{align}
For the offspring ratio, the progenitor count is itself random conditional on
$X_n$.  We therefore write
\begin{equation}\label{eq:controlled-kernel-random-denominator}
        \psi_N(k;\eps)
        :=\Pbf\{N(k)=0\}
          +\Ebf[\psi_\xi(N(k);\eps);N(k)>0].
\end{equation}

\begin{proposition}[Environmentally centered ratio identities]
\label{prop:controlled-ratio-identities}
For every $n\geq0$ and $\eps>0$,
\begin{align}
&\Pbf\left(
 \left|\frac{N_n}{X_n}-\lambda A_n\right|>\eps,
 X_n>0\right)
 =\Ebf[\psi_\phi(X_n;\eps);X_n>0],
\label{eq:controlled-ratio-identity-control}\\
&\Pbf\left(
 \left|\frac{X_{n+1}}{X_n}-m\lambda A_n\right|>\eps,
 X_n>0\right)
 =\Ebf[\psi_D(X_n;\eps);X_n>0],
\label{eq:controlled-ratio-identity-X}\\
&\begin{aligned}
 \Pbf(\mathcal B_{n,3}(\eps))
 &=\Pbf(X_n>0,N_n=0)\\
 &\quad+\Pbf\!\left(X_n>0,N_n>0,\
                  \left|X_{n+1}/N_n-m\right|>\eps\right)\\
 &=\Ebf[\psi_N(X_n;\eps);X_n>0].
\end{aligned}
\label{eq:controlled-ratio-identity-offspring}
\end{align}
On $\{X_n>0,N_n>0\}$,
\begin{equation}\label{eq:controlled-ratio-product-decomposition}
\frac{X_{n+1}}{X_n}-m\lambda A_n
=
\left(\frac{X_{n+1}}{N_n}-m\right)\frac{N_n}{X_n}
+m\left(\frac{N_n}{X_n}-\lambda A_n\right).
\end{equation}
\end{proposition}

\begin{proof}
For every \(k\geq1\),
\[
\begin{aligned}
&\Pbf\left(
 \left|\frac{N_n}{X_n}-\lambda A_n\right|>\eps
 \Bigm|X_n=k\right)
 =\Pbf\left(
 \left|\frac{N(k)}k-\lambda A\right|>\eps\right)
 =\psi_\phi(k;\eps),\\
&\Pbf\left(
 \left|\frac{X_{n+1}}{X_n}-m\lambda A_n\right|>\eps
 \Bigm|X_n=k\right)
 =\Pbf(|D(k)|>\eps k)
 =\psi_D(k;\eps),\\
&\Pbf\bigl(\mathcal B_{n,3}(\eps)\mid X_n=k\bigr)\\
&\qquad=\Pbf\{N(k)=0\}
   +\sum_{\ell\ge1}\Pbf\{N(k)=\ell\}\psi_\xi(\ell;\eps)
   =\psi_N(k;\eps).
\end{aligned}
\]
The last identity separates the case \(N_n=0\). Since the offspring
row is independent of the control and of the past, the conditional
deviation probability given \(X_n=k\) and \(N_n=\ell\ge1\) is
\(\psi_\xi(\ell;\eps)\).
Taking expectations over \(X_n\) on \(\{X_n>0\}\) proves
\eqref{eq:controlled-ratio-identity-control}--
\eqref{eq:controlled-ratio-identity-offspring}.  Finally, on
\(\{X_n>0,N_n>0\}\),
\[
\left(\frac{X_{n+1}}{N_n}-m\right)\frac{N_n}{X_n}
+m\left(\frac{N_n}{X_n}-\lambda A_n\right)
=\frac{X_{n+1}}{X_n}-m\lambda A_n,
\]
which proves \eqref{eq:controlled-ratio-product-decomposition}.
\end{proof}

\begin{lemma}[Kernel transfer]
\label{lem:controlled-kernel-transfer}
Let $Y$ be a nonnegative integer-valued random variable.
\begin{enumerate}[label=\textup{(\roman*)}]
\item If $a_k\leq Cs^k$ for $k\geq1$ and $s\in(0,1)$, then
\[
        \Ebf[a_Y;Y>0]\leq C\Ebf[s^Y;Y>0].
\]
\item If $a_k\leq Ck^{-q}$ for $k\geq1$, then
\[
        \Ebf[a_Y;Y>0]\leq C\Ebf[Y^{-q};Y>0].
\]
The reverse inequality holds under the corresponding pointwise lower bound.
\item If
$a_{n,k}\leq C\min\{1,\eps_n^{-p}k^{-q}\}$, then, for every
$K_n\geq1$,
\begin{equation}\label{eq:controlled-kernel-transfer-cutoff}
        \Ebf[a_{n,Y};Y>0]
        \leq C\Pbf(0<Y\leq K_n)
             +C\eps_n^{-p}K_n^{-q}.
\end{equation}
\end{enumerate}
\end{lemma}

\begin{proof}
Parts \textup{(i)} and \textup{(ii)} follow by conditioning on \(Y\) and
summing the pointwise bounds against its law.  For part \textup{(iii)},
\[
\begin{aligned}
\Ebf[a_{n,Y};Y>0]
&=\Ebf[a_{n,Y};0<Y\leq K_n]
  +\Ebf[a_{n,Y};Y>K_n]\\
&\leq C\Pbf(0<Y\leq K_n)
  +C\eps_n^{-p}K_n^{-q}.
\end{aligned}
\]
\end{proof}

We next bound the three kernels under \((\mathrm{RM}_p)\), with
\(q=p/2\). The moment bounds in \eqref{eq:main-random-ratio-moments}
control \(\psi_\phi\) and \(\psi_D\). For \(\psi_N\) we also use the
inverse-count estimate \eqref{eq:main-random-inverse-count}.

\begin{proposition}[Polynomial one-step kernels]
\label{prop:controlled-poly-kernels}
Assume \eqref{eq:main-random-ratio-moments}, with $q=p/2$.  For every
$0<\eps\leq1$,
\begin{equation}\label{eq:controlled-poly-kernel-bound}
        \psi_\phi(k;\eps)+\psi_D(k;\eps)+\psi_N(k;\eps)
        \leq C\eps^{-p}k^{-q},
        \qquad k\geq1.
\end{equation}
\end{proposition}

\begin{proof}
By the second line of \eqref{eq:main-random-ratio-moments},
\[
\begin{aligned}
\psi_\phi(k;\eps)
&\leq \eps^{-p}k^{-p}\Ebf|N(k)-\lambda Ak|^p
\leq C\eps^{-p}k^{-q}.
\end{aligned}
\]
Conditional Rosenthal's inequality and \eqref{eq:main-random-ratio-moments} give
\[
        \Ebf|D(k)|^p
        \leq C\left\{\Ebf[N(k)^q]
        +\Ebf|N(k)-\lambda Ak|^p\right\}
        \leq C(1+k^q),
\]
and therefore
\(\psi_D(k;\eps)\leq C\eps^{-p}k^{-q}\), since \(p=2q\).
Similarly, conditional Rosenthal gives
\(\psi_\xi(\ell;\eps)\leq C\eps^{-p}\ell^{-q}\) for \(\ell\geq1\).
Recall that, for \(\ell\ge1\),
\[
 \psi_\xi(\ell;\eps)
 =\Pbf\!\left(\left|\frac{1}{\ell}\sum_{j=1}^{\ell}\xi_j-m\right|>\eps\right).
\]
For \(k\ge1\) and \(0<\eps\le1\),
\[
\begin{aligned}
 \psi_N(k;\eps)
 &=\Pbf\{N(k)=0\}+\Ebf[\psi_\xi(N(k);\eps);N(k)>0]\\
 &\le \Pbf\{N(k)=0\}
       +C\eps^{-p}\Ebf[N(k)^{-q};N(k)>0]\\
 &\le C'\eps^{-p}\bigl(\Pbf\{N(k)=0\}
                  +\Ebf[N(k)^{-q};N(k)>0]\bigr)
 \le C''\eps^{-p}k^{-q},
\end{aligned}
\]
where the last inequality follows from \eqref{eq:main-random-inverse-count}.
Combining the three bounds proves \eqref{eq:controlled-poly-kernel-bound}.
\end{proof}

\begin{proof}[Proof of Theorem~\ref{thm:main-random-sharp-ratio}]
Fix \(y\in\R\) and write \(U_n(y):=\{X_n>\e^{\rho n+y}\}\).
For \(0<\eps_n\le1\), conditioning on \(X_n\) and applying the union bound
to the three events in \(\mathcal B_n(\eps_n)\) gives
\[
\begin{aligned}
 \Pbf\bigl(U_n(y)\cap\mathcal B_n(\eps_n)\bigr)
 &\le \Ebf\!\left[
  \psi_\phi(X_n;\eps_n)+\psi_D(X_n;\eps_n)+\psi_N(X_n;\eps_n);U_n(y)\right]\\
 &\le C\eps_n^{-p}\Ebf[X_n^{-q};U_n(y)]\\
 &\le C\e^{-qy}\eps_n^{-p}\e^{-q\rho n}\Pbf(U_n(y)),
\end{aligned}
\]
where the second inequality follows from
\eqref{eq:controlled-poly-kernel-bound}. This proves
\eqref{eq:main-random-sharp-ratio-error}. Under
\eqref{eq:main-random-sharp-separation}, the omitted probability is
\(o(\Pbf(U_n(y)))\). Hence the sharp population asymptotic of
Proposition~\ref{prop:main-random-normalized-limit} gives
\eqref{eq:main-random-sharp-joint}. For fixed \(\eps>0\), use
\(\mathcal B_n(\eps)\subseteq\mathcal B_n(\eps\wedge1)\) and
\((\eps\wedge1)^{-p}\e^{-q\rho n}\to0\).
\end{proof}

\begin{remark}[Conditions tied to the sharp prefactor]
\label{rem:controlled-sharp-tools}
Lemma~\ref{lem:bprei-Stone-kernel} uses the nonlattice law of \(\log M_0\)
and \(0<v_\alpha^2<\infty\) to obtain the constant
\(1/(\alpha v_\alpha\sqrt{2\pi})\) in the integrated local limit.
For a lattice law, the corresponding local estimate depends on the
position of the threshold relative to the lattice.
The proof of Corollary~\ref{cor:bprei-early-prefactor} verifies the
size-biased small-shift condition \eqref{eq:bprei-size-biased-small-shift}.
\end{remark}

Assumption \((\mathrm{RM}_p)\) enters
Theorem~\ref{thm:main-random-sharp-ratio} only through
\eqref{eq:controlled-poly-kernel-bound}.  If the same kernel estimates hold
at a smaller exponent \(p'\geq2\), the proof remains valid with
\(q'=p'/2\) and tolerance condition
\(\eps_n^{-p'}\e^{-q'\rho n}\to0\).  If no common power bound is available,
Lemma~\ref{lem:controlled-kernel-transfer}\textup{(iii)} reduces the problem
to a cutoff estimate involving \(\Pbf(0<X_n\leq K_n)\); the resulting bound
then depends on a separate lower-population estimate.

\subsection{The individual-sum subclass: primitive verification and exact BPREI structure}
\label{subsec:controlled-sum-subclass}

We now impose the individual-sum representation
\eqref{eq:main-individual-sum-control}. It gives sufficient moment
conditions for the innovation bound and an inverse-progenitor estimate
at every positive order. Grouping the offspring by current individual
also gives an exact BPREI representation, to which the harmonic-moment
and generating-function results of Section~\ref{sec:bprei} apply.

Assume
\(C_n(k)=\sum_{j=1}^{k}L_{n,j}\), \(k\in\Nzero\), where the
\(L_{n,j}\)'s are independent copies of a nonnegative integer-valued random
variable \(L\), independent of \((A_n,B_n)\) and of the offspring array, and
where \(\lambda:=\Ebf L\in(0,\infty)\).  The variables \(A_n\) and \(B_n\)
may remain dependent through the same generation mark.  For a generic
generation, write \(C(k):=\sum_{j=1}^{k}L_j\), so that
\(N(k)=AC(k)+B\).

\begin{proposition}[Verification of the sum-scale conditions]
\label{prop:controlled-sum-verification}
Fix \(s>0\) and let \(\mathfrak b_s\) be defined by
\eqref{eq:controlled-bs}.  If \(0<s\leq1\), assume
\(\Ebf A^s+\Ebf B^s+\Ebf L^{r_s}<\infty\); if \(s>1\), assume
\(\Ebf A^s+\Ebf B^s+\Ebf L^s<\infty\).  Then, for every \(k\geq0\),
\begin{align}
\Ebf[(AC(k)+B)^{\mathfrak b_s}]
&\leq K_s(1+k^{\mathfrak b_s}),
\label{eq:controlled-sum-N-bound}\\
\Ebf|A\{C(k)-\lambda k\}+B|^s
&\leq K_s(1+k^{\mathfrak b_s}).
\label{eq:controlled-sum-remainder-bound}
\end{align}
Consequently, Assumption~\ref{ass:controlled-sum-scale} holds if, in
addition, \(\Ebf|\xi-m|^{r_s}<\infty\) when \(0<s\leq1\), and if
\(\Ebf|\xi-m|^s<\infty\) when \(s>1\).
\end{proposition}

\begin{proof}
Write \(R_k:=C(k)-\lambda k=\sum_{j=1}^{k}(L_j-\lambda)\).  The same
random-sum inequalities used in Lemma~\ref{lem:controlled-innovation-bound}
give \(\Ebf|R_k|^s\leq C_sk^{\mathfrak b_s}\).  Independence of \(R_k\)
and \((A,B)\), followed by
\(|AR_k+B|^s\leq c_s(A^s|R_k|^s+B^s)\), proves
\eqref{eq:controlled-sum-remainder-bound}.

For the progenitor-count estimate, put \(b:=\mathfrak b_s\).  Since
\(b<s\), the stated assumptions give finite \(b\)th moments of \(A\),
\(B\), and \(L\).  If \(b\leq1\), concavity gives
\(\Ebf[C(k)^b]\leq(\Ebf C(k))^b=(k\lambda)^b\); if \(b>1\), Minkowski's
inequality gives the same order \(O(k^b)\).  Since \(N(k)=AC(k)+B\), the
inequality \((x+y)^b\leq c_b(x^b+y^b)\) now proves
\eqref{eq:controlled-sum-N-bound}.
\end{proof}

\begin{proposition}[Inverse progenitor counts]
\label{prop:controlled-denominator-sum}
For every $q>0$, there is $C_q<\infty$ such that
\begin{equation}\label{eq:controlled-denominator-sum-bound}
        \Pbf\{N(k)=0\}
        +\Ebf[N(k)^{-q};N(k)>0]
        \leq C_qk^{-q},
        \qquad k\geq1.
\end{equation}
\end{proposition}

\begin{proof}
Choose \(\ell_0\geq1\) such that
\(p_*:=\Pbf(L\geq\ell_0)>0\), and put
\(J_k:=\sum_{j=1}^{k}\one_{\{L_j\geq\ell_0\}}\).  Since
\(A\geq1\) and \(B\geq0\),
\(N(k)=AC(k)+B\geq C(k)\geq\ell_0J_k\).  Hence
\[
\Ebf\left[N(k)^{-q};N(k)>0,J_k\geq\frac{p_*k}{2}\right]
\leq\left(\frac{2}{\ell_0p_*}\right)^qk^{-q},
\]
while
\[
\Ebf\left[N(k)^{-q};N(k)>0,J_k<\frac{p_*k}{2}\right]
\leq\Pbf\left(J_k<\frac{p_*k}{2}\right)
\leq\e^{-c_*k}
\]
for some \(c_*>0\).  Moreover,
\(\{N(k)=0\}\subseteq\{J_k=0\}\), and therefore
\[
        \Pbf\{N(k)=0\}
        \leq(1-p_*)^k\leq\e^{-p_*k}.
\]
Combining the three estimates proves
\eqref{eq:controlled-denominator-sum-bound}.
\end{proof}

For \(q=p/2\), the inverse-count estimate
\eqref{eq:main-random-inverse-count} already follows from
\((\mathrm{RM}_p)\). In the individual-sum model,
Proposition~\ref{prop:controlled-denominator-sum} gives
\eqref{eq:controlled-denominator-sum-bound} at every positive order
without assuming \((\mathrm{RM}_p)\).

For \(p\ge2\), a sufficient condition for \((\mathrm{RM}_p)\) is
\[
 \Ebf A^p+\Ebf B^p+\Ebf L^p+\Ebf|\xi-m|^p<\infty .
\]
Indeed, Proposition~\ref{prop:controlled-sum-verification}, applied with
\(s=p\) and \(\mathfrak b_p=p/2\), gives the two bounds on \(N(k)\) in
\eqref{eq:main-random-ratio-moments}.

\begin{proposition}[Exact BPRE with immigration representation]
\label{prop:controlled-BPREI-representation}
Conditional on the environmental sequence \(\{(A_n,B_n):n\geq0\}\), the
individual-sum controlled process is a BPRE with immigration:
\begin{equation}\label{eq:controlled-BPREI-grouped}
        X_{n+1}
        =\sum_{i=1}^{X_n}Y_{n,i}^{(A_n)}
         +\eta_{n+1}^{(B_n)}.
\end{equation}
Given \(A_n=a\) and \(B_n=b\),
\(Y_{n,i}^{(a)}\) has the law of
\(\sum_{r=1}^{aL_{n,i}}\xi_{n,i,r}\), while
\(\eta_{n+1}^{(b)}\) has the law of
\(\sum_{r=1}^{b}\xi'_{n,r}\).  Their conditional p.g.f.s are
\[
\Ebf[s^{Y_{n,i}^{(a)}}\mid A_n=a]=\ell(f(s)^a),
\qquad
\Ebf[s^{\eta_{n+1}^{(b)}}\mid B_n=b]=f(s)^b.
\]
The induced quenched reproduction mean is \(m\lambda A_n=M_n\), and the
quenched p.g.f. recursion is
\begin{equation}\label{eq:controlled-BPREI-pgf-recursion}
        F_{n+1}^{\cE}(s)
        =f(s)^{B_n}F_n^{\cE}(\ell(f(s)^{A_n})).
\end{equation}
\end{proposition}

\begin{proof}
Fix \(X_n=k\) and \((A_n,B_n)=(a,b)\).  Relabel the offspring array so that
the first \(aL_{n,i}\) progenitors assigned to the \(i\)th current
individual form one group, \(1\leq i\leq k\), and the final \(b\) additive
progenitors form the immigration group.  Conditional on the mark, these group
sums are independent.  Their p.g.f.s are the two functions displayed above,
and summing the groups gives \eqref{eq:controlled-BPREI-grouped}.
Differentiating the reproduction p.g.f. at one gives \(m\lambda a\), while
conditioning on \(X_n\) yields
\eqref{eq:controlled-BPREI-pgf-recursion}.
\end{proof}

Thus, in the notation of Section~\ref{sec:bprei},
\(f_n^{\mathrm{ind}}(s)=\ell(f(s)^{A_n})\),
\(g_n^{\mathrm{ind}}(s)=f(s)^{B_n}\), and
\(M_n=m\lambda A_n\).  The reproduction and immigration laws may be
dependent through the common mark \((A_n,B_n)\), exactly as allowed in
Section~\ref{sec:bprei}.

Recall from Section~\ref{subsec:main-random-slope} that
\(A\in\N=\{1,2,\ldots\}\) almost surely.

\subsubsection{Primitive checks for the induced critical BPREI}
\label{subsec:controlled-induced-checks}

The exact representation above lets us translate the centered finite-variance
critical class of Theorem~\ref{thm:bprei-critical-fv-class} into the
primitive controlled-process variables.  For a generic generation, let
\(Y^{(A)}:=\sum_{r=1}^{AL}\xi_r\) be the induced offspring variable.  Its
quenched mean is \(M=m\lambda A\), and, when \(\xi\) and \(L\) have finite
variances,
\begin{equation}\label{eq:controlled-induced-normalized-variance}
\frac{\Var(Y^{(A)}\mid A)}{M^2}
 =\frac{\Var(\xi)}{m^2\lambda A}
  +\frac{\Var(L)}{\lambda^2}
 \leq
 \frac{\Var(\xi)}{m^2\lambda}
  +\frac{\Var(L)}{\lambda^2}.
\end{equation}
Thus the normalized quenched offspring variance remains bounded for both
bounded and unbounded environmental multipliers.  Also write
\(p_0^{\mathrm{ind}}:=\ell(f(0)^A)\) for the induced quenched probability of
zero offspring.

\begin{proposition}[Primitive critical-class verification]
\label{prop:controlled-induced-critical-verification}
Assume that, for some \(q_{\log}>12\),
\[
 \Ebf\log(m\lambda A)=0,\qquad
 0<\Var(\log(m\lambda A))<\infty,\qquad
 \Ebf\bigl[(\log(m\lambda A))_+^{q_{\log}}\bigr]<\infty.
\]
Suppose also that
\(\Ebf\xi^2+\Ebf L^2<\infty\), \(\Pbf(\xi=0)=0\), and, for some
\(\overline b<\infty\),
\[
        1\leq B\leq\overline b
        \qquad\text{almost surely}.
\]
Then the induced BPREI belongs to \(\mathcal C_{\mathrm{fv}}\) of
Theorem~\ref{thm:bprei-critical-fv-class}.  Consequently, its endpoint-
minimum entrance laws converge to a common law and, for every fixed initial
state \(i\geq1\) and every \(r>0\),
\[
        \sqrt n\,\Ebf_i[X_n^{-r};X_n>0]
        \longrightarrow\kappa_0C_r\in(0,\infty),
\]
where \(C_r\) is independent of \(i\).
\end{proposition}

\begin{proof}
Conditions \textup{(C1)} and \textup{(C4)} follow from the first three
assumptions, while \eqref{eq:controlled-induced-normalized-variance} gives
\textup{(C2)}.  For \(\eta:=\sum_{j=1}^{B}\xi'_j\),
\[
        \eta\geq1\quad\text{a.s.},
        \qquad
        1\leq\Ebf[\eta\mid A,B]=mB\leq m\overline b.
\]
Thus \textup{(C3)} holds, and the conclusion follows from
Theorem~\ref{thm:bprei-critical-fv-class}.
Finally, \(A\geq1\) and \(\Pbf(\xi=0)=0\) imply \(m\geq1\).
If \(\Pbf(L=0)=0\), then \(\lambda\geq1\), and hence
\(m\lambda A\geq1\) almost surely.  The equality
\(\Ebf\log(m\lambda A)=0\) would then force
\(m\lambda A=1\) almost surely, contradicting
\(\Var(\log(m\lambda A))>0\).  Thus \(\Pbf(L=0)>0\).
\end{proof}

\begin{remark}[Bounded and unbounded environmental multipliers]
Suppose first that \(A\leq\overline a\) almost surely and that \(L\), \(B\),
and \(\xi\) have finite moment-generating functions in a neighborhood of the
origin.  Since
\[
        N(k)-\lambda Ak=A\{C(k)-\lambda k\}+B,
        \qquad C(k)=\sum_{j=1}^{k}L_j,
\]
Chernoff bounds give, for every fixed \(\eps>0\),
\[
\Pbf(|N(k)-\lambda Ak|>\eps k)+\Pbf(N(k)>K_\eps k)
\leq C_\eps\e^{-c_\eps k}
\]
for a suitable \(K_\eps<\infty\).  Conditional Chernoff bounds for the
offspring sum on \(\{N(k)\leq K_\eps k\}\) then give the geometric
innovation and offspring kernels, while, since \(A\geq1\) and \(B\geq0\),
\(\Pbf\{N(k)=0\}\leq\Pbf(L=0)^k\).  Thus all four geometric one-step
kernels used later are verified directly; no finite-horizon bridge estimate
is involved.

When \(A\) is unbounded, marginal exponential moments do not in general give
a geometric rate uniformly in the current population size.  The critical
result above uses a different route.  The normalized variance remains
uniformly bounded by \eqref{eq:controlled-induced-normalized-variance}, and
the future-minimum record argument of
Appendix~\ref{sec:critical-verification} gives UIM and localization under the
stated positive-tail condition.  This verifies the post-minimum inputs used by the Laplace and
harmonic-moment limits.
\end{remark}

\subsection{Harmonic-moment and generating-function transfer}
\label{subsec:controlled-transfer}

With the one-step identities and kernel estimates of
Subsection~\ref{subsec:controlled-ratio-kernels} in hand, it remains to
average the kernels over their random denominators.  A polynomial bound of
order \(k^{-q}\) is averaged through the harmonic moment
\(\Ebf[X_n^{-q};X_n>0]\), whereas a geometric bound is averaged through a
positive population or progenitor transform.  In the individual-sum
subclass, the exact BPREI representation identifies these profiles.  We
first complete the polynomial and geometric transfer arguments and then
apply the supercritical and critical BPREI results.

For \(q>0\), write
\(\mathcal H_n^X(q):=\Ebf[X_n^{-q};X_n>0]\).  The ratio identities give
\[
\begin{aligned}
\mathcal R_{n,1}(\eps)
&=\Ebf[\psi_\phi(X_n;\eps);X_n>0],\\
\mathcal R_{n,2}(\eps)
&=\Ebf[\psi_D(X_n;\eps);X_n>0],\\
\mathcal R_{n,3}(\eps)
&=\Ebf[\psi_N(X_n;\eps);X_n>0].
\end{aligned}
\]
For \(s\in(0,1)\), also put
\[
        F_n^+(s):=\Ebf[s^{X_n};X_n>0],
        \qquad
        H_n^+(s):=\Ebf[s^{N_n};N_n>0].
\]
For the contribution from \(X_n>0\) and \(N_n>0\), write
\[
 H_{n,>0}^+(s):=\Ebf[s^{N_n};X_n>0,N_n>0].
\]

\begin{assumption}[Geometric one-step kernels]
\label{ass:controlled-exp-kernels}
For every fixed \(\eps>0\), there are constants \(C_\eps<\infty\) and
\(s_{\phi,\eps},s_{D,\eps},s_{\xi,\eps},s_{0,\eps}\in(0,1)\) such that
\begin{equation}\label{eq:controlled-exp-kernel-bounds}
\begin{gathered}
        \psi_\phi(k;\eps)\leq C_\eps s_{\phi,\eps}^k,
        \qquad
        \psi_D(k;\eps)\leq C_\eps s_{D,\eps}^k,\\
        \psi_\xi(k;\eps)\leq C_\eps s_{\xi,\eps}^k,
        \qquad
        \Pbf\{N(k)=0\}\leq C_\eps s_{0,\eps}^k,
        \qquad k\geq1.
\end{gathered}
\end{equation}
\end{assumption}

A transparent primitive verification is available when \(A\) is essentially
bounded and \(L\), \(B\), and \(\xi\) have moment-generating functions in a
neighborhood of the origin.  Chernoff bounds then give all four inequalities.
The fourth is stated separately because the full offspring-ratio bad event
contains \(N_n=0\).

\begin{proposition}[Generating-function transfer]
\label{prop:controlled-exponential-transfer}
Under Assumption~\ref{ass:controlled-exp-kernels},
\begin{align}
\Pbf\left(
 \left|\frac{N_n}{X_n}-\lambda A_n\right|>\eps,X_n>0\right)
&\leq C_\eps F_n^+(s_{\phi,\eps}),
\label{eq:controlled-exp-transfer-control}\\
\Pbf\left(
 \left|\frac{X_{n+1}}{X_n}-m\lambda A_n\right|>\eps,X_n>0\right)
&\leq C_\eps F_n^+(s_{D,\eps}),
\label{eq:controlled-exp-transfer-X}\\
\Pbf\bigl(\mathcal B_{n,3}(\eps)\bigr)
&\leq C_\eps F_n^+(s_{0,\eps})
     +C_\eps H_n^+(s_{\xi,\eps}).
\label{eq:controlled-exp-transfer-offspring}
\end{align}
\end{proposition}

\begin{proof}
By the ratio identities and \eqref{eq:controlled-exp-kernel-bounds},
\[
\begin{aligned}
\mathcal R_{n,1}(\eps)
&\leq C_\eps\Ebf[s_{\phi,\eps}^{X_n};X_n>0]
 =C_\eps F_n^+(s_{\phi,\eps}),\\
\mathcal R_{n,2}(\eps)
&\leq C_\eps\Ebf[s_{D,\eps}^{X_n};X_n>0]
 =C_\eps F_n^+(s_{D,\eps}).
\end{aligned}
\]
For the offspring ratio, separate the cases \(N_n=0\) and \(N_n>0\).
By \eqref{eq:controlled-exp-kernel-bounds},
\[
\begin{aligned}
 \mathcal R_{n,3}(\eps)
 &=\Pbf(X_n>0,N_n=0)
    +\Ebf[\psi_\xi(N_n;\eps);X_n>0,N_n>0]\\
 &\le C_\eps F_n^+(s_{0,\eps})
    +C_\eps H_{n,>0}^+(s_{\xi,\eps}).
\end{aligned}
\]
Since \(H_{n,>0}^+\le H_n^+\), this proves
\eqref{eq:controlled-exp-transfer-offspring}.
\end{proof}

The same ratio identities give lower bounds. If, for some \(c>0\),
\[
        \psi_\phi(k;\eps)\geq cs^k,
        \qquad k\geq1,
\]
then \(\mathcal R_{n,1}(\eps)\geq cF_n^+(s)\).  The same conclusion holds
for \(\mathcal R_{n,2}(\eps)\) when
\(\psi_D(k;\eps)\geq cs^k\).  If
\[
        \Pbf\{N(k)=0\}\geq cs^k,
        \qquad k\geq1,
\]
then \(\mathcal R_{n,3}(\eps)\geq cF_n^+(s)\), whereas
\[
        \psi_\xi(\ell;\eps)\geq cs^\ell,
        \qquad \ell\geq1,
\]
implies \(\mathcal R_{n,3}(\eps)\geq cH_{n,>0}^+(s)\).

\begin{proof}[Proof of Proposition~\ref{prop:main-random-one-step-transfer}]
Under \((\mathrm{RM}_p)\), Proposition~\ref{prop:controlled-poly-kernels}
and the ratio identities give
\[
        \mathcal R_n(\eps)
        \leq C_\eps\mathcal H_n^X(q),
\]
which proves \eqref{eq:main-random-polynomial-transfer}.  The cutoff bound
\eqref{eq:main-random-cutoff-transfer} follows from
Lemma~\ref{lem:controlled-kernel-transfer}\textup{(iii)}.  Under
Assumption~\ref{ass:controlled-exp-kernels},
Proposition~\ref{prop:controlled-exponential-transfer} gives
\eqref{eq:main-random-exponential-transfer}.  The reverse polynomial and
geometric transfers follow respectively from
Lemma~\ref{lem:controlled-kernel-transfer}\textup{(ii)} and the
reverse-transfer consequences following
Proposition~\ref{prop:controlled-exponential-transfer}.
\end{proof}

\begin{lemma}[Supercritical hypotheses and persistence probabilities]
\label{lem:controlled-induced-parameters}
Assume \eqref{eq:main-individual-sum-control}. If \(\xi,L\ge1\)
almost surely, then, with \(a_1=\Pbf(L=1)\Pbf(\xi=1)\),
\[
 p_0^{\mathrm{ind}}=0,\qquad
 p_1^{\mathrm{ind}}=\one_{\{A=1\}}a_1,\qquad
 h_0^{\mathrm{ind}}=\one_{\{B=0\}}.
\]
For \(r>0\) and \(i\ge1\),
\begin{equation}\label{eq:controlled-induced-cr-gamma}
 c_r=\Ebf[(m\lambda A)^{-r}],\qquad
 \gamma_i=\Pbf(A=1,B=0)a_1^i,
\end{equation}
whereas \(\Pbf_0(X_1=0)=\Pbf(B=0)\).
Under \((\mathrm{RM}_p)\), conditions \textup{(H3)}--\textup{(H5)} of
Assumption~\ref{ass:bprei-harmonic-supercritical} hold. If also
\(\xi,L\ge1\) almost surely, \(0<a_1<1\), and \(\Pbf(A=1,B=0)>0\),
then \textup{(H1)} and \textup{(H2)} hold.
\end{lemma}

\begin{proof}
One induced child requires \(A=L=1\) and \(\xi_1=1\) when \(\xi,L\ge1\);
zero immigration is then equivalent to \(B=0\). The persistence
probabilities follow since \(\gamma_i=\Ebf[h_0^{\mathrm{ind}}
(p_1^{\mathrm{ind}})^i]\) for \(i\ge1\).

Under \((\mathrm{RM}_p)\), conditional Jensen at \(k=1\) gives
\(\Ebf B^p\le\Ebf|N(1)-\lambda A|^p<\infty\). Hence
\(A(L-\lambda)=N(1)-\lambda A-B\) has a finite \(p\)th moment.
Since \(A\ge1\) and \(p\ge2\), both \(L\) and \(\xi\) have finite
second moments. For \(Y=\sum_{j=1}^{AL}\xi_j\) and \(M=m\lambda A\),
\[
 \frac{\Ebf[Y^2\mid A,B]}{M}
 =\frac{\Var(\xi)}m+\frac m\lambda A\Ebf L^2 .
\]
This is integrable by the standing assumption \(\Ebf A<\infty\), so
\(Y\log^+Y\le Y^2\) gives \textup{(H3)}. Also,
\(\Ebf M<\infty\) and \(\Ebf\eta=m\Ebf B<\infty\) give
\textup{(H4)} and \textup{(H5)}. Under the additional conditions,
\(0<a_1<1\) implies \(m\lambda>1\), and so \(0<\Ebf\log M<\infty\).
On \(\{A=1,B=0\}\) we have \(p_1^{\mathrm{ind}}=a_1\in(0,1)\) and
\(h_0^{\mathrm{ind}}=1\). This proves \textup{(H1)}--\textup{(H2)}.
\end{proof}

\begin{proof}[Proof of Theorem~\ref{thm:main-random-harmonic-transfer}]
For \(k\ge1\), write
\[
 (g_1(k),g_2(k),g_3(k))=
 (\psi_\phi(k;\eps),\psi_D(k;\eps),\psi_N(k;\eps)),
 \qquad g_\eps(k)=\sum_{j=1}^3g_j(k),
\]
and set all four kernels to zero at \(k=0\). By
Proposition~\ref{prop:controlled-ratio-identities},
\(\mathcal R_{n,j}(\eps)=\Ebf_i[g_j(X_n)]\) and
\(\mathcal R_n(\eps)=\Ebf_i[g_\eps(X_n)]\).

Suppose first that \((\mathrm H)\) holds and \(i\ge1\).
We use the harmonic-moment argument of Ney and Vidyashankar
\cite{NeyVidyashankar2003,NeyVidyashankar2004}.
If \((U_j)\) are conditionally i.i.d. and centered given \(\mathcal G\),
with \(\Ebf|U_1|^p<\infty\), split each summand at \(|U_j|=T\) and
center both parts conditionally. Applying conditional Rosenthal at order
\(p+2\) to the bounded part and at order \(p\) to the remainder gives
\[
 k^q\Pbf\!\left(\left|\sum_{j=1}^kU_j\right|>\eps k\right)
 \le C_{p,\eps,T}k^{-1}+C_{p,\eps}\Ebf[|U_1|^p;|U_1|>T].
\]
Let \(k\to\infty\) and then \(T\to\infty\). Adding a variable
\(V\) with \(\Ebf|V|^p<\infty\) preserves the zero limit by Markov's
inequality at tolerance \(\eps k/2\).

Condition on \(\mathcal G=\sigma(A,B)\). For the control take
\(U_j=A(L_j-\lambda)\), \(V=B\); for the innovation use the grouped
offspring \(U_j=Y_j^{(A)}-M\), \(V=\eta^{(B)}\) from
\eqref{eq:controlled-BPREI-grouped}. In either case \(U_1+V\) has
finite \(p\)th moment by \((\mathrm{RM}_p)\) and the bound for
\(D(1)\) in the proof of Proposition~\ref{prop:controlled-poly-kernels}.
Since \(\Ebf[U_1+V\mid\mathcal G,V]=V\), conditional Jensen gives
\(\Ebf|V|^p+\Ebf|U_1|^p<\infty\). Thus
\(\psi_\phi(k;\eps),\psi_D(k;\eps)=o(k^{-q})\).
The same estimate applied to \(\xi_j-m\) gives
\(\psi_\xi(\ell;\eps)=o(\ell^{-q})\); since \(N(k)\ge k\),
\(\psi_N(k;\eps)\le\sup_{\ell\ge k}\psi_\xi(\ell;\eps)=o(k^{-q})\).

Given \(\delta>0\), choose \(K\) such that \(g_\eps(k)\le\delta k^{-q}\)
for \(k>K\). Then
\[
 \mathcal R_n(\eps)=\Ebf_i[g_\eps(X_n)]
 \le 3\Pbf_i(X_n\le K)+\delta\Ebf_i X_n^{-q}.
\]
For \(r>q\), \(\Pbf_i(X_n\le K)\le K^r\Ebf_i X_n^{-r}\).
When \(c_q>\gamma_i\), choose \(r>q\) with
\(\gamma_i<c_r<c_q\); the finite-state term is then \(O_K(c_r^n)\)
by Theorem~\ref{thm:bprei-supercritical-harmonic}. When
\(c_q=\gamma_i\), any \(r>q\) gives \(c_r<\gamma_i\), and the
term is \(O_K(\gamma_i^n)\). Dividing by \(c_q^n\) and
\(n\gamma_i^n\), respectively, applying the same theorem at order \(q\),
and letting \(\delta\downarrow0\) proves the two assertions.
When \(c_q<\gamma_i\), Proposition~\ref{prop:main-random-one-step-transfer}
and Theorem~\ref{thm:bprei-supercritical-harmonic} give
\(\mathcal R_n(\eps)\le C_\eps\Ebf_i[X_n^{-q}]\le C_{i,\eps}\gamma_i^n\).
Since the process is nondecreasing, \(\Pbf_i(X_n=i)=\gamma_i^n\).
Restricting \(\Ebf_i[g_\eps(X_n)]\) to this event proves the lower bound
in \eqref{eq:main-induced-super-ratio-profile}.

Suppose now that the induced BPREI is critical and that \((\mathrm C)\)
holds for the given initial state \(i\).
By \eqref{eq:controlled-poly-kernel-bound} at \(\eps\wedge1\) and
monotonicity in the tolerance, \(0\le g_j(k)\le C_\eps k^{-q}\) for
\(k\ge1\). Applying Corollary~\ref{cor:bprei-critical-bounded-kernel}
to the induced BPREI gives
\begin{equation}\label{eq:controlled-critical-transfer-limit}
 \frac{\mathcal R_{n,j}(\eps)}{d_n^-}\longrightarrow
 K_{i,j}(\eps):=\sum_{m\ge0}
        \Ebf_{\nu_i^-}[g_j(Z_m);L_m\ge0]<\infty.
\end{equation}
This constant is positive if and only if there exist \(m\ge0\) and
\(k\ge1\) such that \(g_j(k)>0\) and
\(\Pbf_{\nu_i^-}(Z_m=k,L_m\ge0)>0\).
Since \(\mathcal R_n(\eps)=\sum_{j=1}^3\mathcal R_{n,j}(\eps)\),
its limit after division by \(d_n^-\) is \(\sum_{j=1}^3K_{i,j}(\eps)\).
In the class of Theorem~\ref{thm:bprei-critical-fv-class},
\(\nu_i^-=\nu^-\) and \(\sqrt n\,d_n^-\to\kappa_0\); hence
\(K_{i,j}(\eps)=K_j(\eps)\) is independent of \(i\) and
\(\sqrt n\,\mathcal R_{n,j}(\eps)\to\kappa_0K_j(\eps)\).
\end{proof}

The two regimes are disjoint in the individual-sum primitives.  Indeed,
\begin{equation}\label{eq:controlled-induced-p0-equivalence}
\begin{aligned}
        p_0^{\mathrm{ind}}=0\ \text{almost surely}
        &\quad\Longleftrightarrow\quad
        \Pbf(\xi=0)=\Pbf(L=0)=0,\\
        \Pbf(p_0^{\mathrm{ind}}>0)>0
        &\quad\Longleftrightarrow\quad
        \Pbf(\xi=0)>0\ \text{or}\ \Pbf(L=0)>0.
\end{aligned}
\end{equation}
If \(p_0^{\mathrm{ind}}=0\) almost surely, then \(M=m\lambda A\geq1\)
almost surely.  Criticality forces \(M=1\) almost surely and the induced
reproduction law to be concentrated at one.  In the persistent-immigration
primitive class of
Proposition~\ref{prop:controlled-induced-critical-verification},
\(\Pbf(\xi=0)=0\), so a nondegenerate critical law necessarily has
\(\Pbf(L=0)>0\).

In the individual-sum model, let \(\ell(s)=\Ebf[s^L]\). For
\(0<s<1\) and \(k\ge1\), independence of the \(L_j\) from \((A,B)\)
and \(A\ge1\) give
\begin{equation}\label{eq:controlled-sum-H-upper}
 \Ebf[s^{N(k)}]=\Ebf[s^B\ell(s^A)^k]\le\ell(s)^k,
 \qquad \Pbf\{N(k)=0\}\le\Pbf(L=0)^k.
\end{equation}
Averaging the first bound over \(X_n>0\) gives
\(H_{n,>0}^+(s)\le F_n^+(\ell(s))\).
On \(X_n=0\) we have \(N_n=B_n\), and \(B_n\) is independent of \(X_n\).
Thus the full progenitor transform is
\[
 H_n^+(s)=H_{n,>0}^+(s)+\Pbf(X_n=0)\Ebf[s^B;B>0].
\]
Under the geometric kernel bounds, the offspring-ratio estimate becomes
\[
 \mathcal R_{n,3}(\eps)\le C_\eps F_n^+(s_{0,\eps})
             +C_\eps F_n^+(\ell(s_{\xi,\eps})).
\]
The first term bounds \(\{X_n>0,N_n=0\}\); the second bounds the
deviation on \(\{X_n>0,N_n>0\}\).

\subsection{Negative tilts and the deterministic-slope boundary}
\label{subsec:controlled-boundaries}

\begin{proof}[Proof of Proposition~\ref{prop:main-negative-tilt-boundary}]
Apply \eqref{eq:bprei-harmonic-negative-tilt} with \(q=r\) to the induced
BPREI. The change of measure in that identity requires only
\(-r\in\operatorname{int}\cD_\Lambda\), irrespective of the sign of \(\mu\).
\end{proof}

If \(A_n\equiv a_0\), then \(M_n\equiv m\lambda a_0\) and
\(\Lambda(\vartheta)=\vartheta\log(m\lambda a_0)\), so
\(\Lambda''(\alpha)=0\). Under the hypotheses of
Lemma~\ref{lem:controlled-innovation-bound},
\eqref{eq:controlled-slope-Ls} gives
\(\phi_n(k)/k\to\tau:=a_0\lambda\) in \(L^s\) as \(k\to\infty\).
In the individual-sum model the induced offspring p.g.f. is
\(\ell(f(s)^{a_0})\). Conditional on \(B_n\), the immigration p.g.f. is
\(f(s)^{B_n}\). Under the annealed law this is an ordinary branching
process with immigration, with immigration p.g.f. \(\Ebf[f(s)^B]\). A deterministic
first-order control slope alone does not imply this branching
representation. Section~\ref{sec:ad-controls} treats controls with
deterministic first-order slope under the hypotheses of
Section~\ref{subsec:main-ad}.

\section{Asymptotically deterministic controls: sandwich and transfer calculus}
\label{sec:ad-controls}

For controls with deterministic first-order slope, condition \textup{(AD0)}
bounds the population p.g.f.\ between two auxiliary p.g.f.s. For each
\(e\in\{L,U\}\), the auxiliary process is a branching process with
immigration, with offspring p.g.f.\ \(u_e=l_e\circ f\) and immigration
p.g.f.\ \(R_e=d_e\circ f\). We use these two processes to estimate the
harmonic moments of the population \(X_n\) and the progenitor count \(N_n\).

Lower and upper control p.g.f.\ comparisons are part of the
controlled-branching-process theory developed by Gonz\'alez, Molina and del
Puerto
\cite{GonzalezMolinaDelPuerto2002,GonzalezMolinaDelPuerto2003,
GonzalezMolinaDelPuerto2005L2,GonzalezMolinaDelPuerto2005Critical,
GonzalezMolinaDelPuerto2006Growth,GonzalezMolinaDelPuerto2006Limits}.
At each iteration the comparison evaluates the envelope at a new
argument and introduces another immigration factor \(R_e\).

The one-generation ratio-mixture viewpoint and the harmonic-moment method
used below originate in the Galton--Watson large-deviation programme of
Athreya and Vidyashankar and its subsequent refinements
\cite{AthreyaVidyashankar1993,Athreya1994,
AthreyaVidyashankar1995,AthreyaVidyashankar1997,
NeyVidyashankar2003,NeyVidyashankar2004}.  In the Galton--Watson setting, the
branching property supplies the exact random-sum reduction.  Here the control
sandwich supplies the transfer mechanism instead.  We keep the population
and progenitor transforms separate because the three ratios are indexed by
two different random denominators.

Throughout this section, \(O(\cdot)\) and \(o(\cdot)\) terms involving the
generation index refer to \(n\to\infty\) unless another limit is stated
locally. We use \textup{(AD0)} together with \textup{(AD-S)} in the supercritical
case \(a=m\tau>1\), and with \textup{(AD-C)} in the critical case \(a=1\).
For the ratio bounds we use \eqref{eq:main-ad-poly-conditions} for the
polynomial estimates and \textup{(AD-E)} for the exponential estimates.

\subsection{The control sandwich and the two transforms}
\label{subsec:ad-sandwich}

Recall that \(X_{n+1}=\sum_{j=1}^{\phi_n(X_n)}\xi_{n,j}\), where the
controls are independent across generations and independent of the offspring
array.  Let \(f(s):=\Ebf[s^\xi]\), \(h_k(s):=\Ebf[s^{\phi(k)}]\),
\(F_n(s):=\Ebf[s^{X_n}]\), and \(H_n(s):=\Ebf[s^{N_n}]\), where
\(N_n:=\phi_n(X_n)\). The current control map and offspring row are independent of each other
and, jointly, of \(X_n\); conditioning on \(X_n\) therefore gives
\begin{equation}\label{eq:ad-basic-pgf-identities}
        F_{n+1}(s)=\Ebf[h_{X_n}(f(s))],
        \qquad H_n(s)=\Ebf[h_{X_n}(s)].
\end{equation}
The first transform is used when the random denominator is \(X_n\); the
second is needed for the offspring ratio with denominator \(N_n\).

Under \textup{(AD0)},
\begin{equation}\label{eq:ad-envelope}
        d_L(s)l_L(s)^k\leq h_k(s)\leq d_U(s)l_U(s)^k,
        \qquad k\in\Nzero,\quad 0\leq s\leq1,
\end{equation}
where \(d_e(1)=l_e(1)=1\) for \(e\in\{L,U\}\),
\(l_L'(1)=l_U'(1)=\tau\in(0,\infty)\), and
\(l_L''(1)+l_U''(1)<\infty\) whenever harmonic moments of \(N_n\) are
asserted.  Set \(u_e:=l_e\circ f\), \(R_e:=d_e\circ f\), and let
\begin{equation}\label{eq:ad-envelope-recursion}
        F_{n+1}^e(s)=R_e(s)F_n^e(u_e(s)),
        \qquad F_0^e=F_0,
        \qquad e\in\{L,U\}.
\end{equation}
Both auxiliary offspring laws have mean \(u_e'(1)=m\tau=a\). We call the sandwich
normalized because \(h_k(1)=d_e(1)l_e(1)^k=1\); each envelope therefore
defines a probability-generating-function recursion.

\begin{proposition}[Population and progenitor envelope calculus]
\label{prop:ad-FH-calculus}
Under \textup{(AD0)}, for every \(n\geq0\) and \(s\in[0,1]\),
\begin{equation}\label{eq:ad-F-sandwich}
        F_n^L(s)\leq F_n(s)\leq F_n^U(s),
\end{equation}
and
\begin{equation}\label{eq:ad-H-sandwich}
        d_L(s)F_n^L(l_L(s))\leq H_n(s)\leq d_U(s)F_n^U(l_U(s)).
\end{equation}
Moreover,
\begin{equation}\label{eq:ad-F-product}
        F_n^e(s)=F_0(u_e^{(n)}(s))
        \prod_{j=0}^{n-1}R_e(u_e^{(j)}(s)),
\end{equation}
and, with \(H_n^e(s):=d_e(s)F_n^e(l_e(s))\),
\begin{equation}\label{eq:ad-H-product}
        H_n^e(s)=d_e(s)F_0(u_e^{(n)}(l_e(s)))
        \prod_{j=0}^{n-1}R_e(u_e^{(j)}(l_e(s))).
\end{equation}
Finally, \(H_n^e(f(s))=F_{n+1}^e(s)\).
\end{proposition}

\begin{proof}
The control sandwich and \eqref{eq:ad-basic-pgf-identities} give
\(R_L(s)F_n(u_L(s))\leq F_{n+1}(s)\leq R_U(s)F_n(u_U(s))\), and induction
proves \eqref{eq:ad-F-sandwich}.  Applying the same control sandwich to
\(H_n(s)=\Ebf[h_{X_n}(s)]\), and then using \eqref{eq:ad-F-sandwich}, gives
\eqref{eq:ad-H-sandwich}.  Iteration of \eqref{eq:ad-envelope-recursion}
gives \eqref{eq:ad-F-product}; substitution of \(l_e(s)\) gives
\eqref{eq:ad-H-product} and the shift identity.
\end{proof}

\begin{remark}[Role of the accumulated envelope factor]
\label{rem:ad-product-mechanism}
The product \(\prod_{j=0}^{n-1}R_e(u_e^{(j)}(s))\) cannot generally be
absorbed into a multiplicative constant.  In the supercritical auxiliary
recursion it contributes to the exponential normalization of the positive
transform.  At criticality, the parabolic iteration below gives, for every
fixed \(s\in(0,1)\),
\[
        \sum_{j=0}^{n-1}\log R_e(u_e^{(j)}(s))
        =-\frac{\beta_e}{\gamma_e}\log n+O(1).
\]
Consequently, the product is of order \(n^{-\beta_e/\gamma_e}\); the
additive envelope component therefore determines the critical index
\(\sigma_e=\beta_e/\gamma_e\).
\end{remark}

The normalized factorization is essential for exact iteration.  If a
one-step comparison contains additive errors, repeated substitution
accumulates them with the envelope products already generated.

\begin{proposition}[Accumulated remainder terms]
\label{prop:ad-raw-remainders}
For each \(k\in\Nzero\), let
\(r_k^L,r_k^U:[0,1]\to[0,\infty)\) be nonnegative and suppose
\begin{equation}\label{eq:ad-raw-envelope}
        d_L(s)l_L(s)^k-r_k^L(s)
        \leq h_k(s)
        \leq d_U(s)l_U(s)^k+r_k^U(s).
\end{equation}
Set \(\mathcal V_n^e(s):=\Ebf[r_{X_n}^e(f(s))]\) and
\[
\mathcal E_n^e(s):=
\sum_{t=0}^{n-1}
\left\{\prod_{j=0}^{t-1}R_e(u_e^{(j)}(s))\right\}
\mathcal V_{n-1-t}^e(u_e^{(t)}(s)).
\]
Then
\begin{align}
F_n(s)&\leq F_n^U(s)+\mathcal E_n^U(s),
\label{eq:ad-raw-upper}\\
F_n(s)&\geq F_n^L(s)-\mathcal E_n^L(s).
\label{eq:ad-raw-lower}
\end{align}
\end{proposition}

\begin{proof}
For \(e\in\{L,U\}\), put \(T_e g(s):=R_e(s)g(u_e(s))\).
By \eqref{eq:ad-basic-pgf-identities} and \eqref{eq:ad-raw-envelope},
\[
 T_L F_n-\mathcal V_n^L\le F_{n+1}\le T_U F_n+\mathcal V_n^U.
\]
Each \(T_e\) preserves order, \(F_n^e=T_e^nF_0\), and the accumulated
error in the statement satisfies
\(\mathcal E_0^e=0\) and
\(\mathcal E_{n+1}^e=\mathcal V_n^e+T_e\mathcal E_n^e\).
Induction now gives
\(F_n^L-\mathcal E_n^L\le F_n\le F_n^U+\mathcal E_n^U\).
\end{proof}

For the progenitor transform, the current generation contributes a
further error. For example, the upper comparison is
\[
 H_n(s)\le d_U(s)F_n^U(l_U(s))
       +d_U(s)\mathcal E_n^U(l_U(s))
       +\Ebf[r_{X_n}^U(s)] .
\]
The lower comparison follows in the same way. Extending the harmonic-moment
bounds to this setting requires estimates for the accumulated and current
errors after subtracting the mass at zero.

\subsection{Exact envelope dynamics}
\label{subsec:ad-exact-dynamics}

Fix one of the two envelopes and suppress its index.  We analyze the generic
exact transform recursion
\begin{equation}\label{eq:ad-exact-recursion}
        \mathsf F_{n+1}(s)=R(s)\mathsf F_n(u(s)),
        \qquad \mathsf F_0(1)=1,
\end{equation}
where \(u=l\circ f\) is the transformed reproduction p.g.f.\ and
\(R=d\circ f\) is the transformed immigration p.g.f.  We use
\(u^{(0)}(s):=s\) and \(u^{(k+1)}(s):=u(u^{(k)}(s))\), \(k\geq0\).
Iteration gives
\begin{equation}\label{eq:ad-exact-product}
        \mathsf F_n(s)=\mathsf F_0(u^{(n)}(s))
        \prod_{j=0}^{n-1}R(u^{(j)}(s)).
\end{equation}
Let \(\mathsf Z_n\) denote the auxiliary generation size with p.g.f.\
\(\mathsf F_n\).  The normalization theory of this recursion goes back to
Heathcote, Seneta, and Pakes; see
\cite{Heathcote1965,Heathcote1966,Seneta1968,Seneta1970,
Pakes1971,AthreyaNey1972}.

\subsubsection{Supercritical exact dynamics}

Assume \(a:=u'(1)>1\), and let \(q\in[0,1)\) be the attracting fixed point
of \(u\), with \(u(q)=q\) and \(0<u'(q)<1\).  Define the auxiliary
small-value rate by
\begin{equation}\label{eq:ad-super-theta}
\theta:=
\begin{cases}
R(0)u'(0), & q=0,\\[1mm]
R(q), & q>0\text{ and }R(q)<1,\\[1mm]
u'(q), & q>0\text{ and }R\equiv1,
\end{cases}
\end{equation}
and assume \(0<\theta<1\).  Write
\(u(s)=\sum_{k\geq0}u_ks^k\), let \(\eta\) have p.g.f.\ \(R\), and assume
\begin{equation}\label{eq:ad-super-xlogx}
        \sum_{k\geq1}k\log^+k\,u_k<\infty,
        \qquad
        \Ebf\log(1+\eta)<\infty.
\end{equation}
Then
\begin{equation}\label{eq:ad-super-normalization}
        a^{-n}\mathsf Z_n\longrightarrow W
        \quad\text{almost surely},
        \qquad \Pbf(0<W<\infty)>0,
\end{equation}
provided the initial or immigration mechanism is nontrivial
\cite{AthreyaNey1972}.

In the branch \(q=0\), assume
\begin{equation}\label{eq:ad-super-initial-zero-branch}
        \mathsf F_0(s)=\mu_1s+O(s^2)
        \quad(s\downarrow0),
        \qquad \mu_1>0.
\end{equation}
If the minimum of the support of the initial law is \(j_0>1\), then the same
argument gives the rate \(R(0)u'(0)^{j_0}\); here we treat \(j_0=1\). In the branch \(q>0\), \(R\equiv1\), assume
\(\mathsf F_0'(q)>0\).

The following classical iteration theorem will be used repeatedly; see
Athreya and Ney \cite{AthreyaNey1972}.

\begin{lemma}[Koenigs linearization]
\label{lem:ad-koenigs}
Let \(\vartheta:=u'(q)\).  There is a continuous function \(K_{u,q}\) on
the basin of attraction of \(q\) such that
\begin{equation}\label{eq:ad-koenigs-limit}
        \frac{u^{(n)}(s)-q}{\vartheta^n}
        \longrightarrow K_{u,q}(s)
\end{equation}
locally uniformly on compact subsets of that basin, and
\(K_{u,q}(u(s))=\vartheta K_{u,q}(s)\). On \([0,1)\), \(K_{u,q}\) is strictly increasing, with
\(K_{u,q}(q)=0\) and \(K_{u,q}'(q)=1\). If \(q=0\), write
\(K_u:=K_{u,0}\).
\end{lemma}

\begin{proposition}[Supercritical generating-function profiles]
\label{prop:ad-super-fixed-profile}
Assume \eqref{eq:ad-exact-recursion}--\eqref{eq:ad-super-normalization}.
\begin{enumerate}[label=\textup{(\roman*)}]
\item If \(q>0\) and \(R(q)<1\), then, locally uniformly for \(s<1\),
\begin{equation}\label{eq:ad-super-positive-q-limit}
        R(q)^{-n}\mathsf F_n(s)
        \longrightarrow
        \mathsf Q(s):=\mathsf F_0(q)
        \prod_{j=0}^{\infty}\frac{R(u^{(j)}(s))}{R(q)}.
\end{equation}
\item If \(q=0\), \(R(0)>0\), and
\eqref{eq:ad-super-initial-zero-branch} holds, then, locally uniformly for
\(s<1\),
\begin{equation}\label{eq:ad-super-zero-q-limit}
        \theta^{-n}\mathsf F_n(s)
        \longrightarrow
        \mu_1K_u(s)
        \prod_{j=0}^{\infty}\frac{R(u^{(j)}(s))}{R(0)}.
\end{equation}
\item If \(q>0\), \(R\equiv1\), and \(\mathsf F_0'(q)>0\), then, locally
uniformly for \(s<1\),
\begin{equation}\label{eq:ad-super-noimm-positive-transform}
        u'(q)^{-n}\{\mathsf F_n(s)-\mathsf F_n(0)\}
        \longrightarrow
        \mathsf F_0'(q)\{K_{u,q}(s)-K_{u,q}(0)\}.
\end{equation}
\end{enumerate}
\end{proposition}

\begin{lemma}[Positivity of the supercritical positive profile]
\label{lem:ad-super-subtracted-positive}
In the branch \(q>0\), \(R(q)<1\), the function \(\mathsf Q\) in
\eqref{eq:ad-super-positive-q-limit} is continuous and strictly increasing.
Consequently, for every compact \(K\Subset(q,1)\),
\begin{equation}\label{eq:ad-super-profile-positive}
        \inf_{s\in K,\,0\leq r\leq q}
        \{\mathsf Q(s)-\mathsf Q(r)\}>0.
\end{equation}
The limiting positive profiles in the other two branches are strictly
positive for every \(s>0\).
\end{lemma}

The standard proofs of Proposition~\ref{prop:ad-super-fixed-profile} and
Lemma~\ref{lem:ad-super-subtracted-positive} are recorded in
Appendix~\ref{app:ad-standard-facts}.  In particular, for every compact
\(K\Subset(q,1)\), there are \(m_0<\infty\) and
\(0<c_K<C_K<\infty\) such that, for \(m\geq m_0\), \(s\in K\),
and \(0\leq r\leq q\),
\begin{equation}\label{eq:ad-super-compact-profile}
        c_K\theta^m
        \leq\mathsf F_m(s)-\mathsf F_m(r)
        \leq C_K\theta^m.
\end{equation}
When \(q=0\), \(u(0)=0\) and \(\mathsf F_0(0)=0\), so that
\(\mathsf F_m(0)=0\). Thus \eqref{eq:ad-super-compact-profile} reads
\(c_K\theta^m\le\mathsf F_m(s)\le C_K\theta^m\) for
\(s\in K\Subset(0,1)\) and \(m\ge m_0\).

Set \(\mathsf G_n(t):=\mathsf F_n(\e^{-t})-\mathsf F_n(0)\) and
\begin{equation}\label{eq:ad-super-kappa}
        \kappa_{\mathrm{sc}}
        :=-\frac{\log\theta}{\log a},
        \qquad \theta a^{\kappa_{\mathrm{sc}}}=1.
\end{equation}
The next lemma links the macroscopic scale \(a^{-n}\) to fixed Laplace
arguments.

\begin{lemma}[Supercritical small-value window]
\label{lem:ad-super-window}
There are \(t_0\in(0,1)\) and \(0<c<C<\infty\) such that, for all
sufficiently large \(n\),
\begin{equation}\label{eq:ad-super-window}
        c\theta^nt^{-\kappa_{\mathrm{sc}}}
        \leq\mathsf G_n(t)
        \leq C\theta^nt^{-\kappa_{\mathrm{sc}}},
        \qquad a^{-n}\leq t\leq t_0.
\end{equation}
Moreover,
\begin{equation}\label{eq:ad-super-tail}
        \mathsf G_n(t)\leq C\theta^n\e^{-(t-1)},
        \qquad t\geq1.
\end{equation}
\end{lemma}

\begin{proof}
The proof separates the transformed offspring iteration from the accumulated
immigration factors.  For \(t\leq t_0\), set
\(\delta_j(t):=1-u^{(j)}(\e^{-t})\), and define
\(g(x):=ax-\{1-u(1-x)\}\).  Convexity gives \(g\geq0\) and \(g\)
nondecreasing, while the Kesten--Stigum condition
\eqref{eq:ad-super-xlogx} is equivalent to
\begin{equation}\label{eq:ad-super-KS-integral}
        \int_0^1\frac{g(x)}{x^2}\dd x<\infty.
\end{equation}
Choose \(\underline a\in(1,a)\) and \(\delta_*\in(0,(1-q)/a)\) so small
that, for \(0<x\leq\delta_*\),
\[
        1-u(1-x)=ax\{1+\varepsilon(x)\},\quad
        \underline a\leq a\{1+\varepsilon(x)\}\leq a,
        \quad |\varepsilon(x)|\leq\tfrac12,
        \quad R(1-a\delta_*)\geq\tfrac12.
\]
Choose \(t_0\in(0,1)\) so that \(1-\e^{-t_0}<\delta_*\).
Then \(\delta_0(t)<\delta_*\) for \(0<t\le t_0\).
Let \(j(t):=\inf\{j\geq0:\delta_j(t)\geq\delta_*\}\).  Before \(j(t)\),
\(\underline a\leq\delta_{j+1}(t)/\delta_j(t)\leq a\), and
\eqref{eq:ad-super-KS-integral} gives
\[
        \sup_{0<t\leq t_0}
        \sum_{j<j(t)}|\varepsilon(\delta_j(t))|<\infty.
\]
Hence, for \(j\leq j(t)\),
\begin{equation}\label{eq:ad-super-delta-comparison}
        c a^jt\leq\delta_j(t)\leq C a^jt.
\end{equation}
If \(j(t)<\infty\), then
\begin{equation}\label{eq:ad-super-hitting-time}
        \left|j(t)-\frac{\log(1/t)}{\log a}\right|\leq C,
        \qquad
        c\theta^nt^{-\kappa_{\mathrm{sc}}}
        \leq\theta^{n-j(t)}
        \leq C\theta^nt^{-\kappa_{\mathrm{sc}}}.
\end{equation}
At the first crossing,
\(u^{(j(t))}(\e^{-t})\in
K_*:=[1-a\delta_*,1-\delta_*]\Subset(q,1)\).

Let \(P_k(s):=\prod_{j=0}^{k-1}R(u^{(j)}(s))\).  Since
\(1-R(1-x)\leq\Ebf[\min\{1,\eta x\}]\),
\eqref{eq:ad-super-delta-comparison} and
\(\Ebf\log(1+\eta)<\infty\) give
\[
\sum_{j<j(t)}\{1-R(u^{(j)}(\e^{-t}))\}
\leq C\Ebf\!\left[\sum_{r\geq1}\min\{1,\eta a^{-r}\}\right]
\leq C\Ebf[1+\log(1+\eta)]<\infty.
\]
Therefore
\begin{equation}\label{eq:ad-super-early-product}
        0<c_R\leq P_{j(t)}(\e^{-t})\leq1.
\end{equation}

If \(j(t)>n\), then \eqref{eq:ad-super-delta-comparison} and
\(t\geq a^{-n}\) give \(1\leq a^nt\leq C\).  The normalization
\eqref{eq:ad-super-normalization} yields
\(0<c\leq\mathsf G_n(t)\leq C\), while
\(\theta^nt^{-\kappa_{\mathrm{sc}}}=(a^nt)^{-\kappa_{\mathrm{sc}}}\asymp1\).

If \(k:=j(t)\leq n\), set \(s_t:=\e^{-t}\),
\(s_k:=u^{(k)}(s_t)\in K_*\), and \(r_k:=u^{(k)}(0)\in[0,q]\).  From
\eqref{eq:ad-exact-product},
\begin{align*}
\mathsf G_n(t)
={}&P_k(s_t)\{\mathsf F_{n-k}(s_k)-\mathsf F_{n-k}(r_k)\}\\
&+\{P_k(s_t)-P_k(0)\}\mathsf F_{n-k}(r_k).
\end{align*}
Both terms are nonnegative.  If \(n-k<m_0\), then
\eqref{eq:ad-super-hitting-time} and \(t\geq a^{-n}\) imply
\(a^nt\asymp1\); the normalization \eqref{eq:ad-super-normalization}
therefore gives \(\mathsf G_n(t)\asymp1\), as does
\(\theta^nt^{-\kappa_{\mathrm{sc}}}\).  If \(n-k\geq m_0\),
\eqref{eq:ad-super-compact-profile} makes the first difference comparable
with \(\theta^{n-k}\).  In the two immigration branches
\(\mathsf F_{n-k}(r_k)\leq C\theta^{n-k}\); if \(q=0\), then
\(r_k=0\); and if \(R\equiv1\), the second term vanishes.  Combining
these bounds with \eqref{eq:ad-super-early-product} and
\eqref{eq:ad-super-hitting-time} proves \eqref{eq:ad-super-window}.
Finally, integer-valuedness gives
\(\mathsf G_n(t)\leq\e^{-(t-1)}\mathsf G_n(1)\) for \(t\geq1\), and the
fixed-argument profile gives \(\mathsf G_n(1)\leq C\theta^n\).
\end{proof}

Define
\begin{equation}\label{eq:ad-exact-super-scale}
        \mathfrak b_n^{\mathrm{sc}}(r):=
        \begin{cases}
        \theta^n, & \theta a^r>1,\\[1mm]
        n\theta^n, & \theta a^r=1,\\[1mm]
        a^{-rn}, & \theta a^r<1.
        \end{cases}
\end{equation}

\begin{proposition}[Supercritical harmonic-moment rates for the exact recursion]
\label{prop:ad-exact-super-harmonic}
Under the preceding supercritical assumptions, for every \(r>0\),
\begin{equation}\label{eq:ad-exact-super-harmonic}
        0<\liminf_{n\to\infty}
        \frac{\Ebf[\mathsf Z_n^{-r};\mathsf Z_n>0]}
             {\mathfrak b_n^{\mathrm{sc}}(r)}
        \leq
        \limsup_{n\to\infty}
        \frac{\Ebf[\mathsf Z_n^{-r};\mathsf Z_n>0]}
             {\mathfrak b_n^{\mathrm{sc}}(r)}
        <\infty.
\end{equation}
\end{proposition}

\begin{proof}
Write
\(\Gamma(r)\Ebf[\mathsf Z_n^{-r};\mathsf Z_n>0]
=I_{n,0}+I_{n,1}+I_{n,2}\), where the three integrals are over
\((0,a^{-n})\), \((a^{-n},1)\), and \((1,\infty)\), respectively.  The
normalization \eqref{eq:ad-super-normalization} gives
\(I_{n,0}\asymp a^{-rn}\). By \eqref{eq:ad-super-window},
\(\mathsf G_n(t)\asymp\theta^nt^{-\kappa_{\mathrm{sc}}}\) for
\(a^{-n}\le t\le t_0\). For \(t_0\le t\le1\),
\(\mathsf G_n(1)\le\mathsf G_n(t)\le\mathsf G_n(t_0)\);
Proposition~\ref{prop:ad-super-fixed-profile} and
Lemma~\ref{lem:ad-super-subtracted-positive} show that both endpoints
are of order \(\theta^n\). Since \(t^{-\kappa_{\mathrm{sc}}}\asymp1\)
on this fixed interval, the comparison holds throughout
\([a^{-n},1]\). Hence
\[
I_{n,1}\asymp
\theta^n\int_{a^{-n}}^1t^{r-1-\kappa_{\mathrm{sc}}}\dd t
\asymp
\begin{cases}
\theta^n, & r>\kappa_{\mathrm{sc}},\\
n\theta^n, & r=\kappa_{\mathrm{sc}},\\
a^{-rn}, & 0<r<\kappa_{\mathrm{sc}}.
\end{cases}
\]
Finally, \eqref{eq:ad-super-tail} gives \(I_{n,2}\leq C\theta^n\), and the
positive fixed-argument profile gives a matching lower bound when
\(\theta^n\) is dominant.  Since
\(\theta a^{\kappa_{\mathrm{sc}}}=1\), the three cases are exactly
\eqref{eq:ad-exact-super-scale}.
\end{proof}

\subsubsection{Critical exact dynamics}

It is classical that, for a critical p.g.f.\ \(u\) with
\(u''(1)=2\gamma\in(0,\infty)\),
\(1-u^{(k)}(s)\sim(\gamma k)^{-1}\) for every fixed \(s<1\); see, for
example, Athreya and Ney \cite{AthreyaNey1972}.  Here we also need uniform
control at the moving Laplace arguments \(s=\e^{-t/n}\).  The fixed-argument
estimate gives the \(n^{-\sigma}\) profile, whereas the moving-argument
estimate gives the critical harmonic-moment window.

Assume
\begin{equation}\label{eq:ad-critical-u}
        u(1)=u'(1)=1,
        \qquad \gamma:=\frac12u''(1)\in(0,\infty),
\end{equation}
and, for some \(\eta\in(0,1]\),
\begin{equation}\label{eq:ad-critical-u-regularity}
        u(1-x)=1-x+\gamma x^2+O(x^{2+\eta})
        \qquad(x\downarrow0).
\end{equation}
Recall that \(R=d\circ f\) is the transformed immigration p.g.f., so
\(R(1)=1\).  Assume
\begin{equation}\label{eq:ad-critical-R}
        \beta:=R'(1)\in(0,\infty),
        \qquad
        \log R(1-x)=-\beta x+O(x^2)
        \qquad(x\downarrow0),
\end{equation}
and set \(\sigma:=\beta/\gamma\).

For a fixed \(c>0\), we call \(0<t\leq cn\) the broad window.  We call
\(0<t\leq C_n\) a sublinear window when \(C_n\to\infty\) and
\(C_n/n\to0\).  The product representation shows that the critical profile
is governed by the sums of \(1-u^{(k)}(s)\) and their squares.  The next
lemma provides the fixed-argument and moving-window estimates needed below.

\begin{lemma}[Parabolic iteration at fixed and moving arguments]
\label{lem:ad-parabolic-iteration}
For \(s<1\), let \(\delta_k(s):=1-u^{(k)}(s)\).  Then
\begin{equation}\label{eq:ad-parabolic-fixed}
        \delta_k(s)=\frac1{\gamma k}+\varepsilon_k(s),
        \qquad \sum_{k\geq1}|\varepsilon_k(s)|<\infty,
\end{equation}
locally uniformly for \(s\) in compact subsets of \([0,1)\).

For \(s_{n,t}:=\e^{-t/n}\), set
\(\delta_{n,k}(t):=1-u^{(k)}(s_{n,t})\).  For every fixed \(c>0\), there
are \(0<c_{1,c}<c_{2,c}<\infty\) such that, for all sufficiently large
\(n\),
\begin{equation}\label{eq:ad-parabolic-window-bound}
        c_{1,c}\frac{t}{n+\gamma kt}
        \leq\delta_{n,k}(t)
        \leq c_{2,c}\frac{t}{n+\gamma kt},
        \qquad 0<t\leq cn,\quad0\leq k\leq n.
\end{equation}
Moreover, for every fixed \(c>0\), there are \(K_c<\infty\) and
\(n_c<\infty\) such that, for \(n\geq n_c\),
\begin{equation}\label{eq:ad-parabolic-broad-sums}
\begin{aligned}
\sup_{0<t\leq cn}
\left|
\sum_{k=0}^{n-1}\delta_{n,k}(t)
-\frac1\gamma\log(1+\gamma t)
\right|&\leq K_c,\\
\sup_{0<t\leq cn}
\sum_{k=0}^{n-1}\delta_{n,k}(t)^2&\leq K_c.
\end{aligned}
\end{equation}
For every sublinear window,
\begin{equation}\label{eq:ad-parabolic-window-sums}
\begin{aligned}
\sup_{0<t\leq C_n}
\left|
\sum_{k=0}^{n-1}\delta_{n,k}(t)
-\frac1\gamma\log(1+\gamma t)
\right|&\longrightarrow0,\\
\sup_{0<t\leq C_n}
\sum_{k=0}^{n-1}\delta_{n,k}(t)^2&\longrightarrow0.
\end{aligned}
\end{equation}
\end{lemma}

\begin{proof}
For \(0<x\le1\), put
\[
 D(x):=\frac1{1-u(1-x)}-\frac1x.
\]
The nondegenerate critical p.g.f.\ satisfies \(u(s)>s\) for \(0\le s<1\).
Thus \(D(x)>0\), and \eqref{eq:ad-critical-u-regularity} gives
\(D(x)=\gamma+O(x^\eta)\) as \(x\downarrow0\).
With \(D(0)=\gamma\), continuity on \([0,1]\) gives
\(0<b\le D(x)\le B<\infty\). For \(y_k=\delta_k(s)^{-1}\),
\[
 y_{k+1}-y_k=D(1/y_k),\qquad
 y_0+bk\le y_k\le y_0+Bk .
\]
In particular \(y_k\asymp y_0+k\), uniformly in the starting argument.
Summing \(y_{k+1}-y_k=\gamma+O(y_k^{-\eta})\) for \(s\) in a
compact subset of \([0,1)\) gives
\[
y_k=\gamma k+
\begin{cases}
O(k^{1-\eta}),&0<\eta<1,\\
O(\log k),&\eta=1,
\end{cases}
\quad
\delta_k-\frac1{\gamma k}=
\begin{cases}
O(k^{-1-\eta}),&0<\eta<1,\\
O(k^{-2}\log k),&\eta=1.
\end{cases}
\]
This proves \eqref{eq:ad-parabolic-fixed}, uniformly on compact subsets of
\([0,1)\).

For the moving argument, let
\(y_{n,k}(t):=\delta_{n,k}(t)^{-1}\) and
\(y_{n,0}(t):=(1-\e^{-t/n})^{-1}\). Summing the same recursion from the moving initial value gives
\(y_{n,k}(t)=y_{n,0}(t)+\gamma k+E_{n,k}(t)\), with
\(\lvert E_{n,k}(t)\rvert\le C\sum_{j<k}(y_{n,0}(t)+j)^{-\eta}\).
In particular,
\begin{equation}\label{eq:ad-critical-window-error}
|E_{n,k}(t)|\leq
\begin{cases}
C\{y_{n,0}(t)+k\}^{1-\eta},&0<\eta<1,\\[1mm]
C\log\{1+k/y_{n,0}(t)\},&\eta=1.
\end{cases}
\end{equation}
This yields \eqref{eq:ad-parabolic-window-bound} and, uniformly on the broad
window,
\[
\sum_{k<n}
\left|
\delta_{n,k}(t)-\frac1{y_{n,0}(t)+\gamma k}
\right|\leq K_c,
\qquad
\sum_{k<n}\delta_{n,k}(t)^2\leq K_c.
\]
The harmonic-sum comparison and the bounds
\(c_c n/t\leq y_{n,0}(t)\leq C_c n/t\) give
\eqref{eq:ad-parabolic-broad-sums}.

On a sublinear window,
\[
\sup_{0<t\leq C_n}
\left|
\frac{n(1-\e^{-t/n})}{t}-1
\right|\longrightarrow0.
\]
Together with \eqref{eq:ad-critical-window-error}, this strengthens the last
two bounds to
\[
\sup_{0<t\leq C_n}
\sum_{k<n}
\left|
\delta_{n,k}(t)-\frac1{y_{n,0}(t)+\gamma k}
\right|\longrightarrow0,
\quad
\sup_{0<t\leq C_n}\sum_{k<n}\delta_{n,k}(t)^2\longrightarrow0.
\]
The same harmonic-sum comparison now proves
\eqref{eq:ad-parabolic-window-sums}.
\end{proof}

For \(0\leq x<1\), define \(\rho_R(x):=\log R(1-x)+\beta x\).  Then
\(|\rho_R(x)|\leq Cx^2\) near zero, and the exact product gives, for
\(s\in(0,1)\),
\begin{equation}\label{eq:ad-critical-log-product}
\begin{aligned}
\log\mathsf F_n(s)
={}&\log\mathsf F_0(u^{(n)}(s))
-\beta\sum_{k=0}^{n-1}\{1-u^{(k)}(s)\}\\
&+\sum_{k=0}^{n-1}\rho_R(1-u^{(k)}(s)).
\end{aligned}
\end{equation}
The first sum produces the critical power, and the square-sum estimate
controls the remainder.

\begin{proposition}[Critical generating-function and scaling profiles]
\label{prop:ad-critical-profile}
Under \eqref{eq:ad-critical-u}--\eqref{eq:ad-critical-R}, there is a
continuous function \(\mathsf U:[0,1)\to[0,\infty)\), positive on
\((0,1)\), such that, locally uniformly on \([0,1)\),
\begin{equation}\label{eq:ad-critical-fixed-profile}
        n^\sigma\mathsf F_n(s)\longrightarrow\mathsf U(s).
\end{equation}
Moreover,
\begin{equation}\label{eq:ad-critical-functional-equation}
        R(s)\mathsf U(u(s))=\mathsf U(s).
\end{equation}
For every fixed \(c>0\), there are \(0<c_{1,c}<c_{2,c}<\infty\) such that,
for all sufficiently large \(n\),
\begin{equation}\label{eq:ad-critical-window-two-sided}
        c_{1,c}(1+\gamma t)^{-\sigma}
        \leq\mathsf F_n(\e^{-t/n})
        \leq c_{2,c}(1+\gamma t)^{-\sigma},
        \qquad0<t\leq cn.
\end{equation}
For every sublinear window,
\begin{equation}\label{eq:ad-critical-window-exact}
        \sup_{0<t\leq C_n}
        \left|
        \frac{\mathsf F_n(\e^{-t/n})}
             {(1+\gamma t)^{-\sigma}}-1
        \right|\longrightarrow0.
\end{equation}
\end{proposition}

\begin{proof}
For \(s\) in a compact subset of \((0,1)\),
Lemma~\ref{lem:ad-parabolic-iteration} gives, uniformly,
\(\sum_{k<n}\{1-u^{(k)}(s)\}=\gamma^{-1}\log n+c_\delta(s)+o(1)\), while
\(\sum_k\{1-u^{(k)}(s)\}^2<\infty\).  Hence
\[
        \sum_{k=0}^{n-1}\log R(u^{(k)}(s))
        =-\sigma\log n+c_R(s)+o(1).
\]
Since \(u^{(n)}(s)\to1\), \eqref{eq:ad-critical-log-product} proves
\eqref{eq:ad-critical-fixed-profile} locally uniformly on \((0,1)\), with
a continuous positive limit there.

To include \(s=0\), note that the nondegenerate critical offspring p.g.f.\
\(u\) satisfies \(u(0)>0\).  Thus, for every \(0<b<1\),
\(u([0,b])\) is a compact subset of \((0,1)\).  The exact recursion gives
\[
(n+1)^\sigma\mathsf F_{n+1}(s)
=R(s)\left(\frac{n+1}{n}\right)^\sigma
n^\sigma\mathsf F_n(u(s)),
\]
whose right side converges uniformly on \([0,b]\) to
\(R(s)\mathsf U(u(s))\).  This extends the limiting profile continuously
to \(s=0\) and proves \eqref{eq:ad-critical-fixed-profile} locally uniformly
on \([0,1)\), as well as \eqref{eq:ad-critical-functional-equation} for
\(0\leq s<1\).  In particular, \(R(0)=0\) implies \(\mathsf U(0)=0\).

For \(s=\e^{-t/n}\), \eqref{eq:ad-critical-log-product} and
\eqref{eq:ad-parabolic-broad-sums} yield, for every fixed \(c>0\),
\[
\sup_{0<t\leq cn}
\left|
\log\mathsf F_n(\e^{-t/n})
+\sigma\log(1+\gamma t)
\right|\leq K_c
\]
for all sufficiently large \(n\).  This proves
\eqref{eq:ad-critical-window-two-sided}.  On a sublinear window,
\eqref{eq:ad-parabolic-window-sums} gives
\[
\sup_{0<t\leq C_n}
\left|
\log\mathsf F_n(\e^{-t/n})
+\sigma\log(1+\gamma t)
\right|\longrightarrow0.
\]
Also, \(1-u^{(n)}(\e^{-t/n})\leq C/n\) uniformly on this window, so the
initial factor tends uniformly to one.  Exponentiation proves
\eqref{eq:ad-critical-window-exact}.
\end{proof}

\begin{lemma}[Positivity and boundary growth of the critical profile]
\label{lem:ad-critical-profile-positive}
The function \(\mathsf U\) is strictly increasing, and
\begin{equation}\label{eq:ad-critical-U-boundary}
        \lim_{t\downarrow0}t^\sigma\mathsf U(\e^{-t})
        =\gamma^{-\sigma}.
\end{equation}
Consequently,
\(\mathsf U(\e^{-t})-\mathsf U(0)>0\) for every \(t>0\).
\end{lemma}

\begin{proof}
For \(0<r<s<1\), monotonicity of \(u\) and \(R\) gives
\[
        \frac{\mathsf U(s)}{\mathsf U(r)}
        =\prod_{j=0}^{\infty}
        \frac{R(u^{(j)}(s))}{R(u^{(j)}(r))}>1.
\]
At \(r=0\), the functional equation gives
\(\mathsf U(0)=R(0)\mathsf U(u(0))\).  If \(R(0)>0\), the same product
ratio is valid and is greater than one.  If \(R(0)=0\), instead use
\(\mathsf U(s)>0=\mathsf U(0)\), without a product ratio at zero.  Thus
\(\mathsf U\) is strictly increasing.

Let \(t_j\downarrow0\), and choose \(n_j\uparrow\infty\) so that
\(n_jt_j\to\infty\) and
\(t_j^\sigma|n_j^\sigma\mathsf F_{n_j}(\e^{-t_j})-
\mathsf U(\e^{-t_j})|\to0\).  With \(s_j:=n_jt_j\), the sublinear-window
profile gives
\[
\begin{aligned}
t_j^\sigma\mathsf U(\e^{-t_j})
&=s_j^\sigma\mathsf F_{n_j}(\e^{-s_j/n_j})+o(1)\\
&=s_j^\sigma(1+\gamma s_j)^{-\sigma}\{1+o(1)\}
\longrightarrow\gamma^{-\sigma}.
\end{aligned}
\]
Since \((t_j)\) was arbitrary, \eqref{eq:ad-critical-U-boundary} follows.
\end{proof}

\begin{proposition}[Critical harmonic-moment limits for the exact recursion]
\label{prop:ad-exact-critical-harmonic}
For \(r>\sigma\),
\begin{equation}\label{eq:ad-critical-r-greater-constant}
\lim_{n\to\infty}n^\sigma
\Ebf[\mathsf Z_n^{-r};\mathsf Z_n>0]
=\frac1{\Gamma(r)}\int_0^\infty
\{\mathsf U(\e^{-t})-\mathsf U(0)\}t^{r-1}\dd t.
\end{equation}
For \(0<r<\sigma\),
\begin{equation}\label{eq:ad-critical-r-less-constant}
\lim_{n\to\infty}n^r
\Ebf[\mathsf Z_n^{-r};\mathsf Z_n>0]
=\frac1{\Gamma(r)}\int_0^\infty
(1+\gamma s)^{-\sigma}s^{r-1}\dd s,
\end{equation}
and, at \(r=\sigma\),
\begin{equation}\label{eq:ad-critical-border-constant}
\lim_{n\to\infty}\frac{n^\sigma}{\log n}
\Ebf[\mathsf Z_n^{-\sigma};\mathsf Z_n>0]
=\frac{\gamma^{-\sigma}}{\Gamma(\sigma)}.
\end{equation}
\end{proposition}

\begin{proof}
Let \(\mathsf G_n(t):=\mathsf F_n(\e^{-t})-\mathsf F_n(0)\).  If
\(r>\sigma\), then
\[
h_n(t):=n^\sigma\mathsf G_n(t)t^{r-1}
\longrightarrow
\{\mathsf U(\e^{-t})-\mathsf U(0)\}t^{r-1}.
\]
Moreover,
\[
 0\le h_n(t)\le C\Bigl(
 t^{r-1-\sigma}\one_{\{0<t\le1\}}
 +\e^{-(t-1)}t^{r-1}\one_{\{t>1\}}\Bigr),
\]
which is integrable.  Dominated convergence proves
\eqref{eq:ad-critical-r-greater-constant}.

By the Mellin identity \eqref{eq:bprei-laplace-gamma},
substituting \(t=s/n\) gives
\[
 \Gamma(r)n^r\Ebf[\mathsf Z_n^{-r};\mathsf Z_n>0]
 =\int_0^\infty \mathsf G_n(s/n)s^{r-1}\dd s .
\]

Let \(0<r<\sigma\), fix \(c>0\), and define
\(h_n(s):=\mathsf G_n(s/n)\one_{\{s\leq cn\}}s^{r-1}\).  For every fixed
\(s>0\), \(h_n(s)\to(1+\gamma s)^{-\sigma}s^{r-1}\), while
\(0\leq h_n(s)\leq C(1+s)^{-\sigma}s^{r-1}\).  Hence generalized dominated
convergence gives
\[
\int_0^{cn}\mathsf G_n(s/n)s^{r-1}\dd s
\longrightarrow
\int_0^\infty(1+\gamma s)^{-\sigma}s^{r-1}\dd s.
\]
For the remaining tail, integer-valuedness gives
\[
\begin{aligned}
\int_{cn}^\infty\mathsf G_n(s/n)s^{r-1}\dd s
&=n^r\int_c^\infty\mathsf G_n(t)t^{r-1}\dd t\\
&\leq Cn^{r-\sigma}\int_c^\infty\e^{-(t-c)}t^{r-1}\dd t
\longrightarrow0.
\end{aligned}
\]
This proves \eqref{eq:ad-critical-r-less-constant}.

Finally, let \(r=\sigma\) and \(C_n:=n/\log n\).  Split the scaled integral
into \((0,1)\), \((1,C_n)\), \((C_n,cn)\), and \((cn,\infty)\). Since \(\mathsf F_n(0)=O(n^{-\sigma})\), subtracting the zero atom
on \([1,C_n]\) costs \(O(n^{-\sigma}C_n^\sigma)=o(1)\).
By \eqref{eq:ad-critical-window-exact} and
\((1+\gamma s)^{-\sigma}s^{\sigma-1}\sim\gamma^{-\sigma}/s\), we obtain
\[
\int_1^{C_n}\mathsf G_n(s/n)s^{\sigma-1}\dd s
=\gamma^{-\sigma}\log n+o(\log n).
\]
The broad-window bound gives
\(\int_{C_n}^{cn}\mathsf G_n(s/n)s^{\sigma-1}\dd s=O(\log\log n)\), and
the first and last ranges are \(O(1)\).  Division by \(\log n\) proves
\eqref{eq:ad-critical-border-constant}.
\end{proof}

\subsection{From the auxiliary population to the progenitor transform}
\label{subsec:ad-harmonic-transfer}

Subsection~\ref{subsec:ad-exact-dynamics} established the positive-transform
and harmonic-moment profiles of the auxiliary population \(\mathsf Z_n\),
whose p.g.f. satisfies the exact recursion
\[
        \mathsf F_{n+1}(s)=R(s)\mathsf F_n(u(s)).
\]
To prove Theorems~\ref{thm:main-ad-supercritical} and
\ref{thm:main-ad-critical}, we also need the corresponding profiles for the
auxiliary progenitor count.  This distinction is necessary because the
offspring-ratio deviation \(X_{n+1}/N_n\) is indexed by the progenitor count
\(N_n\), whereas the control and total-population ratios are indexed by
\(X_n\).

Fix one envelope pair \((d,l)\), and recall that
\[
        u=l\circ f,\qquad R=d\circ f.
\]
The corresponding auxiliary progenitor transform is
\begin{equation}\label{eq:ad-exact-control-transform}
        \widehat H_n(s):=d(s)\mathsf F_n(l(s)).
\end{equation}
We transfer the population results of
Subsection~\ref{subsec:ad-exact-dynamics} to \(\widehat H_n\) through two
complementary descriptions.  Analytically, evaluating \(\mathsf F_n\) at
\(l(\e^{-t})\) replaces the Laplace argument \(t\) by
\[
        \lambda_l(t):=-\log l(\e^{-t}).
\]
Lemma~\ref{lem:ad-log-change} shows that
\[
        \lambda_l(t)=\tau_l t+O(t^2),
        \qquad \tau_l=l'(1),
\]
together with the corresponding uniform approximation on the moving windows
used in the critical analysis.  Probabilistically, \(\widehat H_n\) is the
p.g.f. of a random sum consisting of an additive term with p.g.f. \(d\) and
\(\mathsf Z_n\) independent summands with p.g.f. \(l\).  Thus the auxiliary
progenitor count has first-order size \(\tau_l\mathsf Z_n\);
Lemma~\ref{lem:ad-control-random-sum} makes this representation precise.

These two descriptions allow
Proposition~\ref{prop:ad-exact-control-harmonic} to transfer the
supercritical harmonic-moment normalizations of
Proposition~\ref{prop:ad-exact-super-harmonic} and the critical limits of
Proposition~\ref{prop:ad-exact-critical-harmonic} from the auxiliary population to
the auxiliary progenitor count.  In the supercritical case, the three
normalizations are unchanged.  In the critical moving-window limits, the
linear change of Laplace argument replaces \(\gamma\) by
\(\gamma\tau_l\).  Subsection~\ref{subsec:ad-controlled-transfer} will then
apply these conclusions to the lower and upper envelope pairs and transfer
them to the population and progenitor counts of the original controlled
process.

For a nonnegative integer-valued random variable \(Y\),
\begin{equation}\label{eq:ad-laplace-gamma}
        \Gamma(r)\Ebf[Y^{-r};Y>0]
        =\int_0^\infty\Ebf[\e^{-tY};Y>0]t^{r-1}\dd t.
\end{equation}

\begin{lemma}[Logarithmic change induced by a p.g.f.]
\label{lem:ad-log-change}
Let \(l\) be a p.g.f.\ with \(l'(1)=\tau_l\in(0,\infty)\) and
\(l''(1)<\infty\), and set \(\lambda_l(t):=-\log l(\e^{-t})\).  Then, as
\(t\downarrow0\),
\begin{equation}\label{eq:ad-lambda-l-expansion}
        \lambda_l(t)=\tau_lt+O(t^2),
        \qquad \lambda_l'(t)=\tau_l+O(t).
\end{equation}
On a sufficiently small interval, \(\lambda_l\) is strictly increasing; its
inverse \(\Lambda_l\) satisfies
\begin{equation}\label{eq:ad-lambda-l-inverse}
        \Lambda_l(y)=\frac{y}{\tau_l}+O(y^2),
        \qquad \Lambda_l'(y)=\frac1{\tau_l}+O(y).
\end{equation}
For every positive sequence \((C_n)\) with \(C_n/n\to0\),
\begin{equation}\label{eq:ad-lambda-l-window}
        \sup_{0<s\leq C_n}
        \left|\frac{n\lambda_l(s/n)}s-\tau_l\right|
        \longrightarrow0.
\end{equation}
\end{lemma}

\begin{proof}
As \(t\downarrow0\),
\[
l(\e^{-t})=1-\tau_l(1-\e^{-t})+O((1-\e^{-t})^2)
=1-\tau_lt+O(t^2),
\]
so \(\lambda_l(t)=\tau_lt+O(t^2)\) and
\(\lambda_l'(t)=\e^{-t}l'(\e^{-t})/l(\e^{-t})=\tau_l+O(t)\).  The inverse
expansions follow from the inverse-function theorem, and
\[
\sup_{0<s\leq C_n}
\left|\frac{n\lambda_l(s/n)}s-\tau_l\right|
\leq C\frac{C_n}{n}\longrightarrow0.
\]
\end{proof}

\begin{lemma}[Random-sum representation of the progenitor transform]
\label{lem:ad-control-random-sum}
A variable \(\widehat N_n\) with p.g.f.\ \(\widehat H_n\) can be
constructed as follows. Let \(\zeta,L_1,L_2,\ldots\) be mutually
independent and independent of \(\mathsf Z_n\), with \(\zeta\) of
p.g.f.\ \(d\) and each \(L_j\) of p.g.f.\ \(l\). Put
\begin{equation}\label{eq:ad-control-random-sum}
 \widehat N_n:=\zeta+\sum_{j=1}^{\mathsf Z_n}L_j .
\end{equation}
If \(l'(1)=\tau_l\in(0,\infty)\) and \(l''(1)<\infty\), then, for every
\(\delta>0\),
\begin{equation}\label{eq:ad-control-random-sum-concentration}
        \Pbf\left(
        \left|\frac1k\sum_{j=1}^kL_j-\tau_l\right|>\delta
        \right)
        \leq\frac{\Var(L_1)}{\delta^2k},
        \qquad k\geq1.
\end{equation}
\end{lemma}

\begin{proof}
Conditioning on \(\mathsf Z_n\),
\[
\Ebf\left[s^{\zeta+\sum_{j=1}^{\mathsf Z_n}L_j}\right]
=d(s)\Ebf[l(s)^{\mathsf Z_n}]
=d(s)\mathsf F_n(l(s))=\widehat H_n(s).
\]
The second assertion is Chebyshev's inequality.
\end{proof}

For every \(\delta>0\),
\[
\Pbf\left(
\left|\frac{\widehat N_n}{k}-\tau_l\right|>\delta
\,\middle|\,\mathsf Z_n=k
\right)
\leq
\Pbf\left(\zeta>\frac{\delta k}{2}\right)
+\frac{4\Var(L_1)}{\delta^2k}
\longrightarrow0.
\]
The analytic counterpart is \(\lambda_l(t)\sim\tau_lt\).  With
\(\widehat{\mathsf G}_n(t):=\widehat H_n(\e^{-t})-\widehat H_n(0)\),
\begin{equation}\label{eq:ad-control-positive-decomposition}
\begin{aligned}
\widehat{\mathsf G}_n(t)
={}&d(\e^{-t})
\{\mathsf F_n(\e^{-\lambda_l(t)})-\mathsf F_n(l(0))\}\\
&+\{d(\e^{-t})-d(0)\}\mathsf F_n(l(0)).
\end{aligned}
\end{equation}
Both terms are nonnegative.

\begin{proposition}[Harmonic moments of the exact progenitor transform]
\label{prop:ad-exact-control-harmonic}
Let \((\mathsf F_n)\) satisfy \eqref{eq:ad-exact-recursion}, with
\(u=l\circ f\), \(R=d\circ f\), and let
\(\widehat H_n(s)=d(s)\mathsf F_n(l(s))\).  Assume
\(l'(1)=\tau_l\in(0,\infty)\) and \(l''(1)<\infty\).

Under the hypotheses of Proposition~\ref{prop:ad-exact-super-harmonic}, for
every \(r>0\),
\begin{equation}\label{eq:ad-exact-control-super}
        \Ebf[\widehat N_n^{-r};\widehat N_n>0]
        \asymp\mathfrak b_n^{\mathrm{sc}}(r).
\end{equation}
Under the hypotheses of Proposition~\ref{prop:ad-exact-critical-harmonic},
define
\begin{equation}\label{eq:ad-critical-control-profile}
        \widehat{\mathsf U}(s):=d(s)\mathsf U(l(s)).
\end{equation}
Then, for \(r>\sigma\),
\begin{equation}\label{eq:ad-critical-control-r-greater}
\lim_{n\to\infty}n^\sigma
\Ebf[\widehat N_n^{-r};\widehat N_n>0]
=\frac1{\Gamma(r)}\int_0^\infty
\{\widehat{\mathsf U}(\e^{-t})-\widehat{\mathsf U}(0)\}
t^{r-1}\dd t.
\end{equation}
For \(0<r<\sigma\),
\begin{equation}\label{eq:ad-critical-control-r-less}
\lim_{n\to\infty}n^r
\Ebf[\widehat N_n^{-r};\widehat N_n>0]
=\frac1{\Gamma(r)}\int_0^\infty
(1+\gamma\tau_ls)^{-\sigma}s^{r-1}\dd s,
\end{equation}
and, at \(r=\sigma\),
\begin{equation}\label{eq:ad-critical-control-border}
        \lim_{n\to\infty}\frac{n^\sigma}{\log n}
        \Ebf[\widehat N_n^{-\sigma};\widehat N_n>0]
        =\frac{(\gamma\tau_l)^{-\sigma}}{\Gamma(\sigma)}.
\end{equation}
\end{proposition}

\begin{proof}
For the supercritical case, write \(s=\e^{-t}\), \(z=l(s)\), and
\(w=l(0)\). Since \(l'(1)>0\), \(w<z\) for \(t>0\). For \(k\ge1\),
\(z^k-w^k\ge(1-w/z)z^k\). Summing against the law of \(\mathsf Z_n\)
and using \eqref{eq:ad-control-positive-decomposition} gives
\[
 \widehat{\mathsf G}_n(t)
 \ge d(s)(1-w/z)\,\mathsf G_n(\lambda_l(t)).
\]
As \(t\downarrow0\), \(d(s)(1-w/z)\to1-w>0\).
Choose \(t_0>0\) and \(0<c_1<c_2\) so that this factor is bounded
below by a positive constant and \(c_1t\le\lambda_l(t)\le c_2t\)
for \(0<t\le t_0\). The first term of
\eqref{eq:ad-control-positive-decomposition} is at most
\(\mathsf G_n(\lambda_l(t))\). The second term vanishes if \(d\equiv1\);
otherwise \(R=d\circ f\not\equiv1\) and \(\mathsf F_n(w)=O(\theta^n)\) by
Proposition~\ref{prop:ad-super-fixed-profile}. Hence, uniformly on
this interval,
\begin{equation}\label{eq:ad-exact-progenitor-window-comparison}
 c\,\mathsf G_n(c_2t)\le\widehat{\mathsf G}_n(t)
       \le\mathsf G_n(c_1t)+C\theta^n .
\end{equation}
For either positive transform \(G\), integer support gives
\(G(t)\le\e^{-(t-b)}G(b)\) for \(t\ge b>0\), while monotonicity gives
\(\int_0^bG(t)t^{r-1}\dd t\ge b^rG(b)/r\).
Consequently \(\int_b^\infty G(t)t^{r-1}\dd t\le C_{b,r}G(b)\);
for each fixed \(b>0\), integration up to \(b\) is comparable to
integration over \((0,\infty)\), uniformly in \(n\). Integrating
\eqref{eq:ad-exact-progenitor-window-comparison}, changing variables,
and using \eqref{eq:ad-laplace-gamma} and
Proposition~\ref{prop:ad-exact-super-harmonic} proves
\eqref{eq:ad-exact-control-super}, since
\(\theta^n=O(\mathfrak b_n^{\mathrm{sc}}(r))\) for every \(r>0\).

For the critical case and \(r>\sigma\),
\(n^\sigma\widehat{\mathsf G}_n(t)\to
\{\widehat{\mathsf U}(\e^{-t})-\widehat{\mathsf U}(0)\}\), with the same
integrable fixed-domain majorant as in
Proposition~\ref{prop:ad-exact-critical-harmonic}; dominated convergence
gives \eqref{eq:ad-critical-control-r-greater}.

Let \(0<r<\sigma\), fix \(c>0\), and define
\(h_n(s):=\widehat{\mathsf G}_n(s/n)\one_{\{s\leq cn\}}s^{r-1}\).  By
Lemma~\ref{lem:ad-log-change} and the broad critical window,
\(h_n(s)\to(1+\gamma\tau_ls)^{-\sigma}s^{r-1}\), while
\(0\leq h_n(s)\leq C(1+s)^{-\sigma}s^{r-1}\).  Generalized dominated
convergence gives the limit over \((0,cn)\), and
\[
\int_{cn}^\infty\widehat{\mathsf G}_n(s/n)s^{r-1}\dd s
\leq Cn^{r-\sigma}\int_c^\infty\e^{-(t-c)}t^{r-1}\dd t
\longrightarrow0.
\]
This proves \eqref{eq:ad-critical-control-r-less}.

For \(r=\sigma\), set \(C_n:=n/\log n\) and
\(t_{n,s}:=n\lambda_l(s/n)\).  Lemma~\ref{lem:ad-log-change} gives
\[
\sup_{0<s\leq C_n}
\left|\frac{t_{n,s}}s-\tau_l\right|\longrightarrow0,
\quad
\sup_{0<s\leq C_n}
\left|\frac{1+\gamma t_{n,s}}{1+\gamma\tau_ls}-1\right|
\longrightarrow0.
\]
The sublinear critical profile then gives
\[
\sup_{0<s\leq C_n}
\left|
\frac{\widehat H_n(\e^{-s/n})}
     {(1+\gamma\tau_ls)^{-\sigma}}-1
\right|\longrightarrow0.
\]
Since \(\widehat H_n(0)=O(n^{-\sigma})\),
\[
\int_1^{C_n}\widehat{\mathsf G}_n(s/n)s^{\sigma-1}\dd s
=(\gamma\tau_l)^{-\sigma}\log n+o(\log n).
\]
The broad window contributes \(O(\log\log n)\) on \((C_n,cn)\), and the
remaining two ranges contribute \(O(1)\).  This proves
\eqref{eq:ad-critical-control-border}.
\end{proof}

\subsection{Returning to the controlled process}
\label{subsec:ad-controlled-transfer}

We now transfer the exact-envelope profiles to the original population and
progenitor counts.  Under \textup{(AD0)}, the p.g.f.\ sandwiches hold before
subtraction of the zero atoms.  It remains to control the gaps between the
lower and upper zero atoms, first under \textup{(AD-S)} and then under
\textup{(AD-C)}.

\begin{lemma}[Zero-atom gaps for matched supercritical envelopes]
\label{lem:ad-super-zero-gap}
Under \textup{(AD0)} and \textup{(AD-S)},
\begin{equation}\label{eq:ad-super-zero-gap}
        F_n^U(0)-F_n^L(0)=O(\theta^n),
        \qquad H_n^U(0)-H_n^L(0)=O(\theta^n).
\end{equation}
\end{lemma}

\begin{proof}
The two gaps are nonnegative by \eqref{eq:ad-F-sandwich} and
\eqref{eq:ad-H-sandwich} at \(s=0\).
If neither envelope is in the branch \(q_e>0,\ R_e\equiv1\), then
Proposition~\ref{prop:ad-super-fixed-profile}\textup{(i)}--\textup{(ii)},
applied at \(0\) and at \(l_e(0)<1\), gives
\[
\begin{aligned}
0\leq F_n^U(0)-F_n^L(0)
&\leq F_n^U(0)+F_n^L(0)=O(\theta^n),\\
0\leq H_n^U(0)-H_n^L(0)
&\leq d_U(0)F_n^U(l_U(0))+d_L(0)F_n^L(l_L(0))
=O(\theta^n).
\end{aligned}
\]
Suppose instead that
\(q_L=q_U=:q>0\) and \(R_L\equiv R_U\equiv1\).  Then
\(d_L\equiv d_U\equiv1\), \(u_L'(q)=u_U'(q)=\theta\), and, for
\(e\in\{L,U\}\),
\(u_e^{(n)}(s)=q+\theta^nK_{u_e,q}(s)+o(\theta^n)\).  Since \(F_0\) is
differentiable at \(q\),
\[
\begin{aligned}
F_n^U(0)-F_n^L(0)
&=F_0(u_U^{(n)}(0))-F_0(u_L^{(n)}(0))=O(\theta^n),\\
H_n^U(0)-H_n^L(0)
&=F_0(u_U^{(n)}(l_U(0)))-F_0(u_L^{(n)}(l_L(0)))=O(\theta^n).
\end{aligned}
\]
\end{proof}

\begin{lemma}[Critical zero-atom gaps]
\label{lem:ad-critical-zero-gap}
Under \textup{(AD0)} and \textup{(AD-C)},
\begin{equation}\label{eq:ad-critical-zero-gap}
        F_n^U(0)-F_n^L(0)=O(n^{-\sigma}),
        \qquad H_n^U(0)-H_n^L(0)=O(n^{-\sigma}).
\end{equation}
\end{lemma}

\begin{proof}
By \eqref{eq:ad-critical-fixed-profile},
\(F_n^e(0)=O(n^{-\sigma})\) for \(e\in\{L,U\}\).  Since
\(l_e'(1)=\tau>0\), \(l_e(0)<1\), and another application of
\eqref{eq:ad-critical-fixed-profile} gives
\(H_n^e(0)=d_e(0)F_n^e(l_e(0))=O(n^{-\sigma})\).  Hence
\[
\begin{aligned}
0\leq F_n^U(0)-F_n^L(0)
&\leq F_n^U(0)+F_n^L(0)=O(n^{-\sigma}),\\
0\leq H_n^U(0)-H_n^L(0)
&\leq H_n^U(0)+H_n^L(0)=O(n^{-\sigma}).
\end{aligned}
\]
\end{proof}

\begin{lemma}[Finite-state extraction]\label{lem:ad-finite-extraction}
Let \(Y_n\) be nonnegative integer-valued random variables and \(a_n>0\).
Put \(P_n(z):=\Ebf[z^{Y_n};Y_n>0]\).  Suppose that, for some
\(0<u<v<1\) and \(c,C>0\),
\[
        P_n(u)\geq ca_n,
        \qquad P_n(v)\leq Ca_n
\]
for all sufficiently large \(n\).  Then, for every fixed \(s\in(0,1)\),
there is \(c_s>0\) such that \(P_n(s)\geq c_sa_n\) for all sufficiently
large \(n\).
\end{lemma}

\begin{proof}
For an integer \(K\geq1\),
\[
\Ebf[u^{Y_n};Y_n>K]
\leq\left(\frac uv\right)^K P_n(v).
\]
Choose \(K\) so that \(C(u/v)^K\leq c/2\).  Then
\[
\Ebf[u^{Y_n};1\leq Y_n\leq K]\geq\frac c2a_n.
\]
Consequently, for every fixed \(s\in(0,1)\),
\[
\begin{aligned}
P_n(s)
&\geq\min_{1\leq j\leq K}\left(\frac su\right)^j
\Ebf[u^{Y_n};1\leq Y_n\leq K]\\
&\geq c_sa_n.
\end{aligned}
\]
\end{proof}

\begin{remark}[One-sided envelope conclusions]
\label{rem:ad-one-sided}
The common-rate assumptions yield one two-sided normalization.  A one-sided
harmonic-moment or ratio bound follows when the corresponding envelope
comparison also controls the positive transform after subtraction of the
zero atom.  An ordering of the uncorrected p.g.f.s alone yields no one-sided
conclusion for the positive transforms, since that ordering need not be
preserved after the zero atoms are subtracted.  If the lower and upper
envelope parameters differ, the two positive-transform comparisons yield
their respective scales, and no common normalization follows.
\end{remark}

\subsection{Ratio kernels and completion of the main AD theorems}
\label{subsec:ad-ratio-transfer}

For \(k\geq1\) and \(\eps>0\), define
\begin{align}
\psi_\xi(k;\eps)
&:=\Pbf\left(\left|k^{-1}\sum_{j=1}^k\xi_j-m\right|>\eps\right),
\\
\psi_{\mathrm{ad}}(k;\eps)
&:=\Pbf\left(\left|\phi(k)/k-\tau\right|>\eps\right).
\end{align}
For the total ratio, define the complete population-conditioned kernel
\[
\psi_{\mathrm{comb}}(k;\eps)
:=\Pbf\left(
\left|\sum_{j=1}^{\phi(k)}\xi_j-m\tau k\right|>\eps k
\right).
\]
Recall that
\[
\begin{aligned}
J_{n,1}(\eps)
&=\Pbf\left(N_n>0,\left|X_{n+1}/N_n-m\right|>\eps\right),\\
J_{n,2}(\eps)
&=\Pbf\left(X_n>0,\left|N_n/X_n-\tau\right|>\eps\right),\\
J_{n,3}(\eps)
&=\Pbf\left(X_n>0,\left|X_{n+1}/X_n-m\tau\right|>\eps\right).
\end{aligned}
\]

\begin{proposition}[AD ratio identities]
\label{prop:ad-ratio-identities}
For every \(n\) and \(\eps>0\),
\begin{align}
J_{n,1}(\eps)&=\Ebf[\psi_\xi(N_n;\eps);N_n>0],
\label{eq:ad-ratio-offspring-identity}\\
J_{n,2}(\eps)&=\Ebf[\psi_{\mathrm{ad}}(X_n;\eps);X_n>0].
\label{eq:ad-ratio-control-identity}
\end{align}
For the total ratio,
\[
        J_{n,3}(\eps)
        =\Ebf[\psi_{\mathrm{comb}}(X_n;\eps);X_n>0].
\]
On \(\{X_n>0,N_n>0\}\),
\begin{equation}\label{eq:ad-X-over-X-decomposition}
        \frac{X_{n+1}}{X_n}-m\tau
        =\left(\frac{X_{n+1}}{N_n}-m\right)\frac{N_n}{X_n}
         +m\left(\frac{N_n}{X_n}-\tau\right).
\end{equation}
\end{proposition}

\begin{proof}
For every \(k\geq1\),
\[
\begin{aligned}
\Pbf\left(\left|X_{n+1}/N_n-m\right|>\eps\mid N_n=k\right)
&=\psi_\xi(k;\eps),\\
\Pbf\left(\left|N_n/X_n-\tau\right|>\eps\mid X_n=k\right)
&=\psi_{\mathrm{ad}}(k;\eps),\\
\Pbf\left(\left|X_{n+1}/X_n-m\tau\right|>\eps\mid X_n=k\right)
&=\psi_{\mathrm{comb}}(k;\eps).
\end{aligned}
\]
Taking expectations proves \eqref{eq:ad-ratio-offspring-identity}--
\eqref{eq:ad-ratio-control-identity} and the identity for \(J_{n,3}\);
\eqref{eq:ad-X-over-X-decomposition} is algebraic.
\end{proof}

\subsubsection{Exponential transfer}

We work under condition \textup{(AD-E)} of
Subsection~\ref{subsec:main-ad}.

\begin{lemma}[Exponential one-step kernels]
\label{lem:ad-exp-kernels}
For every \(\eps>0\), there are \(C_\eps<\infty\) and
\(s_{\xi,\eps},s_{\phi,\eps}\in(0,1)\) such that
\begin{equation}\label{eq:ad-exp-kernels}
        \psi_\xi(k;\eps)\leq C_\eps s_{\xi,\eps}^k,
        \qquad
        \psi_{\mathrm{ad}}(k;\eps)\leq C_\eps s_{\phi,\eps}^k.
\end{equation}
\end{lemma}

\begin{proof}
Let \(\Lambda_\xi(\eta):=\log\Ebf\e^{\eta\xi}\).  Chernoff's inequality
gives
\[
\begin{aligned}
\psi_\xi(k;\eps)
&\leq
\exp\{-k\sup_{\eta>0}[\eta(m+\eps)-\Lambda_\xi(\eta)]\}\\
&\quad+
\exp\{-k\sup_{\eta<0}[\eta(m-\eps)-\Lambda_\xi(\eta)]\}
\leq2\e^{-c_{\xi,\eps}k}.
\end{aligned}
\]
Write \(\Lambda_k(h):=k^{-1}\log\Ebf[\e^{h\phi(k)}]\). By
\textup{(AD-E)}, \(\Lambda_\phi(0)=0\) and \(\Lambda_\phi'(0)=\tau\).
Choose \(h>0\) small enough that
\[
 \max_{\delta=\pm1}
 \{\Lambda_\phi(\delta h)-\delta h\tau\}<\tfrac12 h\eps .
\]
Since \(\Lambda_k(\pm h)\to\Lambda_\phi(\pm h)\), the same maximum
with \(\Lambda_k\) in place of \(\Lambda_\phi\) is at most
\(3h\eps/4\) for all sufficiently large \(k\). Applying Chernoff's
inequality to the two tails gives
\[
 \psi_{\mathrm{ad}}(k;\eps)
 \le \sum_{\delta=\pm1}
   \exp\!\left\{k\bigl(\Lambda_k(\delta h)-\delta h\tau-h\eps\bigr)\right\}
 \le 2\e^{-kh\eps/4}.
\]
Enlarging the multiplicative constant covers the remaining \(k\).
\end{proof}

\begin{proposition}[Generating-function ratio transfer]
\label{prop:ad-exp-ratio-transfer}
For every \(\eps>0\), there are arguments in \((0,1)\) such that
\begin{align}
J_{n,1}(\eps)&\leq C_\eps\{H_n(s_{1,\eps})-H_n(0)\},
\label{eq:ad-exp-ratio-H}\\
J_{n,2}(\eps)&\leq C_\eps\{F_n(s_{2,\eps})-F_n(0)\}.
\label{eq:ad-exp-ratio-F}
\end{align}
For suitable \(\eps_1,\eps_2>0\),
\begin{equation}\label{eq:ad-exp-ratio-X}
J_{n,3}(\eps)
\leq C_\eps\{H_n(s_{1,\eps_1})-H_n(0)
+F_n(s_{2,\eps_2})-F_n(0)\}.
\end{equation}
\end{proposition}

\begin{proof}
By Proposition~\ref{prop:ad-ratio-identities} and
Lemma~\ref{lem:ad-exp-kernels},
\[
\begin{aligned}
J_{n,1}(\eps)
&\leq C_\eps\Ebf[s_{1,\eps}^{N_n};N_n>0]
=C_\eps\{H_n(s_{1,\eps})-H_n(0)\},\\
J_{n,2}(\eps)
&\leq C_\eps\Ebf[s_{2,\eps}^{X_n};X_n>0]
=C_\eps\{F_n(s_{2,\eps})-F_n(0)\}.
\end{aligned}
\]
Choose \(0<\eps_2<\tau\) and \(\eps_1>0\) such that
\begin{equation}\label{eq:ad-ratio-tolerance-choice}
        \eps_1(\tau+\eps_2)+m\eps_2\leq\eps.
\end{equation}
Then
\[
\begin{aligned}
&\left\{X_n>0,\ \left|X_{n+1}/X_n-m\tau\right|>\eps\right\}\\
&\quad\subseteq
\left\{N_n>0,\ \left|X_{n+1}/N_n-m\right|>\eps_1\right\}
\cup
\left\{X_n>0,\ \left|N_n/X_n-\tau\right|>\eps_2\right\}.
\end{aligned}
\]
On \(\{X_n>0,N_n=0\}\), the control ratio differs from \(\tau\) by
\(\tau>\eps_2\). On \(\{X_n>0,N_n>0\}\), the inclusion follows from
\eqref{eq:ad-X-over-X-decomposition} and
\eqref{eq:ad-ratio-tolerance-choice}. Thus
\(J_{n,3}(\eps)\leq J_{n,1}(\eps_1)+J_{n,2}(\eps_2)\), proving
\eqref{eq:ad-exp-ratio-X}.
\end{proof}

The exact ratio identities in Proposition~\ref{prop:ad-ratio-identities} also
give the corresponding reverse transfers.  For \(c>0\) and \(s\in(0,1)\),
if
\[
        \psi_\xi(k;\eps)\geq cs^k,
        \qquad k\geq1,
\]
then
\[
        J_{n,1}(\eps)\geq c\{H_n(s)-H_n(0)\}.
\]
Similarly,
\[
        \psi_{\mathrm{ad}}(k;\eps)\geq cs^k,
        \qquad k\geq1,
\]
implies
\[
        J_{n,2}(\eps)\geq c\{F_n(s)-F_n(0)\},
\]
while
\[
        \psi_{\mathrm{comb}}(k;\eps)\geq cs^k,
        \qquad k\geq1,
\]
implies
\[
        J_{n,3}(\eps)\geq c\{F_n(s)-F_n(0)\}.
\]
These lower bounds follow by averaging the corresponding pointwise kernel
bounds in the exact identities of Proposition~\ref{prop:ad-ratio-identities}.
In particular, the lower bound for \(J_{n,3}\) uses the identity involving
\(\psi_{\mathrm{comb}}\), not a reversal of the event inclusion used to prove
\eqref{eq:ad-exp-ratio-X}.

\subsubsection{Polynomial transfer}

Let \(p\geq2\) and \(q:=p/2\), and assume
\eqref{eq:main-ad-poly-conditions}.

\begin{lemma}[Polynomial one-step kernels]
\label{lem:ad-poly-kernels}
For every fixed \(\eps>0\),
\begin{equation}\label{eq:ad-poly-kernels}
        \psi_\xi(k;\eps)+\psi_{\mathrm{ad}}(k;\eps)
        \leq C_\eps k^{-q},
        \qquad k\geq1.
\end{equation}
\end{lemma}

\begin{proof}
Rosenthal's inequality and then Markov's inequality give
\[
\begin{aligned}
\psi_\xi(k;\eps)
&\leq\eps^{-p}k^{-p}
\Ebf\left|\sum_{j=1}^k(\xi_j-m)\right|^p
\leq C_\eps k^{-p/2},\\
\psi_{\mathrm{ad}}(k;\eps)
&\leq\eps^{-p}k^{-p}\Ebf|\phi(k)-\tau k|^p
\leq C_\eps k^{-p/2}.
\end{aligned}
\]
Since \(q=p/2\), \eqref{eq:ad-poly-kernels} follows.
\end{proof}

\begin{proposition}[Harmonic-moment ratio transfer]
\label{prop:ad-poly-ratio-transfer}
For every fixed \(\eps>0\),
\begin{align}
J_{n,1}(\eps)&\leq C_\eps\Ebf[N_n^{-q};N_n>0],
\label{eq:ad-poly-ratio-N}\\
J_{n,2}(\eps)&\leq C_\eps\Ebf[X_n^{-q};X_n>0].
\label{eq:ad-poly-ratio-Xden}
\end{align}
For \(\eps_1,\eps_2\) satisfying
\eqref{eq:ad-ratio-tolerance-choice},
\begin{equation}\label{eq:ad-poly-ratio-total}
J_{n,3}(\eps)
\leq C_\eps\{\Ebf[X_n^{-q};X_n>0]
+\Ebf[N_n^{-q};N_n>0]\}.
\end{equation}
\end{proposition}

\begin{proof}
Proposition~\ref{prop:ad-ratio-identities} and
Lemma~\ref{lem:ad-poly-kernels} give
\[
J_{n,1}(\eps)\leq C_\eps\Ebf[N_n^{-q};N_n>0],
\qquad
J_{n,2}(\eps)\leq C_\eps\Ebf[X_n^{-q};X_n>0].
\]
The event inclusion used in the proof of
Proposition~\ref{prop:ad-exp-ratio-transfer} yields
\(J_{n,3}(\eps)\leq J_{n,1}(\eps_1)+J_{n,2}(\eps_2)\), proving
\eqref{eq:ad-poly-ratio-total}.
\end{proof}

The exact ratio identities in Proposition~\ref{prop:ad-ratio-identities} also
give the corresponding reverse polynomial transfers.  For \(c>0\), if
\[
        \psi_\xi(k;\eps)\geq ck^{-q},
        \qquad k\geq1,
\]
then
\[
        J_{n,1}(\eps)\geq c\Ebf[N_n^{-q};N_n>0].
\]
Similarly,
\[
        \psi_{\mathrm{ad}}(k;\eps)\geq ck^{-q},
        \qquad k\geq1,
\]
implies
\[
        J_{n,2}(\eps)\geq c\Ebf[X_n^{-q};X_n>0],
\]
while
\[
        \psi_{\mathrm{comb}}(k;\eps)\geq ck^{-q},
        \qquad k\geq1,
\]
implies
\[
        J_{n,3}(\eps)\geq c\Ebf[X_n^{-q};X_n>0].
\]
These inequalities follow by averaging the corresponding pointwise lower
bounds in the exact identities of Proposition~\ref{prop:ad-ratio-identities}.
In particular, the lower bound for \(J_{n,3}\) uses the identity involving
\(\psi_{\mathrm{comb}}\), not a reversal of the event inclusion used to prove
\eqref{eq:ad-poly-ratio-total}.

\begin{proof}[Proof of Theorem~\ref{thm:main-ad-supercritical}]
For \(Y_n=X_n\) take \(P_n=F_n\), and for \(Y_n=N_n\) take \(P_n=H_n\).
Let \(P_n^e\) be the corresponding envelope p.g.f., \(e=L,U\).
Write \(G_n(t)=P_n(\e^{-t})-P_n(0)\), define \(G_n^e\) similarly, and
put \(\Delta_n=P_n^U(0)-P_n^L(0)\). The p.g.f.\ bounds and
Lemma~\ref{lem:ad-super-zero-gap} give
\begin{equation}\label{eq:ad-super-positive-sandwich}
 G_n^L(t)-\Delta_n\le G_n(t)\le G_n^U(t)+\Delta_n,
 \qquad 0\le\Delta_n\le C_0\theta^n .
\end{equation}
Put \(\kappa=-\log\theta/\log a\). For fixed \(D\ge1\) sufficiently
large and \(t_0>0\) sufficiently small, both envelopes satisfy
\[
 b\theta^nt^{-\kappa}\le G_n^e(t)\le B\theta^nt^{-\kappa},
 \qquad Da^{-n}\le t\le t_0,
\]
for all sufficiently large \(n\), with \(0<b\le B<\infty\).
For \(X_n\) this is \eqref{eq:ad-super-window}. For \(N_n\), apply
\eqref{eq:ad-exact-progenitor-window-comparison} to each envelope and
choose \(D,t_0\) so that the rescaled arguments \(c_1t,c_2t\) both lie
in the window of \eqref{eq:ad-super-window}. The additive \(O(\theta^n)\)
in the upper bound is absorbed since \(t^{-\kappa}\ge t_0^{-\kappa}\).

Choose a fixed \(t_*\in(0,t_0)\) with \(b t_*^{-\kappa}\ge2C_0\).
Then \eqref{eq:ad-super-positive-sandwich} gives
\(G_n(t)\asymp\theta^nt^{-\kappa}\) uniformly on
\([Da^{-n},t_*]\). In particular \(G_n(t_*)=O(\theta^n)\).
For \(t<Da^{-n}\) use \(G_n(t)\le1\), and for \(t\ge t_*\) integer
support gives \(G_n(t)\le\e^{-(t-t_*)}G_n(t_*)\). Thus the two outer
parts of the Mellin integral \eqref{eq:ad-laplace-gamma} are
\(O_r(a^{-rn})\) and \(O_r(\theta^n)\), respectively. Its middle part is
\[
 \int_{Da^{-n}}^{t_*}G_n(t)t^{r-1}\dd t
 \asymp \theta^n\int_{Da^{-n}}^{t_*}t^{r-\kappa-1}\dd t
 \asymp
 \begin{cases}
 a^{-rn},&0<r<\kappa,\\
 n\theta^n,&r=\kappa,\\
 \theta^n,&r>\kappa .
 \end{cases}
\]
Since \(\theta a^\kappa=1\), the middle integral has order
\(b_n^{\mathrm{sc}}(r)\) and dominates the two outer bounds up to a
constant. This proves \eqref{eq:main-ad-super-harmonic} for both counts.

Under \eqref{eq:main-ad-poly-conditions},
Proposition~\ref{prop:ad-poly-ratio-transfer} and the harmonic-moment bounds
just proved give
\[
        J_{n,j}(\eps)\leq C_\eps b_n^{\mathrm{sc}}(q),
        \qquad j=1,2,3,
\]
which proves \eqref{eq:main-ad-super-ratio-profile}.  Under
\textup{(AD-E)}, Proposition~\ref{prop:ad-exp-ratio-transfer} and the matched
positive-transform profiles give
\[
        J_{n,j}(\eps)=O(\theta^n),
        \qquad j=1,2,3.
\]
This completes the proof.
\end{proof}

\begin{proof}[Proof of Theorem~\ref{thm:main-ad-critical}]
For either count, use the notation of the preceding proof. The zero-atom
gap is \(\Delta_n=F_n^U(0)-F_n^L(0)\) for \(X_n\) and
\(\Delta_n=H_n^U(0)-H_n^L(0)\) for \(N_n\).
Lemma~\ref{lem:ad-critical-zero-gap} gives \(0\le\Delta_n\le C_0n^{-\sigma}\),
and \(G_n^L-\Delta_n\le G_n\le G_n^U+\Delta_n\).
Write \(H_{n,e}^Y(r)\) for the harmonic moment under envelope \(e\).
At \(t=1\), \eqref{eq:ad-critical-fixed-profile} and
\eqref{eq:ad-exact-control-transform} give
\(G_n^L(1)+G_n^U(1)=O(n^{-\sigma})\), hence \(G_n(1)=O(n^{-\sigma})\).
Split the Mellin integral \eqref{eq:ad-laplace-gamma} at \(1\).
On \((0,1)\) the integrated zero-atom error is at most \(\Delta_n/r\);
on \([1,\infty)\) each positive transform is at most \(\e^{-(t-1)}\)
times its value at \(1\). Thus, for each fixed \(r>0\),
\[
 H_{n,L}^Y(r)-C_rn^{-\sigma}
 \le H_n^Y(r)\le H_{n,U}^Y(r)+C_rn^{-\sigma}.
\]
For both envelopes, \(\gamma_e=\gamma\) and \(l_e'(1)=\tau\), so the
scale is \(\omega_X=\gamma\) for the population and \(\omega_N=\gamma\tau\)
for the progenitor count.
For \(0<r<\sigma\), the two envelope limits in
Propositions~\ref{prop:ad-exact-critical-harmonic}
and~\ref{prop:ad-exact-control-harmonic} are given by
\[
 \frac{1}{\Gamma(r)}
 \int_0^\infty(1+\omega_Y s)^{-\sigma}s^{r-1}\dd s
 =\omega_Y^{-r}\frac{\Gamma(\sigma-r)}{\Gamma(\sigma)}.
\]
Since \(n^{-\sigma}=o(n^{-r})\), this is also the limit of
\(n^rH_n^Y(r)\). At \(r=\sigma\), both envelopes have limit
\(\omega_Y^{-\sigma}/\Gamma(\sigma)\) under normalization
\(n^\sigma/\log n\), which again makes the error vanish.

For \(r>\sigma\), the upper envelope gives \(H_n^Y(r)=O(n^{-\sigma})\).
By \eqref{eq:ad-critical-fixed-profile} and
\eqref{eq:ad-critical-control-profile}, for fixed \(t>0\),
\[
 n^\sigma G_n^L(t)\longrightarrow V_L(\e^{-t})-V_L(0),
 \qquad
 V_L(s)=
 \begin{cases}
 \mathsf U_L(s), & Y_n=X_n,\\
 d_L(s)\mathsf U_L(l_L(s)), & Y_n=N_n .
 \end{cases}
\]
In both cases \(V_L(s)\to\infty\) as \(s\uparrow1\) by
\eqref{eq:ad-critical-U-boundary}. Choose \(t_*>0\) so that
\(V_L(\e^{-t_*})-V_L(0)>2C_0\). The lower comparison gives
\(G_n(t_*)\ge c n^{-\sigma}\) for all large \(n\).
Monotonicity gives the same lower bound on \([t_*/2,t_*]\);
integration there proves \eqref{eq:main-ad-critical-harmonic}.

Under \eqref{eq:main-ad-poly-conditions},
Proposition~\ref{prop:ad-poly-ratio-transfer} and the harmonic-moment bounds
just proved give
\[
        J_{n,j}(\eps)\leq C_\eps b_n^{\mathrm{cr}}(q),
        \qquad j=1,2,3,
\]
which proves \eqref{eq:main-ad-critical-ratio-profile}.  Under
\textup{(AD-E)}, Proposition~\ref{prop:ad-exp-ratio-transfer} and the matched
critical positive-transform bounds give
\[
        J_{n,j}(\eps)=O(n^{-\sigma}),
        \qquad j=1,2,3.
\]
This completes the proof.
\end{proof}

\begin{proof}[Proof of Corollary~\ref{cor:main-ad-ratio-lower}]
Take \(Y_n=N_n\) for \(j=1\) and \(Y_n=X_n\) for \(j=2,3\).
By Proposition~\ref{prop:ad-ratio-identities},
\(J_{n,j}(\eps)=\Ebf[g_j(Y_n);Y_n>0]\).
If \(g_j(k)\ge ck^{-q}\) for \(k\ge1\), then
\(J_{n,j}(\eps)\ge c\Ebf[Y_n^{-q};Y_n>0]\).
The lower bounds in \eqref{eq:main-ad-super-harmonic} and
\eqref{eq:main-ad-critical-harmonic} give the assertion in the
respective regimes.

If \(g_j(k)\ge cs^k\) for \(k\ge1\), put
\(P_n(z)=\Ebf[z^{Y_n};Y_n>0]\), so \(J_{n,j}(\eps)\ge cP_n(s)\).
Let \(a_n=\theta^n\) in the supercritical case and
\(a_n=n^{-\sigma}\) in the critical case. For the chosen count,
the fixed-argument lower bound in the corresponding theorem proof gives
\(P_n(u)\ge c_1a_n\) at \(u=\e^{-t_*}\). Take any \(v\in(u,1)\).
In the supercritical case the window bound at \(t=-\log v\) gives
\(P_n(v)=O(a_n)\); in the critical case this follows from the
fixed-argument upper-envelope profile and the zero-atom bound.
Lemma~\ref{lem:ad-finite-extraction} now gives \(P_n(s)\ge c_s a_n\)
for every fixed \(s\in(0,1)\) and all sufficiently large \(n\).
\end{proof}

\begin{remark}[Exact ratio constants]
Theorems~\ref{thm:main-ad-supercritical} and
\ref{thm:main-ad-critical} give upper bounds for the ratio-deviation
probabilities; Corollary~\ref{cor:main-ad-ratio-lower} gives matching
lower bounds under its one-step kernel hypotheses.  An exact ratio constant follows if, in addition,
\(k^q\psi(k;\eps)\to c_\eps\), the corresponding normalized harmonic moment
has an exact limit, and fixed finite states are negligible on that scale.
These additional conclusions are not assumed here.
\end{remark}

\section{Examples and discussion}
\label{sec:comparisons-examples}
\label{sec:discussion}

The random-slope and deterministic-slope theories use the same one-step
kernel transfer but different population profiles.  In the first case the
profiles are generated by the environmental product, small-state persistence,
and the finite-horizon-minimum ladder scale; in the second they are generated
by the auxiliary exact recursions and their supercritical or parabolic
small-value behavior.  The examples below identify the common boundary and
illustrate both exact and nonexact controls.

\paragraph{The common boundary.}
We call the envelopes \emph{coincident} when
\(d_L=d_U=:d\) and \(l_L=l_U=:l\); then
\(h_k(s)=d(s)l(s)^k\), and the original process has the exact transformed
recursion \(F_{n+1}(s)=R(s)F_n(u(s))\), with \(u=l\circ f\) and
\(R=d\circ f\).  When the envelopes are distinct, their parameters are
\emph{matched} in the supercritical case when the two auxiliary recursions
satisfy the same branch conditions in \textup{(AD-S)} and have the same
rate \(\theta\); they are matched in the critical case when both
are critical and \(\beta_L/\gamma_L=\beta_U/\gamma_U\). Under
\textup{(AD0)} this equality forces \(\beta_L=\beta_U\) and
\(\gamma_L=\gamma_U\), as shown in Section~\ref{subsec:main-ad}.
Matching determines the harmonic-moment scales and, at criticality, the
limits for \(0<r\le\sigma\) in Theorem~\ref{thm:main-ad-critical}. It does
not imply coincident envelopes or an exact recursion.  If the parameters differ, the two positive-transform
comparisons yield their respective scales and no common normalization follows.

In the individual-sum model, let \(A_n\equiv a_0\in\N\), let \(\ell\) be
the p.g.f. of \(L\), and let \(b\) be the p.g.f. of \(B\).  Then
\[
 h_k(z)=b(z)\ell(z^{a_0})^k,\qquad
 u(s)=\ell(f(s)^{a_0}),\qquad R(s)=b(f(s)),\qquad
 u'(1)=m\lambda a_0.
\]
Thus the deterministic-multiplier model has coincident envelopes and is an
exact member of the asymptotically deterministic class.  For the classical
branching process with child-level immigration \(I\), one instead has
\(u=f\) and \(R=b\); in the critical case,
\(\gamma=f''(1)/2\), \(\beta=\Ebf I\), and
\(\sigma=2\Ebf I/f''(1)\), under the full critical hypotheses.  These
critical regimes are classical; see, for example, \cite{LiZhang2021}.  If
immigration is introduced at the progenitor level through \(\phi(k)=k+B\),
then \(R=b\circ f\).  Equal immigration means need not give the same
profile or constant, since \(b\) and \(b\circ f\) enter different
recursions.

\paragraph{Binomial emigration with immigration.}
Let \(\phi_n(k)=\sum_{j=1}^k\chi_{n,j}+B_n\), where
\(\Pbf(\chi_{n,j}=1)=p\), and assume that the control, offspring, and
immigration variables are independent.  If \(b\) is the p.g.f. of \(B_n\),
then
\[
 h_k(z)=b(z)(1-p+pz)^k,\qquad
 F_{n+1}(s)=b(f(s))F_n(1-p+pf(s)).
\]
The effective mean is \(a=pm\).  If \(pm=1\), then
\(\gamma=pf''(1)/2\), \(\beta=m\Ebf B\), and
\(\sigma=2m\Ebf B/\{pf''(1)\}\).  Under the full conditions in
\textup{(AD-C)}, the population and progenitor harmonic moments have the
critical scales \(n^{-\sigma}\), \(n^{-\sigma}\log n\), and \(n^{-r}\)
according as \(r>\sigma\), \(r=\sigma\), or \(r<\sigma\).
Moreover,
\[
 \frac{N_n}{X_n}-p
 =\frac1{X_n}\sum_{j=1}^{X_n}(\chi_{n,j}-p)+\frac{B_n}{X_n},
\]
while \(X_{n+1}/N_n\) is an offspring empirical mean with random sample
size \(N_n\).  The one-step transfer results therefore give the
corresponding bounds for all three ratios under the polynomial or exponential
one-step conditions stated in Section~\ref{sec:ad-controls}.

\paragraph{A nonexact control with matched critical parameters.}
Here \emph{nonexact} means that there is no single pair of p.g.f.s \((d,l)\)
for which \(h_k(z)=d(z)l(z)^k\) for every \(k\).  Set
\[
 l_L(z)=\tfrac34z+\tfrac14z^3,\qquad
 l_U(z)=\tfrac14+\tfrac34z^2,\qquad d(z)=z,
\]
and let
\(\phi_n(k)=1+\sum_{j=1}^kL_{n,j}^{L}\) for even \(k\), while
\(\phi_n(k)=1+\sum_{j=1}^kL_{n,j}^{U}\) for odd \(k\), where the two
independent control arrays have p.g.f.s \(l_L\) and \(l_U\), respectively.
Thus the choice of the control law is measurable in \(X_n\), and
\(h_k(z)=d(z)l_L(z)^k\) for even \(k\) and
\(h_k(z)=d(z)l_U(z)^k\) for odd \(k\).  Since
\(l_U(z)-l_L(z)=(1-z)^3/4\geq0\), these are distinct lower and upper
envelopes, while \(l_L'(1)=l_U'(1)=3/2\) and
\(l_L''(1)=l_U''(1)=3/2\).

Take \(f(s)=1/3+2s/3\), set \(u_e=l_e\circ f\), and note that
\(R=d\circ f=f\).  For \(e\in\{L,U\}\),
\[
 u_e'(1)=1,\qquad \gamma_e=\tfrac12u_e''(1)=\tfrac13,
 \qquad \beta_e=R'(1)=\tfrac23,
 \qquad \frac{\beta_e}{\gamma_e}=2.
\]
The two auxiliary recursions therefore have matched critical parameters with
common index \(\sigma=2\), although the original parity-dependent control has
no exact factorization.  Theorem~\ref{thm:main-ad-critical} applies to the
harmonic moments and polynomial ratio kernels.  Condition \textup{(AD-E)}
does not hold in general, because
\[
 \frac1k\log\Ebf[\e^{\eta\phi(k)}]
 =\frac{\eta}{k}+
 \begin{cases}
 \log l_L(\e^\eta),&k\ \text{even},\\
 \log l_U(\e^\eta),&k\ \text{odd},
 \end{cases}
\]
and the two subsequential limits differ for \(\eta\neq0\).  The polynomial
one-step conditions hold uniformly, and geometric one-step bounds can be
verified directly for the two bounded control laws.

For comparison, if
\(\phi_n(k)\mid X_n=k\sim\operatorname{Bin}(k,p_k)+B_n\), where
\(p_k\to p\) and \(0<p_-\leq p_k\leq p_+<1\), then
\[
 b(z)(1-p_++p_+z)^k\leq h_k(z)\leq b(z)(1-p_-+p_-z)^k.
\]
If \(p_-<p_+\), these two envelopes have different slopes and thus do
not satisfy \textup{(AD0)}. Comparison at the limiting slope \(p\) gives
\[
 \big|(1-p_k+p_kz)^k-(1-p+pz)^k\big|
 \le k|p_k-p|(1-z),\qquad 0\le z\le1.
\]
Proposition~\ref{prop:ad-raw-remainders} gives the resulting accumulated
population error. The one-generation bound does not establish that this
error is negligible on the harmonic-moment scale.

\paragraph{Discussion.}
The effect of immigration differs across the four regimes considered above.
In the random-slope upper-tail problem, immigration changes the sharp
prefactor \(c_X(\alpha)\), while the environmental large-deviation rate
\(I(\rho)\) remains unchanged.  In the supercritical BPREI harmonic-moment
problem, immigration enters through the persistence probability
\[
        \gamma_i=\Ebf[h_0p_1^i]
\]
and may therefore move the phase boundary \(c_r=\gamma_i\).  In the critical
random-slope regime, the normalization \(d_n^-\) is determined by the
environmental ladder process, whereas immigration affects the endpoint-minimum
entrance law and hence the limiting constant.  In the deterministic-slope
regime, immigration enters through the transform \(R\); consequently,
\(\beta=R'(1)\) and the critical index \(\sigma=\beta/\gamma\) may change the
harmonic-moment decay rate.
If a ratio-deviation event \(\mathcal E_n(\eps)\) satisfies
\(\Pbf(\mathcal E_n(\eps))\leq C_\eps a_n\) with
\(\sum_na_n<\infty\), then
\(\sum_n\Pbf(\mathcal E_n(\eps))<\infty\), and Borel--Cantelli gives
\(\Pbf\{\mathcal E_n(\eps)\ \mathrm{i.o.}\}=0\).  If this summability holds for every \(\eps=1/j\), \(j\geq1\), the
countable intersection gives almost sure convergence of the corresponding
ratio to its center along generations on which its denominator is positive.

Two extensions remain open: heavy immigration may contribute to the
upper-deviation exponent, and a negative tilt satisfying
\(\Lambda'(-r_c)=0\) leads to a weighted finite-horizon-minimum problem not
covered by the present critical theory.

\appendix
\section{Verification of the critical finite-variance inputs}
\label{sec:critical-verification}

The purpose of this appendix is to verify the abstract critical assumptions
\[
        (C)=(\mathrm{SD})+(\mathrm{EL})+(\mathrm{UIM})+(\mathrm{LOC})
\]
for the centered finite-variance BPREI class \(\mathcal C_{\mathrm{fv}}\)
introduced in Section~\ref{subsubsec:bprei-critical-fv-class}, thereby completing the model-specific
verification used in Theorem~\ref{thm:bprei-critical-fv-class}.  Condition (C1)
supplies the Spitzer--Doney persistence scale
\[
        d_n^-\sim\kappa_0n^{-1/2}.
\]
The remaining task is to verify the post-minimum inverse-moment bound,
convolution-tail localization, and convergence of the endpoint-minimum
entrance laws.

The verification is carried out in three stages.  First,
Lemma~\ref{lem:bprei-critical-fv-kernels} and
Theorem~\ref{thm:bprei-critical-fv-occupation} establish local
nonnegative-bridge bounds and an integrated weighted-occupation estimate for
the centered finite-variance environmental walk.  Second,
Proposition~\ref{prop:bprei-critical-fv-uim} combines these environmental
estimates with the normalized quenched-variance bound and persistent
immigration to prove the uniform post-minimum inverse-moment and localization
assumptions \((\mathrm{UIM})\) and \((\mathrm{LOC})\).  Finally,
Proposition~\ref{prop:bprei-critical-fv-entrance} uses reversal at the endpoint
minimum to construct a limiting entrance law \(\nu^-\), prove total-variation
convergence, and show that the limit is independent of the fixed initial
population.  No lattice or nonlattice restriction and no Cram\'er condition
are used in the bridge or entrance-law arguments.

Throughout, write \(\zeta_n:=\log M_n\),
\(S_0:=0\), \(S_n:=\sum_{j=0}^{n-1}\zeta_j\), and
\(L_n:=\min_{0\leq j\leq n}S_j\).  For a fixed horizon \(m\), put
\[
        I_j^{(m)}:=\min_{j\leq r\leq m}S_r,
        \qquad
        \Xi_m:=\sum_{j=0}^{m}\e^{-(S_j-I_j^{(m)})},
\]
and, for integers \(y\geq0\),
\[
        B_{m,y}:=\{L_m\geq0,\ S_m\in[y,y+1)\}.
\]
The quantity \(\Xi_m\) is a weighted occupation functional measuring how
closely the environmental path remains above its successive future minima.
Its cubic bridge moment in
Theorem~\ref{thm:bprei-critical-fv-occupation} is the environmental estimate
used in the proof of Proposition~\ref{prop:bprei-critical-fv-uim}.

\subsection{Centered finite-variance bridge and occupation estimates}
\label{subsec:critical-fv-bridge}

Assume in this subsection that
\begin{equation}\label{eq:bprei-critical-fv-increments}
        \Ebf\zeta_0=0,
        \qquad
        0<\Var(\zeta_0)<\infty.
\end{equation}
We use one notation for the environmental walk and its reflection.  Whenever
\(W=(W_n)_{n\geq0}\) appears below, it denotes either
\[
        W_n=S_n
        \qquad\text{or}\qquad
        W_n=-S_n,
        \qquad n\geq0.
\]
Thus \(-S\) is obtained from \(S\) by reflection about zero, and its
increments have law \(-\zeta_0\).  The reflected walk is introduced because,
after reversing a terminal segment of a path, an upper constraint for the
original walk becomes a nonnegative-persistence constraint for the reflected
walk.  This forward--reflected duality allows the initial and terminal
portions of a bridge to be treated by the same persistence estimates.

For \(x,b\geq0\) and \(n\geq1\), define
\[
\begin{aligned}
        \pi_n^W(x)
        &:=\Pbf(x+W_j\geq0,\ 1\leq j\leq n),\\
        k_n^W(x;b)
        &:=\Pbf(x+W_j\geq0,\ 1\leq j\leq n,
                 \ x+W_n\in[b,b+1)),
\end{aligned}
\]
and set
\(k_0^W(x;b):=\one_{\{x\in[b,b+1)\}}\).

We shall use three classical facts.  First, a concentration inequality for
centered finite-variance sums gives a constant \(C_{\mathrm K}\) such that
\begin{equation}\label{eq:bprei-critical-concentration}
        \sup_{v\in\R}\Pbf(W_n\in[v,v+\lambda))
        \leq C_{\mathrm K}\lambda n^{-1/2},
        \qquad \lambda\geq1,
\end{equation}
for \(W\in\{S,-S\}\); see \cite{Petrov1975}.

Second, for \(W\in\{S,-S\}\), let
\[
        T_0^{W,-}:=0,
        \qquad
        T_{k+1}^{W,-}:=
        \inf\{n>T_k^{W,-}:W_n<W_{T_k^{W,-}}\}
\]
be the strict descending ladder epochs of \(W\), and put
\[
        H_k^{W,-}:=-W_{T_k^{W,-}},
        \qquad k\geq0.
\]
Since the centered nondegenerate walk oscillates, these ladder epochs are
finite almost surely.  Define the associated descending-ladder renewal
function by
\[
        V^W(x):=\sum_{k\geq0}\Pbf(H_k^{W,-}\leq x),
        \qquad x\geq0.
\]
To relate this function to the walk constrained to remain nonnegative, let
\[
        \tau_x^{W,-}:=\inf\{n\geq1:x+W_n<0\}.
\]
The walk is said to be killed at \(\tau_x^{W,-}\), meaning that it is stopped
when it first enters \(( -\infty,0)\).  There exist deterministic constants
\[
        0<c_V\leq C_V<\infty
\]
such that, for \(W\in\{S,-S\}\) and \(x\geq0\),
\begin{equation}\label{eq:bprei-critical-renewal-harmonic}
        c_V(1+x)\leq V^W(x)\leq C_V(1+x),
        \qquad
        \Ebf[V^W(x+W_1);x+W_1\geq0]\leq V^W(x).
\end{equation}
Only this superharmonic inequality and the linear bounds on \(V^W\) are used
below.  These properties follow from standard descending-ladder renewal
theory; see \cite[Chapters~XII and XVIII]{Feller1971} and
\cite{BertoinDoney1994}.

The third fact is a uniform fixed-window bound for the potential measure of
paths constrained to remain nonnegative.  For a Borel set \(A\subseteq\R\),
define
\[
U^+(A):=
\sum_{r\geq0}
\Pbf(S_j\geq0,\ 1\leq j\leq r,\ S_r\in A).
\]
Then
\begin{equation}\label{eq:bprei-critical-fixed-window-potential}
        U_2^+:=\sup_{z\in\R}U^+([z,z+2))<\infty.
\end{equation}
This bound is proved by using Feller duality to represent \(U^+\) as the
renewal measure of the weak ascending ladder heights.  Let \(H_0^+:=0\), let
\(H_j^+\) be the weak ascending ladder heights, and write
\(Y_j^+:=H_j^+-H_{j-1}^+\).  Choose \(\delta>0\) such that
\(p_\delta:=\Pbf(Y_1^+<\delta)<1\), and let
\(N_z:=\inf\{j\geq0:H_j^+\geq z\}\).  The renewal property gives, for
\(z\geq0\),
\[
\begin{aligned}
U^+([z,z+\delta))
&=\Ebf\sum_{j\geq0}\one_{\{H_j^+\in[z,z+\delta)\}}\\
&\leq\Ebf\left[\one_{\{N_z<\infty\}}
       \sum_{\ell\geq0}
       \one_{\{Y_{N_z+1}^+<\delta,\ldots,Y_{N_z+\ell}^+<\delta\}}\right]\\
&\leq\sum_{\ell\geq0}p_\delta^\ell
 =\frac1{1-p_\delta}.
\end{aligned}
\]
An interval of length two is covered by at most
\(\lceil2/\delta\rceil+1\) intervals of length \(\delta\).  Hence
\[
U_2^+\leq\frac{\lceil2/\delta\rceil+1}{1-p_\delta}<\infty,
\]
which proves \eqref{eq:bprei-critical-fixed-window-potential}.

For related local estimates for random walks conditioned to remain
nonnegative, see \cite{Caravenna2005,VatutinWachtel2009,Doney2012}.
The following lemma is proved in the unit-window form required below.

\begin{lemma}[Persistence and unit-window endpoint bounds]
\label{lem:bprei-critical-fv-kernels}
Under \eqref{eq:bprei-critical-fv-increments}, there is \(C<\infty\),
depending only on the increment law, such that, for
\(W\in\{S,-S\}\),
\begin{equation}\label{eq:bprei-critical-fv-kernels}
        \pi_n^W(x)\leq C(1+x)n^{-1/2},
        \qquad
        k_n^W(x;b)\leq C(1+x)(1+b)n^{-3/2}
\end{equation}
for all \(x,b\geq0\) and \(n\geq1\).
\end{lemma}

\begin{proof}
Let
\[
        \tau_x:=\inf\{j\geq1:x+W_j<0\}.
\]
By \eqref{eq:bprei-critical-renewal-harmonic},
\[
        \Ebf[V^W(x+W_n);\tau_x>n]\leq V^W(x).
\]
Fix \(\varepsilon\in(0,1]\).  Since
\(\pi_n^W(x)=\Pbf(\tau_x>n)\), split
\[
\begin{aligned}
\pi_n^W(x)
={}&\Pbf(\tau_x>n,\ x+W_n>\varepsilon\sqrt n)\\
&+\Pbf(\tau_x>n,\ 0\leq x+W_n\leq\varepsilon\sqrt n).
\end{aligned}
\]
For the first term, the lower bound on \(V^W\) and the preceding
superharmonic estimate give
\[
\begin{aligned}
\Pbf(\tau_x>n,\ x+W_n>\varepsilon\sqrt n)
&\leq
\frac{\Ebf[V^W(x+W_n);\tau_x>n]}
     {V^W(\varepsilon\sqrt n)}\\
&\leq C\frac{1+x}{\varepsilon\sqrt n}.
\end{aligned}
\]
For the second term, suppose \(n\geq2\) and put \(n':=\lfloor n/2\rfloor\).  Conditioning on the
position \(x+W_{n'}=z\) reached at time \(n'\), applying the Markov property,
and then dropping the requirement that the path remain nonnegative during the
remaining \(n-n'\) steps gives
\[
\begin{aligned}
&\Pbf(\tau_x>n,\ 0\leq x+W_n\leq\varepsilon\sqrt n)\\
&\quad\leq
\int_{[0,\infty)}
\Pbf(\tau_x>n',\ x+W_{n'}\in\dd z)
\Pbf(z+W_{n-n'}\in[0,\varepsilon\sqrt n])\\
&\quad\leq
\pi_{n'}^W(x)
\sup_{z\in\R}\Pbf(z+W_{n-n'}\in[0,\varepsilon\sqrt n]).
\end{aligned}
\]
Since \(n-n'\asymp n\), the concentration estimate
\eqref{eq:bprei-critical-concentration}, applied to the half-open interval
\([0,1+\varepsilon\sqrt n)\) containing \([0,\varepsilon\sqrt n]\), yields
\[
\sup_{z\in\R}\Pbf(z+W_{n-n'}\in[0,\varepsilon\sqrt n])
\leq C(\varepsilon+n^{-1/2}).
\]
Consequently,
\[
\pi_n^W(x)
\leq C(\varepsilon+n^{-1/2})\pi_{n'}^W(x)
+C\frac{1+x}{\varepsilon\sqrt n}.
\]
For \(f(n):=\sup_{x\geq0}\pi_n^W(x)/(1+x)\), this becomes
\[
f(n)\leq C(\varepsilon+n^{-1/2})f(\lfloor n/2\rfloor)
      +\frac{C}{\varepsilon\sqrt n}.
\]
Choose \(\varepsilon\) and then \(n_0\) so that the first coefficient is at
most \(1/2\) for \(n\geq n_0\geq2\). Multiplication by \(\sqrt n\)
gives a recursion for \(\sqrt n f(n)\) with coefficient
\(\tfrac12\sqrt{n/\lfloor n/2\rfloor}\leq\sqrt3/2<1\).
Induction, absorbing the finitely many initial values, yields
\(f(n)\leq Cn^{-1/2}\).

We next prove the unit-window endpoint bound in
\eqref{eq:bprei-critical-fv-kernels}.  Recall that
\[
        k_n^W(x;b):=
        \Pbf(x+W_j\geq0,\ 1\leq j\leq n,
              \ x+W_n\in[b,b+1)).
\]
Assume \(n\geq3\), and divide the \(n\) increments into three consecutive
blocks by setting
\[
        r_1=r_3:=\lfloor n/3\rfloor,
        \qquad
        r_2:=n-r_1-r_3.
\]
Thus all three block lengths are comparable with \(n\).  The first block
contributes the persistence probability \(\pi_{r_1}^W(x)\).  To treat the
last block, define its reversed partial sums by
\[
        R_j:=W_n-W_{n-j},
        \qquad 0\leq j\leq r_3.
\]
The process \((R_j)_{0\leq j\leq r_3}\) has the same law as
\((W_j)_{0\leq j\leq r_3}\) and is independent of the first and middle
blocks.  On the event defining \(k_n^W(x;b)\), put
\[
        q:=x+W_n\in[b,b+1).
\]
Since the original path remains nonnegative,
\[
        q-R_j=x+W_{n-j}\geq0,
        \qquad 1\leq j\leq r_3.
\]
Because \(q<b+1\), this implies
\[
        b+1-R_j\geq0,
        \qquad 1\leq j\leq r_3.
\]
Thus the reflected walk \((-R_j)\), started from \(b+1\), remains
nonnegative during the last block.  The probability of this necessary event
is \(\pi_{r_3}^{-W}(b+1)\).

Now condition on the increments in the first block and in the reversed last
block.  If
\[
        D_{\rm mid}:=W_{r_1+r_2}-W_{r_1}
\]
denotes the increment over the middle block, then the endpoint condition
requires
\[
D_{\rm mid}\in
[\,b-x-W_{r_1}-R_{r_3},\,
   b+1-x-W_{r_1}-R_{r_3}\,),
\]
which is an interval of length one.  Dropping the requirement that the path
remain nonnegative during the middle block only enlarges the probability.
Independence of the three increment blocks therefore gives
\[
        k_n^W(x;b)
        \leq
        \pi_{r_1}^W(x)\,
        \pi_{r_3}^{-W}(b+1)\,
        \sup_{z\in\R}\Pbf(W_{r_2}\in[z,z+1)).
\]
Applying the persistence bound already proved to the first and reflected last
blocks, and applying \eqref{eq:bprei-critical-concentration} to the middle
block, yields
\[
\begin{aligned}
        k_n^W(x;b)
        &\leq C(1+x)r_1^{-1/2}
                 (1+b)r_3^{-1/2}r_2^{-1/2}\\
        &\leq C(1+x)(1+b)n^{-3/2}.
\end{aligned}
\]
The finitely many cases \(n<3\) are absorbed into the constant.
\end{proof}

Recall that
\(B_{m,y}=\{L_m\geq0,\ S_m\in[y,y+1)\}\).  Hence, by the definition of
\(k_m^S\),
\begin{equation}\label{eq:bprei-critical-Hloc}
        \Pbf(B_{m,y})=k_m^S(0;y)
        \leq C(1+y)m^{-3/2},
        \qquad m\geq1,\ y\in\Nzero.
\end{equation}

We next derive the ladder-duration estimate used in
Theorem~\ref{thm:bprei-critical-fv-occupation}.  Let
\[
        \widehat S_n:=-S_n,
        \qquad n\geq0,
\]
and define its first weak descending ladder duration by
\[
        \tau:=\inf\{t\geq1:\widehat S_t\leq0\}.
\]
When a terminal segment of \(S\) is reversed, its increments have the law of
\(-\zeta_0\); consequently, its first weak descending ladder duration has the
same law as \(\tau\).

Fix \(t\geq2\).  On \(\{\tau=t\}\),
\[
        \widehat S_j>0,
        \qquad 1\leq j\leq t-1,
        \qquad
        \widehat S_t\leq0.
\]
For \(b\in\Nzero\), let
\[
\begin{aligned}
A_{t,b}:=\{\,&\widehat S_j>0,\ 1\leq j\leq t-1,\
\widehat S_{t-1}\in[b,b+1),\\
&\widehat S_t\leq0\,\}.
\end{aligned}
\]
The events \(A_{t,b}\), \(b\geq0\), are disjoint and their union is
\(\{\tau=t\}\).  Since
\[
        \widehat S_t=\widehat S_{t-1}-\zeta_{t-1},
\]
the conditions \(\widehat S_{t-1}\in[b,b+1)\) and
\(\widehat S_t\leq0\) imply
\[
        \zeta_{t-1}\geq\widehat S_{t-1}\geq b.
\]
The path up to time \(t-1\) is independent of the final increment
\(\zeta_{t-1}\).  Therefore
\[
\begin{aligned}
\Pbf(A_{t,b})
&\leq
\Pbf(\widehat S_j\geq0,\ 1\leq j\leq t-1,\
     \widehat S_{t-1}\in[b,b+1))\Pbf(\zeta_0\geq b)\\
&=k_{t-1}^{-S}(0;b)\Pbf(\zeta_0\geq b).
\end{aligned}
\]
Summing over \(b\geq0\) and applying
Lemma~\ref{lem:bprei-critical-fv-kernels} gives
\[
\begin{aligned}
\Pbf(\tau=t)
&\leq\sum_{b\geq0}k_{t-1}^{-S}(0;b)\Pbf(\zeta_0\geq b)\\
&\leq Ct^{-3/2}\sum_{b\geq0}(1+b)\Pbf(\zeta_0\geq b).
\end{aligned}
\]
Since \(\Ebf\zeta_0^2<\infty\),
\[
\sum_{b\geq0}(1+b)\Pbf(\zeta_0\geq b)
\leq C\{1+\Ebf[(\zeta_0^+)^2]\}<\infty.
\]
Thus, after enlarging the constant to include \(t=1\),
\begin{equation}\label{eq:bprei-critical-ladder-duration}
        \Pbf(\tau=t)\leq Ct^{-3/2},
        \qquad t\geq1.
\end{equation}

The next theorem supplies the integrated occupation estimate that is new to
the critical verification.

\begin{theorem}[Integrated occupation on a nonnegative bridge]
\label{thm:bprei-critical-fv-occupation}
Under \eqref{eq:bprei-critical-fv-increments}, there is \(C<\infty\) such
that
\begin{equation}\label{eq:bprei-critical-Hocc}
        \Ebf[\Xi_m^3;B_{m,y}]
        \leq C(1+y)^4m^{-3/2},
        \qquad m\geq1,\ y\in\Nzero.
\end{equation}
\end{theorem}

\begin{proof}
We prove \eqref{eq:bprei-critical-Hocc} by reducing the cubic occupation
moment to a weighted bivariate renewal-bridge estimate.  Reverse the walk at
time \(m\) and decompose the reversed path into its weak descending ladder
pieces.  For each piece, its duration, height decrease, and weighted
occupation will be denoted by \(\tau_i\), \(H_i\), and \(\Phi_i\),
respectively.  On \(B_{m,y}\), the ladder pieces form a bridge whose total
duration is \(m\), whose total height decrease lies in \([y,y+1)\), and
whose occupation functional satisfies
\[
        \Xi_m=1+\sum_i\Phi_i.
\]
For a fixed number \(k\) of pieces, an elementary cubic inequality and
exchangeability reduce the contribution of \(\Xi_m^3\) to \(k^3\) times the
bridge mass obtained by weighting the first piece by \(\Phi_1^3\).  Thus it
is enough to prove that the sum of these first-piece-weighted bridge masses is
bounded by \(C(1+y)^4m^{-3/2}\).  The one-piece estimate below supplies the
\(t^{-3/2}\) duration decay and a summable height weight.  Convolution with
the remaining ordinary ladder pieces preserves the \(m^{-3/2}\) bridge
scale, while the piece-count estimate controls the factor \(k^3\) and
produces the polynomial factor \((1+y)^4\).  The resulting weighted bridge
bound then yields \eqref{eq:bprei-critical-Hocc}.

Fix \(m\), reverse the walk by setting
\(R_u:=S_{m-u}-S_m\), \(0\leq u\leq m\), and put
\(J_u:=\min_{0\leq v\leq u}R_v\).  Then
\begin{equation}\label{eq:bprei-critical-reversal-occupation}
        S_j-I_j^{(m)}=R_{m-j}-J_{m-j},
        \qquad
        \Xi_m=\sum_{u=0}^{m}\e^{-(R_u-J_u)}.
\end{equation}
On \(B_{m,y}\), the reversed walk ends at its weak overall minimum and
\(-R_m=S_m\in[y,y+1)\).

Continue \(R\) beyond \(m\) by independent increments with the same law.
Let \(T_0:=0\) and define the weak descending ladder epochs of \(R\) by
\(T_i:=\inf\{u>T_{i-1}:R_u\leq R_{T_{i-1}}\}\).  Write
\[
        \tau_i:=T_i-T_{i-1},
        \qquad
        H_i:=R_{T_{i-1}}-R_{T_i}\geq0,
        \qquad
        \Phi_i:=\sum_{u=T_{i-1}}^{T_i-1}
                    \e^{-(R_u-R_{T_{i-1}})}.
\]
Since the centered nondegenerate walk \(R\) oscillates,
\(T_i<\infty\) almost surely for every \(i\).  The strong Markov property
at the ladder epochs therefore shows that
\((\tau_i,H_i,\Phi_i)\), \(i\geq1\), are i.i.d.  On \(B_{m,y}\), for a
unique \(N\geq1\),
\begin{equation}\label{eq:bprei-critical-ladder-pieces}
        \sum_{i=1}^{N}\tau_i=m,
        \qquad
        \sum_{i=1}^{N}H_i\in[y,y+1),
        \qquad
        \Xi_m=1+\sum_{i=1}^{N}\Phi_i.
\end{equation}
For \(k\geq1\), define
\[
E_{k,m,y}:=\left\{
        \sum_{i=1}^{k}\tau_i=m,\
        \sum_{i=1}^{k}H_i\in[y,y+1)
        \right\}.
\]
The events \(E_{k,m,y}\), \(k\geq1\), are disjoint and describe
\(B_{m,y}\) according to the number of ladder pieces.  Therefore
\[
\Ebf[\Xi_m^3;B_{m,y}]
=\sum_{k\geq1}
\Ebf\left[\left(1+\sum_{i=1}^{k}\Phi_i\right)^3;E_{k,m,y}\right].
\]
Since \(\Phi_i\geq1\),
\[
\left(1+\sum_{i=1}^{k}\Phi_i\right)^3
\leq8\left(\sum_{i=1}^{k}\Phi_i\right)^3
\leq8k^2\sum_{i=1}^{k}\Phi_i^3.
\]
The triples \((\tau_i,H_i,\Phi_i)\) are i.i.d., and
\(E_{k,m,y}\) is invariant under their permutations.  Hence
\[
\Ebf\left[\left(1+\sum_{i=1}^{k}\Phi_i\right)^3;E_{k,m,y}\right]
\leq8k^3\Ebf[\Phi_1^3;E_{k,m,y}].
\]
Consequently, it remains to prove
\begin{equation}\label{eq:bprei-critical-weighted-piece-count}
\sum_{k\geq1}k^3\Ebf[\Phi_1^3;E_{k,m,y}]
\leq C(1+y)^4m^{-3/2}.
\end{equation}
We now prove the two estimates needed for this weighted bivariate renewal
bridge.

\medskip
\noindent\emph{One weighted ladder piece.}
For integers \(h\geq0\), define
\begin{equation}\label{eq:bprei-critical-nu-h}
        \nu_h:=\sum_{b\geq0}(1+b)
        \Pbf(\zeta_0\in[b+h,b+h+2)).
\end{equation}
The final weak downward crossing of the reversed walk corresponds to a
nonnegative original environmental increment.  Hence no dual replacement is involved in this
definition.  Moreover,
\[
        \sum_{h\geq0}\nu_h
        \leq2\sum_{b\geq0}(1+b)\Pbf(\zeta_0\geq b)<\infty
\]
by finite variance.  We claim that
\begin{equation}\label{eq:bprei-critical-one-piece}
        \Ebf[\Phi_1^3;\tau_1=t,\ H_1\in[h,h+1))
        \leq C t^{-3/2}\nu_h,
        \qquad t\geq1,\ h\geq0.
\end{equation}
For \(t=1\), \(\Phi_1=1\), and the claim follows from the \(b=0\) term in
\eqref{eq:bprei-critical-nu-h}.  Let \(t\geq2\), and order the three observation times in the expansion of
\(\Phi_1^3\) as \(0\leq u_1\leq u_2\leq u_3\leq t-1\), at a cost of at
most \(3!\). Put \(u_0=0\), \(u_4=t-1\), and \(r_i=u_i-u_{i-1}\).
For \(I_a=[a,a+1)\), let \(R_{u_i}\in I_{a_i}\) for \(1\leq i\leq3\),
with \(a_0=0\) and \(R_{t-1}\in I_{a_4}=I_b\). On \(\{\tau_1=t\}\) the
walk stays positive at times \(1,\ldots,t-1\). Allowing nonnegative paths,
the Markov property bounds the mass of this sequence of bins by
\[
 \prod_{i=1}^4 \sup_{x\in I_{a_{i-1}}}k^{-S}_{r_i}(x;a_i)
 \leq C(1+b)\prod_{i=1}^3(1+a_i)^2\prod_{i=1}^4g(r_i),
\]
where \(g(0)=1\) and \(g(r)=r^{-3/2}\) for \(r\geq1\).
Indeed, by Lemma~\ref{lem:bprei-critical-fv-kernels} each intermediate
bin contributes once as an endpoint and once as a starting bin; the
zero-time kernel covers coincident observation times and endpoints.
The product of occupation weights is at most \(\exp(-\sum_{i=1}^3a_i)\).
If \(H_1\in I_h\), the independent final increment of the original sign
lies in \([b+h,b+h+2)\), since its size is \(R_{t-1}+H_1\). This includes
a crossing to zero. Summing first over the bins and then over the gaps gives
\[
 \Ebf[\Phi_1^3;\tau_1=t,H_1\in I_h]
 \leq C\nu_h\!
 \left(\sum_{a\geq0}\e^{-a}(1+a)^2\right)^3 g^{*4}(t-1).
\]
For \(1\leq j\leq4\),
\begin{equation}\label{eq:bprei-critical-gap-convolution}
        g^{*j}(t-1)\leq C_jt^{-3/2},
        \qquad t\geq2.
\end{equation}
Indeed, one of the \(j\) gaps is at least \((t-1)/j\geq t/8\); apply the
factor \(Ct^{-3/2}\) to that gap and sum the remaining gaps using
\(\sum_{r\geq0}g(r)<\infty\).  Combining the height and time sums proves
\eqref{eq:bprei-critical-one-piece}.

\medskip
\noindent\emph{The modified bivariate bridge.}
Let \(\mu\) be the finite measure on first ladder pieces defined by
\[
\mu\{\tau=t,H\in[h,h+1)\}
:=\Ebf[\Phi_1^3;\tau_1=t,H_1\in[h,h+1)).
\]
For \(k\geq1\), let \(\Pbf_*^{(k)}\) denote the product of \(\mu\) for the
first piece and the original ladder-piece law for the remaining \(k-1\)
pieces.  Thus \(\Pbf_*^{(k)}\) is a finite product measure, not necessarily
a probability measure.  Put \(J_y:=[y-1,y+2)\cap[0,\infty)\).  We first
show
\begin{equation}\label{eq:bprei-critical-modified-bridge-mass}
\sum_{k\geq1}\Pbf_*^{(k)}\left(
        \sum_{i=1}^{k}\tau_i=m,
        \ \sum_{i=1}^{k}H_i\in J_y
        \right)
\leq C(1+y)m^{-3/2}.
\end{equation}
For \(t\geq1\), let
\(\mu_t(A):=\Ebf[\Phi_1^3;\tau_1=t,H_1\in A]\), and define the ordinary
ladder-renewal bridge mass by
\[
 B_r(J):=\sum_{\ell\geq0}
 \Pbf\left(\sum_{i=1}^{\ell}\tau_i=r,
            \sum_{i=1}^{\ell}H_i\in J\right),
\]
with the usual empty-sum convention.  Time reversal of an independent
\(r\)-step walk gives
\[
        B_r(J)
        =\Pbf\bigl(S_\ell\geq0\ \text{for all }0\leq\ell\leq r,
                    \ S_r\in J\bigr).
\]
We use this representation below both for the local estimate and, after
summing in \(r\), for Feller duality.  By Tonelli's theorem and the product
definition of \(\Pbf_*^{(k)}\),
\[
\sum_{k\geq1}\Pbf_*^{(k)}\!\left(
 \sum_{i=1}^{k}\tau_i=m,
 \sum_{i=1}^{k}H_i\in J_y\right)
 =\sum_{t=1}^{m}\int_{[0,\infty)}
   B_{m-t}(J_y-h)\,\mu_t(\dd h),
\]
where \(J_y-h:=\{z-h:z\in J_y\}\cap[0,\infty)\).  For \(t\leq m/2\),
\eqref{eq:bprei-critical-Hloc} and
\eqref{eq:bprei-critical-one-piece} give
\[
\begin{aligned}
\sum_{t\leq m/2}\int B_{m-t}(J_y-h)\,\mu_t(\dd h)
&\leq C(1+y)m^{-3/2}
       \sum_{t\geq1}t^{-3/2}\sum_{b\geq0}\nu_b\\
&\leq C(1+y)m^{-3/2}.
\end{aligned}
\]

For \(t>m/2\), decompose the \(h\)-integral over the unit intervals
\([b,b+1)\), \(b\geq0\).  If \(h\in[b,b+1)\), then
\[
        J_y-h\subseteq I_{y,b}:=
        [y-2-b,y+2-b)\cap[0,\infty).
\]
The interval \(I_{y,b}\) has length at most four and is therefore covered
by two intervals of length two.  The one-piece estimate gives
\(\mu_t([b,b+1))\leq Ct^{-3/2}\nu_b\), and hence
\[
\begin{aligned}
\sum_{t>m/2}\int B_{m-t}(J_y-h)\,\mu_t(\dd h)
&\leq Cm^{-3/2}\sum_{b\geq0}\nu_b
        \sum_{r\geq0}B_r(I_{y,b})\\
&\leq Cm^{-3/2}\sum_{b\geq0}\nu_b.
\end{aligned}
\]
The last inequality follows from Feller duality and
\eqref{eq:bprei-critical-fixed-window-potential}.  Since
\(\sum_{b\geq0}\nu_b<\infty\), this proves
\eqref{eq:bprei-critical-modified-bridge-mass}.

We next include the cubic piece-count weight:
\begin{equation}\label{eq:bprei-critical-piece-count}
\sum_{k\geq1} k^3\,
\Pbf_*^{(k)}\left(
        \sum_{i=1}^{k}\tau_i=m,
        \ \sum_{i=1}^{k}H_i\in J_y
        \right)
\leq C(1+y)^4m^{-3/2}.
\end{equation}
Choose \(\delta>0\) with
\(p:=\Pbf(H_1\leq\delta)<1\), put \(\rho:=(1+p)/2\), and set
\begin{equation}\label{eq:bprei-critical-C0}
        C_0:=1+\frac{2}{\delta(1-\rho)}.
\end{equation}
For \(k\leq C_0(1+y)\), the factor \(k^3\) is at most
\(C(1+y)^3\), so \eqref{eq:bprei-critical-modified-bridge-mass} gives the
required contribution.

Suppose now that \(k>C_0(1+y)\).  On the bridge event, at least a fraction
\(\rho\) of the unmodified heights \(H_2,\ldots,H_k\) must not exceed
\(\delta\); otherwise the total height would exceed \(y+2\).  Write
\[
A_k:=\left\{\sum_{i=2}^{k}\one_{\{H_i\leq\delta\}}
             \geq\rho(k-1)\right\}.
\]
The exact time constraint forces at least one duration to be at least
\(m/k\).  If the long piece is \(\ell=1\), the event \(A_k\) is independent
of that modified first piece.  If \(\ell\geq2\), replace \(A_k\) by
\[
A_k^{(\ell)}:=
\left\{\sum_{\substack{2\leq i\leq k\\ i\neq\ell}}
        \one_{\{H_i\leq\delta\}}
        \geq\rho(k-1)-1\right\}.
\]
Then \(A_k\subseteq A_k^{(\ell)}\), and
\(A_k^{(\ell)}\) is independent of the duration and height of piece
\(\ell\).  Moreover,
\[
        \rho(k-1)-1-p(k-2)
        =\frac{1-p}{2}(k-3).
\]
Hoeffding's inequality therefore gives
\(\Pbf(A_k)\leq Ce^{-ck}\) and
\(\Pbf(A_k^{(\ell)})\leq Ce^{-ck}\) for all sufficiently large \(k\);
the finitely many smaller values are absorbed into the constant.

Since the bridge event is contained in the union over possible long pieces,
\[
\Pbf_*^{(k)}(\text{bridge})
\leq\sum_{\ell=1}^{k}
\Pbf_*^{(k)}\left(
\sum_{i=1}^{k}\tau_i=m,\,
\tau_\ell\geq\frac{m}{k},\ A_k^{(\ell)}\right),
\]
where \(A_k^{(1)}:=A_k\).  After the other durations have been fixed, the
long duration is determined and is at least \(m/k\).  Both the modified
first-piece marginal, by
\eqref{eq:bprei-critical-one-piece} and summability of \((\nu_h)\), and the
ordinary duration law, by \eqref{eq:bprei-critical-ladder-duration}, are
bounded by \(Cs^{-3/2}\) at duration \(s\).  Summing over the possible
position of the long piece gives
\[
\Pbf_*^{(k)}\left(
        \sum_{i=1}^{k}\tau_i=m,
        \sum_{i=1}^{k}H_i\in J_y
        \right)
\leq Cm^{-3/2}k^{5/2}e^{-ck}.
\]
Multiplication by \(k^3\) and summation over
\(k>C_0(1+y)\) completes the proof of
\eqref{eq:bprei-critical-piece-count}.

By the definition of \(\Pbf_*^{(k)}\),
\[
\Ebf[\Phi_1^3;E_{k,m,y}]
=
\Pbf_*^{(k)}\left(
 \sum_{i=1}^{k}\tau_i=m,\
 \sum_{i=1}^{k}H_i\in[y,y+1)
\right).
\]
Since \([y,y+1)\subseteq J_y\), the left-hand side of
\eqref{eq:bprei-critical-piece-count} bounds the left-hand side of
\eqref{eq:bprei-critical-weighted-piece-count}.  The reduction above therefore
proves \eqref{eq:bprei-critical-Hocc}.
\end{proof}

\subsection{Uniform post-minimum inverse moments}
\label{subsec:critical-fv-uim}

We next combine the environmental estimate with the branching structure.
Write
\[
\Pbf_{\cE}(\,\cdot\,)
:=\Pbf\left(\,\cdot\,\middle|\sigma(\cE_n:n\geq0)\right),
\]
and define \(\Ebf_{\cE}\) and \(\Var_{\cE}\) analogously.  These are the
quenched probability, expectation, and variance conditional on the full
environmental sequence.

For a generic environmental mark, let
\(\sigma^2(\cE_0):=\Var(\xi_{0,1}\mid\cE_0)\), and assume
\begin{equation}\label{eq:bprei-critical-normalized-variance}
        \frac{\sigma^2(\cE_0)}{M_0^2}\leq\overline C
        \qquad\text{almost surely}.
\end{equation}
For one particle present at generation \(j\), let \(Y_m^{(j)}\) be the
number of its descendants at generation \(m\) when immigration is removed.
A one-step conditional-variance recursion gives
\begin{equation}\label{eq:bprei-critical-clan-variance}
\Var_{\cE}\left(\frac{Y_m^{(j)}}{\Pi_{j,m}}\right)
 =\sum_{k=j}^{m-1}
   \frac{\sigma^2(\cE_k)}{M_k^2}\e^{-(S_k-S_j)}.
\end{equation}
Moreover,
\[
        \Ebf_{\cE}[Y_m^{(j)}]=\Pi_{j,m},
        \qquad
        \Ebf_{\cE}\left[\frac{Y_m^{(j)}}{\Pi_{j,m}}\right]=1.
\]
Put \(\Sigma_j:=\sum_{k=j}^{m-1}\e^{-(S_k-S_j)}\).  By
\eqref{eq:bprei-critical-normalized-variance} and
\eqref{eq:bprei-critical-clan-variance},
\[
\begin{aligned}
\Ebf_{\cE}\left[\left(\frac{Y_m^{(j)}}{\Pi_{j,m}}\right)^2\right]
&=1+\Var_{\cE}\left(\frac{Y_m^{(j)}}{\Pi_{j,m}}\right)\\
&\leq1+\overline C\Sigma_j.
\end{aligned}
\]
Applying the Paley--Zygmund inequality conditionally on \(\cE\), with
parameter \(1/2\), gives
\begin{equation}\label{eq:bprei-critical-clan-PZ}
\begin{aligned}
\Pbf_{\cE}\left(Y_m^{(j)}>\tfrac12\Pi_{j,m}\right)
&=\Pbf_{\cE}\left(
\frac{Y_m^{(j)}}{\Pi_{j,m}}>
\frac12\Ebf_{\cE}\left[\frac{Y_m^{(j)}}{\Pi_{j,m}}\right]
\right)\\
&\geq
\frac{(1-\frac12)^2
\left\{\Ebf_{\cE}[Y_m^{(j)}/\Pi_{j,m}]\right\}^2}
{\Ebf_{\cE}[(Y_m^{(j)}/\Pi_{j,m})^2]}\\
&\geq\frac{1}{4(1+\overline C\Sigma_j)}.
\end{aligned}
\end{equation}

The deterministic selection used below is recorded separately.

\begin{lemma}[Thinned future-minimum epochs]
\label{lem:bprei-critical-future-minima}
Let \(x=(x_0,\ldots,x_m)\) satisfy \(x_0=0\),
\(\min_jx_j\geq0\), and \(x_m=y\geq2\).  Put
\(I_j:=\min_{j\leq r\leq m}x_r\) and
\(K:=\max_{0\leq j<m}(x_{j+1}-x_j)_+\).  There are indices
\(1\leq j_1<\cdots<j_N\leq m-1\) such that each
\(x_{j_a}=I_{j_a}\), their values lie in \([1,y/2]\) and are separated by
at least one,
\begin{equation}\label{eq:bprei-critical-record-count}
        N\geq\frac{y}{2(K+1)}-1,
\end{equation}
and
\begin{equation}\label{eq:bprei-critical-record-occupation}
\sum_{a=1}^{N}\sum_{r=j_a}^{m-1}\e^{-(x_r-x_{j_a})}
\leq\frac{1}{1-\e^{-1}}
      \sum_{r=0}^{m}\e^{-(x_r-I_r)}.
\end{equation}
\end{lemma}

\begin{proof}
If \(y\leq2(K+1)\), take \(N=0\).  Suppose
\(y>2(K+1)\).  The future-infimum path \(j\mapsto I_j\) is nondecreasing,
\(I_0=0\), \(I_m=y\), and
\(I_j=\min\{x_j,I_{j+1}\}\).  Hence
\[
I_{j+1}>I_j
\quad\Longrightarrow\quad
x_j=I_j,
\qquad
0<I_{j+1}-I_j\leq(x_{j+1}-x_j)_+\leq K.
\]
At the first index \(j\) for which \(I_j\geq1\), the increment bound
gives \(I_j<1+K\). The minimum \(I_j\) is attained at some \(v\geq j\)
with \(x_v=I_v=I_j\). The same argument applies when the threshold \(1\)
is replaced by any higher level up to \(y\).
Let \(j_1\) be the first future-minimum epoch with \(x_{j_1}\geq1\), and,
given \(j_a\), let \(j_{a+1}\) be the first later future-minimum epoch with
\(x_{j_{a+1}}\geq x_{j_a}+1\).  The preceding increment bound gives
\[
x_{j_1}\leq1+K,
\qquad
x_{j_{a+1}}\leq x_{j_a}+1+K,
\qquad
x_{j_a}\leq a(K+1).
\]
Retain the indices with \(x_{j_a}\leq y/2\).  Then
\[
N\geq\frac{y}{2(K+1)}-1,
\]
and the retained values lie in \([1,y/2]\) and are one-separated.

For every \(r\), monotonicity of \(I_j\) gives
\(x_{j_a}=I_{j_a}\leq I_r\) whenever \(j_a\leq r\).  Therefore
\[
\begin{aligned}
\sum_{a:j_a\leq r}\e^{-(x_r-x_{j_a})}
&\leq\e^{-(x_r-I_r)}\sum_{\ell\geq0}\e^{-\ell}\\
&=\frac{\e^{-(x_r-I_r)}}{1-\e^{-1}}.
\end{aligned}
\]
Summing over \(0\leq r\leq m-1\) proves
\eqref{eq:bprei-critical-record-occupation}.
\end{proof}

The strict and weak persistence probabilities are linked by an exact
identity.  Let \(u_n:=\Pbf(L_n\geq0)\).  Since \(\tau_n\) is the first time
of the finite-horizon minimum,
\begin{equation}\label{eq:bprei-critical-strict-weak}
        1=\sum_{k=0}^{n}d_k^-u_{n-k}.
\end{equation}
Indeed, on \(\{\tau_n=k\}\) the initial block ends at a strict descending
ladder epoch and the independent remaining block stays weakly nonnegative.
Centered finite variance gives
\(d_k^-\geq c(1+k)^{-1/2}\); since \(u_n\) is nonincreasing,
\eqref{eq:bprei-critical-strict-weak} yields
\begin{equation}\label{eq:bprei-critical-weak-persistence}
        u_n\leq C(1+n)^{-1/2},
        \qquad
        \sum_{j=0}^{m-1}u_j\leq C\sqrt m.
\end{equation}

\begin{proposition}[Uniform inverse moments and localization]
\label{prop:bprei-critical-fv-uim}
Assume \eqref{eq:bprei-critical-fv-increments},
\eqref{eq:bprei-critical-normalized-variance}, and
\(\eta_n\geq1\) almost surely for every \(n\geq1\).  Then, for every
\(s>0\) and every initial law \(\nu\), the following conclusions hold.
\begin{enumerate}[label=\textup{(\roman*)}]
\item If \(\Ebf\e^{\theta(\log M)_+}<\infty\) for some \(\theta>0\), then
\begin{equation}\label{eq:bprei-critical-uim-exp}
\Ebf_\nu[Z_m^{-s};Z_m>0,L_m\geq0]
\leq C_s(1\vee m)^{-3/2}[1+\log(m+1)]^4,
\qquad m\geq0.
\end{equation}
\item If \(\Ebf(\log M)_+^q<\infty\) for some \(q>12\), then
\begin{equation}\label{eq:bprei-critical-uim-poly}
\Ebf_\nu[Z_m^{-s};Z_m>0,L_m\geq0]
\leq C_s(1\vee m)^{-1-\delta_q},
\qquad
\delta_q:=\frac{q-12}{2(q+4)}>0.
\end{equation}
\end{enumerate}
Both majorants are summable and satisfy the convolution-tail localization
condition of Assumption~\ref{ass:bprei-critical-terminal-localization}.
\end{proposition}

\begin{proof}
We first reduce the problem to the population descended from immigrants; this
removes the dependence on the initial law and makes the resulting bounds
uniform in \(\nu\).  We then decompose according to the endpoint events
\(B_{m,y}\).  On \(B_{m,y}\), after excluding paths with a positive
environmental increment larger than a cutoff \(T\),
Lemma~\ref{lem:bprei-critical-future-minima} supplies order
\(y/(T+1)\) separated future-minimum epochs.  Persistent immigration
provides a clan at each such epoch, and the conditional Paley--Zygmund
estimate \eqref{eq:bprei-critical-clan-PZ} gives each tagged clan a positive
quenched probability of contributing a fixed fraction of its environmental
mean at time \(m\).

Since the selected future-minimum heights lie below half of the terminal
height \(S_m\), success of any tagged clan forces \(Z_m\) to be at least
\(\e^{y/2}/2\) on \(B_{m,y}\).  The probability that all tagged clans fail
is controlled by their conditional independence,
the lower bound on the sum of their success probabilities, and the occupation
functional \(\Xi_m\).  The bridge estimates
\eqref{eq:bprei-critical-Hloc} and \eqref{eq:bprei-critical-Hocc} then give
an inverse-moment bound of order \(m^{-3/2}(1+T)^4\) on the event that the
largest positive increment is at most \(T\).  The complementary
large-increment event is controlled directly by the positive tail of
\(\log M\).  Choosing \(T\) according to the available exponential or
polynomial moment yields \eqref{eq:bprei-critical-uim-exp} and
\eqref{eq:bprei-critical-uim-poly}.  The resulting majorants are summable,
and their decay also gives the convolution-tail localization condition.

For \(m=0\), the truncated inverse moment is at most one. Fix \(m\geq1\)
and let \(Z_m^{\mathrm{imm}}\) be the population descended only from immigrants.
Persistent immigration gives
\[
Z_m\geq Z_m^{\mathrm{imm}}\geq1,
\qquad
Z_m^{-s}\leq(Z_m^{\mathrm{imm}})^{-s},
\]
pathwise.  Thus it is enough to bound the inverse moment of
\(Z_m^{\mathrm{imm}}\).  Its law depends only on the environmental and
immigration mechanisms after time zero, so the bound obtained below is
independent of the initial law \(\nu\).

Let \(K_m:=\max_{0\leq j<m}(\zeta_j)_+\), fix \(T\geq1\), and put
\(Y_T:=\lceil4(T+1)^2\rceil\). We first treat
\(B_{m,y}\cap\{K_m\leq T\}\) with \(y\geq Y_T\).
Let \(\mathcal J\) be the epochs selected by
Lemma~\ref{lem:bprei-critical-future-minima} and \(N:=|\mathcal J|\).
Its count bound, with terminal height \(S_m\geq y\), gives
\(N\geq y/[2(T+1)]-1\geq y/[4(T+1)]>0\).
For \(j\in\mathcal J\), tag one immigrant present at generation \(j\), and
let
\[
        E_j:=\left\{Y_m^{(j)}>\tfrac12\Pi_{j,m}\right\}.
\]
The events \((E_j)_{j\in\mathcal J}\) are conditionally independent, and
\eqref{eq:bprei-critical-clan-PZ} gives
\[
        \Pbf_{\cE}(E_j)\geq\frac1{4(1+\overline C\Sigma_j)}.
\]
Hence
\[
\Pbf_{\cE}\left(\bigcap_{j\in\mathcal J}E_j^c\right)
\leq\e^{-\mathcal R},
\qquad
\mathcal R:=\sum_{j\in\mathcal J}
\frac1{4(1+\overline C\Sigma_j)}.
\]
By Cauchy--Schwarz,
\[
\left(\sum_{j\in\mathcal J}1\right)^2
\leq
\left\{\sum_{j\in\mathcal J}(1+\overline C\Sigma_j)\right\}
\left\{\sum_{j\in\mathcal J}\frac1{1+\overline C\Sigma_j}\right\}.
\]
Since \(|\mathcal J|=N\), this gives
\[
\mathcal R
\geq\frac{N^2}{4N+4\overline C\sum_{j\in\mathcal J}\Sigma_j}.
\]
Applying \eqref{eq:bprei-critical-record-occupation} then yields
\[
\mathcal R
\geq\frac{N^2}{4N+4c_1\Xi_m},
\qquad
c_1:=\frac{\overline C}{1-\e^{-1}}.
\]

For every \(j\in\mathcal J\),
Lemma~\ref{lem:bprei-critical-future-minima}, applied to the path with
terminal height \(S_m\), gives
\[
        S_j=I_j^{(m)}\leq\frac{S_m}{2}.
\]
Since \(S_m\geq y\) on \(B_{m,y}\),
\[
        \Pi_{j,m}=\e^{S_m-S_j}
        \geq\e^{S_m/2}\geq\e^{y/2}.
\]
Therefore, if \(E_j\) occurs for at least one \(j\in\mathcal J\), then the
corresponding immigrant clan contributes more than \(\e^{y/2}/2\) particles
at time \(m\), so
\[
        Z_m^{-s}\leq2^s\e^{-sy/2}.
\]
On the event that all \(E_j\) fail, we use only \(Z_m^{-s}\leq1\).
Conditional independence of the tagged clans and the preceding bound on the
joint failure probability therefore give
\begin{equation}\label{eq:bprei-critical-quenched-small}
\Ebf_{\cE}[Z_m^{-s};Z_m>0]
\leq2^s\e^{-sy/2}+\e^{-\mathcal R}
\qquad\text{on }B_{m,y}\cap\{K_m\leq T\},\ y\geq Y_T.
\end{equation}

The selected heights lie in \([1,S_m/2]\), are at least one apart,
and \(S_m<y+1\); hence \(N\leq y\). Using \(\e^{-u}\leq(3/\e)^3u^{-3}\) and
\((a+b)^3\leq4(a^3+b^3)\),
\[
\e^{-\mathcal R}
\leq C\frac{(N+c_1\Xi_m)^3}{N^6}
\leq C(T+1)^6y^{-6}\{y^3+\Xi_m^3\}.
\]
For \(y\geq1\),
\[
\begin{aligned}
y^{-6}\{y^3(1+y)+(1+y)^4\}
&\leq y^{-6}\{2y^4+16y^4\}\\
&\leq18y^{-2}.
\end{aligned}
\]
Thus \eqref{eq:bprei-critical-Hloc} and
\eqref{eq:bprei-critical-Hocc} give
\[
\Ebf[\e^{-\mathcal R};B_{m,y},K_m\leq T]
\leq C(T+1)^6y^{-2}m^{-3/2},
\qquad y\geq Y_T.
\]

We now sum over the disjoint endpoint bins.  For \(y<Y_T\),
\eqref{eq:bprei-critical-Hloc} gives
\[
\sum_{y<Y_T}\Pbf(B_{m,y})
\leq Cm^{-3/2}\sum_{y<Y_T}(1+y)
\leq Cm^{-3/2}(1+T)^4,
\]
because \(Y_T\asymp(T+1)^2\).  The term containing
\(\e^{-sy/2}\) is bounded by \(C_sm^{-3/2}\) using
\eqref{eq:bprei-critical-Hloc}, while
\[
C(T+1)^6m^{-3/2}\sum_{y\geq Y_T}y^{-2}
\leq Cm^{-3/2}(1+T)^4.
\]
Combining these three contributions with the event \(\{K_m>T\}\) yields
\begin{equation}\label{eq:bprei-critical-uim-cutoff}
\Ebf_\nu[Z_m^{-s};Z_m>0,L_m\geq0]
\leq C_sm^{-3/2}(1+T)^4+
\Pbf(L_m\geq0,K_m>T).
\end{equation}

For the large-jump term,
\[
\begin{aligned}
\Pbf(L_m\geq0,K_m>T)
&\leq\sum_{j=0}^{m-1}\Pbf(L_j\geq0,\zeta_j>T)\\
&=\Pbf((\log M)_+>T)\sum_{j=0}^{m-1}u_j\\
&\leq C\sqrt m\,\Pbf((\log M)_+>T),
\end{aligned}
\]
where the equality uses independence of \(\zeta_j\) from the preceding path
and the last inequality uses
\eqref{eq:bprei-critical-weak-persistence}.

Under the exponential-moment assumption,
\[
        \Pbf((\log M)_+>T)\leq C\e^{-\theta T}.
\]
Taking
\[
        T=\max\left\{1,\frac2\theta\log(m+1)\right\}
\]
makes the large-jump contribution \(O(m^{-3/2})\), while the first term in
\eqref{eq:bprei-critical-uim-cutoff} is
\[
        O\left(m^{-3/2}[1+\log(m+1)]^4\right).
\]
This proves \eqref{eq:bprei-critical-uim-exp}.

Under the \(q\)th-moment assumption,
\[
        \Pbf((\log M)_+>T)\leq CT^{-q}.
\]
With \(T=m^\lambda\), the two terms have powers \(-3/2+4\lambda\)
and \(1/2-q\lambda\). Both are below \(-1\) precisely when
\(3/(2q)<\lambda<1/8\), an interval which is nonempty for \(q>12\).
Taking \(\lambda=2/(q+4)\) balances the two terms and gives
\[
\begin{aligned}
m^{-3/2}(1+T)^4+\sqrt m\,T^{-q}
&=O(m^{-1-\delta_q}),\\
\delta_q&=\frac{q-12}{2(q+4)}>0,
\end{aligned}
\]
because
\[
-\frac32+\frac8{q+4}
=\frac12-\frac{2q}{q+4}
=-1-\frac{q-12}{2(q+4)}.
\]
This proves \eqref{eq:bprei-critical-uim-poly}.

It remains to verify that either majorant satisfies
Assumption~\ref{ass:bprei-critical-terminal-localization}.  Since
\(d_n^-\asymp n^{-1/2}\), for either majorant \(\overline a_m(s)\),
\[
\frac1{d_n^-}\sum_{k\leq n/2}d_k^-\overline a_{n-k}(s)
\leq Cn\max_{n/2\leq j\leq n}\overline a_j(s).
\]
The right-hand side is \(O(n^{-1/2}\log^4n)\) in the exponential-moment
case and \(O(n^{-\delta_q})\) in the polynomial-moment case, and hence tends
to zero.
\end{proof}

\begin{remark}[Status of the positive-tail exponent]
The condition \(q>12\) is a sufficient exponent produced by the global
maximum cutoff and the cubic occupation estimate; it is not asserted to be a
threshold.  The positive-tail condition is used only to make the uniform
post-minimum majorant summable.  It is not needed in the entrance-law
construction below.
\end{remark}

\subsection{The endpoint-minimum entrance law}
\label{subsec:critical-fv-entrance}

The entrance law is obtained by looking backward from the terminal
generation.  For every fixed \(U\), endpoint-minimum conditioning produces a
limit for the last \(U+1\) environmental marks.  The population founded by
immigrants in those generations is obtained from the environmental window by
a common Markov kernel, so its law converges in total variation.  It remains
to show that clans founded earlier than \(U\) generations before the endpoint
are negligible, first uniformly in \(n\) and then under the limiting reversed
environment.

Assume now, in addition, that the immigration mean is uniformly bounded:
\begin{equation}\label{eq:bprei-critical-immigration-mean}
        \Ebf[\eta_{n+1}\mid\cE_n]\leq\overline\eta
        \qquad\text{almost surely}.
\end{equation}
Let \(W\) be the walk with increment law \(-\log M\) and, for \(x>0\),
put
\[
        \varphi_n(x):=\Pbf(x+W_j>0,\ 1\leq j\leq n).
\]
The first-attainment convention gives the pathwise identity
\begin{equation}\label{eq:bprei-critical-strict-reversal}
        \{\tau_n=n\}
        =\{S_j>S_n,\ 0\leq j<n\}
        =\{R_u>0,\ 1\leq u\leq n\},
\end{equation}
where \(R_u:=S_{n-u}-S_n\).  In particular, the value \(u=n\) excludes a
tie with \(S_0=0\); no assumption excluding arithmetic ties is needed.

We record the strict-persistence ratio in the normalization used here.  For
fixed \(x>0\), set \(b_0(x):=1\) and, for \(j\geq1\),
\[
 b_j(x):=
 \Pbf\left(W_j=\min_{0\leq\ell\leq j}W_\ell,\ -W_j<x\right).
\]
Decomposing a path that stays above \(-x\) at the last time it attains its
minimum gives the exact convolution
\begin{equation}\label{eq:bprei-critical-strict-ratio-convolution}
        \varphi_n(x)=\sum_{j=0}^{n}b_j(x)d_{n-j}^-.
\end{equation}
This last-minimum decomposition is only an auxiliary proof device; it does not
alter the first-attainment definition of \(\tau_n\).  Reversal of the first
\(j\) increments expresses \(b_j(x)\) as a dual-walk path that stays weakly
nonnegative and ends in \([0,x)\).  Covering that interval by finitely many
unit intervals and applying Lemma~\ref{lem:bprei-critical-fv-kernels} gives
\(b_j(x)\leq C_xj^{-3/2}\).  Thus \(\sum_jb_j(x)<\infty\).  Since
\(d_n^-\sim\kappa_0n^{-1/2}\), split
\eqref{eq:bprei-critical-strict-ratio-convolution} at \(n/2\).  Regular
variation and dominated convergence handle \(j\leq n/2\), while
\[
\frac1{d_n^-}\sum_{j>n/2}b_j(x)d_{n-j}^-
\leq
\frac{C_xn^{-3/2}}{d_n^-}\sum_{\ell\leq n/2}d_\ell^-
\longrightarrow0.
\]
Consequently,
\begin{equation}\label{eq:bprei-critical-BD-ratio}
        \sqrt n\,d_n^-\longrightarrow\kappa_0\in(0,\infty),
        \qquad
        \frac{\varphi_n(x)}{d_n^-}\longrightarrow
        v(x):=\sum_{j=0}^{\infty}b_j(x)\in(0,\infty).
\end{equation}
The first limit is the classical centered finite-variance persistence
asymptotic; see \cite[Chapter~XII]{Feller1971}.  The second is the
strict-persistence form of the Bertoin--Doney ratio theorem
\cite{BertoinDoney1994}, here obtained directly from the local and persistence
bounds already proved in the appendix.  The renewal interpretation uses an
open height interval and the corresponding normalization by strict survival
from zero.  Finally, the persistence bound and the lower bound on \(d_n^-\)
give
\begin{equation}\label{eq:bprei-critical-v-linear}
        v(x)\leq C(1+x),
        \qquad x>0.
\end{equation}

\begin{proposition}[Entrance law at an endpoint minimum]
\label{prop:bprei-critical-fv-entrance}
Assume \eqref{eq:bprei-critical-fv-increments},
\(\eta_n\geq1\) almost surely, and
\eqref{eq:bprei-critical-immigration-mean}.  There is a probability law
\(\nu^-\) on \(\{1,2,\ldots\}\), independent of the fixed initial state
\(i\), such that
\begin{equation}\label{eq:bprei-critical-fv-entrance-TV}
        \|\nu_{i,n}^- -\nu^-\|_{\mathrm{TV}}\longrightarrow0.
\end{equation}
Moreover,
\begin{equation}\label{eq:bprei-critical-endpoint-mean}
        \sup_{n\geq1}\Ebf_i[Z_n\mid\tau_n=n]<\infty.
\end{equation}
\end{proposition}

\begin{proof}
We prove the total-variation convergence by approximating the terminal
population, under the conditioning \(\{\tau_n=n\}\), by the descendants of
immigrants entering during the last \(U+1\) generations.  For each fixed
\(U\), reversal at the endpoint minimum converts this conditioning into
strict positivity of a reversed environmental walk.  The strict-persistence
ratio limit in \eqref{eq:bprei-critical-BD-ratio} then gives a limiting law
for the last \(U+1\) reversed environmental marks.  Since the population
descended from the corresponding recent immigrants is obtained from this
environmental window through the same Markov kernel, its conditional law also
converges in total variation.

It remains to pass from the recent-immigrant population to the full
population.  The local bridge and persistence estimates show that, uniformly
for large \(n\), the expected contribution at time \(n\) of the initial
population and of immigrants entering more than \(U\) generations before the
endpoint is \(O(U^{-1/2}+n^{-1/2})\) under \(\{\tau_n=n\}\).  The
analogous tail under the limiting reversed environment is
\(O(U^{-1/2})\).  Consequently, the full conditioned population is close in
total variation to its recent-immigrant approximation, while the limiting
recent-immigrant populations increase to a finite random variable
\(Z^\infty\).  Letting first \(n\to\infty\) and then \(U\to\infty\)
yields the limiting entrance law.  The same estimates give the uniform
endpoint-mean bound, and the vanishing contribution of the fixed initial
population shows that the limit is independent of \(i\).

Fix \(n\) and reverse the environmental marks:
\(\widetilde\cE_u:=\cE_{n-1-u}\), \(0\leq u<n\).  Put
\(Y_u:=-\log\widetilde M_u\) and
\[
R_u:=S_{n-u}-S_n=\sum_{v=0}^{u-1}Y_v,
\qquad 0\leq u\leq n.
\]
By \eqref{eq:bprei-critical-strict-reversal},
\(\{\tau_n=n\}=\{R_1>0,\ldots,R_n>0\}\).

\medskip
\noindent\emph{Step 1: finite reversed environmental windows.}
Fix \(U\geq0\) and \(n\geq2(U+1)\).  Conditional on the reversed
environmental window
\[
        \widetilde\cE_0,\ldots,\widetilde\cE_U,
\]
the event \(\{\tau_n=n\}\) first requires
\[
        R_1>0,\ldots,R_{U+1}>0.
\]
Given this event and the endpoint \(R_{U+1}\), the remaining
\(n-U-1\) reversed increments are independent of the window, and the
complete reversed path remains strictly positive precisely when a fresh copy
of the reversed walk, started from \(R_{U+1}\), remains positive for those
\(n-U-1\) steps.  The conditional probability of this continuation is
\(\varphi_{n-U-1}(R_{U+1})\).  Since
\(\Pbf(\tau_n=n)=d_n^-\), the conditional law of the window has density
\[
D_n:=g_n(R_{U+1})\one_{\{R_1>0,\ldots,R_{U+1}>0\}},
\qquad
g_n(x):=\frac{\varphi_{n-U-1}(x)}{d_n^-}
\]
with respect to its unconditioned law.  Moreover,
\begin{equation}\label{eq:bprei-critical-entrance-density-factor}
g_n(x)=
\frac{\varphi_{n-U-1}(x)}{d_{n-U-1}^-}
\frac{d_{n-U-1}^-}{d_n^-}.
\end{equation}
The two factors converge to \(v(x)\) and one, respectively.  Also,
\[
0\leq g_n(x)
\leq\frac{C(1+x)(n-U-1)^{-1/2}}{cn^{-1/2}}
\leq C(1+x).
\]
Hence
\[
        0\leq D_n\leq C(1+R_{U+1}^{+}),
\]
and the right-hand side is integrable.  Therefore
\[
D_n\longrightarrow
D:=v(R_{U+1})\one_{\{R_1>0,\ldots,R_{U+1}>0\}}
\quad\text{in }L^1.
\]
Since \(\Ebf D_n=1\), \(\Ebf D=1\).  If \(\mu_{n,U}\) and
\(\mu_U^\uparrow\) denote the corresponding window laws, then
\[
\|\mu_{n,U}-\mu_U^\uparrow\|_{\mathrm{TV}}
=\frac12\Ebf|D_n-D|\longrightarrow0.
\]
The laws \((\mu_U^\uparrow)_U\) are consistent, because restriction commutes
with total-variation limits.  The standard-Borel extension theorem therefore
defines a probability \(\Pbf^\uparrow\) on the infinite reversed
environment.

\medskip
\noindent\emph{Step 2: finite-window populations.}
For \(0\leq u\leq U\), consider the immigrants entering generation
\(n-u\), so that \(u\) is their age at generation \(n\).  Let
\(F_U^{(n)}\) be the total number of their descendants present at generation
\(n\), summed over the reversed ages \(u=0,\ldots,U\).  Thus
\(F_U^{(n)}\) contains only the clans initiated by immigrants entering
during the last \(U+1\) generations and excludes the initial population and
all earlier immigrant clans.  Under the limiting reversed environmental law
\(\Pbf^\uparrow\), define \(F_U^\infty\) by applying the same immigration
and reproduction construction to the first \(U+1\) reversed marks.

Conditional on the environmental window, the law of this collection of
recent clans is described by a Markov kernel \(K_U\) that does not depend on
\(n\).  Hence
\[
\Law(F_U^{(n)}\mid\tau_n=n)=\mu_{n,U}K_U,
\qquad
\Law^\uparrow(F_U^\infty)=\mu_U^\uparrow K_U.
\]
Total variation contracts under application of the same Markov kernel, so
\begin{equation}\label{eq:bprei-critical-finite-window-population}
\begin{aligned}
&\|\Law(F_U^{(n)}\mid\tau_n=n)
  -\Law^\uparrow(F_U^\infty)\|_{\mathrm{TV}}\\
&\qquad\leq\|\mu_{n,U}-\mu_U^\uparrow\|_{\mathrm{TV}}
\longrightarrow0.
\end{aligned}
\end{equation}
On one probability space carrying the infinite reversed environment under
\(\Pbf^\uparrow\), construct all immigration variables and descendant clans.
Then \(F_U^\infty\) is nondecreasing in \(U\); write
\(Z^\infty:=\lim_{U\to\infty}F_U^\infty\).

\medskip
\noindent\emph{Step 3: removal of old clans.}
Let \(\mathcal R_{U,n}:=Z_n-F_U^{(n)}\).  Conditional on the reversed
environment,
\[
\Ebf_i[\mathcal R_{U,n}\mid\widetilde\cE]
\leq i\e^{-R_n}
+\overline\eta\sum_{u=U+1}^{n-1}\e^{-R_u}
\qquad\text{on }\{\tau_n=n\}.
\]
For \(1\leq u<n\), Lemma~\ref{lem:bprei-critical-fv-kernels} and the weak
persistence bound give
\begin{equation}\label{eq:bprei-critical-entrance-two-block}
\begin{aligned}
\Ebf[\e^{-R_u};\tau_n=n]
&\leq\sum_{a\geq0}\e^{-a}k_u^W(0;a)
     \sup_{x\in[a,a+1)}\pi_{n-u}^W(x)\\
&\leq Cu^{-3/2}(n-u)^{-1/2},
\end{aligned}
\qquad
\Ebf[\e^{-R_n};\tau_n=n]\leq Cn^{-3/2}.
\end{equation}
The first term in the remainder estimate accounts for descendants of the
\(i\) initial particles, while the sum accounts for immigrants of reversed
ages \(u>U\).  Using \eqref{eq:bprei-critical-entrance-two-block},
\[
\begin{aligned}
\sum_{u=U+1}^{n-1}\Ebf[\e^{-R_u};\tau_n=n]
&\leq
C\sum_{U<u\leq n/2}u^{-3/2}(n-u)^{-1/2}\\
&\quad+C\sum_{n/2<u<n}u^{-3/2}(n-u)^{-1/2}.
\end{aligned}
\]
On the first range, \((n-u)^{-1/2}\leq Cn^{-1/2}\), and hence the sum is
bounded by
\[
        Cn^{-1/2}\sum_{u>U}u^{-3/2}
        \leq Cn^{-1/2}U^{-1/2}.
\]
On the second range, \(u^{-3/2}\leq Cn^{-3/2}\), so, after putting
\(r=n-u\),
\[
\sum_{n/2<u<n}u^{-3/2}(n-u)^{-1/2}
\leq Cn^{-3/2}\sum_{r\leq n/2}r^{-1/2}
\leq Cn^{-1}.
\]
Dividing by \(\Pbf(\tau_n=n)=d_n^-\geq cn^{-1/2}\) and including the
initial-population term gives
\begin{equation}\label{eq:bprei-critical-entrance-remainder}
\Ebf_i[\mathcal R_{U,n}\mid\tau_n=n]
\leq C(i,\overline\eta)(U^{-1/2}+n^{-1/2}),
\qquad n\geq2(U+1),\ U\geq1.
\end{equation}

By the construction in Step~1, the law of the first \(u\) reversed marks
under \(\Pbf^\uparrow\) has density
\[
        v(R_u)\one_{\{R_1>0,\ldots,R_u>0\}}.
\]
Therefore, using \(v(x)\leq C(1+x)\), dividing \(R_u\) into unit
intervals, and applying Lemma~\ref{lem:bprei-critical-fv-kernels},
\[
\begin{aligned}
\Ebf^\uparrow[\e^{-R_u}]
&=\Ebf[\e^{-R_u}v(R_u);R_1>0,\ldots,R_u>0]\\
&\leq Cu^{-3/2}\sum_{a\geq0}\e^{-a}(1+a)^2
\leq C'u^{-3/2}.
\end{aligned}
\]
Consequently,
\[
\Ebf^\uparrow[Z^\infty-F_U^\infty]
\leq\overline\eta\sum_{u>U}C'u^{-3/2}
\leq C\overline\eta U^{-1/2}.
\]
In particular, \(\Ebf^\uparrow Z^\infty<\infty\) and
\(Z^\infty<\infty\) almost surely.

\medskip
\noindent\emph{Step 4: conclusion.}
Under the natural coupling,
\[
        F_U^{(n)}\leq Z_n,
        \qquad
        F_U^\infty\leq Z^\infty.
\]
For integer-valued random variables \(X\geq Y\) on a common probability
space,
\[
\|\Law(X)-\Law(Y)\|_{\mathrm{TV}}
\leq\Pbf(X\neq Y)
=\Pbf(X-Y>0)
\leq\Ebf[X-Y].
\]
Applying this inequality to the two truncations and then using the triangle
inequality gives
\[
\begin{aligned}
\|\nu_{i,n}^- -\Law^\uparrow(Z^\infty)\|_{\mathrm{TV}}
&\leq\Pbf_i(\mathcal R_{U,n}>0\mid\tau_n=n)\\
&\quad+\|\Law(F_U^{(n)}\mid\tau_n=n)
          -\Law^\uparrow(F_U^\infty)\|_{\mathrm{TV}}\\
&\quad+\Pbf^\uparrow(Z^\infty-F_U^\infty>0).
\end{aligned}
\]
For fixed \(U\), the middle term tends to zero by
\eqref{eq:bprei-critical-finite-window-population}.  The first term is
bounded by
\[
\Ebf_i[\mathcal R_{U,n}\mid\tau_n=n]
\leq C(U^{-1/2}+n^{-1/2})
\]
by \eqref{eq:bprei-critical-entrance-remainder}, while the third is bounded
by
\[
        \Ebf^\uparrow[Z^\infty-F_U^\infty]\leq CU^{-1/2}.
\]
Therefore, first letting \(n\to\infty\) and then \(U\to\infty\) proves
\eqref{eq:bprei-critical-fv-entrance-TV} with
\(\nu^-:=\Law^\uparrow(Z^\infty)\).  Persistent immigration gives
\(Z^\infty\geq1\), so \(\nu^-(\{0\})=0\).

The population \(F_0^{(n)}\) consists of the immigrants entering at
generation \(n\), so its conditional mean is at most \(\overline\eta\).
Repeating the split above with the older-clan sum beginning at \(u=1\) gives
a bound independent of \(n\) for the remaining contribution, and hence
\[
\Ebf_i[Z_n\mid\tau_n=n]
\leq\Ebf_i[F_0^{(n)}\mid\tau_n=n]
   +\Ebf_i[\mathcal R_{0,n}\mid\tau_n=n]
\leq\overline\eta+C(i,\overline\eta).
\]
This proves \eqref{eq:bprei-critical-endpoint-mean}.

Finally, the expected contribution at generation \(n\) of the \(i\) initial
particles is
\[
\Ebf_i[i\e^{-R_n}\mid\tau_n=n]
=\frac{i\Ebf[\e^{-R_n};\tau_n=n]}{d_n^-}
\leq Cin^{-1}\longrightarrow0.
\]
The limiting variable \(Z^\infty\) is constructed solely from the limiting
reversed environment and its immigrant clans.  It therefore does not depend
on \(i\), which proves the asserted independence of the entrance law.
\end{proof}

\begin{remark}
The uniform conditional-mean condition is used only to remove old immigrant
clans.  If the full mark is \((A,B)\), with \(B\) independent of \(A\) and
\(\Ebf B<\infty\), then in the individual-sum construction the
offspring law depends only on \(A\). One may use the coarser environment
generated by the \(A\)-marks: \(\Ebf[\eta\mid A]=m\Ebf B\), and the
same entrance-law proof applies. For this proof, a uniform lower bound
\(\Pbf(\eta\geq1\mid\mathcal E)\geq c>0\) also gives
\(\nu^-(\N)\geq c\) by terminal immigration and total-variation
convergence. Proposition~\ref{prop:bprei-critical-fv-uim} uses persistent
immigration to select a clan at every chosen epoch.
\end{remark}
\section{Standard facts for the auxiliary exact recursion}
\label{app:ad-standard-facts}

This appendix records the standard fixed-argument arguments used in
Subsection~\ref{subsec:ad-exact-dynamics}.  The moving-window and transfer
arguments remain in Section~\ref{sec:ad-controls}; the notation is the same.

For \(y,r>0\), the Mellin integral
\(\int_0^\infty e^{-ty}t^{r-1}\dd t=\Gamma(r)y^{-r}\)
follows by substituting \(v=ty\).
For a nonnegative integer-valued \(Y\) with p.g.f.\ \(F_Y\), Tonelli's
theorem on \(\{Y>0\}\) gives \eqref{eq:ad-laplace-gamma}, since
\(\Ebf[\e^{-tY};Y>0]=F_Y(\e^{-t})-F_Y(0)\).

\begin{proof}[Proof of Proposition~\ref{prop:ad-super-fixed-profile}]
The product formula \eqref{eq:ad-exact-product} gives all three cases.  If
\(q>0\) and \(R(q)<1\), the iterates \(u^{(n)}(s)\) approach \(q\)
geometrically and
\[
        \sum_{j\geq0}
        \left|\frac{R(u^{(j)}(s))}{R(q)}-1\right|<\infty
\]
locally uniformly for \(s<1\).  Hence the normalized product converges and
\(\mathsf F_0(u^{(n)}(s))\to\mathsf F_0(q)\), proving
\eqref{eq:ad-super-positive-q-limit}.

If \(q=0\), then
\(\mathsf F_0(u^{(n)}(s))=\mu_1u^{(n)}(s)+O(u^{(n)}(s)^2)\).  Koenigs'
limit \eqref{eq:ad-koenigs-limit} supplies the factor \(u'(0)^n\), while
the normalized \(R\)-product converges as above; this proves
\eqref{eq:ad-super-zero-q-limit}.

If \(R\equiv1\), the product disappears.  The mean-value theorem applied to
\(\mathsf F_0(u^{(n)}(s))-\mathsf F_0(u^{(n)}(0))\), together with the
locally uniform Koenigs limits at \(q\), gives
\eqref{eq:ad-super-noimm-positive-transform}.
\end{proof}

\begin{proof}[Proof of Lemma~\ref{lem:ad-super-subtracted-positive}]
Continuity follows from locally uniform convergence.  If \(0<r<s<1\), then
monotonicity of \(u\) and \(R\) gives
\[
        \frac{\mathsf Q(s)}{\mathsf Q(r)}
        =\prod_{j=0}^{\infty}
        \frac{R(u^{(j)}(s))}{R(u^{(j)}(r))}>1.
\]
By continuity, \(\mathsf Q(0)\le\mathsf Q(s/2)<\mathsf Q(s)\) for
\(0<s<1\). Thus strict increase extends to \(0\), and compactness gives
\eqref{eq:ad-super-profile-positive}.  In the branch \(q=0\), positivity
follows from \(K_u(s)>0\) and the positive product.  If \(R\equiv1\), it
follows from \(K_{u,q}(s)>K_{u,q}(0)\) for \(s>0\).
\end{proof}

\section*{Acknowledgments}
M.~Gonz\'alez, C.~Minuesa and I.~del Puerto are supported by grant
PID2023-152359NB-I00 funded by MICIU/AEI/10.13039/501100011033 and by the ERDF/EU.

\printbibliography
\end{document}